%% file: main.tex
\documentclass[pdftex,11pt,twoside,a4paper,bibtotoc,halfparskip,openright]{report}

\newcommand{\cu}{\mathcal{U}}
\newcommand{\co}{\mathcal{O}}
\newcommand{\cm}{\mathcal{M}}
\newcommand{\cn}{\mathcal{N}}
\newcommand{\cc}{\mathcal{C}}
\newcommand{\cv}{\mathcal{V}}
\newcommand{\cs}{\mathcal{S}}
\newcommand{\cf}{\mathcal{F}}
\newcommand{\cl}{\mathcal{L}}
\newcommand{\ca}{\mathcal{A}}
\newcommand{\cb}{\mathcal{B}}
\newcommand{\cp}{\mathcal{P}}
\newcommand{\ce}{\mathcal{E}}
\newcommand{\ct}{\mathcal{T}}

\newcommand{\ch}{\mathcal{H}}
\newcommand{\cd}{\mathcal{D}}
\newcommand{\cj}{\mathcal{J}}
\newcommand{\cx}{\mathcal{X}}
\newcommand{\cw}{\mathcal{W}}
\newcommand{\ck}{\mathcal{K}}

\newcommand{\cend}{\mathcal{E}nd}

\newcommand{\fr}{\mathfrak{R}}

\newcommand{\id}{\mathrm{id}}
\newcommand{\ev}{\mathrm{ev}}

\newcommand{\cem}{\mathrm{cem}}
\newcommand{\ctm}{\mathcal{TM}}

\newcommand{\One}{\mathbbm{1}}
\newcommand{\gh}{\hat{\Gamma}}

\newcommand{\sdiff}{\widehat{\mathcal{SD}}}
\newcommand{\scinfty}{\widehat{SC}^\infty}

\newcommand{\pderiv}[2]{\frac{\partial{#1}}{\partial{#2}}}
\newcommand{\opar}{\overline{\partial}}

\newcommand{\realR}{\mathbbm{R}}
\newcommand{\complexC}{\mathbbm{C}}
\newcommand{\intZ}{\mathbbm{Z}}
\newcommand{\fieldK}{\mathbbm{K}}
\newcommand{\naturalN}{\mathbbm{N}}
\newcommand{\expe}{\mathrm{e}}

\newcommand{\olv}{\overline{V}}
\newcommand{\olr}{\overline{\realR}}
\newcommand{\ovlr}[1]{\overline{\realR^{#1}}}
\newcommand{\olc}{\overline{\complexC}}
\newcommand{\olk}{\overline{\fieldK}}
\newcommand{\ovlc}[1]{\overline{\complexC^{#1}}}
\newcommand{\ovlk}[1]{\overline{\fieldK^{#1}}}
\newcommand{\olf}{\overline{f}}
\newcommand{\olcdot}{\overline{\cdot}}
\newcommand{\olpi}{\overline{\Pi}}
\newcommand{\olmu}{\overline{\mu}}

\newcommand{\fg}{\mathfrak{g}}
\newcommand{\fvect}{\mathfrak{vect}}
\newcommand{\fsvect}{\mathfrak{svect}}
\newcommand{\witt}{\mathfrak{witt}}
\newcommand{\fkl}{\mathfrak{k}^L}
\newcommand{\fkm}{\mathfrak{k}^M}
\newcommand{\fns}{\mathfrak{ns}}
\newcommand{\fram}{\mathfrak{r}}

\newcommand{\Hom}{\mathrm{Hom}}
\newcommand{\ihom}{\underline{\mathrm{Hom}}}
\newcommand{\ider}{\underline{\mathrm{der}}}
\newcommand{\Ob}{\mathrm{Ob}}
\newcommand{\Mor}{\mathrm{Mor}}
\newcommand{\Spec}{\mathrm{Spec}\,}
\newcommand{\Aut}{\mathrm{Aut}}

\newcommand{\der}{\mathrm{der}}
\newcommand{\diver}{\mathrm{div}}

\newcommand{\Jac}{\mathrm{Jac}}
\newcommand{\Pic}{\mathrm{Pic}}
\newcommand{\Sym}{\mathrm{Sym}}

\newcommand{\catgr}{\mathsf{Gr}}
\newcommand{\catsets}{\mathsf{Sets}}
\newcommand{\catc}{\mathsf{C}}
\newcommand{\catd}{\mathsf{D}}

\newcommand{\catmod}{\mathsf{Mod}}
\newcommand{\catsmod}{\mathsf{SMod}}
\newcommand{\cattop}{\mathsf{Top}}
\newcommand{\catman}{\mathsf{Man}}
\newcommand{\catbsdom}{\mathsf{BSDom}}
\newcommand{\catbsman}{\mathsf{BSMan}}
\newcommand{\catsman}{\mathsf{SMan}}

\newcommand{\catvbun}{\mathsf{VBun}}
\newcommand{\catsvbun}{\mathsf{SVBun}}
\newcommand{\catsvect}{\mathsf{SVec}}
\newcommand{\catspoint}{\mathsf{SPoint}}
\newcommand{\cattwo}{\mathbf{2}}
\newcommand{\catfinsman}{\mathsf{FinSMan}}
\newcommand{\catsalg}{\mathsf{SAlg}}
\newcommand{\catalg}{\mathsf{Alg}}
\newcommand{\catslie}{\mathsf{SLie}}
\newcommand{\catlie}{\mathsf{Lie}}

\usepackage{bbm}
\usepackage[intlimits,sumlimits]{amsmath}
\usepackage{amsthm}
\usepackage{amscd}
\usepackage{amssymb}
\usepackage{eufrak}
\usepackage[all]{xy}
\usepackage{fancyhdr}

\newtheorem{thm}{Theorem}[section]
\newtheorem{prop}[thm]{Proposition}
\newtheorem{lemma}[thm]{Lemma}
\newtheorem{cor}[thm]{Corollary}
\newtheorem{dfn}[thm]{Definition}

\newtheorem{conj}[thm]{Claim}

\makeatletter
\def\cleardoublepage{\clearpage\if@twoside \ifodd\c@page\else%
    \hbox{}%
    \thispagestyle{empty}
    \newpage%
    \if@twocolumn\hbox{}\newpage\fi\fi\fi}
\makeatother

\usepackage[pdftex]{hyperref}
\hypersetup{pdfauthor={Christoph Sachse},
	    bookmarks=true,
	    pdffitwindow=false}

\begin{document}

\input{titel_alternativ.tex}
\newpage
\thispagestyle{empty}
\mbox{}
\newpage

\tableofcontents
\newpage
\thispagestyle{empty}
\mbox{}
\newpage

\input{intro.tex}
\newpage
\thispagestyle{empty}
\mbox{}
\newpage
\input{basics.tex}

\input{categories.tex}
\input{sconf.tex}
\input{almostcpx.tex}
\input{complex.tex}
\input{sdiff.tex}

\input{quotient.tex}

\addcontentsline{toc}{chapter}{Bibliography}
\bibliography{thesisbib}
\bibliographystyle{alpha}
\newpage
\thispagestyle{empty}
\mbox{}

\end{document}

%% file: titel_alternativ.tex
\thispagestyle{empty}
\begin{center}

%
\LARGE\sc Global Analytic Approach to\\ Super Teichm\"uller Spaces
\\
\hfill\\
\normalsize Der Fakult\"at f\"ur Mathematik und Informatik\\
der Universit\"at Leipzig\\
eingereichte\\
\hfill\\
\Huge Dissertation\\
\hfill\\
\normalsize zur Erlangung des akademischen Grades\\
\hfill\\
\LARGE Doctor rerum naturalium\\
\large (Dr. rer. nat.)\\
\hfill\\
\normalsize im Fachgebiet\\
\hfill\\
\LARGE Mathematik\\
\hfill\\
\normalsize vorgelegt von\\
\hfill\\
\Large Christoph Sachse\\
\hfill\\
\normalsize geboren am \Large 15. M\"arz 1977 \normalsize in \Large
Fulda\\
\hfill\\

\normalsize\upshape

\vspace{0.6cm}
\begin{center}
Die Annahme der Dissertation haben empfohlen:
\vspace*{0.3cm}

\hspace*{1cm}\begin{minipage}[c]{12cm}
\begin{enumerate}
\item Prof. Dr. J\"urgen Jost (MPIMiS, Leipzig) \vspace*{-0.3cm}
\item Prof. Dr. Dimitri Leites (Stockholm University, Schweden) \vspace*{-0.3cm}
\item Prof. Dr. Dan Freed (University of Texas at Austin, USA) \vspace*{-0.3cm}
\end{enumerate}
\end{minipage}
\end{center}
\vspace*{0.4cm}
\begin{minipage}[c]{15cm}
Die Verleihung des akademischen Grades erfolgt auf Beschluss des Rates der Fakult\"at f\"ur \\Mathematik und 
Informatik vom 22.10.2007 mit dem Gesamtpr\"adikat summa cum laude.
\end{minipage}

\end{center}


%% file: intro.tex
\chapter{Introduction}

\begin{flushright}
\textit{Falls ein Wechsel\\
des Universums erforderlich ist,\\
werden wir das anzeigen.}\\
\textsc{H. Schubert, }\textnormal{Kategorien I}\\
\end{flushright}

In this thesis, we adopt and investigate a new formalism for supergeometry which focuses on the
categorical and homological properties of the theory,
and apply it to the construction of Teichm\"uller spaces of certain superconformal structures.
Supergeometry in its mathematically rigorous form is usually formulated as a theory of locally ringed 
topological
spaces, where all rings involved are supercommutative. Once one has established the definition of a 
supermanifold as a ringed space (which, as algebraic geometry tells us, is basically inevitable), one
has the full arsenal of tools from algebraic geometry at hand. The availability of many of these tools is
not just a nice side-effect, many of them are actually indispensible to making the whole theory
viable. One such tool is the functor of points,
which is a well-known method of homological algebra allowing one to define an object of a category via
the morphisms of all other objects of this category to it. The functor of points is, e.g., needed to
define what a Lie supergroup is. Its usefulness had been realized already at a very early stage, at 
least among
those supergeometers who used the ringed space approach, such as Leites \cite{Itttos}, \cite{SoS}, 
Bernstein 
\cite{NoS_QFaSACfM}, Manin \cite{Gftacg}, \cite{M:Topics} and others. Molotkov \cite{I-dZ2ks} was the
first to publish a full-blown reformulation of supergeometry in categorical terms, the main purpose
of which was to enable the construction of infinite-dimensional superspaces. Unfortunately, his amazing
preprint \cite{I-dZ2ks} contains no proofs, and its dense and abstract style made it difficult to approach
at least for many of the more physically oriented researchers in the field. Therefore, no applications
of his approach have been published so far, and infinite-dimensional supergeometry was very rarely used,
although physics offers plenty of opportunities for applications.

\section{Supersymmetry and supergeometry}

Superalgebra and supergeometry were invented in the early seventies as tools for the path integral formalism
of quantum field theories involving fermion fields. The pioneer of this subject was undoubtedly Berezin,
who, being concerned with the method of second quantization, already realized in the sixties that one
could use algebras with anticommuting generators to unify the description of boson and fermion fields.
Algebras containing anticommuting along with commuting generators were, of course, known long ago. The
archetype of such an algebra is the Grassmann algebra $\Lambda_n$, which is generated by $n$ ``odd'',
i.e., anticommuting elements $\theta_1,\ldots,\theta_n$. It is free except for the relations
\[
\theta_i\theta_j=-\theta_j\theta_i
\]
which imply $\theta_i^2=0$. This algebra can also be viewed simply as the exterior algebra of a
vector space $V$ of dimension $n$. So in the realm of linear and commutative algebra, the step from
ordinary algebras to superalgebras is relatively easy: a superalgebra is defined as a 
$\intZ/2\intZ$-graded
algebra whose operations and morphisms preserve the grading. But it was also Berezin who first 
realized that it
could be possible to extend the use of anticommuting quantities to analysis and geometry, which is
far less obvious. The search for a concept of ``supermanifold'' indeed took several years, and was settled
in 1972 by Leites \cite{Itttos}, \cite{BL:Supermanifolds}. The emergence of this new concept
went disregarded for about two years, until Wess and Zumino \cite{WZ:Supergauge} presented their
famous first supersymmetric field theories, which
triggered an avalanche of research on supergeometry and its applications in physics. The majority
of applications for supersymmetry still originates from high engery physics. Supersymmetric
field theories have become standard tools during the past decades. The most famous applications are
probably superstring theory and supergravity. But there are also a variety of applications in 
solid-state physics, see, e.g., the book by Efetov \cite{E:Supersymmetry}.

Although supergeometry as it was invented in the seventies is formulated in a 
way which mimicks ordinary differential geometry as closely as
possible, it can only really be understood with the help of algebraic geometry. An indispensible tool
is the language of ringed spaces. Unlike an ordinary manifold, a supermanifold is not defined as a space
which is locally homeomorphic to some linear space, although it is still locally isomorphic to a certain
model space. The local rings of functions are not commutative algebras anymore, but rather 
supercommuative ones,
which implies, in particular, that they are not reduced. This entails that a 
supermanifold is not
fully described by its topological points, which one can recover as the set of maximal ideals of the
structure sheaf. The notion of a ``point'' itself has to be handled with great care in this context.
Nonetheless, concepts like coordinates, functions, vector fields, differential forms and integration
can all be carried over to supergeometry, but their intuitive interpretation becomes more subtle. 
In particular the existence of some notion of coordinate makes
it possible to formally wrap up much of the theory in the same language as ordinary differential geometry.

This is an advantage for many applications and calculations, but the fact that the odd ``dimensions'' of
supergeometry are really only algebraic, and not topological, creates some pitfalls.
To hide the essentially algebraic nature of the odd variables behind a seemingly analytic formalism can
sometimes obstruct a clear view on the problems.
These subtleties, together with the generally higher level of abstraction enforced by the ringed space
language, led to attempts to formulate supergeometry in different, less algebraic terms. An example is
the de Witt approach to supermanifolds \cite{D:Supermanifolds} which tries to avoid the use of ringed
spaces by formally reformulating ordinary geometry over the ``supernumbers'', which are assumed to be 
elements of a 
Grassmann algebra. In order to escape the ``finite size effects'' one can create by an unfortunate
choice of this algebra, one often assumes it to be infinite. This, however, produces different types of
problems, related to, e.g., convergence questions. As was already argued many times by the more
algebraically oriented supergeometers, one must
require \emph{functoriality} with respect to an exchange of the algebra over which one works. If
one adheres to this point of view, the de Witt approach becomes equivalent to the ringed space approach,
and the de Witt topology on linear superspaces becomes the Grothendieck topology on the category
of superdomains as proposed by Molotkov \cite{I-dZ2ks}.

In much of the physics literature, the rigorous ringed space approach has still not become standard, 
which has several reasons. One main reason is certainly the considerable technical difficulty that inevitably
accompanies this approach. On the other hand, the supergeometric problems appearing in physics often do not
require this whole massive machinery. Many problems can be treated locally, and then it is enough to
manipulate a set of even and odd coordinates tailor-made for the situation at hand. Besides, odd quantities
in physics are almost exclusively associated with fermion fields, i.e., they form spin representations
of the underlying spacetime symmetry group (usually the Poincar\'e group). This is a very special type of
supergeometry. In other words, supersymmetry and supergeometry are two rather different things.
Supergeometry is a genuine extension of commutative geometry, achieved by using supercommutative rings
and algebras instead of commutative ones. Supersymmetry is a certain feature of quantum field theories
where superalgebra --- or, in the case of supergravity,
supergeometry --- is only used to express this feature. Most applications of supergeometry, 
however, still originate in physics and are therefore often
encoded in the language of quantum field theory, which to non-specialists is often cryptic and hard
to decipher. It is maybe for this reason that, despite spectacular successes like \cite{W:Supersymmetry} and
contributions by renowned mathematicians like Manin \cite{Gftacg}, \cite{M:Topics} or Deligne \cite{D:Letter} 
and Bernstein \cite{NoS_QFaSACfM}, the mathematical community concerned with supergeometry has
remained small. We hope that this work might also help in diminishing the prejudice that supergeometry is
just a bizarre game with spinors invented by physicists, and show that it is a real extension of
ordinary commutative geometry which should indeed be of great interest to mathematicians, in particular
differential and algebraic geometers.

\section{The categorical approach}

In this work we will strictly stick to the ringed space formalism. Besides this, we will work out
the categorical approach to supermathematics. This is not a substitute for the explicit description in
terms of sheaves, but rather an extension which allows us to study questions which would not be directly
accessible in the standard formalism. Its investigation and the demonstration of its use for concrete
problems is one of the main results of this thesis. Most of the methods of the categorical approach 
are actually
well known in algebraic geometry and are in use, e.g., in the theory of schemes and non-reduced varieties
\cite{FAG_GFGAE}, \cite{G:Fondements}. As remarked above, this is an approach based not on the 
investigation of single objects
but rather on their interplay with other objects of the category. It requires
considerable technical machinery to establish. Nonetheless, we think that the gains in both
conceptual insight and also completely new technical tools, e.g., infinite-dimensional supermanifolds,
more than justify the work.

On the conceptual side, the categorical approach allows for making many of the cumbersome peculiarities
of the ringed space formalism more transparent. For example, it allows us to enlighten the somewhat awkward
role of ``functions'' on a supermanifold. They are, in analogy to classical geometry, sections of the
structure sheaf, but have no interpretation as maps to the ground field.
One often treats these sections
as being ``Grassmann-valued'', i.e., to be actual maps into some Grassmann algebra $\Lambda_n$.
But then, to avoid side-effects from the choice of the algebra, one has to adopt mysterious 
prescriptions like the one that the algebra must be ``big enough''
(usually meaning that it must contain at least one more generator than the number of odd quantities 
appearing in calculations). The categorical approach makes such unnatural conditions superfluous by
simply requiring functoriality under an exchange of the algbra over which one works. The price to
pay is that one then has to think of functions as functors taking values in \emph{any} Grassmann
algebra. That is, to think of them as being valued in some particular algebra is legitimate as long as one
bears in mind that this means that one only handles one of their infinitely many components this way.
For many practical problems this may, of course, be good enough, but it immediately becomes dangerous 
if one wants to define maps or morphisms by the values they take.

In the same
vein, the functor of points clarifies the origin and necessity of the so-called odd parameters 
appearing when one works
with supermanifolds. These parameters have been used in physics since the very first days of
supersymmetry, but unfortunately most often without wondering where they actually come from. They usually
take the form of ``odd functions depending only on even variables''. As an example, the transition
function between a smooth superdomain with coordinates $(x,\theta)$ and a domain with coordinates $(y,\eta)$ 
would usually be written as
\begin{eqnarray*}
y &=& f(x)+\psi(x)\theta\\
\eta &=& \xi(x)+g(x)\theta,
\end{eqnarray*}
where $f,g$ are ordinary smooth functions, but $\psi,\xi$ are odd, i.e., anticommuting with
other odd objects and nilpotent. But how can this happen for a function depending only on even quantities?
And besides, how can they fit into the picture that this is the transition function of a superringed space,
i.e., stalkwise a homomorphism of superalgebras?
Again, one usually argues with Grassmann-valuedness here. And again, this is only almost true. To cleanly
interpret the odd parameters, one has to think of a family of supermanifolds over a base 
supermanifold, rather than a single one.
The odd coordinates of the base then become odd constants of the fibers of the family, providing the
mysterious parameters. For brevity or other reasons, one usually does not denote them explicitly, but
rather sticks to the rule that in all morphisms of superdomains, like the above, as many ``odd functions''
have to be introduced as necessary in order to not lose generality. If the base of the family is a
superpoint, i.e. a supermanifold of the form $(\{*\},\Lambda_n)$ where $\{*\}$ is the one-point space
and $\Lambda_n$ its sheaf of functions,
then the structure sheaf of the family is just $\Lambda_n\otimes\co_\cm$, where $\cm$ is the fiber with sheaf
$\co_\cm$.
This sheaf can indeed be interpreted as maps from $\co_\cm$ to $\Lambda_n$, as one
often does. But again, it is functoriality under exchange of the base which is the salient feature of
this direct interpretation.

These two examples should make it clear that the categorical approach is not just a clever reformulation
of the same old stuff, but really allows for a deeper and, in particular, cleaner understanding of
supergeometry. Conceptual insight, however, is not the only advantage of this formalism. It also
provides a whole variety of new technical tools and possibilities. Our main motivation to study it
actually arose from the attempt to extend supergeometry to spaces of infinte dimensions. The ringed
space approach is not directly applicable here. This is not a problem of supergeometry, but a 
general one, occurring,
for example, also in complex geometry \cite{D:probleme}. The solution in the case of complex geometry is the
use of so-called functored spaces, which are essentially an application of the functor of points to this
particular problem. The defining property of an infinite-dimensional manifold is still that of being locally
isomorphic to a linear space. But while in finite dimensions, there is --- up to isomorphy --- only one complex
linear space to which the manifold can be locally identified, this is not true in the infinite-dimensional,
e.g., Banach, context.
One can still choose explicit coordinate maps to define the manifold structure. But to endow the underlying
topological space with a sheaf of maps into $\complexC$ is not sufficient anymore. If one wants to
stay in the ringed space picture, one now has to endow it with a functor
from the category of complex Banach domains to the category of sheaves of holomorphic maps over the
topological space in question.

To translate this to supergeometry, one first has to define linear supermanifolds of infinite dimension.
This is where the functor of points comes in. The crucial difference from the ordinary complex case is that 
a supermanifold is not defined by its topological points, so bijections with open sets in super vector
spaces will not solve the problem. If one switches to the functors of points, however, this intuition
works again: two objects are isomorphic if and only if all of their point sets are isomorphic. We therefore
define a supermanifold (of possibly infinite dimension) as a certain type of functor into the category
of sets which is locally isomorphic to the functor of points of a linear supermanifold. Of course,
one then has to specify what ``locally'' means for a functor. This is accomplished, following \cite{I-dZ2ks},
by introducing a Grothendieck pretopology on the appropriate functor category.

Another achievement of the categorical approach is that it allows a complete description of the
diffeomorphism supergroup of a supermanifold. Already in the finite-dimensional cases supergroups are
defined as group objects in the category of supermanifolds, i.e., in categorical terms. Formally
any supergroup is therefore defined by a property of its functor of points, but in most cases one
can construct an explicit ringed space and afterwards show that it represents this functor. In
the infinite-dimensional case this is not possible anymore --- here the functor of points must 
indeed be put to use.

\section{Super Riemann surfaces}

Our initial motivation for studying infinite-dimensional supergeometry was the Teichm\"uller theory of
so-called $N=1$ super Riemann surfaces. They are a particular species of complex $1|1$-dimensional
supermanifolds appearing as world-sheets in $N=1$ superstring theory. In fact, there are several
types of $1|1$-dimensional supermanifolds which could legitimately be called analogues of Riemann surfaces,
since in the super context, there are several types of superconformal structures. The one appearing
in 2D supergravity and superstring theory is related to the superconformal algebra $\fkl(1|1)$,
a superversion of the agebra of contact vector fields. 

Moduli problems can be formulated in
supergeometry in precisely the same way as in ordinary geometry, namely as the search for universal
parameter spaces for families of the objects in question. These moduli spaces can then, of course, also
be supermanifolds, i.e., a superobject may have even as well as odd deformations. Since
the underlying surface of a complex supersurface is a Riemann surface, one can then also study
Teichm\"uller spaces by fixing a marking on the underlying surface.
It was known from works of Vaintrob \cite{V:Deformations}, LeBrun and Rothstein \cite{LBR:Moduli} and
Crane and Rabin \cite{CR:Super} that
$\fkl(1|1)$-surfaces indeed possess a Teichm\"uller space which has
dimension $3g-3|2g-2$. It was, however, shown that one either has to construct this space as a
quotient of a supermanifold by a certain $\intZ_2$-action, i.e., a ``superorbifold'' \cite{LBR:Moduli}, 
or that one
has to settle with a supermanifold that only paratrizes a semiuniversal family. These results were obtained 
mainly with the help of Kodaira-Spencer deformation theory, adapted
to the context of supergeometry. A problem with this approach is that it gives only local information about
the moduli spaces, since deformation theory only constructs infinitesimal neighbourhoods of a given structure.

The idea of this thesis was to attempt to carry a global approach to the construction of
the Teichm\"uller space of Riemann surfaces invented by Fischer and Tromba \cite{TTiRG} 
over to supergeometry. Their
approach is completely based on global analysis and Riemannian geometry. They first construct the
space $\cc_g$ of all complex structures on a smooth closed orientable surface of genus $g$ by
exploiting the fact that all almost complex structures are integrable in two dimensions. The goal is
then to locally divide out the pullback action of the diffeomorphism group on $\cc_g$ and obtain 
Teichm\"uller space as a manifold glued from local slices. In order to achieve this, a couple of
specialties of surfaces have to be used, in particular the fact that one can identify $\cc_g$ with
the space of conformal classes of metrics, and that this space can be given a natural manifold structure
again. But even with these tools at hand, the process of taking the slice is complicated by the fact
that one works with spaces of smooth objects, which do not form Banach manifolds. Hence, one cannot
use the implicit function theorem directly. The trick that circumvents this problem is to use
Sobolev spaces throughout the whole work, and only in the end to show that everything remains true
if one restricts oneself to smooth objects only.

It is clear that if one wants to carry some version of this procedure over to supergeometry, one will
have to use infinite-dimensional supermanifolds. It is not enough to know how to handle infinite-dimensional
super vector spaces, since they are \emph{not} supermanifolds. Even with these tools at hand, 
Fischer and Tromba's
approach cannot simply be repeated. First of all, supercomplex structures and superconformal structures
do not necessarily coincide, and for $\fkl(1|1)$-surfaces, they actually do not. Nonetheless, every 
super Riemann
surface possesses a supercomplex structure, and it turns out that any deformation of its $\fkl(1|1)$-structure
also deforms the complex structure. Therefore, these two moduli problems cannot be separated, and one has
to be treated as a subproblem of the other. Besides, it turns out that not all almost complex structures
in $2|2$ real dimensions are integrable \cite{P:analogs}. So a suitble way to restrict the 
construction to the integrable ones is required. Another severe problem
is the use of objects of finite differentiability class in Tromba's approach. In supergeometry, only
smooth objects can usually be defined. It will be argued below that this restriction can, to a certain
extent, be circumvented if one tailors special categories of superobjects. We could, however, not find
a practicable way to employ these ``$C^k$-supermanifolds'' to the our problem. Luckily, it turns out
that this is unnecessary because one can more or less separate the underlying moduli problem from its ``odd
parts'', and the underlying one can be identified as a classical one. The topological and analytical
subtleties pertain only to this underlying problem. But still, a full understanding of the
diffeomorphism supergroup has to be achieved in order to be able to take a slice for its action.
The investigation of the diffeomorphism supergroup and its pullback action can indeed be seen as one of
the main results of this work.

All this should make it clear that plenty of problems must be overcome to attempt a 
supergeometric version of Fischer and
Tromba's approach. The resulting procedure only loosely follows their tracks, for the reasons outlined
above. In particular, it requires many more algebraic considerations. 
Nonetheless, we finally arrive
at local quotients for the action of the diffeomorphism supergroup which divide out as much of its action
as is possible if one wants the quotient to be a supermanifold. The main result of this thesis
is maybe not the super Teichm\"uller spaces themselves, which had been obtained by less technical means
before. It is rather the tools and methods developed to pursue the global approach which should
prove useful for many similar problems. Applications for such methods can, for example, be found in
the study of classical configuration spaces of supersymmetric field theories, or for super Hilbert spaces
and operators on them. In particular, the categorical approach can be expected to be a powerful tool if
one tries to study supergeometric analogues of classical constructions. Examples are conformal and
topological field theories, vertex operator superalgebras or automorphic forms, to name just a few.

\section{Organization}

Roughly a third of the work we present here is dedicated to
a realization of the categorical programme set up by Molotkov in \cite{I-dZ2ks}, 
in particular to developing the part
of his theory needed for our later applications together with complete proofs and
examples. In Chapter \ref{ch:sgeo}, we give a brief overview of superalgebra and the ringed 
space formulation of supergeometry. Although this is rather standard material by now, there are only
a small number of concise expositions, e.g., \cite{SoS}, \cite{NoS_QFaSACfM}, \cite{SfMAI}. 

In Chapter \ref{ch:categories}, the categorical approach is laid out in detail. The first sections contain
a recapitulation of the necessary tools from homological algebra, like representable functors, inner
Hom-objects and algebraic structures in categories. 

In Section \ref{sect:functsman}, we expound
some of the consequences of the categorical approach, in particular the necessity of the use of
families instead of single objects. 

In Section \ref{sect:catsalg}, the reformulation of linear and
commutative superalgebra in terms of the functor of points is given. 

In Section \ref{sect:sdom} we
return to geometry by defining superdomains of possibly infinite dimension. To be able to do so,
we have to introduce a Grothendieck pretopology on a certain functor category which contains the
functors of points for linear supermanifolds. This description allows the construction of Banach,
Fr\'echet or even just locally convex superdomains. For the rest of the construction, we restrict ourselves to
the Banach case in order not to overload this chapter with formalities. With some minor adaptions,
everything goes through equally well for the Fr\'echet case.

In Section \ref{sect:bansman} we then define Banach supermanifolds as certain functors which are locally
isomorphic to Banach superdomains, along with supersmooth morphisms between them. The last section of
Chapter \ref{ch:categories} contains the definition of super vector bundles within the categorical framework.

Chapter \ref{ch:sconf} is the first one that is specialized to the application to superconformal
geometry. We give a brief review over the classification of superconformal algebras which was obtained 
and expanded by
many authors over the past decades. To only mention a small selection, contributions can be found in
\cite{DswU1cs}, \cite{Ssacc}, \cite{Fdmop}, \cite{Dtfff}, \cite{KvdL:On}, \cite{SS:Comments} and
\cite{LSoST}. 

We then identify
the type of superconformal surface referred to as $N=1$ super Riemann surface in the physics
literature as the type associated with the algebra $\fkl(1|1)$. A subsequent short analysis of
the structure of general complex $1|1$-dimensional supermanifolds and $\fkl(1|1)$-manifolds
shows that the latter are subspecies of the former. While complex supermanifolds of
dimension $1|1$ are equivalent to a Riemann surface together with a holomorphic line bundle on it,
$\fkl(1|1)$-surfaces are precisely those for which this line bundle is a spin bundle. This 
implies in particular that the moduli problem of super
Riemann surfaces cannot be separated from that of supercomplex structures.

Chapters \ref{ch:acs} and \ref{ch:cpx} are concerned with almost complex structures and their integrability
on supermanifolds. They treat the general case and are not specialized to supersurfaces. 

Chapter
\ref{ch:acs} contains two of our main results. The first is Thm.~\ref{thm:sectsvbun}, which asserts that
the functor of smooth sections of a super vector bundle (which is just the functor of points for the
supermanifold of its sections) is represented by an infinite-dimensional super vector space. This
intuitively plausible statement turns out to be rather hard to prove.
The second result
is the construction of a complex supermanifold $\ca(\cm)$ of all almost complex structures on 
given supermanifold $\cm$. Here, the power of the categorical framework becomes fully visible for the first
time. 
We construct $\ca(\cm)$ as a submanifold of the supermanifold of sections of the endomorphism bundle of
the tangent bundle $\ctm$ by adapting the method invented by Abresch and Fischer for ordinary almost
complex manifolds \cite{TTiRG}. 

Chapter \ref{ch:cpx} is rather short, since we can base our analysis of 
integrability on the results of Vaintrob \cite{V:Deformations}, \cite{Acsos}. The main result is the
determination of those deformations of an integrable almost complex structure which preserve integrability.
Here, we follow a method which was proposed in a similar form by \cite{NG:geometry}.

Chapter \ref{ch:sdiff} contains another of the main results of this thesis, namely a complete description
of the structure of the diffeomorphism supergroup $\sdiff(\cm)$ of a supermanifold. We first identify
$\sdiff(\cm)$ as the subfunctor of the inner Hom-object $\Hom_{\catsman}(\cm,\cm)$ consisting of
invertible morphisms and show that it is in fact a restriction of
$\Hom_{\catsman}(\cm,\cm)$ to the group $\Aut(\cm)$ of automorphisms of $\cm$. 

In Section \ref{sect:sdiff},
we begin by analysing the ``higher'' functor points of $\sdiff(\cm)$, i.e., those which are not simply
morphisms $\cm\to\cm$ but contain odd parameters. It turns out that they have a very simple and neat
description as exponentials of vector fields. The topological subtleties usually pertaining to
the exponential in the case when the Lie group is not Banach only occur for the underlying group $\Aut(\cm)$.
This allows us to split the superdiffeomorphism group $\sdiff(\cm)$ into a semidirect product of $\Aut(\cm)$
and a nilpotent group. 

In Section \ref{sect:autm} we complete the description of $\sdiff(\cm)$ by an
analysis of $\Aut(\cm)$. The structure of this group is more intricate, but again it turns out that
one can split it into several semidirect factors by comparing the supermanifold $\cm$ to the exterior
bundle $(M,\wedge^\bullet E)$ to 
which it is isomorphic by Batchelor's theorem \cite{TSOS}. The deviation of $\cm$ from the exterior bundle form
can be expressed by an element of a nilpotent group $N_\cm$ which
is generated by the even nilpotent vector fields on $\cm$. The remaining normal subgroup of $\Aut(\cm)$
can then be expressed as the isomorphisms of $(M,\wedge^\bullet E)$, which are those of a vector bundle $E$ over
the base manifold $M$. This group is again a semidirect product, namely of $\mathrm{Diff}(M)$, the
diffeomorphisms of the underlying manifold, and $\Aut_M(E)$, the smooth automorphisms of $E$ over $M$.

Only with this detailed description of $\sdiff(\cm)$ at hand does it become possible to try to find a
slice for the pullback action of $\sdiff(\cm)$ on the integrable almost complex structures on $\cm$,
which would be a patch of super moduli space. As in the classical case, to divide out all of 
$\sdiff(\cm)$ at once is too ambitious,
since the underlying Riemann surface of $\cm$ may have additional nontrivial
automorphisms. We confine ourselves to the identity component $\sdiff_0(\cm)$ to obtain the super
Teichm\"uller space. 

In Section \ref{sect:sconfaut} we determine the residual automorphisms contained in $\sdiff_0(\cm)$
and find out that there remains a nontrivial group of them. This will make it impossible to
construct the Teichm\"uller space of $\fvect^L(1|1)$-surfaces as the base of a universal family.
The supermanifold $\ct_{\fvect^L(1|1)}^{g,d}$ that we construct instead still parametrizes all 
deformations of a given compact complex
$1|1$-dimensional supermanifold. We did, however, not succeed in showing that it is the base of a
semiuniversal family. We leave this problem for further investigations. The base manifold of
$\ct_{\fvect^L(1|1)}^{g,d}$ is
the family $J(V_g)$ of Jacobian varieties over Teichm\"uller space constructed by 
Earle \cite{E:Families}.
This space can be seen as the Teichm\"uller space of pairs of Riemann surfaces together with holomorphic
line bundles of a given fixed degree on them. 

Finally we restrict this construction to $\fkl(1|1)$-structures and we find that also in this case there
exists no universal family (at least not in the category of supermanifolds). The reason is that 
there remains a residual group of automorphisms of the 
$\fkl(1|1)$-surface which is isomorphic to $\intZ_2$. This is in accord with previous results 
\cite{LBR:Moduli}, \cite{CR:Super}.

\section{Main results}

The main results of this thesis can be summed up as follows.
\begin{enumerate}
\item We realize and extend the categorical programme laid out by Molotkov in \cite{I-dZ2ks}. In particular,
we present proofs to most of the statements claimed in \cite{I-dZ2ks} and construct as an 
example for the viability
of the categorical approach a manifold of all almost complex structures on a given almost 
complex supermanifold.
\item A detailed analysis of the structure of the diffeomorphism supergroup $\sdiff(\cm)$ of a 
smooth finite-dimensional 
supermanifold $\cm$ is presented. This includes an investigation of the group $\Aut(\cm)$ of automorphisms
of $\cm$, as well as an analysis of the functor of points of $\sdiff(\cm)$. The latter is
completely described in terms of a set of generators for the category of supermanifolds.
\item We present an anlysis of the deformations of two types of superconformal surfaces and
investigate the viability of the global approach to Teichm\"uller theory of Tromba \cite{TTiRG} 
in the supergeometric context. We demonstrate that a universal Teichm\"uller family does not exist
for compact complex $1|1$-dimensional supermanifolds, but that a semiuniversal one can be constructed
for $N=1$ super Riemann surfaces, i.e., $\fkl(1|1)$-surfaces.
\end{enumerate}

Except for \cite{I-dZ2ks}, infinite-dimensional supergeometry has, to our knowledge, only been
thoroughly investigated in \cite{S:Supergeometry}. The latter, however, takes a quite different route.
Bits and pieces of the categorical approach have appeared throughout the literature, e.g., the
so-called ``even-rules principle'' \cite{NoS_QFaSACfM} or the quite well-known functor of points.
The full program proposed in \cite{I-dZ2ks} has, however, never been realized and applied before.

A full investigation of the structure of the diffeomorphism supergroup has seemingly also not been 
given before. The group $\Aut(\cm)$ has, of course, been mentioned in many places, but the difficulties
of its infinite dimensionality seem to have obstructed its analysis.

The deformation theory of complex superspaces has been developed in great detail by Vaintrob
\cite{V:Deformations}, \cite{V:moduli}, \cite{V:complex}. In particular, he proved the 
existence of versal deformations in 
the complex $1|1$-dimensional case. In a similar way, Flenner and Sundararaman
investigated deformations of complex analytic superspaces \cite{FS:Analytic}. 
The Teichm\"uller space of $\fkl(1|1)$-surfaces was investigated
by Crane and Rabin \cite{CR:Super}, Rothstein and LeBrun \cite{LBR:Moduli}, Hodgkin \cite{H:structure},
\cite{H:direct}, Nelson and Giddings \cite{NG:geometry},
and many others. The authors of \cite{CR:Super} observed that no universal
family of marked $\fkl(1|1)$-surfaces exists. Rothstein and LeBrun consequently divided out the
resulting ambiguities and introduced for the description of the resulting space the concept of 
a superorbifold. Natanzon in \cite{N:Moduli} developed an approach to the
moduli problems of $N=1$ and $N=2$ super Riemann surfaces based on superanalogs of algebraic curves
and Fuchsian groups. This approach uses completely different tools than the ones employed in this thesis,
but on the other hand the construction of the moduli spaces instead of only the Teichm\"uller spaces,
as stated in \cite{N:Moduli}, would be quite an improvement as compared to our results. A comparison
of the results of Natanzon with ours, however, would have exceeded the scope of this work by far. 
We must leave it as an interesting topic for future research. 

The contribution of this thesis to super Teichm\"uller theory is twofold. On one hand we investigate
the possibility of a global approach to supergeometric moduli problems. This should be understood 
also as a methodological
result, since global techniques may prove valuable for similar problems, e.g., the study of
classical configuration spaces of supersymmetric gauge theories or variational problems in
supergeometry. On the other we study here (to our knowledge for the first time) the Teichm\"uller space
of general complex $1|1$-dimensional supermanifolds and determine the obstructions to the construction
of a universal family of those.

\section{Acknowledgements}

I would like to thank my advisor Prof.~J\"urgen Jost for suggesting
this very interesting and challenging problem to me, as well as his continued support, 
patience and encouragement
during my attempts to solve it. It is hard to grasp how much I learned during the past three years,
both scientifically and personally, and this I owe, first of all, to the friendly and open-minded 
yet competitive and creative environment at the  
Max Planck Institute for Mathematics in the Sciences and especially the hospitality of the IMPRS. 

I am also deeply grateful to Vladimir Molotkov for his invaluable help in understanding his work
on the categorical approach to supermathematics, and for patiently answering my numerous questions. 
Without his advice,
I would most probably have been unable to tackle the problems of infinite-dimensional supergeometry. 

I am also indebted to
Dimitry A. Leites, whose help and explanations have been vital to this work.
\footnote{Despite me not always listening carefully enough!} I am especially grateful for
pointing out the results of Molotkov and Vaintrob to me, among countless other tips. I am grateful
to Joseph Bernstein for the lectures he gave to our seminar about the moduli spaces of SUSY curves and the
approaches to their construction.

I also want to thank the participants of the seminar of the geometry group at the MPIMiS, especially 
Guofang Wang,
Brian Clarke, Guy Buss, Miaomiao Zhu and Alexei Lebedev for countless discussions, feedback and
corrections. In particular
the discussions with Guy about classical Teichm\"uller theory and moduli spaces have yielded substantial
contributions to the results presented here. I am also grateful to Mattias Wendt for some quite
insightful discussions about topoi, the problems of infinite-dimensional ringed spaces, and lots of
other aspects of algebraic geometry.
My thanks also go to Hilke Reiter, Guy and especially Brian for careful checking and 
test-reading.

Finally, I am very grateful to the Klaus-Tschira-Stiftung, whose financial support made this work
and my studies at the Max Planck Research School possible in the first place.

%% file: basics.tex
\chapter{Supergeometry}
\label{ch:sgeo}

This chapter is dedicated to an outline of the general theory of supergeometry, as far as it will be
needed in the following chapters. Here, we first present the ringed space formalism of supergeometry. The
categorical approach will be presented in a separate chapter. Detailed accounts of the materials of this
chapter can be found in \cite{NoS_QFaSACfM},\cite{SfMAI}.

Roughly speaking, supersymmetry deals with $\intZ_2$-graded objects \footnote{By $\intZ_2$, we 
denote the field of residues $\intZ/2\intZ$. The 2-adic numbers, which are sometimes denoted by the
same symbol, do not occur in this work.} and morphisms between them which preserve
the grading. It might seem somewhat arbitrary to restrict oneself to the case of just $\intZ_2$-grading,
and not to consider gradings by arbitrary abelian groups.
But it has turned out that this case plays a special role: extending the scope beyond the linear realm
into geometry, the $\intZ_2$-grading becomes translated into commuting and anticommuting variables. Thus,
supergeometry is a special kind of noncommutative geometry, but its noncommutativity is one of the tamest
possible ones.
While graded objects had been used in algebra for a long time already,
Berezin's idea to construct a theory of analysis and geometry which
uses commuting and anticommuting variables on the same footing was quite revolutionary. As could be
expected, this geometry is considerably more algebraic than classical differential geometry.
In particular, the use of non-reduced ringed
spaces for the definition of supermanifolds becomes inevitable.

\section{Linear superspaces and superalgebras}

Throughout this work, we denote the elements of $\intZ_2$ as $\{\bar{0},\bar{1}\}$. The field $\fieldK$
denotes either $\realR$ or $\complexC$.

\begin{dfn}
\label{def:sring}
A ring $R$ is called a superring, if it is $\intZ_2$-graded, i.e., if it decomposes into a direct sum
of additive subgroups $R=R_{\bar{0}}\oplus R_{\bar{1}}$ for which 
\begin{equation}
R_{\bar{i}}\cdot R_{\bar{j}}\subseteq R_{\bar{i}+\bar{j}}
\end{equation}
holds. A module $M$ over a superring $R$ is called a supermodule, if it is $\intZ_2$-graded,
\begin{equation}
M=M_{\bar{0}}\oplus M_{\bar{1}},
\end{equation}
and if $R_{\bar{i}}\cdot M_{\bar{j}}\subseteq M_{\bar{i}+\bar{j}}$.
The submodule $M_{\bar{0}}$ is called even, $M_{\bar{1}}$ is called odd. A morphism $\phi:M\to M'$ of
$R$-supermodules is a morphism of $R$-modules which preserves the grading.
\end{dfn}

We denote by $\catsmod_R$ (resp. ${}_R\catsmod$) the category of right (resp. left) $R$-supermodules.
Note that any ring $R$ can be considered as a superring by just setting $R_{\bar{0}}:=R$ and
$R_{\bar{1}}:={0}$.

\begin{dfn}
Let $M$ be an $R$-supermodule. An element $m\in M$ is called homogeneous,
if $m\in M_{\bar{0}}$ or $m\in M_{\bar{1}}$. For a homogeneous element, its parity is defined to be
\begin{equation}
p(m):=\left\{\begin{array}{lcl}
0 & \mathrm{if} & m\in M_{\bar{0}}\\
1 & \mathrm{if} & m\in M_{\bar{1}}
\end{array}\right.
\end{equation}
An inhomogeneous element is said to be of indefinite parity.
\end{dfn}

\begin{dfn}
A $\fieldK$-supermodule is called a $\fieldK$-super vector space. 
\end{dfn}

The category of $\fieldK$-super vector spaces will be denoted by $\catsvect_\fieldK$. Particularly important
are the standard super vector spaces $\fieldK^{m|n}$, which have $m$ even and $n$ odd dimensions.

\begin{dfn}
A $\fieldK$-superalgebra $A$ is a $\fieldK$-super vector space endowed with a morphism
\begin{equation}
\mu:A\otimes_\fieldK A\to A.
\end{equation}
The algebra
$A$ is called supercommutative if for all homogeneous elements $a,b$, one has
\begin{equation}
\label{salgcomm}
\mu(a,b)=(-1)^{p(a)p(b)}\mu(b,a).
\end{equation}
$A$ is associative, resp. with unit, if it is associative, resp. has a unit, as an ordinary $\fieldK$-algebra.
\end{dfn}

\begin{dfn}
A Lie superalgebra $L$ over $\fieldK$ is a $\fieldK$-superalgebra whose morphism 
$[\cdot,\cdot]:L\otimes L\to L$ satisfies the following properties:
\begin{enumerate}
\item it is super-antisymmetric: $[a,b]=-(-1)^{p(a)p(b)}[b,a]$ for all homogeneous elements $a,b\in L$,
\item it satisfies the super Jacobi identity, i.e., for all homogeneous $a,b,c\in L$, one has
\begin{equation}
\label{sjacobi}
[a,[b,c]]+(-1)^{p(a)p(b)+p(a)p(c)}[b,[c,a]]+(-1)^{p(a)p(c)+p(b)p(c)}[c,[a,b]]=0.
\end{equation}
\end{enumerate}
\end{dfn}

Expressions like (\ref{salgcomm}) and (\ref{sjacobi}) are extended to inhomogeneous elements by linearity.
It is clear that if a unit exists, it has to be even. If no confusion can arise, 
multiplication in an algebra will be
denoted by either a dot (like $a\cdot b$) or by juxtaposition (like $ab$). In the following, 
we will assume that all superalgebras are associative and have a unit.

\begin{dfn}
A left (resp.~right) module over a $\fieldK$-superalgebra $A$ is a $\fieldK$-super vector space $M$ 
endowed with a morphism
\begin{equation}
\rho:A\otimes_\fieldK M\to M\qquad (\textrm{resp. }\rho:M\otimes_\fieldK A\to M).
\end{equation}
satisfying the usual identities that make $M$ a module over $A$ as an ordinary algebra.
\end{dfn}

For a supercommutative algebra, every left module can be made a right module by defining
\begin{equation}
\label{modstr}
m\cdot a:=(-1)^{p(a)p(m)}a\cdot m
\end{equation}
for all homogeneous elements $a\in A$ and $m\in M$, and extending this definition by linearity. From now on,
we will restrict ourselves to supercommutative algebras with unit. Consequently, we will just speak of
``modules'', all of which are defined to be left-modules whose right module structure is given by
(\ref{modstr}).

\begin{dfn}
\label{def:dprod}
Let $A$ be a supercommutative superalgebra.
The direct sum of $A$-super modules $M=M_{\bar{0}}\oplus M_{\bar{1}}$ and $N=N_{\bar{0}}\oplus N_{\bar{1}}$
is a supermodule whose homogeneous submodules are given by
\begin{equation}
(M\oplus N)_{\bar{i}}=M_{\bar{i}}\oplus N_{\bar{i}}.
\end{equation}
This definition extends to direct sums of families indexed by a set $\mathcal{I}$ in the obvious way.
\end{dfn}

Actually, Def. \ref{def:dprod} is a consequence of the general definition of direct sums as colimits
of certain functors (cf. Section \ref{sect:catprod}).

\subsection{Tensor products of supermodules}

\begin{dfn}
\label{tprod}
Let $M,N$ be $\fieldK$-super vector spaces. The tensor product of $M$ and $N$ is the tensor product of
$M,N$ as ordinary $\fieldK$-super vector spaces, endowed with the grading
\begin{equation}
(M\otimes_\fieldK N)_{\bar{i}}=\bigoplus_{\bar{j}+\bar{k}=\bar{i}}M_{\bar{j}}\otimes_\fieldK N_{\bar{k}}.
\end{equation}
If $A,B$ are $\fieldK$-superalgebras, their tensor product can be given a canoncial $\fieldK$-superalgebra
structure again by setting
\begin{equation}
(a\otimes b)(c\otimes d):=(-1)^{p(b)p(c)}ac\otimes bd.
\end{equation}
\end{dfn}

This extends to modules over supercommutative superalgebras in a natural way.

\begin{dfn}
Let $M,N$ be modules over a $\fieldK$-superalgebra $A$. Their tensor product is 
defined as 
\begin{equation}
(M\otimes_A N):=(M\otimes_\fieldK N)/I,
\end{equation}
where $I$ is the ideal spanned by elements of the form $ma\otimes n-m\otimes an$, $m\in M, n\in N, a\in A$.
\end{dfn}

The definition of the tensor product is chosen such that it has ``good'' properties. The tensor product
of modules over some commutative ring $R$ has several functorial properties, e.g. universality in the
category of multilinear maps. Besides, it comes equipped with a unit object ($R$ itself) and an 
associativity
isomorphism. Categories which allow a product which has these properties are called
\emph{tensor categories}. Usually one supplements this by a commutativity isomorphism
$c_{V,W}:V\otimes W\cong W\otimes V$. If the isomorphisms $c_{V,W}$ satisfy certain compatibility
axioms with respect to associativity and
the unit element\footnote{The so-called Hexagon axioms, see, e.g., \cite{BK:Lectures}}, they
are called a \emph{braiding} of the tensor category. 
The category of supermodules over a supercommutative $\fieldK$-superalgebra $A$
inherits all this from the tensor product in the category of modules over $A$ as an
ordinary algebra, except for one crucial
difference: it has a different braiding, namely \cite{NoS_QFaSACfM}
\begin{eqnarray}
\label{tprodcomiso}
c_{V,W}:V\otimes_A W &\to& W\otimes_A V\\
\nonumber
v\otimes w &\mapsto& (-1)^{p(v)p(w)}w\otimes v.
\end{eqnarray}
If one has several factors arranged in two different permutations $V=V_1\otimes\ldots\otimes V_n$
and $V'=V_{\sigma(1)}\otimes\ldots\otimes V_{\sigma(n)}$, the axioms of a braided tensor 
category require that
there exists a unique isomorphism $\tau:V\cong V'$. This isomorphism is given by a composition of the
associativity and commutativity isomorphisms of the factors, and the axioms thus require that whichever
composition we choose, the result is the same. Since we have inherited associativity from the category
of ordinary modules and commutativity is given by (\ref{tprodcomiso}), $\tau$ must be the map
\begin{equation}
\tau:v_1\otimes\ldots\otimes v_n\mapsto (-1)^Nv_{\sigma(1)}\otimes\ldots\otimes v_{\sigma(n)},
\end{equation}
where $N$ is the number of pairs of indices where $i<j$ but $\sigma(i)>\sigma(j)$, and 
for which $v_i,v_j$ are odd.
This fact is the origin of the \emph{sign rule}, which says that whenever one interchanges two 
neighbouring odd factors in a product of elements of
some superalgebra, one picks up a factor $(-1)$.

\begin{dfn}
\label{def:copf}
Let $R$ be a superring. The change of parity functor is the functor $\Pi:\catsmod_R\to\catsmod_R$ 
(equivalently for left modules) which assigns to a supermodule 
$M=M_{\bar{0}}\oplus M_{\bar{1}}$ the supermodule $\Pi(M)$ with $(\Pi(M))_{\bar{0}}=M_{\bar{1}}$
and $(\Pi(M))_{\bar{1}}=M_{\bar{0}}$.
\end{dfn}

Parity reversal has to be a functor since any morphism between two supermodules has to preserve
parity.
One clearly has $\Pi(\fieldK)=\fieldK^{0|1}$. With the help of definition (\ref{tprod}), one also
easily verifies the useful fact that for any $\fieldK$-super vector space $V$,
\begin{equation}
\fieldK^{0|1}\otimes_\fieldK V\cong\Pi(V).
\end{equation}

\subsection{$\intZ$-graded supermodules}

Some constructions yield modules and spaces which are already naturally $\intZ$-graded, e.g., exterior
or symmetric algebras over super vector spaces, or the complex of differential forms appearing below.
One has to be careful to choose a consistent convention for the interplay of the degree with parity. There
are two choices \cite{NoS_QFaSACfM}, and although they are basically equivalent, they lead to sign 
differences in some formulas. To remove all ambiguities it is sufficient to fix the commutativity isomorphism
of the tensor product. Following \cite{NoS_QFaSACfM}, we choose the following one.

\begin{dfn}
Let $A$ be a commutative $\fieldK$-superalgebra. 
Define a tensor category $\catsmod_A^{gr}$ of $\intZ$-graded objects of $\catsmod_A$ by defining
the tensor product of any two $\intZ$-graded $A$-supermodules $M,N$ to be
the tensor product $M\otimes_A N$ in $\catsmod_A$ with the usual
associativity isomorphism, but with the commutativity isomorphism redefined as
\begin{eqnarray}
c_{M,N}:M\otimes_A N &\to& N\otimes_A M\\
m\otimes n &\mapsto& (-1)^{p(m)p(n)+\deg(m)\deg(n)}n\otimes m.
\end{eqnarray}
\end{dfn}

This redefines the $c_{M,N}$ of $\catsmod_A$ to $(-1)^{\deg(m)\deg(n)}c_{M,N}$ for all 
$m\in M$ and $n\in N$.

An application of this rule is the construction of the symmetric and exterior powers of a super vector
space $V$. The symmetric power $\Sym^n(V)$ is the quotient of $V\otimes\ldots\otimes V$ by the action
of the commutativity isomorphisms $c_{V,V}:V\otimes V\to V\otimes V$ exchanging the elements of all
factors in combination with associativity. Clearly one has
\begin{equation}
\Sym^\bullet(V)=\Sym^\bullet(V_{\bar{0}})\otimes\wedge^\bullet(V_{\bar{1}})
\end{equation}
where $\Sym$ and $\wedge$ on the right hand side are the operations in the category of ordinary vector
spaces.

Analogously the exterior power $\wedge^n(V)$ of a super vector space $V$ is the quotient of 
$V\otimes\ldots\otimes V$ by the action of the symmetric group $S_n$ acting via $c_{V,V}$ on 
the factors, additionally 
multiplying the permutation with its sign character $\epsilon(\sigma)$. The result is
\begin{equation}
\label{svectextalg}
\wedge^\bullet(V)=\wedge^\bullet(V_{\bar{0}})\otimes\Sym^\bullet(V_{\bar{1}}).
\end{equation}

Note that as an object of $\catsmod_\fieldK^{gr}$, $\wedge^\bullet(V)$ is the symmetric algebra of
the object $V\in\catsmod_\fieldK^{gr}$ which is purely of degree one.

\subsection{The inner Hom object for supermodules}
\label{ihomsvect}

If $V,W$ are two $\fieldK$-vector spaces, the set of their morphisms, i.e. linear maps $V\to W$, naturally
carries the structure of a $\fieldK$-vector space itself. The same goes for, e.g., abelian groups, sets and many
other categories. This situation is axiomatized by the so-called inner Hom object \cite{CftWM},\cite{MoHA}.
Let $\catc$ be some category with a product operation $*$, e.g. $\catsets$ with the direct product, or
$\fieldK$-vector spaces with the tensor product over $\fieldK$. Then, for $Y,Z\in\catc$, the inner Hom
object $\ihom(Y,Z)$ is an object satisfying the \emph{adjunction formula} \cite{MoHA}:
\begin{equation}
\label{ihomdef}
\Hom_\catc(X,\ihom(Y,Z))=\Hom_\catc(X*Y,Z)\qquad\forall X\in\catc.
\end{equation}
This formula requires that $\ihom(Y,Z)$, if it exists, must represent the functor
\begin{eqnarray}
\catc^\circ &\to& \catsets\\
X &\mapsto& \Hom(X*Y,Z),
\end{eqnarray}
where $\catc^\circ$ denotes the dual category of $\catc$. This implies that if $\ihom(Y,Z)$ exists, it is
unique up to a unique isomorphism.
In view of the above examples, the interpretation is clear: the inner Hom object $\ihom(Y,Z)$ is 
intended to be an
object which turns the application of a map $\varphi:Y\to Z$ into an \emph{operation} 
$\ihom(Y,Z)*Y\to Z$ in $\catc$.\footnote{A (left) operation of an object $K$ on an object $A$ is a
morphism $K\times A\to A$.}

In the super context, inner Hom objects play an important role as superizations of spaces of maps
between super objects. In this work they will be needed for various $\fieldK$-super modules and for
supermanifolds (e.g., as the diffeomorphism supergroup).

As an example, look at $\fieldK$-super vector spaces first. For $V,W$ two $\fieldK$-super vector spaces
the set $\Hom(V,W)$ naturally inherits the structure of a $\fieldK$-vector space but it is not a 
super vector space: there is no natural grading. The inner Hom object will have to satisfy
\begin{equation}
\Hom_{\catsvect_\fieldK}(U,\ihom(V,W))=\Hom_{\catsvect_\fieldK}(U\otimes_\fieldK V,W)
\end{equation}
for all $\fieldK$-super vector spaces $U$.
Since $\fieldK^{1|0}$ and $\fieldK^{0|1}$ generate the category $\catsvect_\fieldK$ (cf. Prop.
\ref{gensvect}), it is enough to let $U$ run over these two spaces. One finds
\begin{equation}
\Hom(\fieldK^{1|0}\otimes V,W)\cong \Hom(V,W)
\end{equation}
and
\begin{equation}
\Hom(\fieldK^{0|1}\otimes V,W)\cong \Hom(\Pi(V),W).
\end{equation}
Thus the even part $\ihom(V,W)_{\bar{0}}$ consists of the $\fieldK$-linear maps $V\to W$ which preserve the
grading (i.e. the actual morphisms of $\catsvect_\fieldK$), while $\ihom(V,W)_{\bar{1}}$ consists of those
which reverse it. Elements of the even part are therefore called even linear maps, those of the odd part
odd linear maps. One should, however, bear in mind that only the even ones are really morphisms in the
strict sense.

In a similar way \cite{NoS_QFaSACfM} one shows that for modules $M,N$ over a supercommutative super 
algebra $A$ the inner Hom
object $\ihom_A(M,N)$ is the $\fieldK$-super vector space of even, resp. odd, maps $f:M\to N$ for
which
\begin{equation}
f(am)=(-1)^{p(f)p(a)}af(m).
\end{equation}
This enforces $f(ma)=f(m)a$. The $A$-module structure is given by 
\begin{equation}
(af)(m)=af(m). 
\end{equation}
Thus again, 
$\ihom_A(M,N)_{\bar{0}}=\Hom_A(M,N)$. This turns out to be
a general feature of inner Hom objects in categories of super objects: their underlying space is always
the set of actual morphisms between the two objects in question. The parts of them which involve odd elements
are only made visible via products with other super objects.

\subsection{Derivations of superalgebras}
A derivation of a $\fieldK$-superalgebra $A$ is a morphism $D:A\to A$ such that
\begin{equation}
D(a\cdot b)=D(a)\cdot b+a\cdot D(b).
\end{equation}
By this definition, the set of derivations would only be an ordinary Lie algebra over $\fieldK$. 
To make it a 
Lie superalgebra, one has to allow odd derivations, i.e., odd elements of $\ihom_\fieldK(A,A)$.

\begin{dfn}
Let $A$ be a $\fieldK$-superalgebra. The Lie superalgebra $\ider_\fieldK(A,A)$ of derivations of $A$
is the subsuperspace of $\ihom_\fieldK(A,A)$ whose elements $D$ satisfy
\begin{equation}
D(a\cdot b)=(Da)\cdot b+(-1)^{p(D)p(a)}a\cdot Db.
\end{equation}
\end{dfn}

One easily checks that this defines a Lie super algebra, with bracket 
\begin{equation}
[D_1,D_2]=D_1D_2-(-1)^{p(D_1)p(D_2)}D_2D_1.
\end{equation}

\subsection{Dual modules}
\label{sect:dualmod}

In a tensor category with inner Hom objects, one defines the dual object of $M$ as \cite{NoS_QFaSACfM}
\begin{equation}
M^*:=\ihom(M,\underline{1})
\end{equation}
where $\underline{1}$ is the unit of the tensor product. In the case of modules over a supercommutative
superalgebra $A$ one has $\underline{1}=A$. Using the adjunction formula (\ref{ihomdef}), this yields a
canonical morphism
\begin{eqnarray}
M^*\otimes_A N &\to& \ihom(M,N)\\
\omega\otimes n &\mapsto& \left(m\mapsto (-1)^{p(\omega)p(n)}n\omega(m)\right),
\end{eqnarray}
which is an isomorphism if $M$ is free of finite type \cite{NoS_QFaSACfM}.
Dual modules appear naturally in the context of vector bundles on super manifolds.

\subsection{The category of finite-dimensional Grassmann algebras}
\label{subsect:gr}

A particularly important role will be played by the finite dimensional Grassmann algebras $\Lambda_n$.
They are the free commutative superalgebras on $n$ odd generators. Equivalently, they can be thought 
of as the
exterior algebras of $\fieldK$-vector spaces of $\fieldK$-dimension $n$. To make it precise, they are
generated over $\fieldK$ by generators $\xi_1,\ldots,\xi_n$ subject to the relations
\begin{equation}
\xi_i\xi_j=-\xi_j\xi_i,\qquad\textrm{ for } i,j\in\{1,\ldots,n\}.
\end{equation}
Together with their morphisms as superalgebras, the finite-dimensional Grassmann algebras form a 
category $\catgr$.
This category has a terminal object: the algebra $\Lambda_0=\fieldK$. A terminal object has the property that
every other object possesses a unique arrow to it, called the terminal morphism. In the case of $\catgr$, 
this morphism acts by removing all odd generators:
\begin{eqnarray}
\label{finmorgr}
\epsilon_{\Lambda_n}:\Lambda_n &\to& \fieldK\\
\nonumber
1 &\mapsto& 1\\
\nonumber
\xi_i &\mapsto& 0,\qquad\textrm{ for } 1\leq i\leq n.
\end{eqnarray}
Note that $\Lambda_0=\fieldK$ is also the initial object of $\catgr$: there is a unique morphism
\begin{equation}
c_{\Lambda_n}:\fieldK\to\Lambda_n\qquad\textrm{ for all } n\in\naturalN
\end{equation}
which embeds the ground field into $\Lambda_n$ by sending 1 to 1. Thus $\fieldK$ is the \emph{null object}
of $\catgr$.

We will write $\catgr$ only for the real Grassmann algebras, i.e. those with $\fieldK=\realR$. In the case
of $\fieldK=\complexC$ we denote the appropriate category as $\catgr^\complexC$, to avoid confusion. The
complex algebras will be denoted $\Lambda_n^\complexC$. As exterior algebras they are generated over
complex vector spaces of dimension $n$, but it is easy to show that one gets the same if one defines them
as complexifications of the real Grassmann algebras. More precisely, if $V$ is a real vector space, then
\begin{equation}
\Lambda(V\otimes_\realR\complexC)\cong \Lambda(V)\otimes_\realR\complexC.
\end{equation}

\begin{prop}
\label{homgr}
Let $\Lambda_n,\Lambda_m$ be the Grassmann algebras over $\fieldK$ on $n$ and $m$ generators, respectively.
Then there exists an isomorphism of $\fieldK$-vector spaces
\begin{equation}
\Hom_\catgr(\Lambda_n,\Lambda_m)\cong \fieldK^n\otimes_\fieldK \Lambda_{m,\bar{1}}.
\end{equation}
\end{prop}
\begin{proof}
Let $\xi_1,\ldots,\xi_n$ be the free generators of $\Lambda_n$.
A morphism $\phi:\Lambda_n\to\Lambda_m$ is a homomorphism of $\fieldK$-algebras which preserves parity.
Thus $\phi(1)=1$, and $\phi$ is uniquely determined by choosing the images $\phi(\xi_n)$ of its
generators which have to lie in $\Lambda_{m,\bar{1}}$. Setting 
$V=\mathrm{Span}_\fieldK(\xi_1,\ldots,\xi_n)$, we can write
\[
\Hom_\catgr(\Lambda_n,\Lambda_m)\cong \Lambda_{m,\bar{1}}\otimes_\fieldK V^*,
\]
where $V^*$ is the dual space of $V$, which is in our case isomorphic to $\fieldK^n$.
\end{proof}

\section{Supermanifolds}
\label{sect:sman}

Finite-dimensional supermanifolds are most conveniently defined using the language of ringed spaces. 
Although it is possible to give a description in terms of charts and atlases, the ringed space formalism
is an indispensable tool. Furthermore, if one wants to describe singular superspaces,
the ringed space approach seems to be the only viable one. The description of supermanifolds resembles in
this regard the theory of complex analytic spaces \cite{CAS} (where non-reduced spaces also appear).
However, there are important deviations from
ordinary geometry, mainly originating in the fact that one would like to retain the concept of
a coordinate for something that one calls a manifold, but coordinates cannot play the same role as usual
in the context of nilpotent elements. In particular, it turns out that one must be careful with the
notion of a ``function'' on a supermanifold. While a function on an ordinary manifold is described by
the values it takes at every point (and thus coordinates are just a suitable collection of functions whose
values label the points), such an interpretation cannot directly be carried over the super context.
Functions considered as maps from the supermanifold
to the ground field $\fieldK$ would not form a superalgebra; they only inherit the structure of a
$\fieldK$-algebra. 
To make the full geometric structure visible in terms of maps it is therefore often necessary not to
consider single superringed spaces but families of them over a base which is a supermanifold itself.
The use of such families has always been common in physics, but often implicit, by using 
``Grassmann-valued'' functions. We will argue below that this method, although it often yields the
correct results when carefully used, is not really clean. Rather one should clearly
distinguish between single supermanifolds and families of them.
Another (and mathematically rigorous) motivation for the use of families instead of single objects 
will be supplied by the functor of points approach in the next chapter.

\subsection{Superspaces}

\begin{dfn}
A locally superringed space $\cm=(M,\co_\cm)$ is a topological space $M$ endowed with a sheaf $\co_\cm$ 
of local supercommutative rings.
A morphism of superspaces is a morphism of locally ringed spaces which is stalkwise a homomorphism of
supercommutative rings.
\end{dfn}

The structure sheaf $\co_\cm$ contains a canonical subsheaf of nilpotent ideals 
generated by odd elements:
\begin{equation}
\label{nilideal}
\cj = \co_{\cm,\bar{1}}\oplus\co_{\cm,\bar{1}}^2.
\end{equation}
The superspace $\cm_{red}=(M,\co_\cm/\cj)$ is then purely even and if there are no even nilpotent elements
in $\co_\cm$, it is also reduced. It possesses a canonical embedding 
\begin{equation}
\cem:\tilde{\cm}\to\cm.
\end{equation}
The space $\cm_{red}$ is called the underlying space of $\cm$.
In general, $\cm_{red}$ can itself be a non-reduced space. The ``completely reduced'' space is
denoted as $\cm_{rd}$. In the structure sheaves of the superspaces considered in this work
no nilpotent even elements occur, so no confusion can arise. In this case, we speak of $\cm_{red}=\cm_{rd}$
as the underlying manifold.

\subsection{Superdomains and supermanifolds}

As in ordinary algebraic geometry, a supermanifold is a superspace which has only smooth points. The
notion of ``smooth point'', however, is somewhat different in the super context, as it still allows
nilpotent elements in the stalks. But still, a supermanifold is defined by the property that each of
its points
has a neighbourhood which is isomorphic to the same model space.

\begin{dfn}
\label{linsman}
Let $V$ be a real super vector space of dimension $m|n$. To $V$ we associate a locally ringed space
$\olv$ by setting
\begin{equation}
\olv:=(V_{\bar{0}},C^\infty_{V_{\bar{0}}}\otimes_\realR\wedge[\theta_1,\ldots,\theta_n]),
\end{equation}
where $C^\infty_{V_{\bar{0}}}$ denotes the sheaf of smooth functions, and $\theta_1,\ldots,\theta_n$
are a basis of $V_{\bar{1}}$. $\olv$ is called the smooth linear supermanifold associated with $V$.
\end{dfn}

Together with their morphisms as superringed spaces linear supermanifolds form a category. 
An isomorphism $V\to V$ corresponds to an isomorphism $\olv\to\olv$ of ringed spaces, therefore it is 
sufficient to consider the equivalent full subcategory $\ovlk{m|n}$ for $m,n\in\naturalN_0$.

\begin{dfn}
A smooth superdomain $\cu$ of dimension $m|n$ is a restriction of $\ovlr{m|n}$ to an underlying domain
$U\subseteq \realR^m$:
\begin{equation}
\cu=(U,C^\infty_U\otimes_\realR\wedge[\theta_1,\ldots,\theta_n]\big|_{U}).
\end{equation}
\end{dfn}

\begin{dfn}
\label{sman1}
A smooth supermanifold $\cm=(M,\co_\cm)$ of dimension $m|n$ is a superspace which is locally
isomorphic to a linear supermanifold of dimension $m|n$.
\end{dfn}

\begin{dfn}
An open subsupermanifold $\cu$ of $\cm$ is defined by an open subset $U$ of $M$ and the restriction of
$\co_\cm$ to $U$:
\begin{equation}
\cu=(U,\co_\cm\big|_U).
\end{equation}
\end{dfn}

Analogously, one defines real analytic superdomains as open subsupermanifolds of the linear real analytic
supermanifold $(\realR^m,\ca(\realR^m)[\theta_1,\ldots,\theta_n])$, and real analytic supermanifolds as
being locally isomorphic to them.
For complex analytic supermanifolds, one first constructs the reference space
\begin{equation}
\ovlc{m|n}=(\complexC^m,\co_{\complexC^m}\otimes_\complexC\wedge[\theta_1,\ldots,\theta_n]),
\end{equation}
where $\co(\complexC^m)$ denotes the sheaf of holomorphic functions. A complex (analytic) supermanifold
of dimension $m|n$ is then defined by the property of being locally isomorphic to $\ovlc{m|n}$.

It is important to note that, unlike in the purely even case, a super vector space is \emph{not}
a supermanifold. Although there exists a canonical way to produce a linear supermanifold $\olv$ from
any finite-dimensional super vector space $V$ the two objects behave quite differently. 
For example, if $\dim_\realR V=m|n$, the set-theoretical
model underlying $V$ is a linear $m+n$-dimensional space, while that of $\olv$ is only an $m$-dimensional
space. 

Together with their morphisms as superringed spaces, smooth finite-di\-men\-sion\-al supermanifolds form a 
category which we will
denote as $\catfinsman$. For the category of complex supermanifolds, whose morphisms are those of 
topological spaces ringed with local $\complexC$-superalgebras, we write $\catfinsman^\complexC$.

An example for a morphism of supermanifolds is the embedding of its underlying manifold. The
nilpotent elements of each stalk form an ideal in this stalk and these together form an ideal subsheaf
$\cj\subset\co_\cm$ (cf. (\ref{nilideal})). Then one has the canonical epimorphism $q:\co_\cm\to\co_\cm/\cj$ 
of sheaves, and setting
\begin{eqnarray}
\cem=(\id_M,q):M_{rd} &\to& \cm,
\end{eqnarray}
one obtains a canonical embedding of $M_{rd}$ as a closed subsupermanifold of $\cm$. Note that in general
there is no canonical projection $\mathrm{cpr}=(\id_M,\phi):\cm\to M_{rd}$. To construct such a projection,
we need a sheaf morphism $\phi:\co_\cm/J\to\co_\cm$. Of course, such morphisms always exist, but they are
not canonically given. An exception is the case of a split supermanifold (see below), i.e., a supermanifold
which is globally isomorphic to the exterior bundle of a vector bundle over $M_{rd}$.

\subsection{Sections, functions, and values}

One refers to the sections of the structure sheaf as ``superfunctions'', in analogy with the case
of a purely even manifold. One has to be careful with this intuition, however.
The sections of a sheaf of $\fieldK$-superalgebras cannot straightforwardly be interpreted as maps from the
supermanifold $\cm$ to some other (super)space, in particular not to the field $\fieldK$.

In the case of an ordinary manifold $M$, even if is initially defined abstractly just by a sheaf $\co_M$ of 
commutative $\fieldK$-algebras with unit, one can reinterpret the sections of $\co_M$ as actual
functions, i.e. as local morphisms $M\to\fieldK$. Namely one
ascribes to each section $s\in\co_M(U)$ and each point $p\in U$ its \emph{value} $\lambda=s(p)$, which 
is defined
as the unique number in $\fieldK$, such that $s-\lambda$ is not invertible in any neighbourhood of $p$.
In the case of a reduced space the values $s(p)$ for all $p\in U$ determine $s$ uniquely. Conversely,
this fact enables one to choose $n$ suitable coordinate functions, whose values can be used to label the
points of $M$. 
This method does not extend to nonreduced spaces \cite{MoHA},\cite{CAS}, thus in particular not to 
superspaces. 

\begin{dfn}
Let $\cm$ be a supermanifold (real or complex). Let $U\subseteq M$ be an open set of $M$ and let
$f\in\co_\cm(U)$ be a section of the structure sheaf over $U$. Then we define the evaluation of $f$ at some
point $p\in M$ as the map
\begin{equation}
\ev_p:f\mapsto f_{red}(p),
\end{equation}
where $f_{red}$ is the image of $f$ under the sheaf map $\co_\cm\to\co_\cm/\cj$.
\end{dfn}

Evaluation of functions is a sheaf homomorphism which is equivalent
to reducing the structure sheaf. So it makes sense to speak of the values of a section of $\co_\cm$ for
a supermanifold $\cm$ and the value at $p$ is still the unique number $\lambda$ in $\fieldK$ such that 
$s-\lambda$ is not invertible in any neighbourhood of $p$. It just does not specify $s$ uniquely --- 
for example
a purely odd section will have the value zero everywhere, but need not be the zero section.
The fact that its values don't characterise a 
superfunction uniquely also carries over to
sections of modules over the structure sheaf, such as vector fields and differential forms.

\subsection{Coordinates}

\begin{dfn}
Let $\cm$ be a smooth supermanifold, and let $\Phi=(\phi,\varphi):\cm\to\cv\subseteq\ovlr{m|n}$ be a morphism 
to a superdomain. The coordinates of $\Phi$ are defined to be the pullback of the standard
coordinates $(x_1,\ldots,x_m,\theta_1,\ldots,\theta_n)$ of $\ovlr{m|n}$ to $\cm$, i.e. they are
\begin{eqnarray}
y_i &=& \varphi(x_i)\\
\eta_j &=& \varphi(\theta_j).
\end{eqnarray}
\end{dfn}

As remarked in the previous section, the idea to describe the points of a manifold in terms of
tuples of real (or complex) numbers which ``label'' its topological points is
inadequate for nonreduced spaces. In the theory of complex analytic spaces \cite{CAS} one just drops
this kind of intuition and resorts to the algebraic description. For supermanifolds, however, it is
possible to maintain the ``coordinate'' concept in the sense that one may still describe them locally 
in terms of standard (super)functions on model domains which one pulls back along an isomorphism.
The crucial statement that makes this construction work is:

\begin{thm}
\label{scoords}
Let $\cv\subseteq\ovlr{m|n}$ be a superdomain and $\cm$ be an arbitrary supermanifold. Then the map which 
assigns to each morphism $f:\cm\to\cv$ its coordinates $(y_1,\ldots,y_m,\eta_1,\ldots,\eta_n)$ is a bijection
from the set of morphisms $f:\cm\to\cv$ to the set of systems of $m$ even functions $a_i$ and $n$ odd
functions $\beta_j$ on $\cm$, such that the values $(a_1(x),\ldots,a_m(x))$ are in $V\subset\realR^m$ for 
any $x\in M$.
\end{thm}

For a full proof, see \cite{Itttos}, \cite{Gftacg}, \cite{SfMAI}. The key point of the proof 
consists in giving
a natural meaning to the pullback of arbitrary functions $F$ from $\cv\subseteq\ovlr{m|n}$ to $\cm$.
Let $(t_1,\ldots,t_m,\xi_1,\ldots,\xi_n)$ be the coordinates of $\cv$.
Assume that $\cm$ is a domain $\cu\subseteq\ovlr{p|q}$ with coordinates $x_1,\ldots,x_p$ and
$\theta_1,\ldots,\theta_q$. Then any function $f\in\co_\cm(M)$ can be written as
\begin{equation}
f=\sum_{I}f_I\theta_I.
\end{equation}
Here $I\subset\{1,\ldots,q\}$ is a multiindex running through all increasingly ordered 
subsets, $\theta_I$ is the
product of the appropriate odd coordinates in the same order and the $f_I$ are ordinary smooth functions on $U$.
Denote the morphism as $\Phi=(\phi,\varphi):\cm\to\cv$.
Then the pullbacks $a_i=\varphi(t_i)$, $1\leq i\leq m$, of the even coordinates of $\cv$ are even sections 
of $\co_\cm$,
thus they are given by elements of the form $a_{i0}+r_i$, where $a_{i0}$ is just a smooth function 
on $V$ and $r_i$ is even and nilpotent. Now let $F(t_1,\ldots,t_m)$ be a purely even smooth function on
$V$. One then writes its pullback as the formal Taylor series
\begin{equation}
\label{pbeven}
F(a_1,\ldots,a_m)=\sum_{\mathbf{k}\subseteq\{1,\ldots,m\}}\partial^{\mathbf{k}}F(a_{10},\ldots,a_{m0})%
\frac{r_{\mathbf{k}}}{|k|!}.
\end{equation}
The sum runs over all multiindices taking values in $\{1,\ldots,m\}$, $|k|$ is the length of the multiindex,
and $r_\mathbf{k}$ is the product of the corresponding $r_{k_i}$'s.
Since the $r_i$'s are nilpotent, the sum terminates after finitely many summands.

The pullback of
a general function $F=\sum_{I}F_I\xi_I$ on $\cv$ is now uniquely fixed by (\ref{pbeven}) and the requirement
that $\varphi$ be a homomorphism. It can be written as
\begin{equation}
F(a_1,\ldots,a_m,\beta_1,\ldots,\beta_n)=\sum_{I}F_I(a_1,\ldots,a_m)\beta_I,
\end{equation}
with $\beta_j=\varphi(\xi_j)$.

Thus, prescribing $n$ odd and $m$ even functions whose underlying maps form a continuous
map $M\to V$ fixes the morphism $\Phi$ uniquely, if one adopts the above prescription (\ref{pbeven}).
But why should one choose this definition? What other choices are there? In the next sections it will be 
shown that there is actually no other possibility to define coordinates. The Taylor-like formula
(\ref{pbeven}) is forced upon us by the algebraic properties of supercommutative algebras.

\subsection{Splitness}

One obvious way to construct a supermanifold of dimension $m|n$ is to take an ordinary manifold 
$M$ of dimension $m$ and a smooth (resp. holomorphic, whichever category $M$ is in and $\cm$ should
be in) vector bundle $E\to M$ of rank $n$. Then simply set $\cm=(M,\wedge^\bullet E)$, i.e. the 
structure sheaf
to be just the set of sections of the exterior bundle over $E$, and parity to be given by the degree in
the exterior algebra mod 2. This amounts to declaring the fibers of $E$ to be purely odd super vector spaces.
The transition functions of such a  supermanifold
are just smooth transition functions on the base manifold and linear maps in the fibers of $E$,
which then induce the transition functions of $\wedge^\bullet E$. It is clear that not every supermanifold is
of this form, since one may in general have maps like
\begin{equation}
\theta'_i=A^i_j(x)\theta_j+A^i_{jkl}(x)\theta_j\theta_k\theta_l+\ldots
\end{equation}
between the odd coordinates of overlapping patches, where $A^i_j$ and $A^i_{jkl}$ are independent collections
of smooth (resp. holomorphic) functions. So, while $\wedge^\bullet E$ is always naturally $\intZ$-graded,
the structure sheaf of a supermanifold $\cm$ is not. However, it is always filtered, namely by the powers
of the nilpotent ideal $\cj\subset\co_\cm$:
\begin{equation}
\co_\cm\supset\cj\supset\cj^2\supset\cj^3\supset\ldots
\end{equation}
One can reconstruct a vector bundle over the underlying manifold from this filtration.
If we set $\co_M=\co_\cm/\cj$ and $E=\Pi(\cj/\cj^2)$, then $\co_M$ is the structure sheaf of the underlying manifold.
Clearly, $E$ can be given the structure of a locally free module of rank $n$ over $\co_M$, where
$n$ is the odd dimension of the superdomain. Namely, if $f$ is a section of $\co_\cm$ and $v$ is a section of
$\cj$ then the module structure is just given by $[f][v]=[fv]$ where the square brackets denote equivalence classes
with respect to the respective quotienting. This way we assign a vector bundle $E\to M$ to $\cm$. One checks that
this construction is functorial in $\cm$. Therefore we call this vector bundle the \emph{canonical vector bundle} associated
with $\cm$. But as in
the case of the reduced manifold we do not get a canonical morphism $\cm\to(M,\wedge^\bullet E)$.

\begin{dfn}
Let $\cm=(M,\co_\cm)$ be a supermanifold (smooth or complex), and let $E=\Pi(\cj/\cj^2)$ be the vector
bundle over $M$ associated to $\cm$. Then $\cm$ is called split, if it is
isomorphic as a supermanifold to $(M,\wedge^\bullet E)$.
\end{dfn}

\begin{lemma}
Every supermanifold $\cm=(M,\co_\cm)$ gives rise to a canonical vector bundle
$\cm^{split}=(M,\wedge^\bullet(\co_\cm/\cj^2))$ and a canonical morphism 
$P_\cm^{split}:\cm^{split}\to\cm$.
\end{lemma}
\begin{proof}
It is clear from the above discussion that the structure sheaf $\co_\cm$ always gives rise to a 
locally free module $E=\Pi(\cj/\cj^2)$ over the structure sheaf $\co_M$ of the base.
Thus, we can always construct a split supermanifold $(M,\wedge^\bullet E)$ from $\cm$.
Besides, there is always a canonical epimorphism
\begin{equation}
q:\co_\cm\to\co_\cm/\cn_\cm^2.
\end{equation}
Therefore, one always has a morphism of supermanifolds $(\id_M,q\circ\iota):\cm^{split}\to\cm$,
where $\iota:\co_M\oplus E\hookrightarrow\wedge^\bullet E$ is the canonical inclusion.
\end{proof}

The obstructions to splitness were characterized in terms of sheaf cohomology by Rothstein in
\cite{R:Deformations}. It turns out that in the smooth case, there are no obstructions: every smooth
supermanifold is isomorphic to the exterior bundle of a smooth vector bundle. That follows from the
fact that in the smooth category all involved sheaves are fine because they are modules over the
sheaf of smooth functions on the base. This fact is also known as Batchelor's theorem \cite{TSOS}.
In the complex case, that does not hold anymore. In this work, we will be concerned only with
complex $1|1$-dimensional
manifolds, so splitness plays no role for them. But what will play a role is the fact that
also a split supermanifold is not necessarily presented in the form $(M,\wedge^\bullet E)$, its transition
functions may still involve higher nilpotent corrections. In particular, if we allow arbitrary
coordinate changes, then we can never assure that a split supermanifold stays in its exterior bundle form.
The sets of morphisms between supermanifolds, split or not, are much larger than that of the associated
exterior bundles.
Since we want to divide out the action of the superdiffeomorphism group, this makes it impossible to
simply exploit splitness and consider all the structures we are interested in just in the case of
supermanifolds of the form $(M,\wedge^\bullet E)$.

\subsection{Explicit description of morphisms of supermanifolds}
\label{sect:explmor}

The fact that every supermanifold can be ``cut down'' to a vector bundle turns out to be also very
useful for the description of morphisms. In fact, every morphism between supermanifolds can be seen to
be a morphism of vector bundles composed with higher order nilpotent corrections. These corrections
themselves form a unipotent group whose construction will be sketched below. We will not try to prove
every detail here because much of this has already been done elsewhere (albeit in a somewhat different
form, see, e.g., \cite{SoS}, \cite{SfMAI}). Besides, for the $2|2$-dimensional case for which we want 
to employ these results later, things become much simpler and we will work them out specifically then.

Consider a morphism $\cm\to\cn$ of supermanifolds. By locality, it is sufficient to look at the case where
$\cm=\cu$ and $\cn=\cv$ are superdomains. Let $(x_1,\ldots,x_m,\theta_1,\ldots,\theta_n)$ be coordinates
on $\cu$ and $(y_1,\ldots,y_p,\eta_1,\ldots,\eta_q)$ be coordinates on $\cv$. The morphism consists
of a continuous map $\phi:U\to V$ of the underlying topological spaces, and of a sheaf map
$\varphi:\phi^*\co_\cv\to\co_\cu$, where $\phi^*\co_\cv$ is the pullback of $\cv$. 
Let us denote the associated vector bundles of $\cu$ and $\cv$ as $E$ and $F$, respectively.
Denote the images of the coordinates of $\cv$ as
\begin{eqnarray}
\varphi(y_i) &=& a_0^i(x)+a_{jk}^i(x)\theta_i\theta_j+\ldots+a_I^i(x)\theta_I+\ldots\\
\varphi(\eta_j) &=& b_k^j(x)\theta_k+b_{klm}^j(x)\theta_k\theta_l\theta_m+\ldots+b_I^j(x)\theta_I+\ldots.
\end{eqnarray}
Now if we reduce to $\cu^{split}$ by composing $\Phi$ with $P_\cu^{split}$, then only superfunctions on 
$\cu$ which are linear in the $\theta_i$'s are allowed.
This implies that $\varphi$ reduces to a morphism of the associated split manifolds
$(U,\wedge^\bullet E)\to(V,\wedge^\bullet F)$, because we are left with
\begin{eqnarray}
\label{bdlmor}
\varphi(y_i) &=& a_0^i(x)\\
\nonumber
\varphi(\eta_j) &=& b_k^j(x)\theta_k.
\end{eqnarray}
It is a classical consequence of the Weierstra\ss{} approximation theorem that a smooth morphism between 
smooth domains is uniquely determined by the images of the coordinate functions under the pullback.
Therefore, the sections $a_0^i(x)\in C^\infty_U(U)$, $1\leq i\leq p$ determine a smooth morphism 
$\phi:(U,C^\infty_U)\to (V,C^\infty_V)$. The smooth functions $b_k^j$, $1\leq k\leq n$, $1\leq j\leq q$
on the other hand, determine a morphism of vector bundles $\phi^*F\to E$. We will call (\ref{bdlmor}) the
\emph{associated morphism of canonical bundles} of $\varphi$, denoted as
$\varphi^{split}:\cu^{split}\to\cv^{split}$. 

\begin{lemma}
\label{splitmor}
Let $\cm,\cn$ be smooth supermanifolds and let $E\to M$, $F\to N$ be the respective canonical vector
bundles.
Then every morphism $\Phi:\cm\to\cn$ induces an associated morphism of exterior bundles
$\Phi^{split}:(M,\wedge^\bullet E)\to(N,\wedge^\bullet F)$.
\end{lemma}
\begin{proof}
Set $\Phi^{split}=P_\cm^{split}\circ\Phi$. This induces a morphism which has locally the form
(\ref{bdlmor}), i.e. a smooth bundle morphism $(M,E)\to(N,F)$.
\end{proof}

Now let us start to add the nilpotent corrections to the bundle map (\ref{bdlmor}). Let us suppose
we only divide out $\co_\cu\to\co_\cu/\cj_\cu^3$. Then we have transition functions
\begin{eqnarray}
\varphi(y_i) &=& a_0^i(x)+a_{jk}^i(x)\theta_i\theta_j\\
\nonumber
\varphi(\eta_j) &=& b_k^j(x)\theta_k.
\end{eqnarray}
Since $\varphi$ is a homomorphism, we must have
\begin{eqnarray}
\varphi(y_iy_j) &=& \varphi(y_i)\varphi(y_j)\\
\nonumber
&=& a_0^i(x)a_0^j(x)+(a_0^i(x)a_{km}^j(x)+a_{km}^i(x)a_0^j(x))\theta_k\theta_m.
\end{eqnarray}
Therefore, $a_{km}^i(x)$ is a derivation of $C^\infty_U(U)$, i.e. we can write
\begin{equation}
\varphi(y_i)=\left(1+\sum_{k,m}a^i_{km}(x)\theta_k\theta_m\pderiv{}{(a_0^i(x))}\right)a_0^i(x).
\end{equation}
This means, we additionally have to choose $p$ vector fields of degree $2$ in the odd variables to
determine the morphism completely. Allowing terms of order 3 in the $\theta_i$'s will switch on the
first correction to $\varphi(\eta_j)$:
\begin{equation}
\varphi(\eta_j) = b_k^j(x)\theta_k+b_{klm}^j(x)\theta_k\theta_l\theta_m.
\end{equation}
Obviously, this can as well be written as
\begin{equation}
\varphi(\eta_j)=\left(1+\sum_{k,l,m}b_{klm}^j(x)\theta_k\theta_l\theta_m\pderiv{}{(b_k^j(x)\theta_k)}\right)%
b_k^j(x)\theta_k.
\end{equation}
So here again, we have to choose $q$ additional vector fields of degree $2$ in the odd variables to
describe the morphism $\varphi$. We can now inductively consider $\co/\cj^n$ for $n\to\infty$, and it
is clear that in each step $n\to n+1$, we have to choose additional vector fields. Note here the
analogy of this construction with the formal Taylor series (\ref{pbeven}). In fact, our construction here
justifies the ad-hoc definition (\ref{pbeven}) by showing that there is actually no other way to
construct the pullback map.

To write up the result in a readable way, let us denote the images of the coordinates $y_i$, $\eta_j$
under the associated bundle morphism $\varphi^{split}$ (\ref{bdlmor}) as
\begin{eqnarray}
a_i &:=& \varphi^{split}(y_i)\qquad\in\co_\cu/\cj_\cu\\
\beta_j &:=& \varphi^{split}(\eta_j)\qquad\in\Gamma(E),
\end{eqnarray}
where $E:=\Pi(\cj_\cu/\cj_\cu^2)$ is the vector bundle associated with $\cu$. Then the image of
$\co_\cv$ under the sheaf map $\varphi^{split}$ comes endowed with an action of the derivations
\begin{equation}
\label{tvu}
\pderiv{}{a_i},\pderiv{}{\beta_j}.
\end{equation}
We can generate an algebra of even vector fields from the derivations (\ref{tvu}) over the sheaf $\co_\cu$,
which we will denote as $\ct_\cv\cu$. It is, of course, again filtered by the powers of $\cj_\cu$.
Let us write $\ct_\cv\cu^{(j)}$ for those elements of $\ct_\cv\cu$ whose coefficient function is of
order $\geq j$ in the odd variables $\theta_i$. E.g., elements of $\ct_\cv\cu^{(2)}$ are given by
vector fields of the form
\[
f_{ij}^k(x)\theta_i\theta_j\pderiv{}{a_k},\quad g_{ijm}^l(x)\theta_i\theta_j\theta_m\pderiv{}{\beta_l},
\quad h_{ijkl}^n(x)\theta_i\theta_j\theta_k\theta_l\pderiv{}{a_n},\ldots.
\]
Since the $a_k,\beta_l$ are not independent of the $x_m,\theta_n$, the Lie algebra $\ct_\cv\cu$ is
not abelian, although it might look like that at first. 

\begin{lemma}
\label{lem:nilp}
The vector fields $\ct_\cv\cu^{(2)}$ generate a nilpotent group $N_\cv$ under the exponential mapping:
\begin{equation}
\exp:\ct_\cv\cu(2)\to N_\cv.
\end{equation}
This group acts naturally on $\varphi^{split}(\co_\cv)$:
\begin{equation}
N_\cv:\varphi^{split}(\co_\cv)\to\co_\cu.
\end{equation}
\end{lemma}

The exponential mapping is well defined because of the nilpotency of the coefficient functions of
vector fields in $\ct_\cv\cu$, and the inverse of $\exp(X)$ is $\exp(-X)$, as one easily checks. To
illustrate the action of $N_\cv$ on $\varphi^{split}(\co_\cv)$, consider a purely even function
$F(y_1,\ldots,y_p)$ on $\cv$. Its image under $\varphi^{split}$ is $F(a_1,\ldots,a_p)$. Let then
$X=h(x)\theta_1\theta_2\pderiv{}{a_1}$ be an element of $\ct_\cv\cu^{(2)}$. We have
\begin{equation}
\exp(X)=1+h(x)\theta_1\theta_2\pderiv{}{a_1}
\end{equation}
and thus
\begin{equation}
\label{examp}
\exp(X)(F)=F+h(x)\theta_1\theta_2\pderiv{F}{a_1}.
\end{equation}
Then the main result is the following.

\begin{thm}
\label{thm:smanmorph}
Let $\cu$, $\cv$ be two superdomains with coordinates $(x_1,\ldots,x_m,$ $\theta_1,\ldots,\theta_n)$ and
$(y_1,\ldots,y_p,\eta_1,\ldots,\eta_q)$, respectively. Let $(U,\wedge^\bullet E)$ and 
$(V,\wedge^\bullet F)$ be the
associated split supermanifolds. Let $\Phi=(\phi,\varphi):\cu\to\cv$ be a morphism of superdomains, and
let $\varphi^{split}$ be the associated morphism of canonical bundles. Then one can factorize
$\varphi$ into $\varphi^{split}$ and an element of the nilpotent group $N_\cv$ constructed in Lemma
(\ref{lem:nilp}):
\begin{equation}
\varphi=\exp(X)\circ\varphi^{split}
\end{equation}
where $X\in\ct_\cv\cu^{(2)}$ is an element of the Lie algebra of $N_\cv$.
\end{thm}

This theorem has a couple of very useful
consequences. It allows one to control the deviation from the exterior bundle form produced by 
coordinate changes on a split
supermanifold in cases where one cannot retain the form $(M,\wedge^\bullet E)$. Besides, it gives
information about the behaviour of geometric objects, e.g., tensor fields, on the supermanifold under
morphisms. When one starts with a supermanifold of the form $\cm=(M,\wedge^\bullet E)$, then all vector and
tensor fields on $\cm$ have a natural interpretation in terms of the bundle $E$. Applying a transformation
which involves elements of the nilpotent group $N_\cm$, this interpretation no longer holds because
terms are appearing which are proportional to products of the odd variables. But these additional
terms, as one can see in (\ref{examp}), are all infinitesimal: they are proportional to Lie derivatives.
In the case where the morphisms $\Phi$ are superdiffeomorphisms, which will be studied in chapter
\ref{ch:sdiff}, the analogue of Thm.~\ref{thm:smanmorph} takes on a more lucid and practical form:
it splits $\sdiff(\cm)$ into automorphisms of bundles and a subgroup of nilpotent corrections.

\subsection{Products}

The existence of direct products in the category $\catfinsman$ is also assured by Thm.~\ref{scoords}.

\begin{dfn}
Let $\cu\subseteq\ovlr{m|n},\cv\subseteq\ovlr{p|q}$ be two superdomains with underlying domains $U,V$.
Then $\cu\times\cv$ is the superdomain in $\ovlr{m+p|n+q}$ whose underlying domain is $U\times V$.
\end{dfn}

Thm.~\ref{scoords} tells us that the homologically defining property
\begin{equation}
\Hom(\cw,\cu\times\cv)\cong\Hom(\cw,\cu)\times\Hom(\cw,\cv)
\end{equation}
holds for every superdomain $\cw$.

The generalization to supermanifolds is then performed as usually by taking products of domains \cite{SoS}.
Let $\cm=(M,\co_\cm),\cn=(N,\co_\cn)$ be two supermanifolds. Then the topology of $M$ has a base
$B_M=\{U_i\}_{i\in I}$ such that $\cm\big|_{U_i}$ is isomorphic to a superdomain for each $i\in I$.
Likewise, $N$ has such a base $B_N$. Choosing isomorphisms $u_i:\cm\big|_{U_i}\to\cu_i$ and
$v_j:\cn\big|_{V_j}\to\cv_j$ which cover $\cm$ and $\cn$ with superdomains and appropriate isomorphisms
$u_{ik},v_{jl}$, we can construct a supermanifold by gluing the superdomains $\cu_i\times\cv_j$
for all $(i,j)\in I\times J$ via the isomorphisms $(u_{ik},v_{jl})$ for all $(i,j),(k,l)\in I\times J$.
This manifold is defined to be the product manifold $\cm\times\cn$. For more details, see
\cite{SoS}.

\subsection{The category of superpoints}
\label{sect:spoint}

A particular class of finite-dimensional supermanifolds are the superpoints $\cp(\Lambda_n)$.

\begin{dfn}
A finite-di\-men\-sion\-al supermanifold whose underlying manifold is a one-point topological
space is called a superpoint.
\end{dfn}

Comparison with definition \ref{linsman} shows that superpoints are the supermanifolds associated to
purely odd super vector spaces, i.e., to those of dimension $0|n$. 
Together with their morphisms as supermanifolds, finite-dimensional superpoints form a category $\catspoint$
which is a full subcategory of $\catsman$.

\begin{prop}[see also \cite{I-dZ2ks}]
\label{spointeq}
There exists an equivalence of categories
\begin{equation}
\cp:\catgr^\circ\to\catspoint.
\end{equation}
\end{prop}
\begin{proof}
Define the functor $\cp$ on the objects $\Lambda_n$ of $\catgr^\circ$ as
\begin{equation}
\cp(\Lambda_n):=\cp_n:=(\{*\},\Lambda_n).
\end{equation}
To every morphism $\varphi:\Lambda_n\to\Lambda_m$ of Grassmann algebras, assign the morphism
\begin{equation}
\Phi=(\id_{\{*\}},\varphi):\cp_m \to \cp_n.
\end{equation}
To see that this functor establishes an equivalence, note first that it is fully faithful: on the set
of morphisms it is a bijection from $\Hom_\catgr(\Lambda_n,\Lambda_m)$ to $\Hom_\catsman(\cp_m,\cp_n)$.
The last property to check is essential surjectivity, i.e. every superpoint has to be isomorphic to one
of the $\cp_n$. This is clear from the fact that since a superpoint is a supermanifold its structure
sheaf must be a free superalgebra with $n$ odd generators. Since each such algebra is isomorphic to
$\Lambda_n$, the assertion is proved.
\end{proof}

This equivalence enables one to use just the full subcategory $\{\cp_n\big|n\in\naturalN_0\}$ given as
the image of the functor $\cp$ instead of the whole category $\catspoint$. Thus, in the following,
$\catspoint$ denotes just the image of $\catgr^\circ$ under the functor $\cp$ constructed in
Prop.~\ref{spointeq}. Its objects will be denoted $\cp_n$. One could also have defined $\cp$
as a contravariant functor $\catgr\to\catspoint$. Equivalences of this type are also called
\emph{dualities of categories}.

\section{Super vector bundles}

As in ordinary geometry, a smooth super vector bundle over a supermanifold $\cm$ can be viewed in two ways 
\cite{NoS_QFaSACfM}:
\begin{enumerate}
\item as a fiber bundle $\pi:\cv\to\cm$ whose typical fiber is isomorphic to $\ovlr{p|q}$ and whose
structure group is a subsupergroup of $GL(p|q)$
\item as a sheaf $\cv$ of $\co_\cm$-modules which is locally free of rank $p|q$.
\end{enumerate}
Even if the fiber has dimension $p|0$, it is not advisable to think of the bundle as being ``purely even''.
Its sheaf of sections is locally isomorphic to $\co_\cm^p$, and thus contains odd elements as well. For
this reason,
if $\cm$ is not purely even, the grading of the fibres of $\cv$ does not decompose the bundle into a sum
of an even and an odd bundle.

\begin{dfn}
A section of a super vector bundle $\pi:\cv\to\cm$ is a morphism $\sigma:\cm\to\cv$ of smooth supermanifolds
such that $\pi\circ\sigma=\id_\cm$.
\end{dfn}

One should always keep in mind that the fibers of a super vector bundle are, correctly speaking, not super 
vector spaces. They
are linear supermanifolds. Since there is a canonical relation between linear supermanifolds and super
vector spaces this does not always make much of a difference. But there are occasions when it plays a role.
If $\pi:\cv\to\cm$ is a super vector bundle whose fibre is the linear supermanifold $\olv=\ovlr{1|2}$, then
one can also define the bundle whose fiber is $\overline{\Pi(V)}=\ovlr{2|1}$. The underlying topological
spaces of these two bundles are not the same: the first one is a fiber space whose fibers have dimension 1
while the fibers of the second one have dimension two.

\subsection{The tangent and cotangent bundle of a supermanifold}
\label{sect:tangent}

Let $\cm$ be a smooth supermanifold of dimension $p|q$ with structure sheaf $\co_\cm$. 
Its tangent sheaf $\ctm$ is,
as in ordinary geometry, the sheaf of $\realR$-linear derivations of the structure sheaf, i.e., its
stalk at $x\in M$ is $\ider_\realR(\co_{\cm,x},\co_{\cm,x})$. Let 
$(t_1,\ldots,t_p,\theta_1,\ldots,\theta_q)$ be local coordinates on $\cm$. Then one has the coordinate
vectors $\partial/\partial t_i,\partial/\partial \theta_j$, which are defined to act on a smooth function 
$f=\sum_I f_I\theta_I$ as follows:
\begin{equation}
\pderiv{}{t_i}f:=\sum_I\pderiv{f_I}{t_i}\theta_I.
\end{equation}
Writing
\begin{equation}
f=\sum_{j\notin I}(f_I\theta_I+(-1)^{p(f_{j,I})}\theta_jf_{j,I}\theta_I),
\end{equation}
then
\begin{equation}
\pderiv{}{\theta_j}f:=\sum_I (-1)^{p(f_{j,I})}f_{j,I}\theta_I.
\end{equation}

\begin{thm}
Let $\cm=(M,\co_\cm)$ be a supermanifold of dimension $p|q$. Let $(t_1,\ldots,t_p,\theta_1,\ldots,\theta_q)$
be local coordinates on $\cm$. Then its tangent sheaf $\ctm$ is a locally
free module over $\co_\cm$ with basis $\partial/\partial t_i,\partial/\partial \theta_j$ for $1\leq i\leq p$
and $1\leq j\leq q$.
\end{thm}

For a proof see \cite{Itttos}. This theorem actually legitimizes the term ``tangent bundle`` 
for supermanifolds. Sections of the tangent bundle are called vector fields.

\begin{dfn}
The cotangent sheaf $\ct^*\cm$ of a supermanifold $\cm$ is the sheaf of modules dual to $\ctm$ in the sense
of section \ref{sect:dualmod}, i.e., the sheaf $\ihom(\ctm,\co_\cm)$. Sections of $\ct^*\cm$ are called 
(differential) 1-forms.
\end{dfn}

At this point one has to fix the conventions regarding the definition of the duality pairing between vector
fields and forms as well as the handling of the cohomological degree of differential forms. Both
$\ctm$ and $\ct^*\cm$ will be treated as left modules.
We will always denote the pairing between 1-forms and vector fields as a map
\begin{equation}
\label{sforms}
\langle\cdot,\cdot\rangle:\ctm\otimes_{\co_\cm}\ctm^*\to\co_\cm.
\end{equation}
This implies $\langle fX,g\omega\rangle=(-1)^{p(g)p(X)}\langle X,\omega\rangle$ for any vector 
field $X$, 1-form $\omega$ and functions $f,g$. 

The differential is defined as the unique derivation $d:\co_\cm\to\ct^*\cm$ for which
\begin{equation}
\langle X,df\rangle=X(f)
\end{equation}
holds, with $X$ any vector field and $f$ any function. The local dual basis for $\ct^*\cm$ is denoted by
$dt_1,\ldots,dt_p,d\theta_1,\ldots,d\theta_q$. 
\begin{dfn}
The algebra of differential forms on a smooth supermanifold $\cm$ is defined to be
\begin{equation}
\Omega^\bullet_\cm=\wedge^\bullet\ct^*\cm.
\end{equation}
\end{dfn}

Thus the algebra of differential forms is formally defined as usual. The crucial difference is that
$d\theta_i\wedge d\theta_j=d\theta_j\wedge d\theta_i$ (cf. definition \ref{svectextalg}). Therefore, 
if the odd dimension of $\cm$ is not zero, $\Omega^n_\cm\neq 0$ for all $n\geq 0$. There are no top-degree 
forms. Integration on supermanifolds instead involves sections of the Berezinian line bundle \cite{SoS},
\cite{SfMAI}, \cite{NoS_QFaSACfM}.
As in classical geometry, there exists a unique extension of $d$ to a derivation of square zero on the
algebra $(\Omega^\bullet_\cm,\wedge)$, acting by
\begin{equation}
d(\omega\wedge\eta)=(d\omega)\wedge\eta+(-1)^{\deg(\omega)}\omega\wedge (d\eta).
\end{equation}

Many further properties of differential forms carry over to supermanifolds, e.g. the Poincar\'e lemma
holds \cite{NoS_QFaSACfM}.

%% file: categories.tex
\chapter{The categorical approach}
\label{ch:categories}

While ringed spaces might seem easier to handle at first and are also
more wide\-spread in the mathematical literature about supersymmetry, the categorical approach to supergeometry
is, in a way, more fundamental. The ringed space formulation has the disadvantage of hiding certain 
aspects of
supergeometry, in particular the origin of odd parameters appearing in geometric constructions. The
approach of working with ``functorial points'', i.e., morphism sets from other objects of the category
to the object in question, is in a sense also closer to the way the physicists usually do their calculations.

\section{Preliminaries}

\subsection{Notation}

The object class of a category $\catc$ will be denoted as $\Ob(\catc)$.
If $A$ is an object of the category $\catc$, we simply write $A\in\catc$ (instead of 
$A\in\Ob(\catc)$).

The class of all morphisms of a category $\catc$ will be denoted as $\Mor\,\catc$.

Let $F,G:\catc\to\catd$ be two functors. We will denote a functor morphism $\eta:F\to G$ as a family
of maps
\begin{equation}
\eta=\{\eta_A:F(A)\to G(A)\in\Hom_\catd(F(A),G(A))\big|A\in\catc\}.
\end{equation}

For the category of covariant functors $\catc\to\catd$, we write $\catd^\catc$.

The category $\cattwo$ is the category with two objects, $\Ob(\cattwo)=\{1,2\}$, and only identity arrows.

\subsection{Representable functors, Yoneda embedding}

Let $\catc$ be a (nonempty) category, $A$ be an object of $\catc$, and let $F:\catc\to\catsets$ be a functor.
Let further
\begin{equation}
\eta:\Hom_\catc(A,-)\to F
\end{equation}
be a functor morphism of the covariant Hom-functor to $F$.

\begin{thm}
The functor $H^*:\catc^\circ\to\catsets^{\catc}$, given by
\begin{eqnarray*}
A &\mapsto& \Hom(A,-)\\
(f:A\to B) &\mapsto& \left(\begin{array}{rcl}
\Hom(B,-) &\to& \Hom(A,-)\\
u &\mapsto& u\circ f\end{array}\right)
\end{eqnarray*}
defines a full embedding called the Yoneda embedding.
\end{thm}

A proof can be found in many books on category theory or homological algebra, e.g., in \cite{CftWM},
\cite{MoHA}, \cite{KI}. Replacing $\catc$ with $\catc^\circ$, one obtains an embedding 
$H_*:\catc\to\catsets^{\catc^\circ}$. The sets $T(A):=\Hom(A,T)$ are called the 
$A$-points of $T$. If $\catc$
possesses a terminal object $t$, the $t$-points $\Hom(t,M)$ are simply called the \emph{points} of $M$.
The above statement (usually referred to as the Yoneda lemma) allows one to consider
each category $\catc$ as a full subcategory of $\catsets^{\catc^\circ}$.
If one finds it more convenient, one may therefore avoid
working with objects of $\catc$ directly, working with their points instead, which are all sets. Thinking
of familiar categories like manifolds or vector spaces, this might seem an unnecessary complication, 
but in algebraically more involved settings like, e.g., schemes, this approach is a standard tool.

\begin{dfn}
A functor $F:\catc\to\catsets$ is called representable if it is isomorphic to some $\Hom(A,-)$-functor.
$A$ is then called the representing object of $F$.
\end{dfn}

\subsection{Generators}

Luckily, to check the representability of a given functor $F:\catc\to\catsets$, it is in many categories
not necessary to look at all point sets $F(A)$ for all $A$. As an example, consider the 
one-element set $\{*\}$. For any set $X$, one has
\begin{equation}
\Hom_\catsets(\{*\},X)\cong X.
\end{equation}
This reflects the obvious fact that a set can be described by knowing all ways to throw a single element 
into it. Some other
categories also furnish a single object with this property. For example, for the category of smooth manifolds,
the manifold $\Spec\realR=(\{*\},\realR)$ consisting of a single point with structure sheaf $\realR$ does 
the job. It will turn out that $\catsman$ does not possess this nice property.

\begin{dfn}
Let $\catc$ be a category. A set $\{G_i\}_{i\in I}$ for some index set $I$ is called a set of generators
for $\catc$ if for any pair $f,g:A\to B$ of distinct morphisms between objects $A,B\in\catc$ 
there exists $i\in I$ and a morphism $s:G_i\to A$ such that the compositions
\begin{equation}
\xymatrix{G_i \ar[r]^s & A \ar@<0.5ex>[r]^f \ar@<-0.5ex>[r]_g & B}
\end{equation}
are still distinct.
\end{dfn}
A set of generators is thus able to keep all distinct morphisms distinct. Obviously, if a generator set exists, then it need not at all be unique. To find a suitable
one can be non-trivial. The Hom-sets $\Hom(G_i,X)$ can be used to describe the object $X$ but in general they
do not encode the full information about $X$. For example, knowing the set $\Hom(\Spec\fieldK,M)$ for a smooth manifold is not
sufficient to reconstruct $M$. In fact, this set only reproduces the underlying set of $M$ but forgets about
the topology and smooth structure. Nonetheless these sets can be very useful if we retain the missing information
seperately. In the above example we \emph{can} reproduce $M$ if we keep the inner Hom object $\ihom(\Spec\fieldK,M)$ instead of just the Hom-set. In fact, this inner Hom object is canonically isomorphic to $M$ as one easily
checks from the definition.

Conversely, if we give a set $S$ a topology and a manifold structure we may say we define a manifold by assuming
that this set is now an inner Hom object. In the
case of ordinary manifolds this is, of course, a trivial statement. But this is due to the fact that there exists
a single generator in this case. For more subtle objects which cannot be reduced to a single set with structure
this reasoning opens up a way to construct them, namely be specifying all the data needed to describe an inner
Hom-object (by it's functor of points) and then to prove that there exists an object which represents it.

The following Proposition is sometimes useful.

\begin{prop}
\label{gensvect}
$\{\fieldK^{1|1}\}$ is a generator set for $\catsvect_\fieldK$.
\end{prop}
\begin{proof}
First we prove that $\{\fieldK^{1}\}$ is a generator set for the category of
ordinary $\fieldK$-vector spaces. Let $V,W$ be two such vector spaces, and let $f,g:V\to W$ be two distinct 
linear maps.
Since $f\neq g$, there exists $v\in V,\,v\neq 0$, such that $f(v)\neq g(v)$. Then $f,g$ differ on the
one-dimensional linear subspace spanned by $v$ in $V$. Let $z$ denote the basis vector of $\fieldK^{1}$.
Then the linear map $\varphi:\fieldK^1\to V$ given by $\varphi(z)=v$ separates $f$ and $g$.

Now since two morphisms of super vector spaces may differ either in an even or an odd subspace, we need
one even and one odd dimension to separate all distinct morphisms between super vector spaces.
\end{proof}

The set $\{\fieldK^{1|0},\fieldK^{0|1}\}$ is evidently a set of generators as well. Often it is more
convenient to use this set than $\fieldK^{1|1}$.

\subsection{Direct products, direct sums and fiber products}
\label{sect:catprod}

As seen in the previous chapter, the category $\catsman$ has direct products. There is also a categorical
characterization of products \cite{MoHA}. Following the representable functor philosophy, a direct 
product $X\times Y$ is the object $Z$ which represents the functor $\catc\to\catsets$ given by
\begin{equation}
U\mapsto\{X(U)\times Y(U)\}\textrm{ in }\catsets.
\end{equation}
Thus, by definition, the functor of points commutes with direct products. 

An alternative equivalent 
definition states that the direct product of $X$ and $Y$ is the limit \cite{MoHA} of the functor
\begin{eqnarray}
\label{functprod}
F:\cattwo &\to& \catc\\
\nonumber
F(1)=X &,& F(2)=Y.
\end{eqnarray}
That means the product is an object $Z$ with projection morphisms 
$\pi_X:Z\to X$, $\pi_Y:Z\to Y$ which is universal: 
for any other object $Z'$ also possessing projections $\pi'_X,\pi'_Y$ there exists a unique 
morphism $h:Z'\to Z$ such that $\pi_X\circ h=\pi'_X$ and $pi_Y\circ h=\pi'_Y$.

The direct sum $X\oplus Y$ is the ``co''-version of the direct product. In homological terms, 
it is the colimit of the functor (\ref{functprod}). This means it is an object $Z$ with canonical injections
$\iota_X:X\to Z$, $\iota_Y:Y\to Z$ which is universal among all objects with these properties.

Analogously, given morphisms $\phi:X\to T$ and $\psi:Y\to T$, one can define the fiber product 
$X\times_T Y$ as the representing object for the functor
\begin{equation}
U\mapsto\{X(U)\times_T Y(U)\}\textrm{ in }\catsets,
\end{equation}
or as a universal object with projections $p_X:X\times_T Y\to X$, $p_Y:X\times_T Y\to Y$, such that
$\phi\circ p_X=\psi\circ p_Y$.
But there is another way to define the fiber product which is quite insightful.

\begin{dfn}
\label{covert}
Let $\catc$ be some category and $T\in\catc$ be an object. Define the category $\catc/T$ of 
objects over $T$ in the following way. An object of $\catc/T$ is a pair $(X,\phi)$, where $X\in\catc$ is
an object and $\phi$ is
an arrow $X\to T$. A morphism $\eta:(X,\phi)\to(Y,\varphi)$ is a morphism $\eta\in\Hom_\catc(X,Y)$ such
that
\begin{equation}
\xymatrix{X \ar[rr]^\eta \ar[dr]_\phi && Y \ar[dl]^\varphi\\
          &T&}
\end{equation}
commutes. Composition of morphisms consists in the juxtaposition of commutative triangles.
\end{dfn}

The category $\catc/T$ is sometimes also called an arrow category or comma category. It is, in geometric
terms, the category of bundles over the base $T$. Then one can show:

\begin{prop}
\label{fiberp}
Let $\catc$ be some category with direct products and let $T$ be an object of $\catc$. 
Let $\phi:X\to T$ and $\psi:Y\to T$ be two objects of $\catc/T$. Then the fiber product $X\times_S Y$
is the direct product in $\catc/T$, i.e. there exists a canonical isomorphism
\begin{equation}
(X\times_T Y\textrm{ in }\catc)\cong(\phi\times\psi\textrm{ in }\catc/T).
\end{equation}
\end{prop}
\begin{proof}
Denote the canonical projections of the fiber product as above as $p_X:X\times_T Y\to X$ and
$p_Y:X\times_T Y\to Y$. In the category $\catc/T$, the direct product of $\phi$ and $\psi$ is a morphism
$h:Z\to T$ together with morphisms $\pi_X:Z\to X$ and $\pi_Y:Z\to Y$ such that the following diagram commutes:
\begin{equation}
\xymatrix{Z \ar[r]^{\pi_Y} \ar[dr]^h \ar[d]_{\pi_X} & Y \ar[d]^\psi\\
X \ar[r]_\phi & T}.
\end{equation}
Therefore we have $\psi\circ\pi_Y=\phi\circ\pi_X$ and universality then entails that
there exists a unique isomorphism $\eta:Z\to X\times_T Y$.
\end{proof}

Thus, the fiber product of a pair of morphisms of two objects of $\catc$ into the same base can be thought 
of as the direct product in the category of bundles of objects of $\catc$ over that base.

\section{Algebraic structures in a category}

In this entire section, $\catc$ denotes a category with direct products and a terminal object
$t$. An algebraic structure in $\catc$ is an object of $\catc$ which also ``has'' this 
type of structure. For concrete categories, i.e., those whose objects all have a natural description
as sets with additional properties, such constructions are well-known.
For example, a Lie group is a group which is also a manifold and whose operations are diffeomorphisms,
or a topological vector space is a topological space which is also a module over a field and 
whose operations are continuous. 
The functor of points enables one to extend this concept far beyond this concrete context. The
idea is, however, still intriguingly simple: one defines an object $T\in\catc$ to be some
algebraic structure, if it induces this structure on all sets of morphisms of other objects of
$\catc$ into it. A group in $\catc$, for example, is then an object which induces a group structure on all 
Hom-sets of
other objects to it. In the concrete cases, this coincides with the usual intuition.

\subsection{Rings, modules, algebras}

\begin{dfn} An object $R\in\catc$ is called a ring with unit in $\catc$, if there exist morphisms
\begin{eqnarray}
+:R\times R &\to& R\\
\cdot: R\times R &\to& R\\
e: p &\to& R,
\end{eqnarray}
such that for every $Y\in\catc$, the sets $R(Y)=\Hom(Y,R)$ together with the morphisms 
$+_Y:R(Y)\times R(Y)\to R(Y)$ and $\cdot_Y:R(Y)\times R(Y)\to R(Y)$ form a ring and
$e_Y:p(Y)\to R(Y)$ is its unit. The ring is called commutative if all rings $R(Y)$ are commutative.
\end{dfn}

We have implicitly considered $R$ as a functor $\catc\to\catsets^{\catc^\circ}$ in this definition, 
namely as its
associated Hom-functor $\Hom(-,R)$. Consequently, $+,\cdot,e$ have been turned into functor morphisms,
and $+_Y,\cdot_Y$, etc. denote their components.
We are free to do so by the Yoneda lemma, and will continue to switch between an object and its contravariant
Hom-functor whenever this is suitable. 

\begin{dfn}
Let $R$ be a ring in $\catc$. An object $V\in\catc$ is called a left $R$-module if there exists a
morphism $\rho:R\times V\to V$ such that for all $Y\in\catc$, the sets $(V(Y),\rho_Y)$ are left 
$R(Y)$-modules. Right modules are analogously defined. Let $V,V'$ be $R$-modules. A morphism $f:V\to V'$ of
$R$-modules is a functor morphism 
\begin{equation}
f=\{f_Y:V(Y)\to V'(Y)\big|Y\in\catc\},
\end{equation}
such that each $f_Y$ is a morphism of $R(Y)$-modules.
\end{dfn}

Together with such morphisms the right (left) $R$-modules in $\catc$ form a category $\catmod_R(\catc)$
(${}_R\catmod(\catc)$). Since all the points $V(Y)$ of a module $V$ are abelian groups, this category is
additive. Supermodules can be defined completely analogously.

\begin{dfn}
An $R$-module $V$ in $\catc$ is called an $R$-supermodule if it possesses a direct sum decomposition
\begin{equation}
V=V_{\bar{0}}\oplus V_{\bar{1}}.
\end{equation}
$V_{\bar{0}}$ will be called the even submodule, $V_{\bar{1}}$ the odd one. A morphism of $R$-super modules
is a morphism of $R$-modules which preserves the parity.
\end{dfn}

Note that we did not define $R$ itself to be a superring. This is possible, but we do not need it
here.
For supermodules to exist, $\catc$ must obviously allow finite direct sums, which will be the case for
all categories that will be dealt with in this work. Left (resp. right) $R$-super modules in $\catc$ form a 
category ${}_R\catsmod(\catc)$ (resp. $\catsmod_R(\catc)$).
Commutative, resp. supercommutative rings will be sufficient for almost all our purposes. In this case,
each left module can be turned into a right module in a natural way (cf. the convention (\ref{modstr}).

\begin{dfn}
Let $R$ be a commutative ring in $\catc$ and let $A$ be an $R$-module. 
If there exists an $R$-bilinear morphism $\mu:A\times A\to A$ such that for each $Y\in\catc$ the pair
$(A(Y),\mu_Y)$ is an $R(Y)$-algebra, the pair $(A,\mu)$ is called an $R$-algebra. The algebra is 
called (anti-)commutative, Lie, or associative if each of the algebras $A(Y)$ has the corresponding 
property.
\end{dfn}

Modules over $R$-algebras and morphisms of $R$-algebras are defined in the obvious way.

\subsection{Multilinear morphisms}
\label{sect:multmor}

\begin{dfn}
\label{def:multmor}
Let $R$ be a ring in $\catc$ and let $V_1,\ldots,V_n,V$ be $R$-modules. Let $Z$ be some object of $\catc$.
A morphism $f:Z\times V_1\times\ldots\times V_n\to V$ is called a $Z$-family of $R$-$n$-linear morphisms if
for every $Y\in\catc$, the map
\begin{equation}
f_Y:Z(Y)\times V_1(Y)\times\ldots\times V_n(Y)\to V(Y)
\end{equation}
is a $Z(Y)$-family of $R(Y)$-$n$-linear maps, i.e., for every $z\in Z(Y)$, $f_Y(z,\ldots)$ is 
$R(Y)$-$n$-linear. 
\end{dfn}

The set of $Z$-families of $R$-$n$-linear morphisms from $V_1\times\ldots\times V_n$ to $V$ will 
be denoted by $L_R^n(Z;V_1,\ldots,V_n;V)$. It inherits the structure of an abelian group from $V(Y)$:
\begin{equation}
(f+f')_Y=f_Y+f'_Y.
\end{equation}

Besides this, $L_R^n(Z;V_1,\ldots,V_n;V)$ can be given a natural $R(Z)$-module structure. If
$f:Z\times V_1\times\ldots\times V_n\to V$ is an $R$-$n$-linear morphism and $r:Z\to R$ is an element 
of $R(Z)$ one sets
\begin{equation}
(rf)_Y(z,v_1,\ldots,v_n)=r_Y(z)f_Y(z,v_1,\ldots,v_n)
\end{equation}
for $z\in Z(Y)$ and $v_i\in V_i(Y)$. This allows us to consider $L_R^n(Z;-;-)$ as a functor
$(\catmod_R(\catc)^\circ)^n\times\catmod_R(\catc)\to\catmod_{R(Z)}(\catsets)$.

A family of multilinear morphisms over the terminal object $p$ is just called an $R$-$n$-linear morphism.
The abelian groups of such morphisms are denoted simply by $L_R^n(V_1,\ldots,V_n;V)$.

For supermodules things are quite similar. One can show \cite{I-dZ2ks}, \cite{SoaE:ATTC} that the functor
$L_R^n(Z;-;-)$ defined above commutes with finite direct sums. This allows one to equip
the set of $R$-$n$-linear maps between $R$-supermodules with the structure of a $R(Z)$-supermodule
by decomposing it as
\begin{equation}
(L_R^n)_{\bar{i}}(Z;V_1,\ldots,V_n;V)=%
\bigoplus_{\bar{j}+\sum \bar{i}_k=\bar{i}}L_R^n(Z;V_{1,\bar{i}_1},\ldots,V_{n,\bar{i}_n};V_{\bar{j}}).
\end{equation}
The functor $L_R^n(Z;-;-)$ thus becomes a functor into the category of $R(Z)$-super modules.

\subsection{Internal Hom-functors and tensor products}
\label{sect:ihom}

In many of the most interesting cases, the sets of multilinear maps can be internalized, i.e. turned into
modules themselves. Clearly, these will be the inner Hom-objects for the category $\catmod_R(\catc)$. One has
\begin{equation}
L_R^n(V_1,\ldots,V_n;V)=\Hom_{\catmod_R(\catc)}(V_1\times\ldots\times V_n;V).
\end{equation}
Therefore, an inner Hom-object $\cl_R^n(V_1,\ldots,V_n;V)$ has to satisfy
\begin{equation}
\xymatrix{
L_R^1(W;\cl_R^n(V_1,\ldots,V_n;V)) \ar@{=}[d] & L_R^{n+1}(W,V_1,\ldots,V_n;V) \ar@{=}[d]\\
\Hom(W,\cl_R^n(V_1,\ldots,V_n;V)) \ar@{}[r] |{\cong} & \Hom(W\times V_1\times\ldots\times V_n;V)}
\end{equation}
for all $R$-modules $W$, and these isomorphisms must be functorial in $W$. In different terms,
$\cl_R^n(V_1,\ldots,V_n;V)$ has to be a corepresenting object for the functor
\begin{eqnarray}
L_R^{n+1}(-,V_1,\ldots,V_n;V):\catmod_R(\catc)^\circ &\to& \catsets\\
\nonumber
W &\mapsto& L_R^{n+1}(W,V_1,\ldots,V_n;V).
\end{eqnarray}

If these inner Hom-objects do all exist, one can go further and construct an internal functor
\begin{equation}
\cl_R^n:(\catmod_R(\catc)^\circ)^n\times\catmod_R(\catc)\to\catmod_R(\catc),
\end{equation}
assigning to each sequence of domain modules $V_1,\ldots,V_n$ and each target module $V$ the inner 
Hom-object of multilinear maps $V_1\times\ldots\times V_n\to V$.

Closely related are tensor products. The category $\catc$ is said to possess tensor products over a
ring $R\in\catc$ if for every set $V_1,\ldots,V_n$ of $R$-modules there exists an $R$-module
$V_1\otimes_R\ldots\otimes_R V_n$ such that the functors $L_R^1(V_1\otimes_R\ldots\otimes_R V_n;-)$ and
$L_R^n(V_1,\ldots,V_n;-)$ are isomorphic. That means
\begin{equation}
\label{tprodcat}
L^1_R(V_1\otimes_R\ldots\otimes_R V_n;W)\cong L_R^n(V_1,\ldots,V_n;W)\qquad\forall W\in\catmod_R(\catc),
\end{equation}
and these sets have to behave functorially under maps $\phi:W\to W'$. But that is just the condition
that the functor
\begin{eqnarray}
L_R^n(V_1,\ldots,V_n;-):\catmod_R(\catc) &\to& \catsets\\
\nonumber
W &\mapsto& L_R^n(V_1,\ldots,V_n;W)
\end{eqnarray}
be representable. This definition coincides for ordinary rings and modules with the usual one. In
particular, the
tensor product is again a universal object: any other module which satisfies (\ref{tprodcat}) is canonically
isomorphic to $V_1\otimes_R\ldots\otimes_R V_n$.

The existence and properties of internal Hom-functors and tensor products are, in general, more subtle 
than presented here. 
One usually has to assure that they are \emph{coherent}, i.e., they
behave as their ordinary counterparts in $\catmod_R(\catsets)$ with respect to associativity and the
unit of $\otimes$ (if there is one). For the categories
of importance for the problems in this work we can take coherence for granted by the results of
Molotkov \cite{I-dZ2ks}. For an in-depth study of such questions, consult \cite{SoaE:ATTC}.

\section{The functor of points for supermanifolds}
\label{sect:functsman}

\subsection{Interpretation of the functor of points in terms of families}
\label{sect:functpoints}

The sets of points of an object have a nice interpretation which shows that despite their rather abstract
definition they are quite close to geometric reasoning. 

The category $\catc/T$ constructed in Definition \ref{covert} has a terminal object: the identity
arrow $\id_T:T\to T$.
If the category $\catc$ has finite direct products,
which we will always assume, then one has a natural functor
$T^*:\catc\to\catc/T$ which assigns to any object its trivial family over $T$:
\begin{equation}
\label{def:tstar}
T^*:X\mapsto (\pi_T:T\times X\to T).
\end{equation}
To each morphism $\varphi:X\to Y$ it assigns $(\id_T,\varphi):T\times X\to T\times Y$.

\begin{prop}
\label{points}
The map $\Xi:\Hom_\catc(T,X)\to\Hom_{\catc/T}(\id_T,T^*(X))$ which assigns to each $\varphi:T\to X$ the
morphism of families
\begin{equation}
\xymatrix{T \ar[rr]^{(\id_T,\varphi)} \ar[dr]_{\id_T} && T\times X \ar[dl]^{\pi_T}\\
          &T&}
\end{equation}
is a bijection.
\end{prop}

A proof can be found, for example, in \cite{SoaE:ATTC}.
In words: one can identify the $T$-points of $X$ with the ordinary points of $T^*(X)$ and these are nothing
other than the global sections of the projection $\pi_T:T\times X\to T$.

If $\Hom(T,X)$ and $\Hom(S,X)$ are the $T$- and $S$-points of $X$, a morphism $u:T\to S$ induces
a morphism $\Hom(S,X)\to\Hom(T,X)$ of these points via composition. From the family point of view
this morphism corresponds to a base change, i.e., a pullback of the family. To each family $S\times X\to S$ 
the map $u:T\to S$ associates the family $T\times_S X\to T$ and to every $S$-point $\phi:S\to X$
it associates $u\circ\phi:T\to X$. Constructions involving the points of some objects only make sense
if they are compatible with arbitrary base changes. In this case one calls them \emph{geometric}.

Furthermore, we have the following.

\begin{prop}
\label{dirprod}
The functor $T^*:\catc\to\catc/T$ preserves direct products, i.e., there is a canonical isomorphism
\begin{equation}
T^*(X\times Y)\cong T^*(X)\times_T T^*(Y).
\end{equation}
\end{prop}
\begin{proof}
Denote the direct product of $X,Y\in\catc$ by $Z$. This means we have canonical projections 
$\pi_X:Z\to X$ ,$\pi_Y:Z\to Y$, and $Z$ is universal with this property. The functor $T^*$ translates 
this situation into the diagram
\begin{equation}
\xymatrix{T\times X \ar[dr]_{\pi_T} & T\times Z \ar[l]_{T^*(\pi_X)} \ar[r]^{T^*(\pi_Y)} \ar[d]^{\pi_T} & 
T\times Y \ar[dl]^{\pi_T}\\
          &T&}.
\end{equation}
It has to be shown that $T\times Z$ together with the morphims $T^*(\pi_X)$ and $T^*(\pi_Y)$ is universal.
Let $\eta:Q\to T$ be another object of $\catc/T$ with morphisms $P_X:Q\to T\times X$, $P_Y:Q\to T\times Y$, 
such that
\begin{equation}
\xymatrix{T\times X \ar[dr]_{\pi_T} & Q \ar[l]_<<<<{P_X} \ar[r]^>>>>>{P_Y} \ar[d]^{\eta} & 
T\times Y \ar[dl]^{\pi_T}\\
          &T&}
\end{equation}
commutes. Denote by $p_X:T\times X\to X$ and $p_Y:T\times Y\to Y$ the canonical projections.
Then we have morphisms $p_X\circ P_X:Q\to X$ and $p_Y\circ P_Y:Q\to Y$, and thus, there
exists a unique morphism $\alpha:Q\to Z$ such that $\pi_X\circ\alpha=p_X\circ P_X$ and
$\pi_Y\circ\alpha=p_Y\circ P_Y$.

But then we have morphisms $\alpha:Q\to Z$ and $\eta:Q\to T$ and thus there exists a unique morphism 
$h:Q\to T\times Z$.
\end{proof}

By virtue of Prop.~\ref{fiberp}, we can then conclude that $T^*$ maps direct products $X\times Y$ into fiber
products $X\times_T Y$. This is important if one  has an algebraic structure on an object of the category 
$\catc$, e.g., if some $X\in\catc$ is a group in $\catc$. Then $T^*$ will produce an object with the
same structure in $\catc/T$.

\subsection{Generators for the category of supermanifolds}

In order to be able to actually work with the functor of points of a supermanifold we need a manageable
set of generators. A single generator is not enough, as it turns out.

\begin{thm}
\label{gensman}
The objects of the category $\catspoint$ (cf. section \ref{sect:spoint}) form a set of generators for
the category $\catfinsman$ of finite dimensional super manifolds.
\end{thm}
\begin{proof}
The claim is local, therefore it is sufficient to check it for superdomains. Let $\cu,\cv$ be two 
superdomains
and $f,g:\cu\to\cv$ be a pair of distinct morphisms. It is well known that $\Spec\fieldK=(\{*\},\fieldK)$ is
a generator for the categories of real smooth (resp.~complex analytic) finite-dimensional manifolds:
these are completely described by their $\fieldK$-points. Thus if the underlying smooth maps of $f,g$
are distinct they will be separated already by $\cp_0=\Spec\fieldK$. Let us therefore assume that the
underlying maps of $f$ and $g$ are identical. Since $f\neq g$, there must exist a (topological) point
$p\in U$, which gets mapped by $f$ and $g$ to the point $q\in V$ such that the sheaf maps
\begin{equation}
\tilde{f},\tilde{g}:\co_{\cv}\to\co_{\cu}
\end{equation}
differ on the stalk at $q$, and this difference must either be visible in the images of the odd coordinates
or in the even nilpotent corrections to the images of the even coordinates.
Let $x$ collectively denote the even coordinates of $\cu$ and let $\theta_1,\ldots,\theta_n$ be the 
odd coordinates of $\cu$. Let $y_1,\ldots,y_p$ be the even and $\eta_1,\ldots,\eta_r$ be the odd coordinates
of $\cv$. Let us first assume that
$\tilde{f},\tilde{g}$ differ in the image of some $\eta_i$, which, without loss of generality,
we assume to be $\eta_1$. Then we can write
\begin{eqnarray}
\tilde{f}(\eta_1) &=& \alpha_i(x)\theta_i+\alpha_{ijk}(x)\theta_i\theta_j\theta_k+\ldots\\
\tilde{g}(\eta_1) &=& \beta_i(x)\theta_i+\beta_{ijk}(x)\theta_i\theta_j\theta_k+\ldots
\end{eqnarray}
where $\alpha_I,\beta_I$ are germs of smooth (resp.~holomorphic) functions around $p\in U$. 
At least for one index $J$,
$\alpha_J\neq\beta_J$ can be assumed. 
Since these germs are entirely defined
by their values, there must be a point $p'$ in any neighbourhood of $p$ where the values of 
$\alpha_J$ and $\beta_J$ are different.
To separate
$f$ and $g$, we may then define a map $\Phi=(\phi,\varphi):\cp_N\to\cu$ by choosing $N\geq n$ and
setting $\phi(\{*\})=p'$. If $\xi_1,\ldots,\xi_N$ are the coordinates of $\cp_N$, then the sheaf map
is defined to be
\begin{eqnarray}
\nonumber
\varphi(\theta_i) &:=& \xi_i\\
\label{phimap2}
\varphi(f) &:=& f(p)
\end{eqnarray}
where $f$ is any germ of a function on the underlying manifold $U$. The image of $\eta_1$ under
$\tilde{f},\tilde{g}$ will be mapped to a different element of $\wedge[\xi_1,\ldots,\xi_N]$'s. Thus the two
morphisms remain distinct. It is clear that one must be able to choose an arbitrarily large $N$ in order to
separate all morphisms between finite-dimensional supermanifolds: to make the difference between
$\alpha_J$ and $\beta_J$ visible one needs at least as many odd coordinates as the length $|J|$ of the
index.

If we assume instead that $\tilde{f},\tilde{g}$ differ on the image of an even coordinate $y_k$, say $y_1$,
we can write
\begin{eqnarray}
\tilde{f}(y_1) &=& \alpha_0(x)+\alpha_{ij}(x)\theta_i\theta_j+\ldots\\
\tilde{g}(y_1) &=& \beta_0(x)+\beta_{ij}(x)\theta_i\theta_j+\ldots,
\end{eqnarray}
and for some $J\neq 0$, we have $\alpha_J\neq\beta_J$. By the same arguments as before, there exists a
$p''$ in any neighbourhood of $p$ for which the map (\ref{phimap2}) separates $\tilde{f},\tilde{g}$.
\end{proof}

In general, we will have to keep track of an infinite tower of sets of points, one set
for each $\Lambda_n$, $n\in\naturalN_0$. This is still much less comfortable than ordinary geometry, but
at least allows a clean handling of, e.g., odd parameters or supergroups. It will also provide a means to
extend the notion of a supermanifold beyond finite dimensions.

Theorem \ref{gensman} also means that a supermanifold can be defined as a contravariant functor 
$\catspoint\to\catsets$ which is representable, i.e., a family of sets $S_n$ such that there exists a 
supermanifold $\cm$ for which $S_n\cong\Hom(\cp_n,\cm)$ for all $n\in\naturalN_0$. By Prop.~\ref{spointeq},
we can as well use covariant functors $\catgr\to\catsets$, and this is precisely the way we will
proceed in the next section, following largely the program proposed by Molotkov \cite{I-dZ2ks}.

\subsection{Consequences of the family point of view for supermanifolds}
\label{sect:smanfpv}

Thm.~\ref{gensman} shows that it is sufficient to consider the $\cp_n$-points of any supermanifold $\cm$ to
describe it entirely. One way to construct a supermanifold is therefore to propose a functor
$F:\catspoint\to\catsets$ and show that it is representable. How to do this in detail will be discussed
later. At this point we only want to use the family point of view to explain the origin of the odd
parameters occuring in coordinate changes on supermanifolds.

The $\cp_n$-points of a supermanifold $\cm$ are morphisms $\cp_n\to\cm$. By the interpretation of
Prop.~\ref{points} they may be viewed
as \emph{local} sections of the bundle $\cp_n\times\cm\to\cp_n$. That means to study a supermanifold 
one studies families $\phi:\cm\to\cp_n$ of locally superringed spaces which are locally on $\cm$ isomorphic
to projections $\cp_n\times\cu\to\cp_n$, where $\cu$ is a superdomain. Now suppose one has found two such 
local sections,
$\cp_n\times\cu\to\cp_n$ and $\cp_n\times\cv\to\cp_n$, whose domains $U,V\subset M$ overlap. Then on the
overlap one needs an invertible transition function, i.e., an isomorphism $\psi$ such that
\begin{equation}
\xymatrix{\cp_n\times \cu \ar[rr]_{\psi} \ar[dr]_{\pi_{\cp_n}} && \cp_n\times \cv \ar[dl]^{\pi_{\cp_n}}\\
          &\cp_n&}
\end{equation}
commutes. This means that we can redefine $\psi$ as $\pi_{\cp_n}\times\psi$ where $\psi$ is now a map 
$\psi:\cp_n\times\cu\to\cv$. That it is an isomorphism means that there exists an inverse
map $\psi^{-1}:\cp_n\times\cv\to\cu$, such that the composition
\begin{equation}
\begin{CD}
\cp_n\times\cu @>{\pi_{\cp_n}\times\psi}>> \cp_n\times\cv @>{\pi_{\cp_n}\times\psi^{-1}}>> \cp_n\times\cu
\end{CD}
\end{equation}
is the identity $\id_{\cp_n\times\cu}$. The map $\psi$ may therefore involve the odd generators of
$\Lambda_n$, i.e., it is stalkwise a map $\co_{\cu,p}\otimes\Lambda_n\to\co_{\cv,\psi(p)}$. This
``mixing in'' of the odd coordinates of the base causes the occurence of the so-called odd parameters in
coordinate transformations on a supermanifold.

One may also interpret the occurence of odd parameters in terms of inner Hom-objects. For two supermanifolds
$\cm$ and $\cn$ an inner Hom object $\ihom(\cm,\cn)$ would have to be a supermanifold of supersmooth maps 
satisfying 
\[
\Hom(\cx,\ihom(\cm,\cn))=\Hom(\cx\times\cm,\cn)
\]
for all supermanifolds $\cx$. Since we
know that the superpoints generate $\catfinsman$, it will be sufficient to verify that one has
\begin{equation}
\label{ihomsman}
\Hom(\cp_n,\ihom(\cm,\cn))=\Hom(\cp_n\times\cm,\cn)\qquad\forall n\in\naturalN_0.
\end{equation}
Thus, a morphism $\cp_n\times\cm\to\cn$ may be thought of as a ``higher point'' of a supermanifold of 
morphisms $\ihom(\cm,\cn)$, as opposed to the $\fieldK$-points which are just the ordinary morphisms
$\ihom(\cm,\cn)_{red}=\Hom(\cm,\cn)$. This point of view will be pursued further when the supergroup
of superdiffeomorphisms will be studied, which is the restriction of the 
inner Hom object $\ihom(\cm,\cm)$ to the (ordinary) group $\Aut(\cm)$. We will also take it up when 
supermanifolds of sections of vector bundles on a supermanifold are constructed.

\section{The categorical formulation of linear and commutative superalgebra}
\label{sect:catsalg}

As a first step towards a fully categorical formulation of supergeometry, we will translate linear
and commutative
superalgebra into the language of the functor of points. By the remark following Prop. \ref{gensvect},
the only points that would matter for a linear superspace are the $\Lambda_0$- and $\Lambda_1$-points. 
In view of the later
extension of the construction to algebras and general supermanifolds, we will nonetheless use
a bigger set of generators right away, namely all finite dimensional Grassmann algebras. The construction
outlined below can be viewed as a systematic treatment of the so-called even rules principle, which
is a way to do superalgebra without having to handle odd quantities \cite{NoS_QFaSACfM}, \cite{SfMAI}.

\subsection{The rings $\olr$ and $\olc$}
All considerations in this and the following sections refer to the category $\catsets^\catgr$ of covariant
functors $\catgr\to\catsets$. 

For the case when we work over the reals, i.e. with the category $\catgr$ of real Grassmann algebras,
we define in $\catsets^\catgr$ a commutative ring $\olr$ with unity by setting
\begin{eqnarray}
\olr(\Lambda) &:=& (\Lambda\otimes_\realR \realR)_{\bar{0}}=\Lambda_{\bar{0}},\\
\olr(\varphi) &:=& \varphi\big|_{\Lambda_{\bar{0}}}
\end{eqnarray}
for $\varphi:\Lambda\to\Lambda'$ a morphism in $\catgr$.\footnote{To tensor with $\realR$ over
$\realR$ is, of course, not necessary in this definition. We merely used this notation to stress the 
analogy with the construction of modules in the next section.} The ring structure is directly inherited from
the one of the $\Lambda_{\bar{0}}$.

If we work over $\complexC$, i.e., with $\catgr^\complexC$, then we analogously define 
\begin{eqnarray}
\olc(\Lambda^\complexC) &:=& (\Lambda^\complexC\otimes_\complexC \complexC)_{\bar{0}}=%
\Lambda^\complexC_{\bar{0}},\\
\olc(\varphi) &:=& \varphi\big|_{\Lambda^\complexC_{\bar{0}}}.
\end{eqnarray}

These two rings will play the role usually played by the ground field in the definition of linear spaces.

\subsection{Superrepresentable $\olk$-modules in $\catsets^\catgr$}

In this section we will introduce a particular class of $\olr$- and $\olc$-modules in the 
functor category $\catsets^\catgr$, which, as it will turn out, can be used instead of super vector spaces
in the categorical approach.

Let $V$ be some real or complex super vector space. We define a functor $\olv\in\catsets^\catgr$ by
setting
\begin{eqnarray}
\label{def:olv}
\olv(\Lambda) &:=& (\Lambda\otimes_\fieldK V)_{\bar{0}}=\left(\Lambda_{\bar{0}}\otimes V_{\bar{0}}\right)\oplus%
\left(\Lambda_{\bar{1}}\otimes V_{\bar{1}}\right)\\
\nonumber
\olv(\varphi) &:=& (\varphi\otimes\mathrm{id}_{V})\big|_{\olv(\Lambda)}\qquad\mathrm{for}\quad%
\varphi:\Lambda\to\Lambda'.
\end{eqnarray}
Here $\Lambda$ means the $\fieldK$-version of the Grassmann algebras.
All sets $\olv(\Lambda)$ are naturally $\Lambda_{\bar{0}}$-modules, and thus, $\olv$ is a $\olk$-module.

Let $f:V_1\times\ldots\times V_n\to V$ be a multilinear map of $\fieldK$-super vector spaces. To $f$,
one assigns the functor morphism $\olf:\olv_1\times\ldots\times\olv_n\to\olv_n$ whose components
\begin{equation}
\olf_\Lambda:\olv_1(\Lambda)\times\ldots\times\olv_n(\Lambda)\to\olv(\Lambda)
\end{equation}
are defined by
\begin{equation}
\label{fbar}
\olf_\Lambda(\lambda_1\otimes v_1,\ldots,\lambda_n\otimes v_n)=\lambda_n\cdots\lambda_1\otimes%
f(v_1,\ldots,v_n)
\end{equation}
for all $\lambda_i\otimes v_i\in\olv_i(\Lambda)$. All maps $\olf_\Lambda$ are clearly $\Lambda_{\bar{0}}$-linear,
hence $\olf$ is a $\olk$-$n$-linear morphism in $\catmod_{\olk}(\catsets^\catgr)$.

\begin{prop}[see also \cite{I-dZ2ks}]
\label{isokmod}
The assignment $f\mapsto\olf$, which is a map
\begin{equation}
L^n_{\fieldK,\bar{0}}(V_1,\ldots,V_n;V)\to L^n_{\olk}(\olv_1,\ldots,\olv_n;\olv)
\end{equation}
for any tuple $V_1,\ldots,V_n,V$ of $\fieldK$-super vector spaces, is an isomorphism of $\fieldK$-modules.
\end{prop}
\begin{proof}\footnote{I am grateful to V. Molotkov for pointing out an error in my original proof of this
Proposition and for sending me a correct version.}
Remember that the $\fieldK$-module structure on $L^n_{\olk}(\olv_1,\ldots,\olv_n;\olv)$ is due to the
fact that $L^n_{\olk}(\olv_1,\ldots,\olv_n;\olv)$ is a family of multilinear maps over the final object
$\fieldK$ of $\catgr$, and thus, $L^n_{\olk}(\olv_1,\ldots,\olv_n;\olv)$ has the natural structure of a
$\olk(\fieldK)=\fieldK$-module.

Definition (\ref{fbar}) assigns to every $f\in L^n_{\fieldK,\bar{0}}(V_1,\ldots,V_n;V)$ a functor
morphism $\olf$. We have to
show that one can reconstruct a unique $f$ from a given $\olf\in L^n_{\olk}(\olv_1,\ldots,\olv_n;\olv)$. 
Let a seqeuence $\lambda_i\otimes v_i\in \olv_i(\Lambda)$
be given, $1\leq i\leq n$ and let $j\leq n$ of these $v_i$ be odd, assuming for simplicity that these are
the first $j$. By $\Lambda_{\bar{0}}$-linearity we then have
\begin{equation}
\olf_\Lambda(\lambda_1\otimes v_1,\ldots,\lambda_n\otimes v_n)=%
\lambda_n\cdots\lambda_{j+1}\olf_\Lambda(\lambda_1\otimes v_1,\ldots,\lambda_j\otimes v_j,%
1\otimes v_{j+1},\ldots,1\otimes v_n).
\end{equation}
Now consider the morphism $\eta:\Lambda_j\to\Lambda$ which is defined by $\eta(\theta_k)=\lambda_k$,
where $\theta_k$, $1\leq k\leq j$ are the odd generators of $\Lambda_j$. In order to prove the
statement of the Proposition it will now be enough to show that there exists a unique
\[
g\in L^n_{\fieldK,\bar{0}}(V_1,\ldots,V_n;V)
\]
such that
\begin{equation}
\label{unig}
f_{\Lambda_j}(\theta_1\otimes v_1,\ldots,\theta_j\otimes v_j,1\otimes v_{j+1},\ldots,1\otimes v_n)=%
\theta_i\cdots\theta_1\otimes g(v_1,\ldots,v_n).
\end{equation}
This is sufficient because the fact that $f$ is a functor morphism implies that we have a commutative
square
\begin{equation}
\label{lambfunc}
\xymatrix{
\olv_1(\Lambda_j)\times\ldots\times\olv_n(\Lambda_j) \ar[rrr]^{\olv_1(\eta)\times\ldots\times\olv_n(\eta)}
\ar[d]_{\olf_{\Lambda_j}} &&& \olv_1(\Lambda)\times\ldots\times\olv_n(\Lambda) \ar[d]^{\olf_\Lambda} \\
\olv(\Lambda_j) \ar[rrr]^{\olv(\eta)} &&& \olv(\Lambda)
}
\end{equation}
and this then entails that
\[
\olf_\Lambda(\lambda_1\otimes v_1,\ldots,\lambda_n\otimes v_n)=%
\lambda_n\cdots\lambda_1\otimes g(v_1,\ldots,v_n)
\]
for a unique even linear map $g$.

To prove that a unique $g$ as claimed in eq. (\ref{unig}) exists, we first observe that the most general
expression that could appear on the right hand side of eq. (\ref{unig}) reads
\begin{multline}
\label{gensum}
f_{\Lambda_j}(\theta_1\otimes v_1,\ldots,\theta_j\otimes v_j,1\otimes v_{j+1},\ldots,1\otimes v_n)=%
\theta_i\cdots\theta_1\otimes g(v_1,\ldots,v_n)+\\
+\sum_{\substack{m<j \\ j_1>\ldots> j_m}}%
\theta_{j_1}\cdots\theta_{j_m}\otimes g_{j_1\cdots j_m}(v_1,\ldots,v_n),
\end{multline}
where all $g_{j_1\cdots j_m}$ are linear maps. To show that the sum on the right hand side equals zero, 
we again use functoriality (\ref{lambfunc}), this time for the morphism 
\begin{eqnarray*}
\varphi_l:\Lambda_j &\to& \Lambda_j\\
\varphi_l(\theta_k) &=& 0\qquad\mathrm{if}\quad k=l\\
\varphi_l(\theta_k) &=& \theta_k\qquad\mathrm{if}\quad k\neq l.
\end{eqnarray*}
This evidently kills all summands which contain $\theta_l$ and yields
\[
0=\sum_{\substack{m<j \\ l\notin\{j_1,\ldots, j_m\}}}%
\theta_{j_1}\cdots\theta_{j_m}\otimes g_{j_1\cdots j_m}(v_1,\ldots,v_n).
\]
We obtain such an equation for each $1\leq l\leq j$, therefore the sum on the right hand side of
(\ref{gensum}) must equal zero.
\end{proof}

This Proposition is of central importance for all further constructions. As a first application, we obtain
the following Corollary which shows that $\olk$-modules in $\catsets^\catgr$ are just the right objects
to fully capture the structure of a super vector space, i.e., of a $\fieldK$-supermodule in $\catsets$.

\begin{cor}[see also \cite{I-dZ2ks}]
\label{cor:barfunct}
The assignment $V\mapsto\olv$ and $f\mapsto\olf$ defines a fully faithful functor
\begin{equation}
\olcdot:\catsmod_\fieldK(\catsets)\to\catmod_{\olk}(\catsets^\catgr).
\end{equation}
\end{cor}
\begin{proof}
It has to be shown that the assignment $f\mapsto\olf$ is a bijection
\begin{equation}
(L_{\fieldK})_{\bar{0}}(V;V')\to L^1_{\olk}(\olv;\overline{V'})
\end{equation}
for any pair $V,V'$ of $\fieldK$-super vector spaces. This is a special case of Prop.~\ref{isokmod}
if one inserts there $V_1=V$ and $V=V'$.
\end{proof}

We do not get an equivalence here, because Prop.~\ref{isokmod} only holds for superrepresentable
$\olk$-modules (see Def.~\ref{def:srep}).
On the other hand we can conclude much more than the above Corollary, since Prop.~\ref{isokmod} does not just 
make a statement about linear maps
between two super vector spaces but about multilinear maps on an arbitrary finite number of arguments.

\begin{cor}
\label{barfunct}
The functor $\olcdot:\catsmod_\fieldK(\catsets)\to\catmod_{\olk}(\catsets^\catgr)$ induces fully
faithful functors
\begin{eqnarray}
\olcdot:\catslie_\fieldK(\catsets) &\to& \catlie_{\olk}(\catsets^\catgr)\\
\olcdot:\catsalg_\fieldK(\catsets) &\to& \catalg_{\olk}(\catsets^\catgr)
\end{eqnarray}
between the categories of $\fieldK$-super Lie algebras and ordinary $\olk$-Lie algebras in $\catsets^\catgr$,
and between $\fieldK$-super algebras and ordinary $\olk$-algebras.
\end{cor}
\begin{proof}
A $\fieldK$-super algebra $A$ is a $\fieldK$-super vector space with a bilinear product 
$\mu:A\otimes A\to A$. Clearly, the functor $\olcdot$ assigns to $A$ a functor $\overline{A}$ and to
$\mu$ a $\olk$-bilinear morphism $\overline{\mu}$ such that $\overline{A}$ becomes a $\olk$-algebra.
By Prop.~\ref{isokmod} we are done if we can show that a morphism $f:(A,\mu)\to(A',\mu')$ of 
$\fieldK$-super algebras will be mapped to a morphism of the corresponding $\olk$-algebras. This means
we want to verify that for any $\Lambda$,
\begin{equation}
\olf_\Lambda(\overline{\mu}_\Lambda(\lambda_1\otimes a_1,\lambda_2\otimes a_2))=%
\overline{\mu}'_\Lambda(\olf_\Lambda(\lambda_1\otimes a_1),\olf(\lambda_2\otimes a_2)).
\end{equation}
for all $\lambda_i\otimes a_i\in \overline{A}(\Lambda)$. The left hand side is
\begin{equation}
\label{lab1}
\olf_\Lambda(\lambda_2\lambda_1\otimes\mu(a_1,a_2))=\lambda_2\lambda_1\otimes f(\mu(a_1,a_2)).
\end{equation}
The right hand side can be written as
\begin{equation}
\overline{\mu}'_\Lambda(\lambda_1\otimes f(a_1),\lambda_2\otimes f(a_2))=%
\lambda_2\lambda_1\otimes\mu'(f(a_1),f(a_2)),
\end{equation}
and that is the same as (\ref{lab1}) since $f$ was a homomorphism of super algebras.
\end{proof}

One could extend the list of categories between which $\olcdot$ maps fully faithfully to any kind of
$\fieldK$-multilinear superalgebraic structure which is defined by relations involving finitely 
many arguments, e.g., associative superalgebras \cite{I-dZ2ks}.

\begin{dfn}
\label{def:srep}
A $\olk$-module $\cv$ in $\catsets^\catgr$ is called superrepresentable if it is isomorphic to
$\olv$ for some $\fieldK$-super vector space $V$.
\end{dfn}

Due to Prop.~\ref{isokmod}, the superrepresentable $\olk$-modules form a full subcategory in 
$\catmod_{\olk}(\catsets^\catgr)$, and $\olcdot$ is an equivalence between this subcategory and
$\catsmod_\fieldK(\catsets)$. Non-superrepresentable $\olk$-modules do indeed exist. An example that
will prove useful later on is $\olv^{nil}$. Let $V$ be some $\fieldK$-super vector space
and let $\Lambda^{nil}$ be the nilpotent ideal in $\Lambda$. Then one can define a $\olk$-module by
setting
\begin{eqnarray}
\label{vnildef}
\olv^{nil}(\Lambda) &:=& (\Lambda^{nil}\otimes_\fieldK V)_{\bar{0}}\\
\olv(\varphi) &:=& (\varphi\otimes\mathrm{id}_{V})\big|_{\olv^{nil}(\Lambda)}\qquad\mathrm{for}\quad%
\varphi:\Lambda\to\Lambda'.
\end{eqnarray}
For every $\Lambda$, one has
\begin{equation}
\olv(\Lambda)=V_{\bar{0}}\oplus\olv^{nil}(\Lambda)=V_{\bar{0}}\oplus%
\left(\Lambda_{\bar{0}}^{nil}\otimes V_{\bar{0}}\right)\oplus%
\left(\Lambda_{\bar{1}}^{nil}\otimes V_{\bar{1}}\right).
\end{equation}
Hence $\olv^{nil}$ is superrepresentable if and only if $V_{\bar{0}}=0$, in which case $V$ itself
superrepresents it.

Finally, the change of parity functor $\Pi$ can be carried over to the category of 
superrepresentable $\olk$-modules in an obvious way.

\begin{dfn}
\label{def:copf2}
Let $\mathsf{SRepMod}_{\olk}\subset\catsmod_{\olk}$ be the full subcategory of superrepresentable 
$\olk$-modules. The change of parity functor $\olpi$ is defined as
\begin{eqnarray}
\olpi:\mathsf{SRepMod}_{\olk} &\to& \mathsf{SRepMod}_{\olk}\\
\olv &\mapsto& \overline{(\Pi(V))}.
\end{eqnarray}
\end{dfn}

\section{Superdomains}
\label{sect:sdom}

In this section, superdomains will be defined in purely categorical terms as open regions of
superrepresentable $\olk$-modules. Since we can define $\fieldK$-super vector spaces and thus $\olk$-modules
of arbitrary dimensions, infinite-dimensional superdomains will also become available. Later, these will
again serve as the building blocks of infinite-dimensional supermanifolds.

To make the notion of an open subobject sensible, one has to define an analogue of a topology on 
$\olk$-modules.
Since the objects of $\catmod_{\olk}(\catsets^\catgr)$ have no immediate interpretation as sets, this seems
somewhat strange at first. The solution is, once more, to use the sets of functorial points, each of which
is just an ordinary set and may thus be given a topology in the usual sense. Given an open set in each of
these spaces, these open sets will have to behave functorially with respect to $\catgr$ in order to form
an \emph{open subfunctor} in some $\olk$-module. The concept one arrives at, pursuing this intuition, 
is that of a Grothendieck (pre-)topology.

\subsection{The site $\cattop^\catgr$}

Following the idea outlined above, we study functors $\catgr\to\cattop$, i.e. from the finite-dimensional
Grassmann algebras into topological spaces.

\begin{dfn}
Let $\cf,\cf'$ be functors in $\cattop^\catgr$. $\cf'$ is called a subfunctor of $\cf$, if
\begin{enumerate} 
\item for every $\Lambda\in\catgr$, $\cf'(\Lambda)$ is a topological subspace of $\cf(\Lambda)$, and
\item the family of inclusions $\{\cf'(\Lambda)\subset\cf(\Lambda)\big|\Lambda\in\catgr\}$ forms a
functor morphism.
\end{enumerate}
In this case, one just writes $\cf'\subset\cf$. $\cf'$ is called an open subfunctor of $\cf$ if, in 
addition, each $\cf'(\Lambda)$ is open in $\cf(\Lambda)$.
\end{dfn}

\begin{dfn}
Let $\cf',\cf''$ be open subfunctors of $\cf\in\cattop^\catgr$. Then the intersection $\cf'\cap\cf''$ is
the functor whose points are
\begin{equation}
(\cf'\cap\cf'')(\Lambda):=\cf'(\Lambda)\cap\cf''(\Lambda).
\end{equation}
The union $\cf'\cup\cf''$ is the functor defined by
\begin{equation}
(\cf'\cup\cf'')(\Lambda):=\cf'(\Lambda)\cup\cf''(\Lambda).
\end{equation}
A morphism $\varphi:\Lambda\to\Lambda'$ is mapped by
$\cf'\cap\cf''$ resp. $\cf'\cup\cf''$ to $\cf(\varphi)\big|_{\cf'(\Lambda)\cap\cf''(\Lambda)}$ resp.
$\cf(\varphi)\big|_{\cf'(\Lambda)\cup\cf''(\Lambda)}$.
\end{dfn}

Clearly, both $\cf'\cap\cf''$ and $\cf'\cup\cf''$ are again open subfunctors of $\cf$. The functor
$\mathrm{emp}:\Lambda\to\emptyset$ is the initial object in $\cattop^\catgr$.

\begin{dfn}
A functor morphism $g:\cf''\to\cf$ is called open if there exists a factorization
\[
\begin{CD}
g:\cf'' @>f>> \cf'\subset\cf
\end{CD}
\]
such that $f$ is an isomorphism and $\cf'$ is an open subfunctor of $\cf$.
\end{dfn}

Finally, the notion of an open covering can be carried over straightforwardly.

\begin{dfn}
A family
$\{u_\alpha:\cu_\alpha\to\cf\}$ of open functor morphisms is called an open covering of $\cf$ if for
each $\Lambda\in\catgr$, the family of maps
\[
u_{\alpha,\Lambda}:\cu_\alpha(\Lambda)\to\cf(\Lambda)
\]
is an open covering of the topological space $\cf(\Lambda)$.
\end{dfn}

It is not hard to check that this assignment of a set of coverings to each functor $\cf\in\cattop^\catgr$
endows $\cattop^\catgr$ with a Grothendieck topology \cite{I-dZ2ks}, \cite{KI}, 
\cite{FAG_GFGAE}, \cite{SGA4}, \cite{SiGaL}, and thus turns it into a \emph{site}.\footnote{Categories
allowing the definition of a Grothendieck topology are called topoi (or toposes). The choice of
a Grothendieck topology turns the topos into a site.}

An example of an open subfunctor can be constructed in the following way: let  $\cf\in\cattop^\catgr$ be
an arbitrary functor, and let $U\subset\cf(\realR)$ be an open subset of its underlying set. Then we can
construct an open subfunctor $\cu\subset\cf$ by setting
\begin{eqnarray}
\nonumber
\cu(\realR) &:=& U\\
\label{restfunc}
\cu(\Lambda) &:=& \cf(\epsilon_\Lambda)^{-1}(U)\subset\cf(\Lambda)\\
\nonumber
\cu(\varphi) &:=& \cf(\varphi)\big|_{\cu(\Lambda)}\qquad\mathrm{for}\,\,\varphi:\Lambda\to\Lambda'.
\end{eqnarray}
Here, $\epsilon_\Lambda:\Lambda\to\realR$ is the final morphism of $\Lambda\in\catgr$ (cf.~(\ref{finmorgr})).
It is clear from the definition that the inclusion $\cu\subset\cf$ is indeed a functor morphism. We will
denote subfunctors of this form by $\cu=\cf\big|_U$ and call them restrictions. It will turn out that in
the cases we are interested in, such subfunctors are the only useful ones.

\subsection{Superdomains and supersmooth morphisms}

Now everything is prepared for the introduction of (possibly in\-fin\-ite-di\-men\-sion\-al) superdomains.

\begin{dfn}
Let $\cv$ be a superrepresentable $\olk$-module in $\cattop^\catgr$. $\cv$ will be called a locally
convex, resp.~Fr\'echet, resp.~Banach $\olk$-module if for every $\Lambda\in\catgr$, the topological
vector space $\cv(\Lambda)$ is locally convex, resp.~Fr\'echet, resp.~Banach.
\end{dfn}

\begin{dfn}
An open subfunctor $\cf$ of a locally convex (resp.~Fr\'echet, resp.~Banach) $\olk$-module in
$\cattop^\catgr$ is called a real (or complex, whichever $\olk$ is) locally convex 
(resp.~Fr\'echet, resp.~Banach) superdomain.
\end{dfn}

\begin{dfn}
A functor $\cf\in\cattop^\catgr$ is called locally isomorphic to real (or complex) locally convex
(resp.~Fr\'echet, resp.~Banach) superdomains if there exists an open covering 
$\{u_\alpha:\cu_\alpha\to\cf\}$ of $\cf$ such that each $\cu_\alpha$ is a locally convex
(resp.~Fr\'echet, resp.~Banach) superdomain.
\end{dfn}

Restrictions, as it turns out, are the only open subfunctors a superrepresentable $\olk$-module has.

\begin{prop}
\label{subf}
Any open subfunctor $\cu\subset\olv$ of a Banach (resp.~Fr\'echet, resp.~locally convex) $\olk$-module $\olv$
is a restriction $\olv\big|_U$, where $U=\cu(\fieldK)$ (cf. (\ref{restfunc})).
\end{prop}
\begin{proof}
Clearly, one can write 
\begin{equation}
\olv=\olv\big|_{V_{\bar{0}}}
\end{equation}
for a superrepresentable $\olk$-module $\olv$ which is represented by $V$, since
\begin{equation}
\olv(\Lambda)=\olv(\epsilon_\Lambda)^{-1}(V_{\bar{0}}).
\end{equation}
Let now $\cu$ be an arbitrary open subfunctor of $\olv$ and $U\subset V_{\bar{0}}$ be its $\fieldK$-points.
The inclusion $\cu\subset\olv$ must be a functor morphism, therefore the diagram
\begin{equation}
\begin{CD}
\cu(\Lambda) @>\subset>> \olv(\Lambda)\\
@VV{\cu(\epsilon_\Lambda)}V @VV{\olv(\epsilon_\Lambda)}V\\
U @>\subset>> \olv(\fieldK)=V_{\bar{0}}
\end{CD}
\end{equation}
has to commute for all $\Lambda\in\catgr$. This enforces
\begin{equation}
\cu(\Lambda)=\olv(\epsilon_\Lambda)^{-1}(U)
\end{equation}
for all $\Lambda$. For any morphism $\varphi:\Lambda\to\Lambda'$ of Grassmann algebras, we also have
\begin{equation}
\begin{CD}
\cu(\Lambda) @>\subset>> \olv(\Lambda)\\
@VV\cu(\varphi)V @VV\olv(\varphi)V\\
\cu(\Lambda') @>\subset>> \olv(\Lambda')
\end{CD},
\end{equation}
which commutes again, because the inclusion is a functor morphism. So,
\begin{equation}
\cu(\varphi)=\olv(\varphi)\big|_{\cu(\Lambda)}.
\end{equation}
\end{proof}

By the properties of unions and intersections of open subfunctors, we obtain the following corollary.

\begin{cor}
Let $\cf\in\cattop^\catgr$ be a functor which is locally isomorphic to Banach 
(resp.~Fr\'echet, resp.~locally convex) superdomains. Then every open subfunctor $\cu\subset\cf$ is 
a restriction $\cf\big|_U$.
\end{cor}
\begin{proof}
Let $\{u_\alpha:\cu_\alpha\to\cf\}$ be an open cover of $\cf$ by superdomains of the appropriate type
and let $\cu$ be an arbitrary open subfunctor of $\cf$. Then
every intersection $\cu\cap\cu_\alpha$ is a superdomain, i.e. is a restriction $\cf\big|_{U\cap U_\alpha}$,
where $U,U_\alpha$ are the underlying open sets of the respective functors. Therefore,
\[
\cu(\Lambda)=\bigcup_{\alpha}(\cf(\epsilon_\Lambda))^{-1}(U_\alpha\cap U)=(\cf(\epsilon_\Lambda))^{-1}(U)
\]
By the same argument as in Prop.~\ref{subf} (functoriality of inclusions),
the images of the morphisms of $\catgr$ under $\cu$ must be the restrictions of those of 
$\cf$.
\end{proof}

It the rest of this Chapter, we will focus exclusivly on
Banach superdomains. This is mainly for the reason not to overload the construction with technicalities.
After the suitable corrections the constructions described below work equally well in the Fre\'echet case.

\begin{dfn}
\label{bsdom}
Let $\cv\big|_U$ and $\cv'\big|_{U'}$ be two real Banach superdomains. A functor morphism 
$f:\cv\big|_U\to\cv'\big|_{U'}$ is called supersmooth if 
\begin{enumerate}
\item the map
\begin{equation}
f_\Lambda:\cv\big|_U(\Lambda)\to\cv'\big|_{U'}(\Lambda)
\end{equation}
is smooth for every $\Lambda\in\catgr$, and
\item for every $u\in \cv\big|_U(\Lambda)$, the derivative
\begin{equation}
Df_\Lambda(u):\cv(\Lambda)\to\cv'(\Lambda)
\end{equation}
is $\Lambda_{\bar{0}}$-linear.
\end{enumerate}
\end{dfn}

The second condition is necessary and sufficient to turn the sets of differential morphisms
\begin{eqnarray}
(Df)_\Lambda:\cv\big|_U(\Lambda)\times\cv(\Lambda) &\to& \cv'(\Lambda)\\
(Df)_\Lambda(u,v) &=& (Df_\Lambda(u))(v)
\end{eqnarray}
into a $\cv\big|_U$-family of $\olr$-linear morphisms $\cv\to\cv'$.

Together with supersmooth morphisms, smooth Banach superdomains form a category $\catbsdom$. Replacing 
``smooth'' with ``real analytic'' in 
Def.~\ref{bsdom} leads to the definition of the category of real analytic
superdomains. 

For complex analytic domains, there seem to be two different approaches. One can start with superdomains
in $\cattop^\catgr$ which are isomorphic to open subfunctors of superrepresentable $\olc$-modules, and
define morphisms to be complex superanalytic functor morphisms. On the other hand, one could as well use
the category $\catgr^\complexC$ from the very beginning on, studying only functors in 
$\cattop^{\catgr^\complexC}$ and using morphisms which are analytic in their complex coordinates. 
However, the two resulting categories are equivalent: the $\Lambda$-points of some superrepresentable
$\olc$ module $V$ are $(V\otimes_\realR\Lambda)_{\bar{0}}$. But
\begin{equation}
(V\otimes_\realR\Lambda)_{\bar{0}}\cong%
(V^\realR\otimes_\realR\complexC\otimes_\realR\Lambda)_{\bar{0}}\cong%
(V^\realR\otimes_\realR\Lambda^\complexC)_{\bar{0}},
\end{equation}
where $V^\realR$ is the realification of the super vector space $V=V^\complexC$. Thus, whether one applies
analytic maps to the last of these three sets of points, or applies complex analytic ones to the first
does not make a difference.

\subsection{Remark on superdomains of finite differentiability class}

In the usual ringed space formalism, it is always stressed that it is impossible to define supermanifolds
of class $C^k$, because the formula (\ref{pbeven}), which is crucial for the definition of supermanifolds,
involves an arbitrarily high number of derivatives. There is, however, a way to circumvent this problem:
by restricting the category $\catsman$ to those whose odd dimension is no greater than $N<\infty$. 
That seems a bit
unnatural, but yields a perfectly consistent theory as well. In the categorical approach, one simply
uses the categories $\catgr_{(N)}$ of Grassmann algebras with no more than $N$ generators. Then the entire
theory of superdomains outlined above works just as well with functors from $\catsets^{\catgr_{(N)}}$ and
$\cattop^{\catgr_{(N)}}$. In fact, working with such functors is even simpler than with the general ones,
since one only has to handle a finite number of point sets. In view of formula (\ref{pbeven}), if only
$N$ odd dimensions can occur, then the functions $F$ which can still sensibly be pulled back must be of
class $C^{\lceil (N+1)/2\rceil}$. So, in this case, supermanifolds of this class or higher can be
consistently defined.

It then becomes meaningful to define superfunctions of finite differentiability class (but at least of
class $\lceil (N+1)/2\rceil$) in the category $\catsman_{(N)}$. It even seems possible to construct
``super-Sobolev spaces'' of such functions by imposing pointwise a norm and completing $C^k$-spaces
of functions with respect to it \cite{mol-priv}. With the help of V.~Molotkov, I have tried to construct
such spaces, and could find no inconsistency in doing this so far. A problem that arises in practical 
applications of this construction is that one would usually not
want to restrict the number of possible odd dimensions. In particular in physical contexts, such a restriction
seems unjustifiable. For the solution of the moduli problem later in this work, it was therefore not
employed, although it seemed at first that it would enable us to directly carry over the global
analytic approach of Tromba \cite{TTiRG}. It also seemed not to be any simpler to work with the categories
$\catsman_{(N)}$ at first and subsequently let $N\to\infty$. This approach might, however, be viable for
certain problems of supergeometry, e.g., for the study of solutions of variational problems 
involving functionals
on supermanifolds, like action functionals of supersymmetric field theories.

\section{Banach supermanifolds}
\label{sect:bansman}

Actually, the superrepresentable $\olk$-modules in $\cattop^\catgr$ do no take their
values just in $\cattop$, but rather in $\catman$, the category of smooth (resp.~complex)
Banach manifolds. Since we will only be interested in functors which are locally isomorphic to
superdomains, this implies that we may restrict ourselves to functors in the category $\catman^\catgr$.

\begin{dfn}[Molotkov]
\label{def:sman}
Let $\cf$ be a functor in $\catman^\catgr$. An open covering $\ca=\{u_\alpha:\cu_\alpha\to\cf\}_{\alpha\in I}$
of $\cf$ is called a supersmooth atlas on $\cf$ if
\begin{enumerate}
\item every $\cu_\alpha$ is a Banach superdomain,
\item for every pair $\alpha,\beta\in I$, the fiber product 
\begin{equation}
\cu_{\alpha\beta}=\cu_\alpha\times_\cf\cu_\beta\in\catman^\catgr
\end{equation}
can be given the structure of a superdomain such that the projections 
$\Pi_\alpha:\cu_{\alpha\beta}\to\cu_\alpha$ and $\Pi_\beta:\cu_{\alpha\beta}\to\cu_\beta$ are supersmooth.
\end{enumerate}
The maps $u_\alpha:\cu_\alpha\to\cf$ are called charts on $\cf$.
\end{dfn}

\begin{dfn}
Two supersmooth atlases $\ca,\ca'$ on the functor $\cf\in\catman^\catgr$ are said to be
equivalent if their union $\ca\cup\ca'$ is again a supersmooth atlas on $\cf$. A supermanifold $\cm$
is a functor in $\catman^\catgr$ endowed with an equivalence class of atlases.
\end{dfn}

The second condition in Definition \ref{def:sman} needs some explanation. One might think at first that
the fiber product of any two supersmooth superdomains is automatically again supersmooth. But we may
only assume here that $\cf$ is a functor in $\catman^\catgr$, and thus we can a priori only 
assume the fiber product to exist in $\catman^\catgr$. The projection morphisms $\Pi_\alpha,\Pi_\beta$
are therefore only guaranteed to be functor morphisms in $\catman^\catgr$. The second condition therefore
requires
the fiber product of $\cu_\alpha,\cu_\beta$ to exist in the subcategory $\catbsdom\subset\catman^\catgr$.
Since there do really exist functor isomorphisms in $\catman^\catgr$ which are not supersmooth
\cite{mol-priv}, this is not automatic.

\begin{dfn}
Let $\cm,\cm'$ be Banach supermanifolds. A functor morphism $f:\cm\to\cm'$ is called supersmooth if for
each pair of charts $u:\cu\to\cm$, $u':\cu'\to\cm'$, the pullback 
\begin{equation}
\xymatrix{
\cu\times_{\cm'}\cu' \ar[d]_{\Pi}\ar[rr]^{\Pi'} && \cu' \ar[d]^{u'}\\
\cu \ar[r]^{u} & \cm \ar[r]^{f} & \cm'
}
\end{equation}
can be given the structure of a Banach superdomain such that its projections $\Pi,\Pi'$ are supersmooth.
\end{dfn}

It is clear that the composition of two supersmooth morphisms of Banach supermanifolds is again
supersmooth. Thus Banach supermanifolds form a category $\catbsman$. If nothing else is specified, we
will from now on only write $\catsman$ for the category of Banach supermanifolds, since it 
will mainly be them we will
be concerned with. The set of supersmooth morphisms $f:\cm\to\cm'$ will be denoted as
$SC^\infty(\cm,\cm'):=\Hom_{\catbsman}(\cm,\cm')$.

According to Prop.~\ref{subf}, every open submanifold of a Banach supermanifold $\cm$ is of the form
$\cu=\cm\big|_U$, where $U$ is an open submanifold of the underlying manifold $M=\cm(\fieldK)$. Clearly,
the restriction of the supermanifold structure of $\cm$ to $\cu$ induces on the latter the
unique supermanifold structure which makes the inclusion $\cu\subset\cm$ a supersmooth morphism.

\subsection{Linear algebra in the category $\catsman$}

Evidently, the functor $\overline{\cdot}$ defined in Def.~\ref{def:olv} takes its values in $\catsman$ and
not just in $\catsets^\catgr$. In addition, if $f:V_1\times\ldots\times V_n\to V$ is a $\fieldK$-n-linear map
of super vector spaces, then $\overline{f}:\olv_1\times\ldots\times \olv_n\to \olv$ is a supersmooth
morphism. Therefore, Corollary \ref{cor:barfunct} tells us that $\olcdot$ is a fully faithful functor
$\catsmod_{\fieldK}(\catman)\to\catmod_{\olk}(\catsman)$. But since all supermanifolds are locally isomorphic
to \emph{superrepresentable} $\olk$-modules, we get more.

\begin{thm}[see also \cite{I-dZ2ks}]
\label{thm:smodsman}
The functor
\begin{equation}
\olcdot:\catsmod_{\fieldK}(\catman)\to\catmod_{\olk}(\catsman)
\end{equation}
is an equivalence of categories.
\end{thm}
\begin{proof}
As seen above, every $\fieldK$-super vector space defines a $\olk$-module in $\catsman$. Conversely, let
a $\olk$-module $\cv$ in $\catsman$ be given. It only has to be shown that $\cv$ is superrepresentable when
considered as a $\olk$-module in in $\catsets^\catgr$. But this follows at once from the fact that each
of the sets $\cv(\Lambda)$ is already locally isomorphic to the $\Lambda$-points $\olv(\Lambda)$ of a 
superrepresentable module $\olv$. The $\olk$-module structure on $\cv$ then requires all transition functions
of the supermanifold to be linear. Therefore, taking any subfunctor of 
$\cv$ which is isomorphic to a superdomain $\olv\big|_U$, we can conclude that $\cv\cong\olv$.
\end{proof}

This theorem can be viewed as the categorical version of the statement that supermanifolds must be
locally modelled on linear superspaces. This was already obvious in the ringed space formulation for
the finite-dimensional case, but Thm.~\ref{thm:smodsman} extends it to the infinite-dimensional case.
The analogue of Corollary \ref{barfunct} also holds.

\begin{cor}
The functor $\olcdot:\catsmod_\fieldK(\catman)\to\catsmod_{\olk}(\catsman)$ induces equivalences
\begin{eqnarray}
\olcdot:\catslie_\fieldK(\catman) &\to& \catlie_{\olk}(\catsman)\\
\olcdot:\catsalg_\fieldK(\catman) &\to& \catalg_{\olk}(\catsman)
\end{eqnarray}
between the categories of $\fieldK$-super Lie algebras and ordinary $\olk$-Lie algebras in $\catsman$,
and between $\fieldK$-super algebras and ordinary $\olk$-algebras.
\end{cor}
\begin{proof}
This follows immediately from Cor.~\ref{barfunct} and Thm.~\ref{thm:smodsman}.
\end{proof}

In fact, it is Thm.~\ref{thm:smodsman} that assures that coherent tensor products and coherent 
inner Hom-functors
exist in the category $\catsman$, since they exist in $\catsmod_\fieldK(\catman)$. For a deeper discussion,
see \cite{I-dZ2ks} or \cite{SoaE:ATTC}. Finally, the change of parity functor can be extended to linear
supermanifolds.

\begin{dfn}
\label{def:copf3}
The change of parity functor $\olpi$ is defined on isomorphism classes of $\olk$-modules in $\catsman$ as
\begin{eqnarray}
\olpi:\catmod_{\olk}(\catsman) &\to& \catmod_{\olk}(\catsman)\\
\cv\cong\olv &\mapsto& \overline{(\Pi(V))}.
\end{eqnarray}
\end{dfn}

\subsection{Superpoints, again}

Superpoints were introduced in section \ref{sect:spoint} as linear supermanifolds corresponding to purely odd
super vector spaces. In Prop.~\ref{spointeq}, it was shown that the category $\catspoint$ is dual to
the category $\catgr$, and a duality was chosen, namely
\begin{eqnarray}
\label{cpfunct}
\cp:\catgr^\circ &\to& \catspoint\\
\nonumber
\Lambda &\mapsto& \Spec(\Lambda)=(\{*\},\Lambda).
\end{eqnarray}
In terms of the categorical definition, we have chosen $\cp(\Lambda_n)=\ovlk{0|n}$ for the Grassmann algebra
on $n$ generators over $\fieldK$. A supermanifold is, in categorical terms, still a functor 
$\catgr\to\catman$. Thus, $\cp$ can be considered as a bifunctor
\begin{equation}
\cp:\catgr^\circ\times\catgr\to\catman.
\end{equation}

\begin{prop}
\label{prop:isocp}
There exists an isomorphism of bifunctors
\begin{equation}
\cp\cong\Hom_\catgr(-,-).
\end{equation}
\end{prop}
\begin{proof}
We have chosen $\cp(\Lambda_n)=\overline{\fieldK^{0|n}}$. Therefore, the $\Lambda_m$-points of 
$\cp(\Lambda_n)$ are
\begin{equation}
\cp(\Lambda_n)(\Lambda_m)=(\Lambda_m\otimes\fieldK^{0|n})_{\bar{0}}%
\cong \Lambda_{m,\bar{1}}\otimes\fieldK^n\cong\Hom_\catgr(\Lambda_n,\Lambda_m),
\end{equation}
where the last isomorphism was proved in Prop.~\ref{homgr}. 
\end{proof}

The following easy fact is also very useful.

\begin{lemma}
\label{cpriso}
There exists an isomorphism of supermanifolds
\begin{equation}
\cp(\fieldK)\times\cm\cong\cm.
\end{equation}
\end{lemma}
\begin{proof}
Each of the sets of points $\cp(\fieldK)(\Lambda)$ is just the one-point set:
\[
\cp(\fieldK)(\Lambda)=(\Lambda\otimes\{0\})_{\bar{0}}=\{0\}.
\]
Therefore,
\[
(\cp(\fieldK)\times\cm)(\Lambda)=\cp(\fieldK)(\Lambda)\times\cm(\Lambda)\cong\cm(\Lambda).
\]
\end{proof}

\section{Connection to the Berezin-Leites theory in the finite-dimensional case}

It is instructive to recover the standard ringed space version of a supermanifold from the categorical 
construction. A rough sketch of the idea can be found in \cite{I-dZ2ks}. We will not try to
prove every statement here, since we will not need to rely on this construction later on. 
This section should rather be understood as a heuristic discussion.

One defines an $\olr$-superalgebra $\fr$ in $\catsman$ by setting
\begin{eqnarray}
\fr(\Lambda) &:=& \Lambda\\
\fr(\varphi) &:=& \varphi\quad\textrm{ for }\varphi:\Lambda\to\Lambda'.
\end{eqnarray} 
The $\olr$-superalgebra structure is provided by the $\Lambda_{\bar{0}}$-superalgebra structures on
each $\Lambda$. Note that up to now, we never defined super objects in any of our categories of
superobjects. It was never necessary -- in fact one of great advantages of the categorical approach is
that one can work with purely even objects. This is expressed by Thm.~\ref{thm:smodsman}. Every
super vector space, superalgebra etc. is an ordinary algebra in $\catsets^\catgr$. In this sense
$\fr$ is ``super super''. The functor $\fr$ is still representable as we will show now, but the
superalgebra representing it is non-supercommutative.

As an $\olr$-module, we have
\begin{equation}
\fr\cong\olr\oplus\olpi(\olr)\cong\ovlr{1|1}. 
\end{equation}
We want to find a superalgebra structure on $\realR^{1|1}$ which represents the one on $\fr$.
Denoting the standard basis of $\realR^{1|1}$ as $\{1,\theta\}$, we can write
\begin{equation}
\fr(\Lambda)=\Lambda_{\bar{0}}\oplus\Lambda_{\bar{1}}\cong%
\Lambda_{\bar{0}}\otimes 1\oplus \Lambda_{\bar{1}}\otimes\theta.
\end{equation}
Let $\olmu:\fr\times\fr\to\fr$ denote the multiplication in $\fr$ and let $\mu$ denote the
hypothetical multiplication in $\realR^{1|1}$ that we want to determine. Let 
$\lambda_1,\lambda_2\in\Lambda_{\bar{0}}$ be given. We have
\begin{equation}
\olmu_\Lambda(\lambda_1\otimes 1,\lambda_2\otimes 1)=\lambda_2\lambda_1\otimes\mu(1,1).
\end{equation}
This must coincide with the product $\lambda_1\lambda_2\otimes 1$, by construction of $\fr$.
Since $\lambda_1,\lambda_2$ are even, this requires $\mu(1,1)=1$. Likewise, for
$\lambda_1\in\Lambda_{\bar{0}}$ and $\lambda_2\in\Lambda_{\bar{1}}$ we have
\begin{equation}
\olmu_\Lambda(\lambda_1\otimes 1,\lambda_2\otimes\theta)=\lambda_2\lambda_1\otimes\mu(1,\theta).
\end{equation}
This must coincide with $\lambda_1\lambda_2\otimes\theta$, and thus we must have $\mu(1,\theta)=\theta$.
Analogously we find $\mu(\theta,1)=\theta$.

Let then $\lambda_1,\lambda_2\in\Lambda_{\bar{1}}$ be given. The product is
\begin{equation}
\olmu_\Lambda(\lambda_1\otimes\theta,\lambda_2\otimes\theta)=\lambda_2\lambda_1\otimes\mu(\theta,\theta).
\end{equation}
This must be equal to $\lambda_1\lambda_2\otimes 1$, which enforces $\mu(\theta,\theta)=-1$.

The super space $\realR^{1|1}$ endowed with this multiplication will be denoted as $\complexC^s$.
As an $\realR$-algebra it is isomorphic to $\complexC$, but as an $\realR$-superalgebra, it is
isomorphic to $\complexC$ with $i$ declared odd. It is non-supercommutative, in particular, every
non-zero element is invertible (which makes it a kind of super analog of a skew field). Note that,
although $\complexC^s$ is non-supercommutative, the $\olr$-superalgebra $\fr$ \emph{is}
supercommutative. Just like passing to their functors of points turns supercommutative algebras
into commutative $\olr$-algebras, the special non-supercommutativity of $\complexC^s$ is weakened
to supercommutativity of its functor of points as an $\olr$-algebra. It would be an interesting
question to study whether this kind of reasoning can be iterated to yield something like 
``super super super'' objects and whether these would possess any geometric interpretation. 
Molotkov in \cite{I-dZ2ks} proposes a formalism to investigate such
questions, but a conclusive answer has yet to be found.

The reason why we introduced $\fr$ is that we need
a superalgebra in $\catsman$ in order to induce the structure of a superalgebra on certain sets
of morphisms which we want to interpret as the superfunctions on a supermanifold $\cm$. Following
Molotkov \cite{I-dZ2ks}, we define an $\realR$-superalgebra 
\begin{equation}
SC^\infty(\cm):=SC^\infty(\cm,\fr),
\end{equation}
which will be called the superalgebra of superfunctions on $\cm$. Since $\fr$ is a supercommutative
$\olr$-superalgebra, $SC^\infty(\cm)$
is canonically equipped with the structure of a supercommutative $SC^\infty(\cm,\olr)$-superalgebra.
Moreover, we can embed $\realR\hookrightarrow SC^\infty(\cm)$ as the constant functions $\cm\to\realR$.
More precisely, for any $r\in\realR$, we define a morphism $f_r:\cm\to\realR$ by setting
\begin{equation}
(f_r)_\Lambda(m)=r\quad\textrm{ for all }m\in\cm(\Lambda).
\end{equation}
These are obviously supersmooth morphisms. Via this embedding, $SC^\infty(\cm,\fr)$ becomes endowed
with an $\realR$-superalgebra structure.

The following example is borrowed from \cite{I-dZ2ks}. Consider a superdomain $\cu\subset\ovlr{m|n}$.
One has maps
\begin{eqnarray}
x_i:\ovlr{m|n} &\to& \ovlr{1|0}\hookrightarrow\fr,\qquad 1\leq i\leq m,\\
\theta_j:\ovlr{m|n} &\to& \ovlr{0|1}\hookrightarrow\fr,\qquad 1\leq j\leq n,
\end{eqnarray}
where the first arrow in every line represents the canonical projection onto the $i$-th even and
$j$-th odd coordinate, respectively. As in ordinary geometry, the sheaf of functions can be
generated from these coordinate maps. One can show that \cite{I-dZ2ks}
\begin{equation}
SC^\infty(\cu)\cong C^\infty(x_1,\ldots,x_m)\otimes\wedge^\bullet[\theta_1,\ldots,\theta_n],
\end{equation}
where $C^\infty(x_1,\ldots,x_m)=C^\infty_U$.

Now let $\cm$ be a supermanifold and let $M$ be its underlying topological space (i.e., the
topological space underlying the base manifold $\cm(\realR)$). Then we can assign to every open set
$U\subset M$ the $\realR$-superalgebra $SC^\infty(\cm\big|_U)$. This yields a presheaf
$P(\cm)$ on $M$. The standard procedure of ``sheafification'', i.e., taking for every point 
$x\in M$ the direct limit over all open sets containing it produces the
associated sheaf $S(\cm)$. It is clear that any morphism $f:\cm\to\cm'$ of supermanifolds
induces, via its associated map $M\to M'$ of the underlying spaces, a morphism of sheaves
$S(f):S(\cm)\to S(\cm')$. Therefore the assignment 
\begin{eqnarray}
\cs:\cm &\mapsto& S(\cm)\\
f &\mapsto& S(f)
\end{eqnarray}
defines a functor from the category
of supermanifolds to the category of topological spaces locally ringed by supercommutative superalgebras.

Denote by $\catfinsman$ the category of finite-dimensional supermanifolds defined by the categorical
construction. Moreover, we call the category of finite-dimensional topological spaces locally ringed by
supercommutative superalgebras the category of Berezin-Leites supermanifolds. Then we have the following:

\begin{thm}[Molotkov \cite{I-dZ2ks}]
The functor $\cs$ establishes an equivalence between the category $\catfinsman$ and the category
of Berezin-Leites supermanifolds.
\end{thm}

We do not want to go into the proof of this Theorem here. Let us just point out that it really only
holds in the finite-dimensional case, for the same reasons as in ordinary geometry: while in
any finite dimension $m|n$, there exists up to isomorphy only one super vector space to which a
supermanifold can be identified (in the functor of points sense), this is no longer the case in
infinite dimensions.

\section{Super vector bundles}
\label{sect:svbun}

Following Molotkov \cite{I-dZ2ks}, we will define
super vector bundles, like supermanifolds, in terms of atlases of trivial super vector bundles.

\begin{dfn}
A smooth trivial super vector bundle over $\cm$ is defined to be a triple $(\cm\times\cv,\cm,\Pi_\cm)$,
where $\cm$ is a supermanifold, $\cv$ is a superrepresentable $\olr$-module and $\Pi_\cm:\cm\times\cv\to\cm$ 
is the canonical projection. A morphism $(\cm\times\cv,\cm,\Pi_\cm)\to(\cm'\times\cv',\cm',\Pi_{\cm'})$ of
trivial super vector bundles consists of a pair of supersmooth morphisms
\begin{eqnarray}
f:\cm\times\cv &\to& \cm'\times\cv'\\
g:\cm &\to& \cm',
\end{eqnarray}
such that $\Pi_{\cm'}\circ f=g\circ\Pi_\cm$, and such that $\Pi_{\cv'}\circ f:\cm\times\cv\to\cv'$ is
a $\cm$-family of $\olr$-linear morphisms (cf.~definition \ref{def:multmor}).
\end{dfn}

By definition, trivial super vector bundles are therefore particular functors $\catgr\to\catvbun$, where
$\catvbun$ is the category of smooth vector bundles. Every functor $\ce\in\catvbun^\catgr$ gives rise to
a functor $\cm:\catgr\to\catman$ in a canonical way: every bundle $\ce(\Lambda)$ possesses a projection
map $\pi(\Lambda):\ce(\Lambda)\to M(\Lambda)$, where $M(\Lambda)$ is an ordinary manifold. Then, clearly,
setting $\cm(\Lambda)=M(\Lambda)$ defines a functor in $\catman^\catgr$.

\begin{dfn}
Let $\ce,\ce'$ be functors in $\catvbun^\catgr$, and let $\cm,\cm'$ be their associated functors in
$\catman^\catgr$. Then $\ce$ is said to be an open subfunctor of $\ce'$, denoted $\ce\subset\ce'$, if
\begin{enumerate}
\item $\cm$ is an open subfunctor of $\cm'$, and
\item for each $\Lambda\in\catgr$ we have $\pi(\Lambda)^{-1}(\cm(\Lambda))=\pi'(\Lambda)^{-1}(\cm(\Lambda))$,
\end{enumerate}
where $\pi(\Lambda)$ is the projection of the bundle $\ce(\Lambda)$ to its base $\cm(\Lambda)$.
\end{dfn}

This gives us again the notion of an open morphism $\ce''\to\ce$ of functors in $\catvbun^\catgr$: it
is called open if it can be factorized as a composition
\begin{equation}
\begin{CD}
\ce'' @>f>> \ce' \subset\ce
\end{CD},
\end{equation}
where $f$ is an isomorphism of functors and $\ce'$ is an open subfunctor of $\ce$. An open 
covering $\{\ce_\alpha\}_{\alpha\in A}$ of a $\ce\in\catvbun^\catgr$
is then a collection of open morphisms $\{\phi_\alpha:\ce_\alpha\to\ce\}_{\alpha\in A}$, such 
that the associated maps
$\{\pi\circ\phi_\alpha\}_{\alpha\in A}$ are an open covering of the functor $\cm:\catgr\to\catman$ 
associated with 
$\ce$. In analogy to supermanifolds, a supervector bundle is a functor in $\catvbun^\catgr$ endowed with an
atlas of trivial open subbundles.

\begin{dfn}
\label{def:svbun}
Let $\ce$ be a functor in $\catvbun^\catgr$, and let $\cm\in\catman^\catgr$ be its associated functor
of base manifolds. Let $\ca=\{\phi_\alpha:\ce_\alpha\to\ce\}_{\alpha\in A}$ be an open covering of $\ce$. 
Then this covering is an atlas of a
super vector bundle $\ce$ over the supermanifold $\cm$ if the following conditions hold:
\begin{enumerate}
\item each of the $\ce_\alpha$ is a trivial super vector bundle $\cu_\alpha\times\cv_\alpha$, and
$\cv_\alpha\cong\cv_\beta$ for all $\alpha,\beta\in A$, and
\item for each $\alpha,\beta\in A$, the overlaps
\begin{equation}
\begin{CD}
\ce_\alpha\times_\ce\ce_\beta @>{\Pi_\alpha}>> \ce_\alpha\\
@V{\Pi_\beta}VV @VV{\phi_\alpha}V\\
\ce_\beta @>{\phi_\beta}>> \ce
\end{CD}
\end{equation}
can be given the structure of a trivial super vector bundle in such a way that the projections
$\Pi_\alpha,\Pi_\beta$ become morphisms of trivial super vector bundles.
\end{enumerate}
Two atlases $\ca$ and $\ca'$ are equivalent, if their union $\ca\cup\ca'$ is again an atlas. A super
vector bundle $\ce$ is a functor in $\catvbun^\catgr$ together with an equivalence class of atlases.
\end{dfn}

The second condition is necessary because the fiber product in the diagram is constructed as the fiber
product in $\catvbun^\catgr$. We thus have to make sure that it actually exists in the subcategory of
trivial super vector bundles (compare to the discussion following Def.~\ref{def:sman}). Note also that
the requirement that the transition functions be morphisms of trivial super vector bundles automatically
turns $\cm$ into a supermanifold. In general, one would usually start with a given base supermanifold and
construct a super vector bundle on it by choosing a local trivialization which is compatible with the 
transition functions of the base.

\begin{dfn}
\label{def:svbunmor}
Let $\ce,\ce'$ be super vector bundles with open coverings $\{\phi_\alpha:\ce_\alpha\to\ce\}_{\alpha\in A}$ 
and
$\{\phi_{\alpha'}:\ce'_{\alpha'}\to\ce'\}_{\alpha'\in A'}$. A functor morphism $\Phi:\ce\to\ce'$ in 
$\catvbun^\catgr$ is a
morphism of super vector bundles if for all $\alpha\in A$ and all $\alpha'\in A'$, the pullbacks
\begin{equation}
\xymatrix{
\ce_\alpha\times_{\ce'}\ce_{\alpha'} \ar[rr]^{\Pi_{\alpha'}} \ar[d]_{\Pi_{\alpha}} &&%
\cu_{\alpha'} \ar[d]^{\phi_{\alpha'}}\\
\cu_\alpha \ar[r]^{\phi_\alpha} & \ce \ar[r]^{\Phi} & \ce'
}
\end{equation}
can be chosen such that $\ce_\alpha\times_{\ce'}\ce_{\alpha'}$ is a trivial super vector bundle and
the projections $\Pi_\alpha,\Pi_{\alpha'}$ are morphisms of trivial super vector bundles.
\end{dfn}

Definitions \ref{def:svbun} and \ref{def:svbunmor} yield a category $\catsvbun$, which is obviously
a subcategory of $\catvbun^\catgr$, but not a full one. One can define super vector bundles in terms
of cocycles with values in a Lie supergroup as well \cite{I-dZ2ks}, but we will not attempt to do this here.

We have a natural functor 
\begin{eqnarray}
B:\catsvbun &\to& \catsman\\
(\pi:\ce\to\cm) &\mapsto& \cm
\end{eqnarray}
which assigns to every super vector bundle its base supermanifold. The resulting full subcategories of
bundles are denoted by $\catsvbun(\cm):=B^{-1}(\cm)$.

\begin{prop}
\label{trivsvbun}
A super vector bundle $\pi:\ce\to\cm$ is trivial if and only if all of its $\Lambda$-points
$\pi_\Lambda:\ce(\Lambda)\to\cm(\Lambda)$ are trivial bundles.
\end{prop}
\begin{proof}
The bundle $\pi:\ce\to\cm$ is trivial if and only if there exists an isomorphism $f:\ce\to\cm\times\cv$
for some superrepresentable $\fieldK$-module $\cv$ such that $\pi=\pi_\cm\circ f$. This means that for 
every $\Lambda\in\catgr$, the components of $f$ must make the diagram
\begin{equation}
\xymatrix{f_\Lambda:\ce(\Lambda)\ar[rr] \ar[dr]_{\pi_\Lambda} & &
\cm(\Lambda)\times\cv(\Lambda) \ar[dl]^{\pi_{\cm,\Lambda}}\\
&\cm(\Lambda)&}
\end{equation}
commutative. That is precisely the condition for the triviality of the ordinary vector bundle 
$\pi_\Lambda:\ce(\Lambda)\to\cm(\Lambda)$.
\end{proof}

\subsection{Pullback of super vector bundles}

Let $\pi:\ce\to\cm$ be a super vector bundle, and let $f:\cm'\to\cm$ be a supersmooth morphism of
supermanifolds. Then we define the pullback of $\ce$ along $f$ as the functor
\begin{eqnarray}
f^*\ce:\catgr&\to&\catvbun\\
\Lambda &\mapsto& f_\Lambda^*(\ce(\Lambda)).
\end{eqnarray}
It is clear that by this construction, $f^*\ce$ is again a super vector bundle: pulling back each of the
trivial open subbundles $\ce_\alpha$ which make up the atlas of $\ce$ gives an atlas on $f^*\ce$
\cite{I-dZ2ks}. It is equally clear that the pullback bundle thus defined is indeed the pullback in
the homological sense: it completes the cartesian square
\begin{equation}
\begin{CD}
f^*\ce @>{\Pi_\ce}>> \ce\\
@V{\Pi_{f^*\ce}}VV @VV{\pi}V\\
\cm' @>f>> \cm
\end{CD}\qquad,
\end{equation}
where $\Pi_\ce$ is the canonical projection defined pointwise by the pullback bundles
$f_\Lambda^*\ce(\Lambda)$. As in the ordinary case, any given morphism $f:\cm'\to\cm$ gives in 
this way rise to a functor
\begin{equation}
f^*:\catsvbun(\cm)\to\catsvbun(\cm').
\end{equation}

In spite of the formal similarity, it is somewhat dangerous to think of the $\olr$-modules $\cv$ in
$\Pi_\cm:\cm\times\cv\to\cm$ as ``fibers'' in the set-theoretical sense, since they are parametrized 
not just by the ``points'', i.e.,
the underlying manifold of $\cm$, but by the odd dimensions of $\cm$ as well. However, this becomes
true again on the underlying manifold:
looking at an underlying point $x:\cp(\realR)\to\cm$ of $\cm$, we note that, since each $\cp(\realR)(\Lambda)$
consists of a single element,
\begin{equation}
x^*\ce\cong\cv,
\end{equation}
where $\cv$ is the typical fiber of $\ce$, i.e. a superrepresentable $\olr$-module. This means that the
pullback of $\ce$ along the inclusion $M_{red}\hookrightarrow \cm$ yields a canonical $\intZ_2$-graded
vector bundle over the ordinary manifold $M_{red}$.

\subsection{The tangent bundle $\ctm$ of a supermanifold $\cm$}

The tangent bundle $\ctm$ of a supermanifold $\cm$ is defined in the categorical framework as a functor
$\ctm:\catgr\to\catvbun$ in the following way: for every $\Lambda\in\catgr$ and every 
$\varphi:\Lambda\to\Lambda'$, set
\begin{eqnarray}
\label{tdef1}
\ctm(\Lambda) &:=& T(\cm(\Lambda)),\\
\nonumber
\ctm(\varphi) &:=& D(\cm(\varphi)):T(\cm(\Lambda))\to T(\cm(\Lambda')).
\end{eqnarray}
To every morphism $f:\cm\to\cm'$ of supermanifolds, we assign a functor morphism
\begin{eqnarray}
\label{tdef2}
\cd f:\ctm &\to& \ctm'\\
\nonumber
(\cd f)_\Lambda &:=& Df_\Lambda:T(\cm(\Lambda))\to T(\cm'(\Lambda)).
\end{eqnarray}
The assignments (\ref{tdef1}) and (\ref{tdef2}) define a functor $\ct:\catsman\to\catvbun^\catgr$ which
will be called the tangent functor. For the definition of a super vector bundle to make sense, we would
certainly expect the tangent bundle to be in $\catsvbun$, not just in $\catvbun^\catgr$. This is indeed
the case:

\begin{prop}
The tangent functor is a functor $\ct:\catsman\to\catsvbun$.
\end{prop}
\begin{proof}
Choose a supersmooth atlas $\{u_\alpha:\cu_\alpha\to\cm\}_{\alpha\in A}$ of $\cm$ which satisfies the 
conditions of Def.~\ref{def:sman}. Then all $\cu_\alpha$ are open domains in some superrepresentable
$\olr$-module $\cv$, so their tangent bundles are trivial and isomorphic to
\begin{equation}
\ct\cu_\alpha\cong\cu_\alpha\times\cv.
\end{equation}
It is clear that the tangent bundles $\{\ct\cu_\alpha\}_{\alpha\in A}$ of the coordinate domains form
an atlas of open subfunctors for the functor $\ctm\in\catvbun^\catgr$. It has to be shown that this
atlas satisfies the conditions of Def.~\ref{def:svbun}.

By Definition \ref{def:sman}, each intersection $\cu_\alpha\times_\cm\cu_\beta$ has the structure of a 
superdomain itself, and the projections to $\cu_\alpha,\cu_\beta$ are supersmooth. Thus, the intersection
$\ct\cu_\alpha\times_{\ctm}\ct\cu_\beta$ has the structure of a trivial super vector bundle as well.
Moreover, the maps $\cd\Pi_\alpha$ and $\cd\Pi_\beta$ are $\cu_\alpha\times_\cm\cu_\beta$-families of
morphisms of $\olr$-modules (cf. definition \ref{bsdom}). Therefore, we have the commutative square
\begin{equation}
\xymatrix{
\ct\cu_\alpha\times_{\ctm}\ct\cu_\beta \ar[r]^>>>>>{\cd\Pi_\beta} \ar[d]_{\cd\Pi_\alpha} &%
\ct\cu_\beta \ar[d]\\
\ct\cu_\alpha \ar[r] & \ctm
},
\end{equation}
which is precisely the second condition of \ref{def:svbun}.
\end{proof}

One can also show \cite{I-dZ2ks} that the sections of the tangent bundle $\ctm$ thus constructed have 
the properties of vector fields in the sense that an action of these sections on the sections of
any ``natural bundle'' (like functions, tensor fields, etc.) can be defined which possesses the properties
of the Lie derivative \cite{GLS:Invariant}.

\subsection{The endomorphism bundle $\cend(\ce)$ of a super vector bundle}

To construct bundles of morphisms, it is most convenient to use the intuition of
Section \ref{sect:functpoints}, in particular Prop.~\ref{dirprod}. In the spirit of the family point
of view, one
should consider a super vector bundle $\pi:\ce\to\cm$ as an object in the category $\catsman/\cm$. 
The fact that
it is a vector bundle and not an arbitrary fiber bundle is then translated into an 
$\cm^*(\olr)=\cm\times\olr$ action on $\ce$ (compare to Section \ref{sect:functpoints}). The dual bundle 
$\ce^*$ is defined to be
\begin{equation}
\ce^*:=\ihom_{\catsvbun(\cm)}(\ce,\cm^*(\olr)),
\end{equation}
i.e., it is an inner Hom-object in the category of super vector bundles over $\cm$. In the same vein, 
one defines the endomorphism bundle of $\ce$.

\begin{dfn}
Let $\pi:\ce\to\cm$ be a super vector bundle. The endomorphism bundle of $\ce$ is defined to be the
inner Hom-object
\begin{equation}
\cend(\ce):=\ihom_{\catsvbun(\cm)}(\ce,\ce).
\end{equation}
\end{dfn}

The existence of inner Hom-objects in the categories $\catsvbun/\cm$ can be deduced from their existence
in $\catsvect$ and local triviality, but for nontrivial super vector bundles, this is a rather tedious 
job. We take these objects for granted, relying on the work of Molotkov \cite{I-dZ2ks}, \cite{G:aGoTt}
and the
general theory of topoi, expounded for example in \cite{SoaE:ATTC} and \cite{SiGaL}.

\subsection{The change of parity functor $\olpi$ for super vector bundles}

The functor $\olpi$ defined in Definitions \ref{def:copf2} and \ref{def:copf3} extends to super vector
bundles. For any trivial super vector bundle $\cm\times\cv$, we set
\begin{equation}
\olpi(\cm\times\cv):=\cm\times\olpi(\cv).
\end{equation}

To show that an arbitrary super vector bundle $\pi:\ce\to\cm$ gets mapped into a well defined 
parity-reversed version $\olpi(\ce)$, note first that for any pair $V,V'$ of $\realR$-super vector spaces, 
one has the natural isomorphism
\begin{equation}
\Hom_{\catsvect_\realR}(V,V')\cong\Hom_{\catsvect_\realR}(\Pi(V),\Pi(V')),
\end{equation}
which entails
\begin{equation}
\label{pr2}
\ihom_{\catsvect_\realR}(V,V')=\cl_{\olr}(V;V')\cong\cl_{\olr}(\Pi(V);\Pi(V')).
\end{equation}
As remarked above, Thm.~\ref{thm:smodsman} ensures the existence of inner Hom-objects in the category
$\catmod_{\olk}(\catsman)$. This allows us to extend equation (\ref{pr2}) to an isomorphism of families
of $\olr$-linear maps (cf.~Sections \ref{sect:multmor} and \ref{sect:ihom})
\begin{equation}
\label{pr3}
L_{\olr}(\cm;\cv;\cv')\cong L_{\olr}(\cm;\Pi(\cv);\Pi(\cv')).
\end{equation}
This is precisely what we need, because $L_{\olr}(\cm;\cv;\cv')$ is, of course, isomorphic to the set
of morphisms of trivial super vector bundles $\Hom_{\catsvbun}(\cm\times\cv,\cm\times\cv')$.

Let now $\pi:\ce\to\cm$ be an aribtrary super vector bundle, and let $\{\ce_\alpha\}_{\alpha\in A}$ be an
atlas of trivial super vector bundles. Then we can just take the atlas 
$\{\Pi(\ce_\alpha)\}_{\alpha\in A}$ of parity reversed bundles to define the bundle $\Pi(\ce)$. This
atlas satisfies the conditions of Definition \ref{def:svbun}: to each overlap $\ce_\alpha\times_\ce\ce_\beta$,
one assigns its parity reversed counterpart, and to each of the projections
$\Pi_\alpha:\ce_\alpha\times_\ce\ce_\beta\to\ce_\alpha$, one assigns its image under the isomorphism
(\ref{pr3}). For a more detailed and formal exposition, cf.~\cite{G:aGoTt}, \cite{I-dZ2ks}.

%% file: sconf.tex
\chapter{Superconformal surfaces}
\label{ch:sconf}

The widely used term \emph{super Riemann surface} suggests that there should exist a unique
generalization of Riemann surfaces to supergeometry. This
is not at all the case. Not all of the features of a Riemann surface have a unique super analog.
In particular, there are infinitely many types of superconformal structures (four families and
several exceptional ones) which can serve as super analogs of the conformal structure of a
Riemann surface. Superconformal structures are not the same as supercomplex structures anymore,
instead, a supercomplex structure is a particular superconformal structure.

The super Riemann surfaces used in superstring theory are only one of these infinitely many 
kinds of superconformal surfaces, whose specific features make it especially useful for physics.
They are the surfaces associated to the algebra $\fkl(1|1)$, which can be seen as a $1|1$-dimensional
analog of the contact vector fields.
In particular, the algebra $\fkl(1|1)$ possesses a non-trivial central extension, which
is an indispensible ingredient of the field theory constructed on it.

\section{The superconformal algebras}

The history of the superalgebras which are nowadays subsumed under the term ``superconformal''
reaches back to the days of the so-called dual models, a topic in the study of hadronic and
mesonic resonances whose
ideas and methods directly led to the emergence of string theory during the seventies. 
The superconformal algebras comprise four main series, accompanied by several exceptional ones. The main
series were discovered in the early seventies, first by Neveu and Schwarz \cite{Fdmop} and Ramond
\cite{Dtfff}, and then by Ademollo, Brink et al.
\cite{DswU1cs}, \cite{Ssacc}. Schwimmer and Seiberg later found an additional one-parameter family of 
deformations for one type of these series. The exceptional ones were announced in \cite{S:Five},
\cite{LSoST}. The completeness of the classification by the four series was conjectured
in \cite{KvdL:On}, and their central extensions were completely classified. An extended list, including
the exceptional algebras and several additional extensions, was conjectured to be the complete
classification in \cite{LSoST} and seems to be the definitive answer today. The authors of
\cite{LSoST} also argue for the use of the term ``stringy'' instead of ``superconformal'', since 
the latter seems to suggest that these algebras preserve some geometric structure up to a factor. But
only a few of them indeed do that, while the others merely inherit their name from containing the algebra
of conformal transformations of $\complexC^\times$.

\subsection{Definition}

Although the geometric structures whose moduli space we want to construct in this work are preserved by only
two particular superconformal algebras, we will give a brief overview of their definitions and
explicit realizations in general in this section. Each of these algebras defines its own species 
of superconformal
structure on a supersurface, thus each of them could rightfully be associated with its own kind of
super Riemann surface. Each of them therefore poses a moduli problem, of which the one studied in this
work (and dubbed ``the'' super moduli space in many physics papers) is just one particular case.

We study the complex punctured superspace $\ovlc{1|n}\setminus\{0\}$, i.e., the superspace $\ovlc{1|n}$ with
the stalk at zero removed.
Denote by $(z,\theta_1,\ldots,\theta_n)$ its coordinates and by 
$\cp^L(n):=\complexC[z,z^{-1},\theta_1,\ldots,\theta_n]$ the algebra of Laurent polynomials on
$\ovlc{1|n}\setminus\{0\}$. Then the algebra $\der(\cp^L(n))$ of its derivations is
\begin{equation}
\fvect^L(1|n)=\left\{X=f_0\pderiv{}{z}+\sum_{i=1}^n f_i\pderiv{}{\theta_i}\mid%
f_i\in\cp^L(n),\,\, 0\leq i\leq n\right\},
\end{equation}
the algebra of vector fields on the punctured complex superplane with Laurent coefficients. 
All superconformal algebras can be realized as subalgebras of $\fvect^L(1|n)$.

\begin{dfn}
The Witt algebra (or centerless Virasoro algebra) $\witt$ is the algebra
\begin{equation}
\der(\cp^L(0))=\fvect^L(1|0)=\left\{X=f(z)\pderiv{}{z}\mid f(z)\in\cp^L(0)\right\}
\end{equation}
of derivations of the Laurent polynomials on $\complexC^\times$.
\end{dfn}

The following definition is borrowed from \cite{KvdL:On}.

\begin{dfn}
\label{def:sconf}
A Lie superalgebra $\fg$ is called superconformal, if
\begin{enumerate}
\item $\fg$ is simple,
\item $\fg$ contains $\witt$ as a subalgebra, and
\item $\fg$ has growth 1.
\end{enumerate}
\end{dfn}

The third condition means that whenever one takes a finite set of elements $x_1,\ldots,x_k\in\fg$ and
computes the linear span $V_j$ of commutators of the $x_i$ of length $\leq j$, then 
$\dim V_j\leq C(x_1,\ldots,x_k)\cdot j$, where $C$ is a constant independent of $j$.

Since it will turn out that every superconformal algebra is a subalgebra of some $\fvect^L(1|n)$, each of
them has a standard $\intZ$-grading, but some have additional nonstandard gradings. Actually, the 
definition of stringy
superalgebras given in \cite{LSoST} slightly extends the one of superconformal algebras above. The authors 
of \cite{LSoST} carry over
a definition of O. Mathieu, who introduced the notion of a \emph{deep algebra} for simple $\intZ$-graded
Lie algebras \cite{M:Classification} to the case of superalgebras. Then an infinitely deeply $\intZ$-graded
superalgebra is called \emph{stringy}, according to \cite{LSoST}, if it possesses a
root vector which does not act locally nilpotently. This will be the case for all superconformal algebras
(after Def.~\ref{def:sconf}), since they have the root vector $\pderiv{}{z}$. A detailed account can
 be found in \cite{LSoST} and
the references therein. For our purposes, it will suffice to just outline the explicit construction of
the superconformal algebras and the geometric structures they preserve.

\subsection{The algebras $\fvect^L(1|n)$ and $\fsvect^L(1|n)$}

Obviously, the algebras $\fvect^L(1|n)$ are all superconformal. They comprise the first series of those
algebras and can be viewed as the algebra of holomorphic vector fields on
$\ovlc{1|n}\setminus\{0\}$. 

The second series is the divergence free vector fields. For a vector field
$X=f_0\pderiv{}{z}+\sum_{i=1}^n f_i\pderiv{}{\theta_i}\in\fvect^L(1|n)$, its divergence is defined as
\begin{equation}
\diver\, X=\pderiv{f_0}{z}+\sum_{i=1}^n(-1)^{p(f_i)}\pderiv{f_i}{\theta_i}.
\end{equation}
Then the algebras
\begin{equation}
\fsvect^L_\lambda(1|n):=\{X\in\fvect^L(1|n)\mid\diver(z^\lambda X)=0\}
\end{equation}
are superconformal if $\lambda\notin\intZ$. If one defines the standard holomorphic part of the
volume element with constant coefficients as
\begin{equation}
\mathrm{dvol}(z,\theta_1,\ldots,\theta_n)=dz\otimes\pderiv{}{\theta_1}\otimes\cdots\otimes\pderiv{}{\theta_n},
\end{equation}
then
\begin{equation}
L_X(z^\lambda\cdot \mathrm{dvol}(z,\theta_1,\ldots,\theta_n))=0\qquad\forall X\in\fsvect^L_\lambda(1|n).
\end{equation}
Thus one can view these as the algebras which preserve the $z^\lambda$-twisted holomorphic part of the
standard volume element.

One easily verifies that $\fsvect^L_\lambda(1|n)\cong\fsvect^L_\mu(1|n)$ if and only if
$\lambda-\mu\in\intZ$. If $\lambda\in\intZ$, the algebra $\fsvect^L_\lambda(1|n)$ is not
simple (so it is not superconformal according to Definition \ref{def:sconf}). It contains, however, a
simple ideal ${\fsvect^{L}_\lambda}'(1|n)$ of codimension $(1|0)$ if $n$ is even, and of codimension 
$(0|1)$, if $n$ is odd. This ideal is described by the exact sequence
\begin{equation}
0 \longrightarrow {\fsvect^{L}_\lambda}'(1|n) \longrightarrow \fsvect^L_\lambda(1|n) \longrightarrow%
f(z)\theta_1\cdots\theta_n\pderiv{}{z} \longrightarrow 0.
\end{equation} 

\subsection{The algebras of contact vector fields and M\"obius contact fields}

These comprise the third and fourth series of superconformal algebras. They owe their names to the
fact that they are supergeometric generalizations of the classical contact structures. 
Define the following one-forms on $\ovlc{1|n}\setminus\{0\}$:
\begin{eqnarray}
\label{conform}
\alpha_n &=& dz+\sum_{i=1}^n\theta_i d\theta_i,\\
\label{mconform}
\alpha_n^M &=& dz+\sum_{i=1}^{n-1}\theta_i d\theta_i+z\theta_n d\theta_n.
\end{eqnarray}
The first one will be called a \emph{contact form}, the second one a \emph{M\"obius contact form}. Then
we define the algebra of contact vector fields on $\ovlc{1|n}\setminus\{0\}$ as
\begin{equation}
\label{def:fkl}
\fkl(1|n) := \left\{X\in\fvect^L(1|n)\mid L_X(\alpha_n)=f_X\cdot\alpha_n\,\,%
\textrm{for some }f_X\in\cp^L(n)\right\},
\end{equation}
and the algebra of M\"obius contact vector fields as
\begin{equation}
\fkm(1|n) := \left\{X\in\fvect^L(1|n)\mid L_X(\alpha_n^M)=f_X\cdot\alpha_n^M\,\,%
\textrm{for some }f_X\in\cp^L(n)\right\}.
\end{equation}
These definitions mean that these vector fields preserve the forms (\ref{conform}) and (\ref{mconform}) up
to a factor, or equivalently, that they preserve their kernels. Therefore one often describes them
as the algebras which preserve the Pfaff equations 
\begin{equation}
\alpha_n(X)=0\qquad\mathrm{for}\,\,X\in\fvect^L(1|n),
\end{equation}
resp. $\alpha_n^M(X)=0$. It is well known that in the case of $\fkl(1|0)=\witt$, i.e., in 
the absence of odd dimensions, 
the contact vector fields coincide with the conformal vector fields on $\complexC^\times$. There is no
analog of the M\"obius superalgebra for the non-super case.

All algebras $\fkl(1|n)$ and $\fkm(1|n)$, except for $\fkl(1|4)$ and $\fkm(1|5)$, are simple. The latter two
contain a simple ideal of codimension $(1|0)$, denoted ${\fkl}'(1|4)$, resp. of codimension $(0|1)$ denoted
${\fkm}'(1|5)$. These two
simple subalgebras are, in fact, two of the four exceptional superconformal algebras. The other two
are $\mathfrak{m}^L(1)$, the algebra preserving the form $\beta=d\theta_1+zd\theta_2+\theta_2dz$ on
$\ovlc{1|2}\setminus\{0\}$ up to a factor, and $\mathfrak{kas}^L\subset\fkl(1|6)$, which is perhaps the
only truly exceptional one, generated by certain polynomial functions on $\ovlc{1|6}\setminus\{0\}$
(see \cite{S:Five}, \cite{LSoST} for the definitions and \cite{LJ:Super} for an extensive review).

For an explicit construction of these algebras, another description is more
useful. First, note that $\cp^L(N)$ can be endowed with a Lie superalgebra structure by introducing the
\emph{contact bracket} on it:
\begin{equation}
\{f,g\}_{kb}:=(2-E)(f)\pderiv{g}{z}-\pderiv{f}{z}(2-E)(g)-\{f,g\}_{pb}.
\end{equation}
Here, $E=\sum_{i=1}^N\theta_i\pderiv{}{\theta_i}$ is the so-called Euler operator, and
$\{f,g\}_{pb}$ is the \emph{Poisson bracket}:
\begin{equation}
\{f,g\}_{pb}:=-(-1)^{p(f)}\sum_{i=1}^N\pderiv{f}{\theta_i}\pderiv{g}{\theta_i}
\end{equation}

In \cite{LJ:Super}, it is shown that $\cp^L(N)$ with its contact bracket is isomorphic to $\fkl(1|N)$.
Specifically, we have:
\begin{prop}
\label{p1}
To every $f\in\cp^L(N)$, assign a vector field $K_f$ by setting
\begin{equation}
K_f:=(2-E)(f)\pderiv{}{z}-H_f+\pderiv{f}{z}E,
\label{kf}
\end{equation}
where
\begin{equation}
H_f:=-(-1)^{p(f)}\sum_{i=1}^N\pderiv{f}{\theta_i}\pderiv{}{\theta_i}
\end{equation}
assigns to each $f$ a Hamiltonian vector field. Then 
\begin{equation}
L_{K_f}(\alpha_N)=2\pderiv{f}{z}\alpha_N,
\end{equation}
and
\begin{equation}
\fkl(1|N)=\mathrm{Span}\{K_f\big|f\in\cp^L(N)\}.
\end{equation}
Furthermore,
\begin{equation}
[K_f,K_g]=K_{\{f,g\}_{kb}}.
\end{equation}
\end{prop}

We do not want to prove this statement, which can also be extended to spaces with more than one
even coordinate, but refer to \cite{LSoST}, \cite{LJ:Super} for an extensive analysis. 
For the M\"obius algebras, an analogous description is also given in \cite{LSoST}, \cite{LJ:Super}.

\subsection{Central extensions, critical dimensions}

These infinitely many superconformal algebras all define superconformal surfaces of dimension $1|n$,
but not all of these surfaces are suitable for superconformal field theories. A superconformal algebra
can only be the algebra of a two-dimensional superconformal field theory if it possesses a nontrivial
central extension. These extensions have been classified (for the algebras known at that time) 
in \cite{KvdL:On}. For a complete review, see \cite{LJ:Super}. In general, 
the maximal dimension of a nontrivial central extension of a Lie algebra $\fg$ is the dimension of its
second cohomology group $H^2(\fg)$. To find suitable representatives for these cohomology classes can
be a quite nontrivial job. We will only list the extensions here, for information on how to obtain
them, and how to represent them by superfields, see \cite{KvdL:On}, \cite{LJ:Super},
\cite{S:Thecocycles}.

Of the contact algebras, the following ones have nontrivial central extensions: 
$\witt$ has a one-dimensional extension, the Virasoro algebra $\mathfrak{vir}$. The
algebras $\fkl(1|1)$, $\fkl(1|2)$, $\fkl(1|3)$ each have a one-dimensional extension, 
while ${\fkl}'(1|4)$ has
a 3-dimensional one.
These extensions are called the Neveu-Schwarz algebras $\fns(1)$, $\fns(2)$, $\fns(3)$ and $\fns(4)$. The
latter is one particular of the three extensions of $\fkl(1|4)$. 
For $n\geq 5$, none of the $\fkl(1|n)$ possesses an extension.

Of the M\"obius contact algebras, each of the algebras
$\fkm(1|1)$, $\fkm(1|2)$, $\fkm(1|3)$, $\fkm(1|4)$ has a one-dimensional extension. 
These extensions are dubbed the Ramond algebras
$\fram(n)$. None of the $\fkm(1|n)$ for $n\geq 5$ has a nontrivial extension. 

Of the other superconformal algebras, $\fvect^L(1|1)$, $\fvect^L(1|2)$, $\fsvect^L_\lambda(1|2)$ and 
$\mathfrak{m}^L(1)$ each have a unique nontrivial central extension.

This is the complete list, so altogether, there are at best 15 possibilites to define superconformal
field theories if one adopts Definition \ref{def:sconf} for a superconformal algebra. There might,
however, be other possbilities if one allows loop or Kac-Moody algebras. Yet, not all of these 15
theories can describe the worldsheets of a superstring theory. In order to be able to define such a
theory free of anomalies, there must exist a spacetime of dimension $D_{crit}\geq 0$ (obviously, one would
prefer $D_{crit}\geq 4$), in which the (super-)conformal anomaly cancels \cite{GSW:Superstring},
\cite{DP:Thegeometry}. This is quite a severe restriction. The Virasora algebra has $D_{crit}=26$,
$\fns(1)$ and $\fram(1)$ both have $D_{crit}=10$, and $\fns(2)$ and $\fram(2)$ have $D_{crit}=2$. For
the other centrally extendable algebras, the critical dimension is $\leq 1$ \cite{LS:Critical}.

\section{Spin surfaces and super Riemann surfaces}

\subsection{Complex supersurfaces}
\label{sect:compss}

In all of the following work we will be concerned with orientable smooth supersurfaces of
dimension $2|2$. Let us assume we are given such a surface $\cm$ together with a complex structure, that
is, we have a complex supermanifold $\cm$ of dimension $1|1$. Let $(z,\theta)$ and $(z',\theta')$ be two
overlapping complex coordinate charts on $\cm$. Then the transition function between these must have the form
\begin{eqnarray}
\label{transf}
z' &=& f(z)\\
\nonumber
\theta' &=& g(z)\theta,
\end{eqnarray}
because it must be a morphism of superalgebras on each stalk. Writing the structure sheaf as
$\co_\cm=\co_{\cm,\bar{0}}\oplus\co_{\cm,\bar{1}}$, we find that $\co_{\cm,\bar{0}}$ is simply the sheaf
$\co_M$ of holomorphic functions on the underlying Riemann surface $M=\cm_{red}$,\footnote{In complex
supergeometry, also even nilpotent elements can occur in the structure sheaf. In this case, 
$\cm_{red}$ denotes the complex analytic space where one has only divided out the ideal of odd elements from
the structure sheaf, while $\cm_{rd}$ denotes the underlying space, i.e. the completely reduced one.
In this work, we do not have to make a distinction, since no non-reduced ordinary complex spaces occur.}  
while $\co_{\cm,\bar{1}}$ is
a sheaf of locally free modules of rank $0|1$ over $\co_{\cm,\bar{0}}$. Thus, setting 
\begin{equation}
L:=\Pi(\co_{\cm,\bar{1}})
\end{equation}
turns $\cm$ into a Riemann surface endowed with a locally free sheaf of $\co_M$-modules of rank 1, i.e., a
line bundle. 

Conversely, it is obvious that starting with a pair $(M,L)$ consisting of a Riemann surface and a holomorphic
line bundle $L$, we can produce a $1|1$-dimensional complex supermanifold $\cm$ by just setting
$\cm=(M,\co_\cm=\co_M\oplus\Pi(\Gamma(L)))$, where $\Gamma(L)$ is the sheaf of holomorphic sections of $L$.
Thus we have shown

\begin{prop}
\label{isosrs}
There is a bijection between the set of pairs $(M,L)$ of Riemann surfaces with a holomorphic line bundle
and the set of complex supermanifolds $\cm$ of dimension $1|1$.
\end{prop}

In fact, we will solely be concerned with compact (super) Riemann surfaces without boundary later on, but
Prop.~\ref{isosrs} holds also in the noncompact case.

This is the point where one has to clarify an important conceptual issue. 
In much of the literature (e.g., \cite{CR:Super}), the transition functions between 
$(z,\theta)$ and $(z',\theta')$ are not given by (\ref{transf}), but rather by expressions like
\begin{eqnarray}
\label{transf2}
z' &=& f(z) + \psi(z)\theta\\
\nonumber
\theta' &=& \eta(z)+ g(z)\theta,
\end{eqnarray}
where $\psi(z),\eta(z)$ are functions taking ``odd values''. As remarked earlier, this point of view
reflects the idea of functions taking values in some Grassmann algebra which we prefer to avoid
in this work. Instead, we want to keep a clear distinction between the notion of a single supermanifold 
and that of a family of supermanifolds. 

On a single ringed space $\cm$, where each
stalk is just an exterior algebra over the algebra of germs of holomorphic functions, it does not make
any sense to speak of functions depending only on $z$ as being odd. If, however, one looks at a
family $\ct\times\cm$, where $\ct$ is supermanifold with structure sheaf $\co_\ct$, then the transition
maps on $\cm$ may contain odd germs from the stalk of the base $\ct$. Then $\psi(z),\eta(z)$ can be
understood as being proportional to such an odd germ. The geometric picture describing this situation 
most accurately
is to think of the presence of these odd germs as deformations of the transition function (\ref{transf})
along odd dimensions of the base.
Thus the transition function (\ref{transf}) describes the family $\ct\times\cm$ restricted to the
base $\ct_{red}$, while (\ref{transf2}) describes the full family, which can be understood as
a deformation involving odd and even parameters.

To illustrate this point, consider the following example. Let $\cm$ be a complex
$1|1$-dimensional supermanifold, and let $\ct=\Spec\Lambda_1^\complexC=(\{*\},\Lambda_1^\complexC)$ be the
complex superpoint with one odd dimension. Let $\tau$ be the generator of $\Lambda_1$. 
Now consider the family $\ct\times\cm$. Let $\cu,\cv$ be two
superdomains on $\cm$. The transition
function between $\cu$ and $\cv$ is then a morphism of families: a map of superdomains parametrized
by the base. This means it is a map $\Phi:\ct\times\cu\to\ct\times\cv$ which makes the diagram
\begin{equation}
\xymatrix{\ct\times\cu \ar[dr]_{\Pi_\ct} \ar[rr]^\Phi && \ct\times\cv \ar[dl]^{\Pi_\ct}\\
& \ct &}
\end{equation}
commutative. This implies that $\Phi$ can be written as $\Pi_\ct\times(\Phi_2:\ct\times\cu\to\cv)$. Fix a
point $p$ in the underlying domain $U$ which gets mapped to $\phi(p)$ by the underlying homeomorphism
$\phi$ of $\Phi_2$. The sheaf map $\varphi$ associated with $\Phi_2$ maps the stalk $\co_{\cv,\phi(p)}$ into
$\co_{\cu,p}\otimes\Lambda_1^\complexC$.
Let $z,\theta$ be the generators of the stalk $\co_{\cv,\phi(p)}$, then their image
under $\varphi$ is
\begin{eqnarray}
z' &=& f(z)+\tau g(z)\theta\\
\theta' &=& h(z)\theta+\tau k(z)
\end{eqnarray}
where $f,g,h,k$ are germs of holomorphic functions in the stalk $\co_{\cu,p}$. This is precisely the
form of (\ref{transf2}).

For smooth supermanifolds, it always suffices to study just families over the superpoints 
$\cp(\Lambda_n),\,n\in\naturalN_0$. One may view this fact
as the extension of Ehresmann's theorem \cite{E:connexions} to the super case: there are no 
nontrivial deformations of a smooth structure by real parameters.
For complex supermanifolds, however, even families over purely even complex base spaces can be
nontrivial, just as in the case of classical complex geometry.

This discussion should make it clear that, although Prop.~\ref{isosrs} states that the set of single
complex $1|1$-dimensional supermanifolds (i.e., families over the point $\Spec\complexC$) is 
in one-to-one correspondence with pairs of Riemann surfaces and line bundles,
the Teichm\"uller and moduli spaces of these two types of structures will be different. These spaces
describe deformations of structures, and as seen above, there are more deformations of a super object
than just the ones of the classical underlying object.

All supersurfaces of complex dimension $1|n$ are superconformal surfaces, namely they are 
$\fvect^L(1|n)$-surfaces.
All other superconformal surfaces are subspecies of the $\fvect^L(1|n)$-surfaces, since all 
superconformal algebras can be realized as subalgebras of algebras of vector fields.

\subsection{$\fkl(1|1)$-surfaces and spin curves}

The notion of a super Riemann surface (SRS) was introduced by Friedan \cite{F:Notes} in the context
of superstring theory, which can be viewed as a special type of 2D supergravity. In some of the
mathematical literature \cite{M:Topics}, \cite{Gftacg}, the term $\mathrm{SUSY}_1$-curve is 
used for an SRS. It is defined as a 
complex $1|1$-dimensional supermanifold $\cm$ with 
the additional property that it possesses a maximally non-integrable distribution 
$\cd\subset \ct\cm$ of rank $0|1$. This means that the pairing given by
\begin{eqnarray}
\label{ddef}
\cd\otimes\cd &\to& \ct\cm/\cd\\
\nonumber
X\otimes Y &\mapsto& [X,Y]/\cd,
\end{eqnarray}
where $[\cdot,\cdot]$ is the Lie bracket, is an isomorphism. Recalling that a distribution $\cd$ is 
integrable if and only if $[\cd,\cd]\subseteq\cd$,
it is obvious why a distribution with the above property is called maximally non-integrable. In ordinary
geometry, such a distribution corresponds precisely to a contact structure, and we will show that in
the super case, it is preserved by the contact algebra $\fkl(1|1)$.

A theorem of LeBrun and Rothstein \cite{LBR:Moduli} states that, given a distribution $\cd$ with 
properties (\ref{ddef}), one can always find a local coordinate
system $(z,\theta)$ such that $\cd$ is locally generated by the odd vector field
\begin{equation}
D=\pderiv{}{\theta}+\theta\pderiv{}{z}.
\end{equation}
That means that the transition functions of a super Riemann surface have to preserve $D$ up to an
invertible factor. Writing again
\begin{eqnarray}
z' &=& f(z)\\
\nonumber
\theta' &=& g(z)\theta,
\end{eqnarray}
one deduces that 
\begin{equation}
D=\theta f'(z)\pderiv{}{z'}+g(z)\pderiv{}{\theta'}.
\end{equation}
So in order to have $D\sim D'$, we have to require
\begin{equation}
\label{srstrans}
f'(z)=g(z)^2.
\end{equation}
This differs from the result in \cite{CR:Super} because we only look at a single SRS here,
while the authors of \cite{CR:Super} implicitly study a family of SRS parametrized by a supermanifold 
(cf.~the discussion
in the previous section). Equation (\ref{srstrans}) states that under super coordinate transformations, the
odd coordinate $\theta$ transforms like a section of a spin bundle $S=K^{1/2}$. By an argument analogous
to that in Section \ref{sect:compss}, we conclude

\begin{prop}
\label{isofkl}
There exists a bijection between the set of super Riemann surfaces and the set of pairs $(M,S)$, where
$M$ is Riemann surface and $S$ is a spin bundle on $M$.
\end{prop}

Riemann surfaces with a spin structure are also called \emph{spin curves}. For an extensive investigation,
see, e.g., \cite{A:Riemann}. On a Riemann surface, a spin bundle is simply a holomorphic line bundle 
which transforms by the square root of the transition functions of the canonical bundle $K=T^*M$. 
If the surface $M$ has genus $g$, then
there exist $2^{2g}$ spin structures on $M$, one for each element of $H^1(M,\intZ_2)$. The reason is
that to fix a square root of $K$ uniquely, one has to choose the sign for the square roots of the transition
functions which one picks up running along any homologically nontrivial 1-cycle on $M$.
Often a spin bundle
is denoted as $K^{1/2}$, but that notation is, of course, not unique if the genus of the surface
is greater than zero. The same goes for other half powers of $K$, like $K^{3/2}$. All this implies that
a Riemann surface $M$ is able to carry at least $2^{2g}$ non-equivalent SRS-structures. But again, by
the same arguments as outlined in the previous section, it is the possibility of odd deformation parameters
which makes the Teichm\"uller and moduli spaces for SRS interesting, even though the underlying 
spaces of these
can at this point already be expected to be just the spin Teichm\"uller and spin moduli spaces.

Super Riemann surfaces are, if one uses the naming of a superconformal surface after the superconformal
structure it carries, $\fkl(1|1)$-surfaces. By definition, the algebra $\fkl(1|1)$ preserves the
contact form $\omega=dz+\theta d\theta$ up to a factor. This means it preserves its kernel. 
Using our convention (\ref{sforms}) for the pairing between forms and vector fields, we see that
\begin{equation}
\langle f(z,\theta)\pderiv{}{z}+g(z,\theta)\pderiv{}{\theta}, dz+\theta d\theta\rangle=0
\end{equation}
implies $f=g\theta$. Therefore the kernel of $\omega$ consists of elements of the form
\begin{equation}
g(z,\theta)\left(\theta\pderiv{}{z}+\pderiv{}{\theta}\right),
\end{equation}
and therefore coincides with the distribution $\cd$ defined above. 

\subsection{Families of spin curves and of $\fkl(1|1)$-curves}

It is interesting to directly compare the situation for families of spin curves and SRS, because it
shows how far the relation between these two types of objects reaches, as well as where the differences occur.
We follow here the exposition given in \cite{M:Topics}, Chapter 2.

Let $\pi:\cx\to \cb$ be a family of complex supermanifolds. The associated sheaf map 
$\pi^*:\co_\cb\to\co_\cx$ of this projection embeds the structure sheaf of the base into that of 
the total space. The relative tangent sheaf $\ct_{\cx/\cb}$ is then defined as
the subsheaf of vector fields in $\ct_\cx$ which annihilate $\pi^*(\co_\cb)$. One may think of them as those
vector fields which point ``vertically'' along the fibers, such that the images of germs of functions from
$\co_\cb$ are treated as constants by these vector fields.

A family of super Riemann surfaces is a family $\pi:\cx\to\cb$ of complex supermanifolds such that the
relative tangent sheaf contains a distribution $\cd\subset\ct_{\cx/\cb}$ of rank $0|1$ which is
maximally non-integrable (compare with Definition \ref{ddef}). One may think of this as a distribution
satisfying Definition \ref{ddef} on \emph{every fiber}. One must be warned, however, that the concept of
``fibers'' parametrized by a supermanifold differs from that of fibers over an ordinary manifold or
topological space, since the supermanifold is not specified by its topological points. It parametrizes
the fibers both by even as well as odd parameters, and only the former ones can be thought of as
a parametrization in the sense of families of topological spaces. 

By the considerations of the previous section, it is clear how to construct a family of $\fkl(1|1)$-curves
over a purely even base $B_0$. One starts with a family $\pi_0:X_0\to B_0$ of relative dimension
$1|0$, i.e., a family of ordinary complex 1-dimensional manifolds. Choosing a line bundle $L$ on $X_0$
and declaring it odd turns every fiber into a complex supermanifold of dimension $1|1$. To obtain a
family of $\fkl(1|1)$-surfaces, we choose a line bundle $L$ for which there exists an isomorphism
\begin{equation}
\alpha:L\otimes L\to \ct^*_{X_0/B_0},
\end{equation}
which means nothing else than the requirement that $L$ restricts to a square root of the canonical bundle
of each fiber. Bundles with this property are called \emph{theta characteristics}\footnote{Often
the points of the half-period lattice on the Jacobian $\Jac(\Sigma)$ of a Riemann surface $\Sigma$
are called the theta characteristics of $\Sigma$. This almost coincides with the above usage of the term:
the divisor classes of line bundles which square to the canonical bundle are given by the 
half-period lattice shifted
by the vector of Riemann constants. For more details see Section \ref{sect:earle} or \cite{GH:Principles},
\cite{ACGH:Geometry}.} of $\pi_0$. Denoting 
the sheaf of sections of $L$ by $\Gamma(L)$, we obtain a family of particular
complex supermanifolds by setting $\cx=(X_0,\co_{X_0}\oplus\Pi(\Gamma(L)))$ and a projection $p:\cx\to X_0$,
as well as a projection $\pi=\pi_0\circ p:\cx\to B_0$.
This family can explicitly be endowed with the structure of a family of super Riemann surfaces: take
a local relative coordinate $z$ on $X_0$ and a section $\theta$ of $L$ such that 
$\alpha(\theta,\theta)=d_{X/B_0}z$ ($d_{X/B_0}$ denotes the relative differential). Such a section 
always exists, because $\alpha$ is an isomorphism. Now we can set $(z,\theta)$ as 
a local relative coordinate system on $\cx$, and define
$D:=\pderiv{}{\theta}+\theta\pderiv{}{z}$. Then the distribution $\co_XD\subset\ct_{\cx/B_0}$ satisfies
the condition of maximal non-integrability, and we have turned $X$ into a family of super Riemann surfaces. 
This construction does not depend on the relative coordinate
$z$ that one starts with: if we had started with $z'=f(z)$, then we would have ended up with
$\theta'=\sqrt{\pderiv{f}{z}}\theta$, and thus with
\begin{equation}
D'=\left(\pderiv{f}{z}\right)^{-\frac{1}{2}}D.
\end{equation}
Since $f$ is conformal, i.e., holomorphic and with nowhere vanishing derivative, the coordinate system
$(z',\theta')$ would thus have produced the same distribution in $\ct_\cx$. In \cite{M:Topics} it is
shown that the converse also holds:
every family $\pi:\cx\to \cb$ of super Riemann surfaces reduces to a family $X_{red}\to B_{red}$ with
a canonically determined theta characteristic.

\begin{thm}
Let a family $\pi_0:X_0\to B$ of complex supermanifolds of relative dimension $1|0$ be given. Then 
there exists a bijection between the following two sets:
\begin{enumerate}
\item relative theta characteristics of $\pi_0$ up to equivalence
\item $\fkl(1|1)$-families $\pi:X\to B$ such that $X_{red}=X_{0,red}$ for which 
$\co_{X.\bar{0}}\cong\co_{X_0}$.
\end{enumerate}
\end{thm}

A universal family of super Riemann surfaces above their moduli space must therefore restrict to a family
spin curves, as well. Indeed, the underlying spaces of the Teichm\"uller and moduli spaces of SRS turn
out to be just the Teichm\"uller and moduli spaces of spin curves. The reference for the above Theorem and
the preceding discussion is \cite{M:Topics} and references therein, in particular \cite{D:Letter}.

%% file: almostcpx.tex
\chapter{The supermanifold $\ca(\cm)$}
\label{ch:acs}

In this chapter, we will construct a supermanifold $\ca(\cm)$ of all almost complex structures on a given
almost complex supermanifold $\cm$. In subsequent Chapters we will identify the tangent spaces to the
Teichm\"uller spaces that we wish to construct as certain subspaces of the tangent spaces of $\ca(\cm)$.
The Teichm\"uller spaces themselves can therefore be thought of as being glued together from local
submanifolds of $\ca(\cm)$.

\section{Supermanifolds of sections of super vector bundles}

Let $\pi:\ce\to\cm$ be a smooth super vector bundle over a compact supermanifold $\cm$. We would like
to give the set of sections 
\begin{equation}
\Gamma(\cm,\ce):=\{\sigma:\cm\to\ce\big|\,\pi\circ\sigma=\id_\cm\}
\end{equation}
the ``structure of a supermanifold''. As always in the super context, to obtain this set is by far not 
sufficient for this purpose; it will only form the underlying space of such a supermanifold.
The starting point of the construction is the inner Hom-object $\scinfty(\cm,\ce)$ for the set 
$\Hom(\cm,\ce)=SC^\infty(\cm,\ce)$. By definition, we have
\begin{eqnarray}
\scinfty(\cm,\ce)(\Lambda)=SC^\infty(\cp(\Lambda)\times\cm,\ce).
\end{eqnarray}
This inner Hom-object is in general not a Banach supermanifold, for the same reason as in ordinary geometry: 
spaces of smooth maps usually only form Fr\'echet manifolds.

Next, define a functor $\gh(\cm,\ce):\catgr\to\catsets$ on the objects of $\catgr$ by setting
\begin{equation}
\gh(\cm,\ce)(\Lambda):=\Gamma(\cp(\Lambda)\times\cm,\Pi_\cm^*\ce).
\end{equation}
Here, $\Pi_\cm^*\ce$ denotes the pullback of $\ce$ along the projection
$\Pi_\cm:\cp(\Lambda)\times\cm\to\cm$.
For a morphism $\varphi:\Lambda\to\Lambda'$, we define
\begin{eqnarray}
\label{svbungr}
\gh(\cm,\ce)(\varphi):\gh(\cm,\ce)(\Lambda) &\to& \gh(\cm,\ce)(\Lambda')\\
\nonumber
\sigma &\mapsto& \sigma\circ(\cp(\varphi)\times\id_\cm).
\end{eqnarray}

The functor $\gh(\cm,\ce)$ represents the global sections of the bundle $\pi:\ce\to\cm$ as a functor
in $\catsets^\catgr$.
As with $\scinfty(\cm,\ce)$, this functor cannot be a Banach supermanifold, but at most a Fr\'echet
supermanifold. Nonetheless, it can be easily constructed, since it turns out that it is actually a
superrepresentable $\olr$-supermodule.

\begin{thm}
\label{thm:sectsvbun}
Let $\pi:\ce\to\cm$ be a smooth super vector bundle. Then the functor $\gh(\cm,\ce)$ is a superrepresentable
$\olr$-module, i.e., there exists an $\realR$-super vector space $E$ such that
\begin{equation}
\gh(\cm,\ce)\cong \overline{E}.
\end{equation}
\end{thm}
\begin{proof}
We first show that $\gh(\cm,\ce)$ is indeed an $\olr$-module. Let $\sigma:\cp(\Lambda)\times\cm\to\ce$
be a given section, and let $\lambda\in\Lambda_{\bar{0}}$ be an element of $\olr(\Lambda)$. Because
we have $\cp(\Lambda)(\Lambda')\cong\Hom(\Lambda,\Lambda')$ (cf.~Prop.~\ref{homgr}), 
every $p\in\cp(\Lambda)(\Lambda')$ may
be viewed as a morphism $p:\Lambda\to\Lambda'$. We then define the functor morphism
\begin{equation}
\lambda\cdot\sigma:\cp(\Lambda)\times\cm\to\ce
\end{equation}
componentwise:
\begin{eqnarray}
\nonumber
(\lambda\cdot\sigma)_{\Lambda'}:\cp(\Lambda)(\Lambda')\times\cm(\Lambda') &\to& \ce(\Lambda')\\
(p,m) &\mapsto& p(\lambda)\cdot \sigma_{\Lambda'}(p,m).
\end{eqnarray}
Here, we have used the fact that $\sigma_{\Lambda'}(p,m)\in (\pi_{\Lambda'})^{-1}(m)\cong\cv(\Lambda')$,
i.e., that the fibres over all points $m\in\cm(\Lambda')$ carry the structure of 
$\olr(\Lambda')=\Lambda'_{\bar{0}}$-modules.

Now consider the set of sections
\begin{equation}
E:=\Gamma(\cm,\ce\oplus\Pi\ce)=\{\sigma:\cm\to\ce\oplus\Pi\ce\,\big|\,\pi\circ\sigma=\id_\cm\}.
\end{equation}
This set carries the structure of an $\realR$-super vector space: every section taking values in
$\ce$ is defined to be even, every one taking values in $\Pi\ce$ to be odd. We claim that
$\overline{E}\cong\gh(\cm,\ce)$, i.e., that there exists a bijection between the sets
$\overline{E}(\Lambda)$ and $\gh(\cm,\ce)(\Lambda)$ for all $\Lambda\in\catgr$.

We will prove this locally, i.e., on every stalk, using the ringed space formalism. Denote the rank of
the super vector bundle by $a|b$.
Let $p\in M$ be
a point in the underlying manifold $\cm_{rd}$, and let $\sigma:\cp(\Lambda_n)\times\cm\to\ce$ be
an element of $\gh(\cm,\ce)(\Lambda_n)$. Then
$\sigma$ is determined by a continuous map of the underlying spaces $s:\{*\}\times\cm_{rd}\to\ce_{rd}$,
and by a sheaf map $\sigma^*:\co_{\ce}\to\co_{\cm}\times\Lambda_n$. The sheaf map induces a stalk 
map between the stalk at $(\{*\},p)$ and the stalk at $s(p)$:
\begin{eqnarray}
\label{vstalk}
\sigma^*_p:\co_{\ce,s(p)} &\to& \co_{\cm,p}\times\Lambda_n\\
\nonumber
f &\mapsto& \sum_{I\subseteq\{1,\ldots,n\}}\tau_I\alpha_I(f),
\end{eqnarray}
where each $\alpha_I$ is a homomorphism $\co_{\ce,s(p)}\to\co_{\cm,p}$ of superalgebras, and the sum
runs over all increasingly ordered subsets. The $\tau_1,\ldots,\tau_n$ are the odd generators of $\Lambda_n$,
and $\tau_I$ is the product of all the $\tau$'s indexed by $I$ in the same order (cf.~Thm.~\ref{sdiff}
for this notation). By Thm.~\ref{sdiff} we know that there exists a homomorphism
$\alpha_0:\co_{\ce,s(p)}\to\co_{\cm,p}$ of superalgebras, as well as $2^{n-1}$ odd and $2^{n-1}-1$ even derivations
of $\co_{\ce,s(p)}$ such that the stalk map at $p$ induced by any morphism 
$\sigma:\cp(\Lambda_n)\times\cm\to\ce$ can be written as
\begin{equation}
\label{sigmap}
\sigma^*_p=\alpha_0\circ\exp\left(\sum_{I\subseteq\{1,\ldots,n\}}\tau_IX_I\right).
\end{equation}
Here, each $X_I$ is a derivation of parity $|I|$ of $\co_{\ce,s(p)}$.

Now the fact that $\sigma$ is a section implies that the composition with $\pi:\ce\to\cm$  
must satisfy $\pi\circ\sigma=\Pi_\cm$. That means that for any germ $f\in\co_{\cm,p}$, the composition of
sheaf maps satisfies $\sigma^*\pi^*(f)=\Pi_\cm^*(f)$. Inserting $\pi^*(f)$ into 
(\ref{sigmap}) yields immediately the conditions
\begin{equation}
\alpha_0(\pi^*(f))=f,\qquad X_{I}(\pi^*(f))=0.
\end{equation}
The first of these conditions means that the ``underlying'' section $\sigma:\cm\to\ce$, described 
by $\alpha_0$, is a local inverse of the projection $\pi$, just as one would have expected. The second one
tells us that all the derivations $X_I$ in (\ref{sigmap}) have to be \emph{relative}, i.e. they must
annihilate the functions depending only on coordinates of the base.

Since a super vector bundle of rank $a|b$ is
a locally free module of this rank over $\co_\cm$, we know that the set of all homomorphisms
$\alpha_0$ is
\begin{equation}
\label{salghom}
\Hom_{SAlg}(\co_{\ce,s(p)},\co_{\cm,p})\cong (\co_{\cm,\bar{0},p})^a\oplus(\co_{\cm,\bar{1},p})^b.
\end{equation}
The relative derivations $X:\co_{\ce,s(p)}\to\co_{\ce,s(p)}$ form a module of rank $a|b$ over
$\co_{\ce,s(p)}$. Each of the $X_I$ appears composed with $\alpha_0$, which maps $\co_{\ce,s(p)}$
to $\co_{\cm,p}$, so we can consider each of the
$\alpha_I$ in (\ref{vstalk}) actually as an element of an $\co_{\cm,p}$-module of rank $a|b$. So
we can write
\begin{equation}
\alpha_I\in\left\{\begin{array}{ll}
(\co_{\cm,\bar{0},p})^a\oplus(\co_{\cm,\bar{1},p})^b & \textrm{ if }|I|\textrm{ even}\\
(\co_{\cm,\bar{1},p})^a\oplus(\co_{\cm,\bar{0},p})^b & \textrm{ if }|I|\textrm{ odd}.
\end{array}\right.
\end{equation}
Therefore, writing
\begin{equation}
V=V_{\bar{0}}\oplus V_{\bar{1}}=\left[(\co_{\cm,\bar{0},p})^a\oplus(\co_{\cm,\bar{1},p})^b\right]\oplus\\
\left[(\co_{\cm,\bar{1},p})^a\oplus(\co_{\cm,\bar{0},p})^b\right]
\end{equation}
shows us that we may identify the set of all stalk maps $\sigma^*_p$ as in (\ref{vstalk}) with the vector
space $\olv(\Lambda_n)$. For $\Lambda_n=\realR$, we thus obtain the space $V_{\bar{0}}$, which is
the space (\ref{salghom}). Since these are just the sections $\cm\to\ce$, we have $V_{\bar{0}}=E_{\bar{0}}$. 
For $\Lambda_n=\Lambda_1$, we obtain $V_{\bar{0}}\oplus\Pi(V_{\bar{0}})$
because $V_{\bar{1}}=\Pi(V_{\bar{0}})$, and this is isomorphic to $E_{\bar{0}}\oplus E_{\bar{1}}$.
It is then clear that $\bar{E}(\Lambda)\cong \bar{V}(\Lambda)$ for all $\Lambda\in\catgr$, which
completes the proof.
\end{proof}

Theorem \ref{thm:sectsvbun} reflects the intuitive expectation that the sections of a super vector
bundle on a super manifold should inherit the structure of a super vector space from the fibers, as it
is the case for ordinary vector bundles. That the proof is somewhat involved shows, on the other hand,
that things are considerably more complicated in the super setting. This is due to the fact that one
does not have a collection of super vector spaces, parametrized by some topological space (where the
sections would obviously form a super vector space even if they were not even continuous), 
but rather a collection parametrized by the smooth structure
of the base manifold $\cm$ \emph{and} by the odd elements of the structure sheaf $\co_\cm$.
As a corollary, we obtain the following theorem of Molotkov \cite{I-dZ2ks}.
\begin{cor}
\label{trivbd}
Let $\cm$ be a supermanifold and $\cv$ be a superrepresentable $\olr$-module.
Then there exists an isomorphism of functors 
\begin{equation}
\scinfty(\cm,\cv)\cong \overline{SC^\infty(\cm,\cv\oplus\Pi(\cv))},
\end{equation}
which is natural in $\cm$ and $\cv$.
\end{cor}
\begin{proof}
This follows directly from the application of Thm.~\ref{thm:sectsvbun} to the trivial super vector
bundle $\Pi_\cm:\cm\times\cv\to\cm$.
\end{proof}

Actually, as in the classical case, the $\olr$-module
$\gh(\cm,\ce)$ of global sections of a trivial super vector bundle is a module over the globally
defined superfunctions, but this fact will not be used in the following.

\section{Construction of the supermanifold $\ca(\cm)$}

\subsection{Almost complex structures and $Q$-structures}
\label{sect:qstruct}

An almost complex structure on a supermanifold $\cm$ is a section
\[
J:\cm\to\cend(\ctm)
\]
of the endomorphism bundle of $\ctm$, such that $J^2=-\One$. Here,
$\One$ denotes the section of $\cend(\ctm)$ which is constantly the identity.

On a supermanifold of dimension $2n|2n$, there might be odd sections as well as even ones which satisfy
$J^2=-\One$. So in the case of a smooth $2|2$-dimensional supermanifold, the case we want to study in this
work, there might exist odd almost complex structures. Although they are an interesting case in their
own right, we will ignore them in this work. The main reason is that the geometry defined by them is 
quite different from the one defined by an even almost complex structure. The latter
corresponds to a reduction of the frame group of a supermanifold from $GL(2n|2m;\realR)$ to
$GL(n|m;\complexC)$ and can therefore be viewed as a superization of the concept of an almost complex 
structure. An odd one, however, corresponds to a reduction of the frame group from
$GL(2n|2n;\realR)$ to the supergroup $Q(n)$ (the so-called queer analogue of the general linear
group \cite{SoS}), so the term ``Q-structure'' seems more appropriate for them 
(see \cite{Acsos}, \cite{Docsos} for this terminology and an in-depth discussion). Q-structures
are rigid \cite{Acsos}. The geometry defined them is completely different from the one of
complex supermanifolds. In particular, it is non-supercommutative.

\subsection{Construction of $\ca(\cm)$}

We denote from now on by $\ca(\cm)(\Lambda)$ the almost complex structures in $\gh(\cend(\ctm))(\Lambda)$,
i.e. the sections 
\begin{equation}
J:\cp(\Lambda)\times\cm\to\cp(\Lambda)\times\cend(\ctm)
\end{equation}
such that $J^2=-\One_\Lambda$. In view of the definition given above, these sections
should more precisely be called almost complex structures on the trivial family over $\cp(\Lambda)$ with 
fiber $\cm$. As always, the higher $\Lambda$-points of a superspace of maps between two super objects
$\cm,\cn$ are maps $\cp(\Lambda)\times\cm\to\cn$.

It will be shown that the almost complex structures on an almost complex supermanifold $\cm$ form a 
submanifold in $\gh(\cend(\ctm))$. The strategy will be to extend the technique invented by U.~Abresch
and A.E.~Fischer \cite{TTiRG} for the construction of holomorphic coordinates on the set of almost complex 
structures of an ordinary manifold. This will turn each of the sets $\ca(\cm)(\Lambda)$ into a complex
manifold, and it will subsequently be shown that the Abresch-Fischer charts which provide these
manifold structures can be constructed in a functorial manner with respect to $\Lambda$. This in turn 
allows us to make them the points of supercharts on the functor $\ca(\cm)$.

\begin{lemma}
The points $\ca(\cm)(\Lambda)$ form a subfunctor of $\gh(\cend(\ctm))$ in $\catsets^\catgr$.
\end{lemma}
\begin{proof}
Each set $\ca(\cm)(\Lambda)$ is obviously a subset of $\gh(\cend(\ctm))(\Lambda)$. It remains to be
shown that the family of inclusions $\{\ca(\cm)(\Lambda)\subset\gh(\cend)(\Lambda)|\Lambda\in\catgr\}$
is a functor morphism. This means we have to prove that for a morphism $\varphi:\Lambda\to\Lambda'$, we can
define a morphism 
\begin{equation}
\label{tfing}
\ca(\cm)(\varphi):\ca(\cm)(\Lambda)\to\ca(\cm)(\Lambda')
\end{equation}
in such a way that it is compatible with the image of $\varphi$ under $\gh(\cend(\ctm))$. This image is
the map (recall that $\cp$ is contravariant)
\begin{eqnarray}
\gh(\cend(\ctm))(\varphi):\gh(\cend(\ctm))(\Lambda) &\to& \gh(\cend(\ctm))(\Lambda')\\
\nonumber
S &\mapsto& S\circ(\cp(\varphi)\times\id_\cm).
\end{eqnarray}
We will show that defining $\ca(\cm)(\varphi)$ as the restriction
\begin{equation}
\ca(\cm)(\varphi)=\gh(\cend(\ctm))(\varphi)\big|_{\ca(\cm)(\Lambda)}
\end{equation}
does the job. $\gh(\cend(\ctm))$ is an $\olr$-algebra, therefore $\gh(\cend(\ctm))(\varphi)$ is a
map between a $\Lambda_{\bar{0}}$- and a $\Lambda'_{\bar{0}}$-algebra whose multiplications are the
components of a functor morphism. Therefore, if $\sigma\in\ca(\cm)(\Lambda)$ is given and we denote for
brevity $\gh(\cend(\ctm))(\varphi)$ by $\varphi$ and by $\mu$ the multiplication
in $\gh(\cend(\ctm))$, we have a commutative square
\begin{equation}
\xymatrix{(\sigma,\sigma) \ar@{|->}[r]^{\varphi} \ar@{|->}[d]_{\mu_\Lambda} & 
(\varphi(\sigma),\varphi(\sigma)) \ar@{|->}[d]^{\mu_{\Lambda'}}\\
-\One_\Lambda \ar@{|->}[r]^{\varphi} & -\One_{\Lambda'}}.
\end{equation}
Thus, $\gh(\cend(\ctm))(\varphi)$ preserves the property of a section to be
of square $-\One$, and restricting it to $\ca(\cm)(\Lambda)$ yields a well defined map $\ca(\cm)(\varphi)$,
turning the sets $\ca(\cm)(\Lambda)$ into a subfunctor of $\gh(\cend(\ctm))$.
\end{proof}

We first show that we can turn each of the sets $\ca(\cm)(\Lambda)$ into an ordinary manifold:

\begin{thm}
For each $\Lambda\in\catgr$, the set $\ca(\cm)(\Lambda)$ is a submanifold of \newline
$\gh(\cend(\ctm))(\Lambda)$.
\label{thm1}
\end{thm}
\begin{proof}
$\ca(\cm)(\Lambda)$ is a subset of $\gh(\cend(\ctm))(\Lambda)$. In the following, we keep some 
$J_0\in\ca(\cm)(\Lambda)$ fixed.
If $\ca(\cm)(\Lambda)$ is indeed a manifold, then its
tangent space at $J_0$ will be
\begin{equation}
T_{J_0}\ca(\cm)(\Lambda)=\left\{H\in\gh(\cend)(\Lambda)\big|J_0H+HJ_0=0\right\}
\label{tjzero}
\end{equation}
as one sees from formally differentiating $J^2=-\One_\Lambda$. For brevity, we will write
just $\One$ instead of $\One_\Lambda$ in this proof, since this is the only component of $\One$ that
will appear. 

Now let $0\in U$ be a neighbourhood of zero in $T_{J_0}\ca(\cm)(\Lambda)$ such that $(\One+H)$ is
invertible for all $H\in U$. Then define
\begin{eqnarray}
\label{psimap}
\psi_{J_0}:U &\to& \ca(\cm)(\Lambda)\\
\nonumber
H &\mapsto& (\One+H)J_0(\One+H)^{-1}=:J.
\end{eqnarray}
Clearly, $J^2=-\One$ if and only if $J_0^2=-\One$, so the range of $\psi$ is in $\ca(\cm)(\Lambda)$.
Solving for $H$, one finds the inverse of $\psi$ to be
\begin{equation}
\psi_{J_0}^{-1}(J)=(J-J_0)(J+J_0)^{-1}
\label{invpsi}
\end{equation}
This already proves that $\psi$ provides a chart around $J_0$, because we have constructed a bijective 
map from a neighbourhood around $J_0$ in $\ca(\cm)(\Lambda)$ to the linear space 
$T_{J_0}\ca(\cm)(\Lambda)$. Such charts are available around any $J_0\in\ca(\cm)(\Lambda)$. We have to
show that the overlaps are given by smooth maps. Therefore let $J$ be in the intersection of two charts,
i.e.,
\begin{equation}
(\One+H)J_0(\One+H)^{-1}=J=(\One+K)J_1(\One+K)^{-1},
\end{equation}
with $K\in T_{J_1}\ca(\cm)(\Lambda)$. Using (\ref{invpsi}), we find
\begin{equation}
\label{transmap}
K=\left[(\One+H)J_0(\One+H)^{-1}-J_1\right]\left[(\One+H)J_0(\One+H)^{-1}+J_1\right],
\end{equation}
which is, clearly, smooth with respect to $H$ whenever $(\One+H)$ is invertible.

To see that $\ca(\cm)(\Lambda)$ is actually a submanifold of $\gh(\cend)(\Lambda)$, we must show that
there exists a tubular neighbourhood of $\ca(\cm)(\Lambda)$ around $J_0\in\gh(\cend)(\Lambda)$.
This means we must find a chart $\phi$ for $V\subset\gh(\cend)(\Lambda)$ 
around $J_0$, which locally flattens $\ca(\cm)(\Lambda)$, i.e., such a chart $\phi(V)$ admits a direct sum
decomposition into a subspace which provides the chart for $\ca(\cm)(\Lambda)\cap V$ and a complement.
For this, note that every $H\in \gh(\cend(\ctm))(\Lambda)$ may be decomposed as $H=H^++H^-$, with
\[
H^+=\frac{1}{2}(H+J_0HJ_0),\qquad H^-=\frac{1}{2}(H-J_0HJ_0)
\]
such that $J_0H^++H^+J_0=0$, i.e., an anticommuting part with respect to $J_0$, and
$J_0H^--H^-J_0=0$, i.e., a commuting part. This splits
\begin{equation}
\gh(\cend)(\Lambda)=T^+_{J_0}\oplus T^-_{J_0}
\label{splitend}
\end{equation}
where $T^+_{J_0}$ denotes the space of endomorphisms anticommuting with $J_0$, and 
$T^-_{J_0}$ denotes the space of those commuting with $J_0$. Now take a neighbourhood 
$V\subset\gh(\cend)(\Lambda)$ such that $V\cap\ca(\cm)(\Lambda)=U$. 
Then $\id:V\to V$ is a chart around
$J_0$ for $\gh(\cend)(\Lambda)$, and $\psi_{J_0}^{-1}:U\to T_{J_0}^+$ is a chart for
$\ca(\cm)(\Lambda)\cap V$ whose image lies in a direct subspace of $V$. It remains to check that
there exists a neighbourhood around $J_0$ in $T_{J_0}^+$ which does not contain any points of 
$\ca(\cm)(\Lambda)$. Indeed, consider an $H\in T_{J_0}^+$ and check the equation
\begin{equation}
(J_0+H)^2=J_0^2+HJ_0+J_0H+H^2=-\One.
\end{equation}
Since $H$ commutes with $J_0$, this is equivalent to
\[
H(2J_0+H)=0.
\]
We are not interested in the $H=0$ solution, so we look for solutions of $2J_0+H=0$. 
It is clear that there exists a neighbourhood of $0\in T_{J_0}^+$ for which this equation has no
solutions, because $J_0$ is invertible.
This shows that
$\ca(\cm)(\Lambda)$ is indeed a submanifold of $\gh(\cend)(\Lambda)$.
\end{proof}

In fact, we get even more, namely complex coordinates on $\ca(\cm)(\Lambda)$. The construction in the
following proof is an adaptation of the idea developed in \cite{TTiRG} for ordinary almost complex
manifolds.

\begin{thm}
\label{cstr}
$\ca(\cm)(\Lambda)$ is a complex manifold.
\end{thm}
\begin{proof}
This will be shown by directly constructing holomorphic coordinates. First, observe that
$\ca(\cm)(\Lambda)$ is already an almost complex manifold: for any $J\in\ca(\cm)(\Lambda)$,
the tensor $J$ itself provides the almost complex structure on its tangent space:
\begin{eqnarray}
\label{phimap}
\Phi:T_{J}\ca(\cm)(\Lambda) &\to& T_{J}\ca(\cm)(\Lambda)\\
\nonumber
H &\mapsto& JH
\end{eqnarray}
This is a well defined map, because for each $H\in T_{J}\ca(\cm)(\Lambda)$, we have
\[
J(JH)+(JH)J=-H+(JH)J=-H-HJJ=0,
\]
so the image is again in $T_{J}\ca(\cm)(\Lambda)$. Now we claim that
\begin{equation}
\label{pbacs}
\psi_{J_0}^*(\Phi)=\Phi_{J_0},
\end{equation}
where and $\psi_{J_0}$ is the chart defined in (\ref{psimap}) and
$\Phi_{J_0}$ denotes the fixed almost complex structure on the vector space $T_{J_0}\ca(\cm)(\Lambda)$.
If (\ref{pbacs}) holds true, then $\psi_{J_0}$ is actually a holomorphic coordinate chart around $J_0$ 
for the
almost complex structure $\Phi$ (\ref{phimap}). Such a chart can be found around any $J_0$, and on two
overlapping charts, we have the transition function (\ref{transmap}), which is clearly holomorphic in
$H$ for $H$ near zero. Therefore the statement (\ref{pbacs}) is equivalent to the assertion of the theorem.

Take an arbitrary $K\in T_{J_0}\ca(\cm)(\Lambda)$. We have to show that
\[
D\psi_{J_0}(H)^{-1}\Phi_J D\psi_{J_0}(H)K=\Phi_{J_0}(K)=J_0K
\]
where $H\in T_{J_0}\ca(\cm)(\Lambda)$ is mapped to $J$ under $\psi_{J_0}$.

The map $G\mapsto G^{-1}$, defined on the invertible operators, has the derivative
$K\mapsto -G^{-1}KG^{-1}$. So
\begin{eqnarray*}
D\psi_{J_0}(H)K &=& KJ_0(\One+H)^{-1}-(\One+H)J_0(\One+H)^{-1}K(\One+H)^{-1}\\
&=& KJ_0(\One+H)^{-1}-JK(\One+H)^{-1}
\end{eqnarray*}
and
\[
JD\psi_{J_0}(H)K=-J(J+J_0)K(\One+H)^{-1}
\]
To find $D\psi_{J_0}(H)^{-1}$, we calculate the derivative of $\psi_{J_0}^{-1}$ from (\ref{invpsi}):
\begin{eqnarray*}
D\psi_{J_0}^{-1}(J)L &=& L(J+J_0)^{-1}-(J-J_0)(J+J_0)^{-1}L(J+J_0)^{-1}\\
&=& L(J+J_0)^{-1}-HL(J+J_0)^{-1}\\
&=& (\One-H)L(J+J_0)^{-1}
\end{eqnarray*}
So, altogether, one has
\begin{eqnarray}
\nonumber
D\psi_{J_0}(H)^{-1}\Phi_J D\psi_{J_0}(H)K &=& D\psi_{J_0}(H)^{-1}JD\psi_{J_0}(H)K\\
\nonumber
&=&(\One-H)(-J(J+J_0)K(\One+H)^{-1})(J+J_0)^{-1}\\
\nonumber
&=&-(\One-H)\left(J(J+J_0)K)\right)((J+J_0)(\One+H))^{-1}\\
\nonumber
&=&-(\One-H)\left((\One+H)J_0(\One+H)^{-1}(J+J_0)K\right)\cdot\\
&&\cdot((J+J_0)(\One+H))^{-1}
\label{dpsi}
\end{eqnarray}
Now use
\[
(\One-H)(\One+H)J_0(\One+H)^{-1}=J_0(\One-H)
\]
to obtain for (\ref{dpsi})
\begin{equation}
-J_0(\One-H)(J+J_0)K((J+J_0)(\One+H))^{-1}
\label{dpsi2}
\end{equation}
However
\begin{eqnarray*}
(\One-H)(J+J_0) &=& (\One-(J-J_0)(J+J_0)^{-1})(J+J_0)\\
&=& ((J+J_0)-(J-J_0))\\
&=& 2J_0
\end{eqnarray*}
and additionally,
\[
(J+J_0)(\One+H)=(J+J_0)\left(\One+(J-J_0)(J+J_0)^{-1}\right)
\]
which, together with
\[
(J+J_0)(J-J_0)=-(J-J_0)(J+J_0)
\]
yields
\[
(J+J_0)(\One+H)=((J+J_0)-(J-J_0))=2J_0
\]
But $J_0^{-1}=-J_0$, so
\[
((J+J_0)(\One+H))^{-1}=-\frac{1}{2}J_0.
\]
Inserting everything into (\ref{dpsi2}), one obtains finally
\begin{eqnarray*}
D\psi(H)^{-1}\Phi_J D\psi(H)K &=& -J_0(2J_0)K(-\frac{1}{2}J_0)\\
&=& -KJ_0\\
&=& J_0K,
\end{eqnarray*}
which was to be shown.
\end{proof}

In order to show that these complex manifold structures on the sets $\ca(\cm)(\Lambda)$ assemble to form 
a complex supermanifold $\ca(\cm)$, we have to show that the construction of Abresch-Fischer charts
carried out in the above proofs can be performed in a functorial manner. Since every open subfunctor of
a functor $\cf\in\cattop^\catgr$ which is locally isomorphic to superdomains is a restriction 
(Prop.~\ref{subf}), this means that around every $J_0\in\ca(\cm)(\Lambda)$, we have to find an open domain $U$
such that $\ca(\cm)\big|_U$ is isomorphic to an open subfunctor of a superrepresentable $\olc$-module.
In order to achieve this, the following two lemmas will be of use.

\begin{lemma}
\label{preim}
Let $\olv$ be a superrepresentable $\olk$-module in $\catsets^\catgr$. Let $S$ be a subset of 
its underlying points: $S\subset\olv(\fieldK)$.
Then $\olv(\epsilon_\Lambda)^{-1}(S)\subset\olv(\Lambda)$ has the form 
$S+\olv^{nil}(\Lambda)$. Here, the (not superrepresentable) $\fieldK$-module $\olv^{nil}$ is defined
by (\ref{vnildef}).
\end{lemma}
\begin{proof}
Since $\olv$ is superrepresentable, there exists a super vector space $V$ (a $\fieldK$-supermodule in
$\catsets$), such that
\[
\olv(\Lambda)=(\Lambda\otimes V)_{\bar{0}}\qquad\mathrm{and}\qquad\olv(\varphi)=\varphi\otimes\id_V
\]
for any morphism $\varphi:\Lambda\to\Lambda'$.
For $S\subset\olv(\fieldK)$, this means simply that $S\subset V_{\bar{0}}$, and additionally
\[
\olv(\Lambda)=(\Lambda\otimes V)_{\bar{0}}=\Lambda_{\bar{0}}\otimes V_{\bar{0}}+%
\Lambda_{\bar{1}}\otimes V_{\bar{1}}.
\]
The terminal morphism $\epsilon_\Lambda:\Lambda\to\fieldK$ maps all odd generators of $\Lambda$ to
zero. Therefore, the general preimage of $S$ under $\epsilon_\Lambda$ contains 
$S=(\Lambda_{\bar{0}}\cap\fieldK)\otimes S$, plus an arbitrary term proportional
to at least one odd generator of $\Lambda$. These are
$\Lambda_{\bar{1}}\otimes V_{\bar{1}}$ and  
$(\Lambda_{\bar{0}}\cap\Lambda^{nil})\otimes V_{\bar{0}}$. The sum of the latter two sets forms
$(\Lambda^{nil}\otimes V)_{\bar{0}}=\olv^{nil}(\Lambda)$.
\end{proof}

\begin{lemma}
\label{alginv}
Let $\cb$ be an associative $\olk$-algebra with unit in $\catsets^\catgr$. Let $a\in\cb(\fieldK)$ 
be an invertible elment of its underlying algebra,
i.e. there exists $a^{-1}\in\cb(\fieldK)$ such that $aa^{-1}=a^{-1}a=1$ where 1 is the unit in
$\cb(\fieldK)$. Then $\cb(\epsilon_\Lambda)^{-1}(a)$ is a set of invertible elements in
$\cb(\Lambda)$ for all $\Lambda\in\catgr$.
\end{lemma}
\begin{proof}
Due to Corollary \ref{barfunct}, every associative $\olr$-algebra with unit in 
$\catsets^\catgr$ arises from an ordinary associative super algebra $B$, that is, a super vector space
with a bilinear morphism
\[
\mu:B\times B\to B
\]
Then $\cb(\Lambda)=(\Lambda\otimes B)_{\bar{0}}$, and the algebra operation becomes a functor morphism
$\bar{\mu}:\cb\times\cb\to\cb$, given pointwise by a map
\[
\bar{\mu}_\Lambda:\cb(\Lambda)\times\cb(\Lambda)\to\cb(\Lambda)
\]
This map is uniquely given by the requirements that it be $\Lambda_{\bar{0}}$-linear and that it
fulfill
\[
\bar{\mu}_\Lambda(\lambda_1\otimes a_1,\lambda_2\otimes a_2)=\lambda_2\lambda_1\otimes \mu(a_1,a_2).
\]
Now by Lemma \ref{preim}, the preimage of $a\in\cb(\fieldK)$ under $\epsilon_\Lambda$ is a set of 
the form $a+\cb^{nil}(\Lambda)$. That means any element of it can be written as $a+c$, where $c$
contains nilpotent elements of $\Lambda$, which means that $\bar{\mu}_\Lambda(c,c)=0$. Then
$a+c$ has the inverse $a^{-1}-a^{-1}ca^{-1}$:
\begin{eqnarray*}
\bar{\mu}_\Lambda(a+c,a^{-1}-a^{-1}ca^{-1}) &=& \bar{\mu}_\Lambda(a,a^{-1})+%
\bar{\mu}_\Lambda(a,-a^{-1}ca^{-1})+\\
&&\bar{\mu}_\Lambda(c,a^{-1})+\bar{\mu}_\Lambda(c,-a^{-1}ca^{-1})\\
&=& 1-\bar{\mu}_\Lambda(1,ca^{-1})+\bar{\mu}_\Lambda(c,a^{-1})-\bar{\mu}_\Lambda(c,a^{-1}ca^{-1})\\
&=& 1
\end{eqnarray*}
We have used the fact that $\bar{\mu}_\Lambda$ is $\Lambda_{\bar{0}}$-linear, as well as associativity.
The last term vanishes because $c$ appears twice and thus introduces the same nilpotent 
elements twice.
\end{proof}

Lemma \ref{alginv} is the categorified version of the fact that in an algebra which contains nilpotent
elements, an element is invertible if and only if its reduced element is. Lemma \ref{preim} shows
that for a superrepresentable $\olk$-module $\cv$, it makes sense to regard an element $v\in\cv(\fieldK)$ 
of its underlying space also as an element of all higher sets of points $\cv(\Lambda)$.
To construct a chart on $\ca(\cm)$, we first specify a $\olc$-module which the chart will map to. 

\begin{prop}
\label{that}
Fix a $J_0\in\ca(\cm)(\realR)$. We define a functor $\hat{T}_{J_0}$ in $\catsets^\catgr$ by
\begin{eqnarray}
\hat{T}_{J_0}(\Lambda) &=& T_{J_0}\ca(\Lambda)\\
\hat{T}_{J_0}(\varphi) &=& \gh(\cend(\ctm))(\varphi)\big|_{\hat{T}_{J_0}(\Lambda)}
\end{eqnarray}
for any morphism $\varphi:\Lambda\to\Lambda'$.
Then $\hat{T}_{J_0}$ is a superrepresentable $\olc$-module.
\end{prop}
\begin{proof}
We know that $\gh(\cend(\ctm))$ is superrepresentable, so there exists associative superalgebra 
$(\ce,\mu)$ such that $\gh(\cend(\ctm))\cong\overline{\ce}$ and such that the 
multiplication $\mu$ becomes the composition
of endomorphisms. Since $J_0$ is an element of the underlying space of 
$\gh(\cend(\ctm))$ it is an element of $\ce_{\bar{0}}$. 
We may split the algebra $\ce$ into a direct sum of linear subspaces $\ce=\ce_+\oplus\ce_-$ by setting
\begin{eqnarray*}
\ce_+ &:=& \{a\in\ce\mid aJ_0+J_0a=0\}\\
\ce_- &:=& \{a\in\ce\mid aJ_0-J_0a=0\}
\end{eqnarray*}
This decomposition is direct, since any $a$ may be written uniquely as 
\[
a=\frac{1}{2}(a-J_0aJ_0)+\frac{1}{2}(a+J_0aJ_0)
\]
with the first summand in $\ce_-$ and the second one in $\ce_+$ and $\ce_+\cap\ce_-=\{0\}$. We want
to show that $\hat{T}_{J_0}\cong\overline{\ce_+}$. Let $H=\lambda\otimes h$ be an element of
$\overline{\ce_+}(\Lambda)$, i.e., we have $\mu(J_0,h)=-\mu(h,J_0)$. In $\overline{\ce}$, $\mu$
is replaced by $\overline{\mu}$ and we find
\[
\overline{\mu}_\Lambda(J_0,H)=\lambda\otimes\mu(J_0,h)=-\lambda\otimes(h,J_0)=-\overline{\mu}(H,J_0).
\]
Therefore one has $\hat{T}_{J_0}(\Lambda)=\overline{\ce_+}(\Lambda)$. 

Finally, the functors $\hat{T}_{J_0}$ and
$\overline{\ce_+}$ are well-defined because $\mu$ is a functor morphism. This means that for
every morphism $\varphi:\Lambda\to\Lambda'$ one
has a commutative diagram
\begin{equation}
\begin{CD}
\ce(\Lambda)\times\ce(\Lambda) @>{\mu_\Lambda}>> \ce(\Lambda)\\
@V{\ce(\varphi)}VV @VV{\ce(\varphi)}V\\
\ce(\Lambda') @>{\mu_{\Lambda'}}>> \ce(\Lambda')
\end{CD}\qquad.
\end{equation}
The
property of an element in $\ce(\Lambda)$ to anticommute with $J_0$ is therefore preserved under images
$\ce(\varphi)$ of morphisms $\varphi\in\catgr$.
\end{proof}

The $\olc$-supermodule $\hat{T}_{J_0}$ is the functor that represents the tangent space $T_{J_0}\ca(\cm)$.
Since each of the Abresch-Fischer charts around $J_0$ maps precisely this tangent space 
onto $\ca(\cm)$, we have to prove that $\ca(\cm)$ is locally (around $J_0$) isomorphic to $T_{J_0}\ca(\cm)$.

\begin{thm}
Let $J_0$ be a fixed element of $\ca(\cm)(\realR)$. Then there exists an open subfunctor $\cv$ around
$J_0$ in $\ca(\cm)$ which is isomorphic to an open subfunctor around $0$ in $\hat{T}_{J_0}$.
\end{thm}
\begin{proof}
First, we construct a coordinate neighbourhood of $J_0$ in $\ca(\cm)(\realR)$. According to Thm.~\ref{thm1},
such a local coordinate system is given by a neighbourhood $U$ of zero in $\hat{T}(\realR)$ via
the map
\begin{eqnarray}
\nonumber
\psi_{J_0}:U &\to& \ca(\cm)(\realR)\\
\label{psimap2}
H &\mapsto& (\One+H)J_0(\One+H)^{-1}=:J.
\end{eqnarray}
We will refer to this map and its analog for the higher $\Lambda_n$-points simply as $\psi$.
One observes that $U$ is only limited by the requirement that $(\One+H)$ be invertible. In the
following we will keep $U$ fixed, and we denote the image of $U$ under $\psi$ as 
$V\subset\ca(\cm)(\realR)$.

The claim that we want to prove states that there exists an isomorphism
\begin{equation}
\ca(\cm)(\epsilon_{\Lambda_n})^{-1}(V)\cong\hat{T}_{J_0}(\epsilon_{\Lambda_n})^{-1}(U)
\end{equation}
which provides an Abresch-Fischer chart on each point set $\ca(\cm)(\epsilon_{\Lambda_n})^{-1}(V)$.

First, we characterise the set $(\ca(\cm)(\epsilon_\Lambda))^{-1}(J_0)$. According to Lemma \ref{preim}, 
it is a set of the form $J_0+\ca_{J_0}^{nil}$, where $\ca_{J_0}^{nil}$ are elements from 
$\gh(\cend(\ctm))(\Lambda)$ which are proportional to nilpotent elements of $\Lambda$. So any element
$\tilde{J}_0\in(\ca(\cm)(\epsilon_\Lambda))^{-1}(J_0)$ has the form $J_0+J_0^{nil}$, and since
\begin{equation}
\tilde{J}_0^2=(J_0+J_0^{nil})^2=-\One+J_0J_0^{nil}+J_0^{nil}J_0
\label{jnil}
\end{equation}
we see that $J_0^{nil}$ must anticommute with $J_0$. So,
\begin{equation}
\ca_{J_0}^{nil}(\Lambda)=\{H\in\gh(\cend)^{nil}(\Lambda)\mid HJ_0+J_0H=0\},
\label{jnil2}
\end{equation}
implying that the points of $\ca(\cm)$ ``above'' $J_0$ (i.e. those whose reduced element is $J_0$) are 
those elements of $\gh(\cend(\ctm))(\Lambda)$ which anticommute with $J_0$ and which are purely nilpotent.
These are nothing else than $(T_{J_0}\ca(\cm)(\Lambda))^{nil}$.

Fix an arbitrary $\tilde{J}_0\in(\ca(\cm)(\epsilon_\Lambda))^{-1}(J_0)$. It must be of the form
$\tilde{J}_0=J_0+J_0^{nil}$. If it can be covered by our proposed coordinate map $\psi$, it must be
possible to write it as
\begin{equation}
\tilde{J}_0=(\One+K)J_0(\One+K)^{-1}
\label{oj1}
\end{equation}
for some $K\in\hat{T}_{J_0}(\Lambda)$.
Since $\epsilon_\Lambda:\gh(\cend(\ctm))(\Lambda)\to\gh(\cend(\ctm))(\realR)$ is a homomorphism of 
associative algebras, the above equation implies that
\[
J_0=\epsilon_\Lambda(\One+K)J_0\epsilon_\Lambda(\One+K)^{-1}
\]
So $\epsilon_\Lambda(K)=0$, i.e. $K$ is proportional to nilpotent elements. 
Then $(\One+K)^{-1}=\One-K$, and setting $K=-\frac{1}{2}J_0J_0^{nil}$, we find the desired element of
$\hat{T}_{J_0}(\Lambda)$, whose image under $\psi$ is $\tilde{J}_0$. The invertibility of $\One+K$
is assured by Lemma \ref{alginv}.

Now consider another $J\in V$, and an arbitrary $\tilde{J}\in(\ca(\cm)(\epsilon_\Lambda))^{-1}(J)$. We again
want to find $K\in \hat{T}_{J_0}(\Lambda)$, such that
\[
\tilde{J}=(\One+K)J_0(\One+K)^{-1}.
\]
Applying $\epsilon_\Lambda$, we see that this implies
\[
J=\epsilon_\Lambda(\One+K)J_0\epsilon_\Lambda(\One+K)^{-1}
\]
which means that $\epsilon_\Lambda(K)=H$, where $H$ is the element of $T_{J_0}\ca(\realR)$ whose
image under $\psi$ is $J$. So $K=H+K^{nil}$, where $K^{nil}$ is some nilpotent term.
Applying the same reasoning to $\tilde{J}$ as for $\tilde{J}_0$, we find 
$K\in\hat{T}(\Lambda)$ such that $\tilde{J}$ is the image of $K$ under $\psi$ and the underlying
endomorphism of $K$ is $H\in U$. Here again, Lemma \ref{alginv} assures the existence of $(\One+K)^{-1}$.

We can now conclude that any $J\in\ca(\cm)(\epsilon_\Lambda)^{-1}(V)$ can be written as
\begin{equation}
\label{last}
J=(\One+K)J_0(\One+K)^{-1}
\end{equation}
with $K\in\hat{T}_{J_0}(\epsilon_{\Lambda})^{-1}(U)$. Conversely, inserting any such $K$ into (\ref{last})
will produce a $J$ above $V$. This proves the claim.
\end{proof}

The following Theorem finishes the construction of $\ca(\cm)$ as a complex supermanifold.

\begin{thm}
The Abresch-Fisher coordinate charts around each $J_0\in\ca(\cm)(\Lambda)$ provide holomorphic complex
coordinates for $\ca(\cm)$. Therefore $\ca(\cm)$ is a complex supermanifold.
\end{thm}
\begin{proof}
It was shown that $\hat{T}_{J_0}$ is a superrepresentable $\olc$-module. Therefore, the Abresch-Fischer
charts identify $\ca(\cm)$ locally with a linear complex supermanifold. The compatibility of the charts is
ensured by Thm.~\ref{cstr}: the almost complex structure induced by any chart coincides with the globally
defined almost complex structure $\Phi_{J_0}:T_{J_0}\ca(\cm)\to T_{J_0}\ca(\cm)$. This makes all transition functions
automatically holomorphic.
\end{proof}

The equations (\ref{jnil}) and (\ref{jnil2}) above give us an important piece of information about the
structure of $\ca(\cm)$: in its odd directions, it coincides with its tangent space. Of course, every
supermanifold can be thought of as being ``linear'' in its odd directions, as was already argued above. But
our particular construction of a chart for $\ca(\cm)$ as an open domain in a tangent space tells us that
we can write every $J\in\ca(\cm)(\Lambda)$ as a sum of a $J_0\in\ca(\cm)(\realR)$ and a nilpotent
tangent vector, which is a consequence of the fact that the tangent space at $J$ is defined by an 
algebraic property.

\begin{cor}
\label{acstang}
Let $J$ be an element of $\ca(\cm)(\realR)$ and let 
\begin{equation}
U_J(\Lambda)=\ca(\cm)(\epsilon_\Lambda)^{-1}(J)
\end{equation}
be the set of all elements of $\ca(\cm)(\Lambda)$ which reduce to $J$ under the morphism
$\epsilon_\Lambda:\Lambda\to\realR$. Then each $J'\in U_J$ can be written as
\begin{equation}
J'=J+H^{nil},
\end{equation}
where $H^{nil}$ is a nilpotent element of $\hat{T}_J$.
\end{cor}
\begin{proof}
This follows directly from equations (\ref{jnil}) and (\ref{jnil2}).
\end{proof}

Therefore, the Abresch-Fischer charts are really only interesting on the underlying manifold
$\ca(\cm)(\realR)$. For the higher points $\ca(\cm)(\Lambda)$, they simply directly identify all the points
above some point $p\in\ca(\cm)(\realR)$ with the part of its tangent space which is proportional to odd
generators of $\Lambda$.

%% file: complex.tex
\chapter{Integrability of almost complex structures}
\label{ch:cpx}

The manifold $\ca(\cm)$ constructed in the previous chapter comprises all almost complex structures
that the supermanifold $\cm$ can be endowed with. Any superconformal structure requires, however, a
complex structure, i.e., an \emph{integrable} almost complex structure. The goal of this chapter is
therefore to state the conditions for the integrability of an almost complex structure on a supermanifold.
These conditions are direct super analogues of the ordinary ones, as was shown by Vaintrob \cite{Acsos}.
Although we suspect that it exists we do not try to construct an explicit supermanifold structure on
the sets of integrable structures $\cc(\cm)(\Lambda)$.

\section{General case}

As remarked in Section \ref{sect:qstruct}, we do not consider odd almost complex structures --- all
structures appearing here are even.
Let $J$ be an almost complex structure on the supermanifold $\cm$. Then, as ordinary complex geometry,
$J$ splits the complexified tangent bundle $\ctm^\complexC$ into two eigendistributions:
\begin{equation}
\label{split}
\ctm^\complexC=\ctm^{1,0}\oplus\ctm^{0,1},
\end{equation}
where
\begin{eqnarray}
\ctm^{1,0} &:=& \{X\in\ctm^\complexC\big| JX=iX\},\\
\ctm^{0,1} &:=& \{X\in\ctm^\complexC\big| JX=-iX\}.
\end{eqnarray}
In complete analogy to classical almost complex geometry, this splitting induces a bigrading on the
complex $\Omega^\bullet$ of differential forms. A form $\omega\in\Omega^n$ is a super-antisymmetric
map (cf.~Section \ref{sect:tangent})
\begin{equation}
\omega:\ctm^{\otimes n}\to\co_\cm.
\end{equation}
One extends forms to $\ctm^\complexC$ by $\complexC$-linearity. This yields the complexified sheaf
of differential forms, which, when applied to sections of $\ctm^\complexC$, take values in
$\co_\cm\otimes\complexC$.
The splitting (\ref{split}) allows us to then decompose any $\omega\in\Omega^n\otimes\complexC$ as
\begin{equation}
\omega=\omega^{0,n}+\omega^{1,n-1}+\ldots+\omega^{n,0},
\end{equation}
where each $\omega^{p,q}$ is a super-antisymmetric map
\begin{equation}
\omega^{p,q}:(\ctm^{1,0})^{\otimes p}\otimes (\ctm^{0,1})^{\otimes q}\to\co_\cm\otimes\complexC.
\end{equation}
Thus, the space of complexified $n$-forms decomposes as
\begin{equation}
\Omega^n\otimes\complexC=\bigoplus_{p+q=n}\Omega^{p,q}.
\end{equation}
If we denote by $\pi_{p,q}$ the projection onto the summand with indices $p,q$, then we can define the
super versions of the Dolbeault operators as
\begin{eqnarray}
\partial &:=& \pi_{p+1,q}\circ d\\
\bar{\partial} &:=& \pi_{p,q+1}\circ d,
\end{eqnarray}
where the exterior differential $d$ has to be understood as being extended by $\complexC$-linearity 
to complexified forms.

The Nijenhuis tensor associated with an even almost complex structure $J$ is also 
defined by exactly the same
formula as in classical geometry. It is a (2-1)-tensor field which is most conveniently defined by the way it
acts on an arbitrary pair $X,Y$ of vector fields:
\begin{equation}
N_J(X,Y)=[JX,JY]-J[JX,Y]-J[X,JY]-[X,Y].
\end{equation}
Then we have the following result.

\begin{thm}[Vaintrob, \cite{Acsos}]
\label{thm:intacs}
Let $J$ be an almost complex structure on a supermanifold $\cm$. Then the following conditions are
equivalent:
\begin{enumerate}
\item $N_J\equiv 0$,
\item $\ctm^{1,0}$ is a sheaf of Lie subalgebras of $\ctm^\complexC$,
\item $\ctm^{0,1}$ is a sheaf of Lie subalgebras of $\ctm^\complexC$,
\item $d=\partial+\bar{\partial}$, $\partial^2= 0$, $\bar{\partial}^2= 0$,
\item There exists a torsion-free connection on $\cm$ with respect to which $J$ is horizontal.
\end{enumerate}
\end{thm}
In case these conditions hold, $J$ is called integrable.

In particular, the second and third conditions prove to be of great value for our later considerations.
Of course, the term ``integrable'' was chosen to denote an almost complex structure which can be induced
by a complex structure, i.e., a complex atlas on $\cm$. This is expressed by

\begin{thm}[Vaintrob, \cite{Acsos}]
Let $J$ be an integrable almost complex structure on a smooth supermanifold $\cm$. Then there exists a
complex structure on $\cm$ such that $J$ coincides with the almost complex structures locally induced
by the chart maps, i.e., it acts as the multiplication operator by $i$ on the tangent bundle of the
complex supermanifold $\cm$.
\end{thm}

With these beautiful results at hand, we can directly start to investigate the conditions for an almost
complex structure on a $2|2$ smooth supermanifold to be integrable.

\section{The $2|2$-dimensional case}

To find the set or, even better, supermanifold of integrable almost complex structures on an arbitrary
almost complex supermanifold $\cm$ is in general a
rather tough job. Is it at least as hard as finding them for the underlying manifold. But in
the smooth $2|2$-dimensional case that we are interested in, things are much easier.

Prop.~\ref{isosrs} tells us that every complex $1|1$-dimensional supermanifold can be interpreted as
a pair $(M,L)$ consisting of a Riemann surface and a holomorphic line bundle $L$, and vice versa. This
provides us at once with a description of an integrable almost complex structure $J\in\ca(\cm)(\realR)$.
Every smooth supermanifold of dimension $2|2$ can be realized as a smooth surface with a smooth real
vector bundle of rank 2 (by Batchelor's theorem \cite{TSOS}). Denote this bundle as $\pi:E\to M$. An
integrable almost complex structure on $\cm$ must turn both the total space $E$ as well as the base $M$
into complex manifolds in a manner such that $\pi$ becomes holomorphic.

It may seem at first that we need two complex structures to accomplish this, one on $E$ and one
on $M$, which have to be compatible. But the fact that we start with a supermanifold $\cm$ gives us
additional data: a splitting of the tangent sheaf $\ctm=\ctm_{\bar{0}}\oplus\ctm_{\bar{1}}$
into an even and an odd part. Reducing the structure sheaf $\co_\cm\to\co_M$ turns $\ctm$ into a sheaf
of locally free modules over $\co_M$ of rank $2|2$. This sheaf can be considered as the total space
of the rank 4 smooth vector bundle $E$ over $M$ by changing the parity of its odd generators. This
vector bundle inherits the splitting $E=E_{\bar{0}}\oplus E_{\bar{1}}$ from the tangent sheaf
$\ctm$. An almost complex structure on $\cm$ is then integrable if and only if it is both integrable on the 
total space
$E$ and when it is restricted to the subbundle $E_{\bar{0}}$ (which is, of course, the tangent bundle of
the underlying surface $M$). This is, clearly, more
restrictive than the conditions in four ordinary real dimensions. The additional condition will show 
up below when we determine the integrable deformations of an integrable structure.

We note here another important fact: the smooth vector bundle $E$ of rank 2 which determines $\cm$ possesses
a fixed degree $d$ which completely classifies it up to isomorphism. This degree is, along with the 
topological invariants of the surface, an additional topological invariant
of $\cm$. Since we are interested only in orientable vector bundles (only they can be made complex),
we see that the set of isomorphism classes of smooth supermanifolds $\cm$ whose underlying 
manifold is some fixed closed surface $M$ and which possess integrable almost complex
structures is isomorphic to $\intZ$. After having endowed $\cm$ with a complex atlas, the degree becomes
the degree of the line bundle $L$. The degree does not classify $L$ up to isomorphism anymore, rather, 
the set of these isomorphism classes is isomorphic to $\Jac(M)$, the Jacobian variety of $M$. This
will be further discussed in Chapter \ref{ch:quotient}.

\subsection{\texorpdfstring{The tangent space to $\cc(\cm)(\Lambda)$}{The tangent space to C(M)}}

From now on, we assume that $J$ is an integrable almost complex structure on $\cp(\Lambda)\times\cm$,
where $\cm$ is a smooth $2|2$-dimensional
supermanifold $\cm$. This means that $J$ turns $\cp(\Lambda)\times\cm$ into a family of complex supermanifolds
parametrized by $\cp(\Lambda)$.
Denote by $z,\theta$ the local 
complex coordinates induced by $J$, and by $x,y,\xi_1,\xi_2$
the corresponding real coordinates, i.e. $z=x+iy$ and $\theta=\xi_1+i\xi_2$. The almost complex structure
can, in real coordinates, always be written locally as
\begin{equation}
J=\left(\begin{array}{cc}
0 & -\One\\
\One & 0\end{array}\right),
\end{equation} 
where each of the entries is to be understood as a $2\times 2$-matrix of smooth superfunctions. A smooth
superfunction is, in this case, a local section of $\co_\cm\otimes\Lambda$. Let
\begin{equation}
H=\left(\begin{array}{cc}
A & B\\
C & D\end{array}\right)
\end{equation} 
be an arbitrary endomomorphism of the tangent bundle, i.e., a section of the form
$\cp(\Lambda)\times\cm\to\cend(\ctm)$. For $H$ to be an element of $T_J\ca(\cm)(\Lambda)$,
it is necessary and sufficient that $HJ+JH=0$, which implies that $H$ can be put into the form
\begin{equation}
H=\left(\begin{array}{cc}
A & B\\
B & -A\end{array}\right).
\end{equation}
It will be more convenient to work in the complex picture. We have
\begin{equation}
\frac{1}{2}\left(\begin{array}{cccc}
1 & 0 & -i & 0\\
0 & 1 & 0 & -i\\
1 & 0 & i & 0\\
0 & 1 & 0 & i
\end{array}\right)\left(\begin{array}{c}
\partial_x\\
\partial_{\xi_1}\\
\partial_y\\
\partial_{\xi_2}
\end{array}\right)=\left(\begin{array}{c}
\partial_z\\
\partial_{\theta}\\
\partial_{\bar{z}}\\
\partial_{\bar{\theta}}
\end{array}\right)
\end{equation}
and therefore, in the $J$-eigenbasis of the complexified tangent bundle, $H$ takes the form
\begin{equation}
\label{hform}
\frac{1}{2}\left(\begin{array}{cc}
\One & -i\One\\
\One & i\One
\end{array}\right)\left(\begin{array}{cc}
A & B\\
B & -A
\end{array}\right)\left(\begin{array}{cc}
\One & \One\\
i\One & -i\One
\end{array}\right)=\left(\begin{array}{cc}
0 & A-iB\\
A+iB & 0
\end{array}\right).
\end{equation}
A tangent vector $H$ of $J$ at $\ca(\cm)(\Lambda)$ is therefore locally described by just 4 local
sections of $\co_\cm\otimes\complexC\otimes\Lambda$, two even and two odd ones. 
We will denote these four components of $A-iB$ as
\begin{equation}
\label{hcomp}
A-iB=\left(\begin{array}{cc}
\alpha & \beta\\
\gamma & \delta
\end{array}\right).
\end{equation}
Note that $\alpha$ and $\delta$ are even, while $\beta$ and $\gamma$ are odd.

Now consider an infinitesimal deformation of $J$ along $H$, i.e. set $J'=J+itH$ for a small real 
parameter $t$. Then the eigenbasis of $J'$ is given by
\begin{eqnarray}
\label{neweig}
\partial_{z'} &=& (\One+\frac{t}{2}H)\partial_z\\
\nonumber
\partial_{\theta'} &=& (\One+\frac{t}{2}H)\partial_\theta.
\end{eqnarray}

To identify the tangent space $T_J\cc(\cm)(\Lambda)$, we must find out whether $J'$ is again integrable. 
Vaintrob's Theorem (Thm.~\ref{thm:intacs}) will be the crucial tool here.

\begin{prop}
\label{tcm}
Let $J'=J+itH$ be an infinitesimal deformation of an integrable almost complex structure $J$ on $\cm$,
and let $H$ be locally given in the form (\ref{hcomp}).
Then $J'$ is again integrable if and only if the coefficient functions of $H$ satisfy
\begin{eqnarray}
\label{int1}
\pderiv{\beta}{\bar{z}}=\pderiv{\alpha}{\bar{\theta}},
&\qquad&\pderiv{\delta}{\bar{z}}=\pderiv{\gamma}{\bar{\theta}}\\
\label{int2}
\pderiv{\beta}{\bar{\theta}}=0, &\qquad& \pderiv{\delta}{\bar{\theta}}=0.
\end{eqnarray}
\end{prop}
\begin{proof}
By Thm.~\ref{thm:intacs}, it is sufficient to check whether the new eigenbasis (\ref{neweig}) is
again closed under Lie bracket. Consider first
\begin{equation}
[\partial_{z'},\partial_{\theta'}] = \frac{t}{2}\left((\partial_z\overline{\beta})\partial_{\bar{z}}+
(\partial_z\overline{\delta})\partial_{\bar{\theta}}-(\partial_\theta\overline{\alpha})\partial_{\bar{z}}-
(\partial_\theta\overline{\gamma})\partial_{\bar{\theta}}\right).
\end{equation}
This bracket must produce vector fields proportional to $\partial_{z'},\partial_{\theta'}$. Clearly,
the only possibility to achieve this is that bracket vanishes, so we have to require
\begin{equation}
\label{indcond1}
\pderiv{\beta}{\bar{z}}=\pderiv{\alpha}{\bar{\theta}},
\qquad\pderiv{\delta}{\bar{z}}=\pderiv{\gamma}{\bar{\theta}}.
\end{equation}
Likewise, we obtain
\begin{equation}
[\partial_{\theta'},\partial_{\theta'}]=\frac{t}{2}\left((\partial_\theta\overline{\beta})\partial_{\bar{z}}+
(\partial_\theta\overline{\delta})\partial_{\bar{\theta}}\right),
\end{equation}
which yields the conditions
\begin{equation}
\label{indcond2}
\pderiv{\beta}{\bar{\theta}}=0,\qquad\pderiv{\delta}{\bar{\theta}}=0.
\end{equation}
\end{proof}

The last condition (\ref{indcond2}) would be absent if we would study the problem of four ordinary real
dimensions, because the commutator of an even vector field with itself always vanishes. Its presence
shows that an integrable almost complex structure on a supermanifold is a more special structure than just
a complex structure on the total space of a smooth vector bundle over a surface (cf.~also the discussion
in the previous section). Prop.~\ref{tcm} also makes it clear that there exist non-integrable almost complex
structures on a smooth $2|2$-dimensional supermanifold.

\begin{cor}
\label{compsp}
The almost complex structure on $\ca(\cm)$ can be restricted to an almost complex structure on $\cc(\cm)$.
\end{cor}
\begin{proof}
The almost complex structure on $\ca(\cm)(\Lambda)$ was given on the tangent space 
$T_J\ca(\cm)(\Lambda)$ by
\begin{eqnarray}
\Phi:T_J\ca(\cm)(\Lambda) &\to& T_J\ca(\cm)(\Lambda)\\
H &\mapsto& JH.
\end{eqnarray}
Locally, this corresponds just to replacing the entries $\alpha,\beta,\gamma,\delta$ of $H$ by
$i\alpha,i\beta,i\gamma,i\delta$, and the conjugated ones by $-i\overline{\alpha}$, etc. Obviously, 
these functions also satisfy the conditions (\ref{int1}) and (\ref{int2}).
\end{proof}

The conditions (\ref{int1}) and (\ref{int2}) are linear: they determine subspaces in each tangent space
$T_J\ca(\cm)(\Lambda)$. In order to show that $\cc(\cm)$ is also a supermanifold, we need to first show that
the tangent spaces determined in Prop.~\ref{tcm} form a $\olc$-submodule of the $\olc$-module 
$\hat{T}_{J}$ (compare with Prop.~\ref{that}) which served as a model space for $\ca(\cm)$.

\begin{prop}
The integrable deformations (those satisfying (\ref{int1}) and \linebreak
(\ref{int2})) of an integrable almost
complex structure $J$ form a superrepresentable $\olc$-module $\hat{T}_J^{int}$.
\end{prop}
\begin{proof}
The integrable deformations form linear subspaces of the spaces $T_J\ca(\cm)(\Lambda)$, and 
Cor.~\ref{compsp} tells us that these spaces are in fact all complex. What remains to be shown is that the
inclusion of these spaces into the spaces $T_J\ca(\cm)(\Lambda)$ is a functor morphism. This is equivalent
to showing that for each morphism $\varphi:\Lambda\to\Lambda'$, the restriction of $\hat{T}_J(\varphi)$
to the integrable deformations $\hat{T}_J^{int}(\Lambda)$ yields only integrable deformations in 
$\hat{T}_J^{int}(\Lambda')$.

Let $H$ be an integrable deformation in $T_J\ca(\cm)(\Lambda)$, and let $\varphi:\Lambda\to\Lambda'$ be some
given morphism. The action of $\gh(\cend(\ctm))(\varphi)$ on $H$ is described by (\ref{svbungr}), and
the action of $\hat{T}_J(\varphi)$ is then just the projection of $\gh(\cend(\ctm))(\varphi)(H)$ onto the 
subspace of endomorphisms anticommuting with $J$. Let $\tau_1,\ldots,\tau_n$ denote the generators of
$\Lambda$. The dependence of $H$ on $\Lambda$ is encoded, in its local expression, by the dependence of
the entries $\alpha,\beta,\gamma,\delta$ on the $\tau_i$. Each of these is a local section of
$\co_\cm\otimes\Lambda$ (where $\Lambda$ denotes here the locally constant sheaf with stalk $\Lambda$).
Thus each of the functions can be written as
\begin{equation}
\alpha=\sum_{I\subset\{1,\ldots,n\}}\tau_I\alpha_I,
\end{equation}
where the sum runs over all increasingly ordered subsets, $\tau_I$ is the product of the appropriate
$\tau_i$'s and $\alpha_I$ is a local section of $\co_\cm$ of parity $|I|+p(\alpha)$ (cf.~also Thm.~\ref{sdiff} 
and the definitions preceding it). Morphisms $\Lambda\to\Lambda'$ affect only the generators
$\tau_i$, not the coefficient functions $\alpha_I$ which contain the dependence of $\alpha$ on 
$z,\bar{z},\theta,\bar{\theta}$. Therefore, if $H$ satisfies the integrablility conditions (\ref{int1})
and (\ref{int2}) this property will be preserved under $\varphi:\Lambda\to\Lambda'$.
\end{proof}

\subsection{\texorpdfstring{A supermanifold structure for $\cc(\cm)$}{A supermanifold structure for C(M)}}

To obtain a complex supermanifold $\cc(\cm)$, we would now have to show
that we can find a chart around every $J\in\cc(\cm)(\realR)$ which is isomorphic to
a domain in $\hat{T}_J^{int}$. Unfortunately, it does \emph{not} suffice to restrict
the Abresch-Fischer charts constructed in Chapter \ref{ch:acs} to the submodule $\hat{T}_J^{int}$.
The images of integrable deformations are not necessarily integrable almost complex structures 
under this map. This is not a problem of supergeometry, but occurs as well for ordinary almost complex 
manifolds of dimension $\geq 4$.

It is therefore very difficult to directly find a chart on $\cc(\cm)$, and indeed we did not succeed in
finding one. On the other hand, it seems to be intuitively clear that $\cc(\cm)$ should be a (super)manifold.
Looking at an ordinary almost complex manifold $M$, the set of integrable almost complex structures decomposes
into
orbits of the diffeomorphism group. So if the diffeomorphisms act freely, 
the set of integrable structures is a (possibly infinite)
union of sets which are diffeomorphic to $\mathrm{Diff}(M)$. Formally this may be considered
as a manifold. However, for moduli questions this ``construction'' seems to be rather useless, 
since
it is the structure of the set of $\mathrm{Diff}$-orbits that we are interested in, and that is
precisely the structure that we are neglecting in this way.
We will not try to address this problem for the $2|2$-dimensional
case. Instead, in the next Chapter we will determine the deformations which are both integrable and
transversal to the action of the diffeomorphism supergroup, and then directly construct a patch of super
Teichm\"uller space.

%% file: sdiff.tex
\chapter{The supergroup $\sdiff(\cm)$}
\label{ch:sdiff}

As in the case of an ordinary manifold, the superdiffeomorphisms of a supermanifold $\cm$ form a highly
nontrivial supermanifold. Its exhaustive investigation lies far beyond the scope of this work and the 
present knowledge
of its author, and many problems are unsolved even for the ordinary case. In this Chapter, we will 
only outline the structure of the diffeomorphism supergroup
and its action on various tensor fields on $\cm$. The basic ideas underlying the constructions are once
more due to V. Molotkov \cite{I-dZ2ks}. To describe $\sdiff(\cm)$ as an actual supermanifold, one would
have to use Fr\'echet supercharts. We will not do this here, but rather content ourselves with the construction of
$\sdiff(\cm)$ as a group object in $\catsets^\catgr$. For the Fr\'echet approach, see \cite{I-dZ2ks},
\cite{SoS}. The main goal of this Chapter is to provide the prerequisites for the study of the
existence of slices for the pullback action of the
identity component $\sdiff_0(\cm)$ of the diffeomorphism supergroup on the integrable almost complex
structures $\cc(\cm)$ on a given smooth supersurface.

\section{Inner Hom-functors for $\catsman$}

\subsection{Generators for $\catsman$}

The first step towards the definition of the diffeomorphism supergroup is the construction of an
inner Hom-functor for supermanifolds. According to the adjunction formula \ref{ihomdef}, such an 
inner Hom-functor $\ihom(\cm,\cn)$ has to satisfy
\begin{equation}
\Hom(\ct,\ihom(\cm,\cn))\cong\Hom(\ct\times\cm,\cn)\qquad\forall\,\ct\in\catsman.
\end{equation}
Since $\Hom(\cm,\cn)$ was denoted as $SC^\infty(\cm,\cn)$, we follow the notation of \cite{I-dZ2ks} and
denote $\ihom(\cm,\cn)$ as $\scinfty(\cm,\cn)$.

\begin{thm}[Molotkov \cite{I-dZ2ks}]
\label{catgensman}
Consider the functor
\begin{eqnarray}
\label{functf}
F:\catsets^{\catsman^\circ} &\to& \catsets^{\catspoint^\circ}\\
\nonumber
\Phi &\mapsto& \Phi\big|_{\catspoint},
\end{eqnarray}
which simply restricts each functor $\Phi$ to the full subcategory $\catspoint^\circ$. Then there exists
an isomorphism between the composition of functors
\begin{equation}
\label{funct1}
\xymatrix@1{\catsman \ar[r]^{H_*} & \catsets^{\catsman^\circ} \ar[r]^F & \catsets^{\catspoint^\circ}%
\ar[r]^\sim & \catsets^\catgr},
\end{equation}
where $H_*$ is the Yoneda embedding, and the forgetful functor
\begin{equation}
\label{funct2}
\xymatrix{N:\catsman \ar[r] & \catman^\catgr \ar[r] & \catsets^\catgr}.
\end{equation}
\end{thm}

We do not want to prove this here, since it would require some more technicalities that we
will not need for other purposes and instead rely on Molotkov's results \cite{I-dZ2ks}.
This theorem is the categorical analogue of the statement that the superpoints are generators for the
category $\catsman$, extending Thm.~\ref{gensman} to the infinite-dimensional 
case. A direct consequence is the following statement, which says that superpoints
completely describe a supermanifold by their morphisms into it, as is the case for ordinary
manifolds and the point $\Spec\fieldK$. The difference is just that 
there is a whole $\intZ_{\geq 0}$-family of superpoints, and each of them yields one set of points. 
The full information
about the supermanifold is then stored in this whole tower of points and their functoriality.

\begin{cor}[Molotkov \cite{I-dZ2ks}]
For every supermanifold $\cm$ and every $\Lambda\in\catgr$, there exists an isomorphism
\begin{equation}
\cm(\Lambda)\cong SC^\infty(\cp(\Lambda),\cm).
\end{equation}
\end{cor}
\begin{proof}
By applying the forgetful functor $N$ (\ref{funct2}), a supermanifold $\cm$ can be seen
as a functor $\catsets^\catgr$. 
The Yoneda embedding maps $\cm$ to the functor $\Hom_\catsman(-,\cm)$ in $\catsets^{\catsman^\circ}$,
which we can turn into a functor $\catsets^{\catspoint^\circ}$ by restriction (i.e., by applying
the functor $F$ from (\ref{functf}).
Then
Theorem \ref{catgensman} asserts that for every $\Lambda\in\catgr$, there exists an isomorphism of
sets
\begin{equation}
SC^\infty(\cp(\Lambda),\cm)=\Hom_\catsman(\cp(\Lambda),\cm)\cong\cm(\Lambda).
\end{equation}
\end{proof}

\subsection{The inner Hom-functor}

The inner Hom-functor for supermanifolds was already mentioned in an informal way in Sections
\ref{sect:functpoints} and \ref{sect:smanfpv} (see (\ref{ihomsman})). Thm.~\ref{catgensman} now
legitimates the following definition.

\begin{dfn}
Let $\cm,\cn$ be two supermanifolds. We define the inner Hom-functor $\scinfty(\cm,\cn)$ in 
$\catsets^\catgr$ by setting
\begin{equation}
\scinfty(\cm,\cn)(\Lambda) := SC^\infty(\cp(\Lambda)\times\cm,\cn)
\end{equation}
for all $\Lambda\in\catgr$. To each $\varphi:\Lambda\to\Lambda'$, assign
\begin{eqnarray}
\label{scinftymap}
\scinfty(\cm,\cn)(\varphi):\scinfty(\cm,\cn)(\Lambda) &\to& \scinfty(\cm,\cn)(\Lambda')\\
\nonumber
\sigma &\mapsto& \sigma\circ(\cp(\varphi)\times\id_\cm).
\end{eqnarray}
\end{dfn}

It is clear that this functor satisfies the adjunction formula (\ref{ihomdef}), thus it is an inner
Hom-functor. In general, it is not the inner Hom-functor in the category $\catsman$, since it is
usually impossible to give $\scinfty(\cm,\cn)$ the structure of a Banach supermanifold. As for ordinary
manifolds, the set of maps between them can usually only be endowed with the structure of a Fr\'echet
manifold. We will not try to address these topological subtleties here, since in the case of the
action of the superdiffeomorphism group $\sdiff(\cm)$ on $\cc(\cm)$, they only play a role for the 
underlying manifold. Rather, we will deal with this case in a different way. For our purposes, it is 
sufficient to construct $\scinfty(\cm,\cn)$ as an inner Hom-functor in $\catsets^\catgr$.

Note that the underlying set of $\scinfty(\cm,\cn)$ is just
\begin{equation}
\scinfty(\cm,\cn)(\fieldK)=SC^\infty(\cp(\fieldK)\times\cm,\cn)\cong SC^\infty(\cm,\cn),
\end{equation}
since $\cp(\fieldK)\times\cm\cong\cm$ by Lemma \ref{cpriso}.
Again, the underlying space of the inner-Hom functor consists precisely of the actual morphisms of
the two objects, as was already the case for other super things, like super vector spaces.

\subsection{Composition of morphisms}

Let $\cm,\cm',\cm''$ be supermanifolds. Then there exists a composition map
\begin{equation}
\label{comp}
\circ:\scinfty(\cm,\cm')\times\scinfty(\cm',\cm'')\to\scinfty(\cm,\cm''),
\end{equation}
which, of course, must be defined pointwise. Let $\Lambda\in\catgr$ be fixed, and let
$f:\cp(\Lambda)\times\cm\to\cm'$ and $g:\cp(\Lambda)\times\cm'\to\cm''$ be two supersmooth maps.
Then we define $g\circ f$ as the following composition:
\begin{equation}
\label{def:comp}
\xymatrix{
g\circ f:\cp(\Lambda)\times\cm \ar[rr]^-{\id_{\cp(\Lambda)}\times f} && \cp(\Lambda)\times\cm'%
\ar[r]^-{g} & \cm''}.
\end{equation}

On the functor $\scinfty(\cm,\cm)$ of supersmooth morphisms of $\cm$ into itself, this composition
is obviously associative. The following proposition also establishes the existence of a unit, 
making $\scinfty(\cm,\cm)$ into a monoid in $\catsets^\catgr$.

\begin{prop}
The functor morphism
\begin{eqnarray}
e:\cp(\realR) &\to& \scinfty(\cm,\cm)\\
e_\Lambda:\{0\}=\cp(\realR)(\Lambda) &\mapsto& (\Pi_\cm:\cp(\Lambda)\times\cm\to\cm)
\end{eqnarray}
in $\catsets^\catgr$ is the unit for the composition $\circ$ defined in (\ref{def:comp}).
\end{prop}
\begin{proof}
This is clear from the definition (\ref{def:comp}).
\end{proof}

The unit element allows one to define inverse morphisms as follows.

\begin{dfn}
Let $f\in\scinfty(\cm,\cm)(\Lambda)$ be a morphism $f:\cp(\Lambda)\times\cm\to\cm$. The inverse
$f^{-1}:\cp(\Lambda)\times\cm\to\cm$ is defined to be the morphism such that
\begin{equation}
(\id_{\cp(\Lambda)}\times f)\circ f^{-1}=(\id_{\cp(\Lambda)}\times f^{-1})\circ f=e_\Lambda(\{0\})=\Pi_\cm.
\end{equation}
\end{dfn}

\subsection{Geometric interpretation}

An interpretation of the functor of points of a supermanifold was already sketched in Sections
\ref{sect:functpoints} and \ref{sect:smanfpv}. It was argued that, for supermanifolds $\cm,\ct$,
the $\ct$-points $\Hom(\ct,\cm)$ of $\cm$ should be thought of as sections of the projection
$\pi_\ct:\ct\times\cm\to\ct$. By Thm.~\ref{catgensman} we know that only $\cp(\Lambda)$-points need
to be considered. If we assign to every $f\in\scinfty(\cm,\cm')(\Lambda)$ a morphism
\begin{equation}
\label{extmap}
\Pi_\cp\times f:\cp(\Lambda)\times\cm\to\cp(\Lambda)\times\cm'
\end{equation}
of families over $\cp(\Lambda)$, then we clearly obtain a bijection
\begin{equation}
\scinfty(\cm,\cm')(\Lambda)\to\Hom_{\catsman/\cp(\Lambda)}(\cp(\Lambda)^*(\cm),\cp(\Lambda)^*(\cm')).
\end{equation}
The functor $\ct^*:\catsman\to\catsman/\ct$ was defined in (\ref{def:tstar}) as the map which assigns to
every supermanifold $\cm$ its trivial family over $\ct$.

For $\cp(\realR)=\Spec\realR$, everything reduces to the case of ordinary morphisms of supermanifolds, 
because $\cp(\realR)\times\cm\cong\cm$. The composition of morphisms consists in the composition of
morphisms of families: let $f\in\scinfty(\cm,\cm')(\Lambda)$ and $g\in\scinfty(\cm',\cm'')(\Lambda)$
be given. Then their composition $g\circ f$ corresponds to
\begin{equation}
\xymatrix{\cp(\Lambda)\times\cm \ar[rr]^{\Pi_{\cp(\Lambda)}\times f} \ar[drr]_{\Pi_{\cp(\Lambda)}} &&%
\cp(\Lambda)\times\cm' \ar[rr]^{\Pi_{\cp(\Lambda)}\times g} \ar[d]^{\Pi_{\cp(\Lambda)}} &&%
\cp(\Lambda)\times\cm'' \ar[dll]^{\Pi_{\cp(\Lambda)}}\\
&& \cp(\Lambda)&&}.
\end{equation}
The fact that the composition $\circ$ is associative and has a unit can be seen as a consequence of 
this interpretation: the $\Lambda$-component of the functor morphism $\circ$ becomes identified with
the ordinary composition of morphisms in the category $\catsman/\cp(\Lambda)$, which is associative by
definition.
Inserting the unit of $\scinfty(\cm,\cm)(\Lambda)$ into the map (\ref{extmap}), it becomes just the
identity $\id_{\cp(\Lambda)\times\cm}$, as one expects for the neutral element of a space of
morphisms.

\section{The group of superdiffeomorphisms}
\label{sect:sdiff}

From now on, we will only consider the case of a given finite-dimensional 
supermanifold $\cm$.

Define for each $\Lambda\in\catgr$ a set $\sdiff(\cm)(\Lambda)$ by setting
\begin{equation}
\sdiff(\cm)(\Lambda)=\{f\in\scinfty(\cm,\cm)(\Lambda)\mid f\textrm{ invertible}\}.
\end{equation}
Clearly, each of these sets is a group. Therefore if we can show that they form a functor in
$\catsets^\catgr$, this functor will be a supergroup (a group object in $\catsets^\catgr$). To show that
this is indeed the case and that the sets $\sdiff(\cm)(\Lambda)$ actually form the subfunctor in 
$\scinfty(\cm,\cm)$ whose underlying set is the group $\Aut(\cm)$, i.e., the isomorphisms of $\cm$ in
$\catsman$, will be the goal of this section. Although a categorical proof exists \cite{I-dZ2ks}, we
will study the elements of $\sdiff(\cm)(\Lambda)$ directly as morphisms of ringed spaces, which will give
us a somewhat deeper insight into their structure than a purely formal approach. 

The first thing to show is that $\sdiff(\cm)$ is a subfunctor of $\scinfty(\cm,\cm)$. This means we
have to prove that the restriction of the morphisms (\ref{scinftymap}) to $\sdiff(\cm)$ is well-defined,
i.e., that for each $\varphi:\Lambda\to\Lambda'$ the image of 
$\scinfty(\cm,\cm)(\varphi)\big|_{\sdiff(\cm)(\Lambda)}$ lies in $\sdiff(\cm)(\Lambda')$.

\begin{prop}
\label{sdhom}
For each $\Lambda\in\catgr$ and each morphism $\varphi:\Lambda\to\Lambda'$, the restriction of 
$\scinfty(\cm,\cm)(\varphi)$ to $\sdiff(\cm)(\Lambda)$ induces a group homomorphism
\begin{equation}
\sdiff(\cm)(\varphi):\sdiff(\cm)(\Lambda)\to\sdiff(\cm)(\Lambda').
\end{equation}
\end{prop}
\begin{proof}
Applying the definition (\ref{scinftymap}) to the neutral element $\Pi_\cm:\cp(\Lambda)\times\cm\to\cm$,
we see immediately that
\begin{equation}
\Pi_\cm\circ(\cp(\varphi)\times\id_\cm)=\Pi_\cm,
\end{equation}
i.e., $\scinfty(\cm,\cm)(\varphi)$ maps the unit element to the unit element. Now let 
$f,g\in\sdiff(\cm)(\Lambda)$ be given. We have to show that
\begin{equation}
\scinfty(\cm,\cm)(\varphi)(g\circ f)=(\scinfty(\cm,\cm)(\varphi)(g))\circ(\scinfty(\cm,\cm)(\varphi)(f)).
\end{equation}
It is most insightful to compare the definition of the two functors. The left hand side corresponds to 
the composition
\begin{equation}
\label{f2}
\xymatrix@1{\cp(\Lambda')\times\cm \ar[rr]^{(\cp(\varphi),\id_\cm)} &&%
\cp(\Lambda)\times\cm \ar[rr]^{(\Pi_{\cp(\Lambda)},f)} && \cp(\Lambda)\times\cm \ar[r]^-{g} &\cm},
\end{equation}
while the right hand side corresponds to
\begin{multline}
\label{f3}
\xymatrix@1{\cp(\Lambda')\times\cm \ar[rrr]^(.43){(\Pi_{\cp(\Lambda')},\cp(\varphi),\id_\cm)} &&&%
\cp(\Lambda')\times\cp(\Lambda)\times\cm \ar[rr]^(.7){(\Pi_{\cp(\Lambda')},f)} &&}\\
\xymatrix{&\ar[r]& \cp(\Lambda')\times\cm \ar[rr]^{(\cp(\varphi),\id_\cm)} &&%
\cp(\Lambda)\times\cm \ar[r]^-{g} &\cm}.
\end{multline}
Let now $m\in\cm(\Lambda'')$ be some $\Lambda''$-point of $\cm$, $p\in\cp(\Lambda')(\Lambda'')$ be
a $\Lambda''$-point of $\cp(\Lambda')$ and let $q\in\cp(\Lambda)(\Lambda'')$ be its image under
$\cp(\varphi)$, i.e., $q=\cp(\varphi)(p)$. Then (\ref{f2}) will map the pair $(p,m)$ to
\begin{equation}
\xymatrix@1{(p,m) \ar@{|->}[r] & (q,m) \ar@{|->}[r] & (q,f_{\Lambda''}(q,m)) \ar@{|->}[r] &%
g(q,f_{\Lambda''}(q,m))}.
\end{equation}
On the other hand, (\ref{f3}) will map $(p,m)$ as
\begin{equation}
\xymatrix@-=12pt{(p,m) \ar@{|->}[r] & (p,q,m) \ar@{|->}[r] & (p,f_{\Lambda''}(q,m)) \ar@{|->}[r] &%
(q,f_{\Lambda''}(q,m)) \ar@{|->}[r] & g(q,f_{\Lambda''}(q,m))}.
\end{equation}
This shows that all components of the two functors (\ref{f2}) and (\ref{f3}) are indeed identical.
\end{proof}

\begin{cor}
$\sdiff(\cm)$ is a subfunctor of $\scinfty(\cm,\cm)$ and a group object in $\catsets^\catgr$.
\end{cor}
\begin{proof}
By Prop.~\ref{sdhom}, for $\varphi:\Lambda\to\Lambda'$, $\scinfty(\cm,\cm)(\varphi)$ maps invertible 
morphisms to invertible morphisms, so the restriction of $\scinfty(\cm,\cm)(\varphi)$ to
$\sdiff(\cm)(\Lambda)$ is well-defined. This means that the inclusion 
$\sdiff(\cm)\subset\scinfty(\cm,\cm)(\varphi)$ is a functor morphism, and thus $\sdiff(\cm)$ is a
subfunctor. Since each $\sdiff(\cm)(\Lambda)$ is a group and each $\sdiff(\cm)(\varphi)$ is a group
homomorphism, the second assertion is clear.
\end{proof}

\subsection{Fine structure of supersmooth morphisms}

We will now analyse the structure of a morphism $f:\cp(\Lambda)\times\cm\to\cm$ explicitly, i.e.,
by studying the supermanifolds involved as ringed spaces. This will give us a nice interpretation of the
``higher points'' $\scinfty(\cm,\cm)(\Lambda)$ in terms of odd parameters. As a byproduct, we find a
factorization theorem for supersmooth morphisms.

In what follows we will keep $\Lambda$ as
\[
\Lambda=\Lambda_n=\realR[\tau_1,\ldots,\tau_n].
\]
Now, let $\varphi:\cp(\Lambda)\times\cm\to\cm$ be given.
It consists of a continuous map
\begin{equation}
\phi:\{*\}\times M\to M
\end{equation}
of the underlying topological spaces and a sheaf morphism
\begin{equation}
\varphi:\co_\cm\to\phi_*\co_{\cp(\Lambda_n)\times\cm},
\end{equation}
which we also denote by $\varphi$. This should not cause any confusion, since the sheaf morphism is the 
only one we really have to deal with.\footnote{It can be shown \cite{SoS}, \cite{SfMAI} that the sheaf map
in fact defines the continuous map of the underlying spaces, so one does not really have freedom in specifying
it.} Then, for every topological point $p\in M$, $\varphi$ consists of a stalk map
\begin{equation}
\varphi_p:\co_{\cm,\phi(p)}\to\co_{\cp(\Lambda)\times\cm,(\{*\},p)},
\end{equation}
which is a homomorphism of superalgebras. Clearly,
\begin{equation}
\co_{\cp(\Lambda)\times\cm,(\{*\},p)}\cong\Lambda\otimes_\realR\co_{\cm,p}
\end{equation}
since $\cp(\Lambda)$ is just the one-point supermanifold with structure sheaf $\Lambda$. Let $f$ be a
germ in $\co_{\cm,\phi(p)}$. Then $\varphi_p(f)$ is of the general form
\begin{equation}
\label{sdiffeo}
\varphi_p(f)=\sum_{I\subseteq\{1,\ldots,n\}}\tau_I\alpha_I(f).
\end{equation}
The sum runs over all increasingly ordered subsets, including the empty one. For a subset $I=\{i_1,\ldots,i_k\}$ with
$i_1<\ldots<i_k$, we write
\begin{equation}
\label{taui}
\tau_I=\tau_{i_1}\tau_{i_2}\cdots\tau_{i_k}.
\end{equation}
The $\alpha_I$ are, for all $I$, homomorphisms of superalgebras
\begin{equation}
\alpha_I:\co_{\cm,\phi(p)}\to\co_{\cm,p}.
\end{equation}
Each $\alpha_I$ has parity $|I|$, i.e., those whose index is of even length are even, 
while the others are odd.
So, (\ref{sdiffeo}) reads explicitly as
\begin{equation}
\varphi_p(f)=\alpha_0(f)+\tau_1\alpha_1(f)+\ldots+\tau_1\cdots\tau_n\alpha_{1\ldots n}(f),
\end{equation}
and the image has the same parity as $f$. 

If we apply $\scinfty(\cm,\cm)(\epsilon_{\Lambda})$ to $\varphi$, then we obtain a morphism
\[
\scinfty(\cm,\cm)(\epsilon_{\Lambda})(\varphi)=:\varphi_\realR:\cm\to\cm,
\]
i.e. a morphism of $\cm$ into itself as a superringed space. We will call $\varphi_\realR$ the 
\emph{underlying
morphism} of $\varphi$. It is of course not just a smooth selfmap of the underlying manifold $M_{rd}$, but
a morphism of supermanifolds. Clearly, if $\varphi$ has the form (\ref{sdiffeo}) on the stalk at $p$, then
$\varphi_\realR$ acts on this stalk simply as
\begin{equation}
\label{phir}
\varphi_{\realR,p}(f)=\alpha_0(f),
\end{equation}
since $\scinfty(\cm,\cm)(\epsilon_{\Lambda})$ acts by annihilating all odd generators in the structure 
sheaf of $\cp(\Lambda)$.

The space of all supersmooth morphisms $\cm\to\cm$ can have a quite complicated structure. Each set of
higher points $\scinfty(\cm,\cm)(\Lambda)$, however, has a remarkably simple structure as a bundle over
$\scinfty(\cm,\cm)(\realR)$.
Since the proof of the result for general $\Lambda$ is somewhat tedious and cumbersome, we will study the 
semigroups $\scinfty(\cm,\cm)(\Lambda_1)$ and $\scinfty(\cm,\cm)(\Lambda_2)$ separately first.

\begin{prop}
\label{l1sdiff}
Let $\varphi:\cp(\Lambda_1)\times\cm\to\cm$ be a $\Lambda_1$-point of $\scinfty(\cm,\cm)$. Then $\varphi$ is
uniquely determined by its underlying morphism $\varphi_\realR$ and an odd vector field $X$ on $\cm$.
Specifically,
\begin{equation}
\varphi=\varphi_\realR\circ(1+\tau X),
\end{equation}
where $\tau$ is the odd generator of $\Lambda_1$.
\end{prop}
\begin{proof}
For any $p\in M$, we know by (\ref{sdiffeo}) that the induced stalk map $\varphi_{\phi(p)}$ acts on a 
germ of a function as
\begin{equation}
\varphi_{\phi(p)}(f)=\alpha_0(f)+\tau\alpha_1(f).
\end{equation}
The homomorphism property for $\varphi_{\phi(p)}$ requires that for any two germs $f,g\in\co_{\cm,p}$, 
we must have
\begin{equation}
\varphi_{\phi(p)}(fg)=(\alpha_0(f)+\tau\alpha_1(f))(\alpha_0(g)+\tau\alpha_1(g)).
\end{equation}
This requires
\begin{eqnarray}
\label{a0}
\alpha_0(fg) &=& \alpha_0(f)\alpha_0(g),\\
\label{a1}
\alpha_1(fg) &=& \alpha_1(f)\alpha_0(g)+(-1)^{p(f)}\alpha_0(f)\alpha_1(g).
\end{eqnarray}
Therefore, $\alpha_0$ is a homomorphism $\co_{\cm,\phi(p)}\to\co_{\cm,p}$ and determines 
$\varphi_{\realR,\phi(p)}$,
as already shown in (\ref{phir}). $\alpha_1$ is an odd derivation of $\co_{\cm,\phi(p)}$ composed with 
$\alpha_0$, i.e., there exists a germ $X_{\phi(p)}$ of a smooth odd vector field at $\phi(p)$ such that
\begin{equation}
\alpha_1(f)=\alpha_0(X_{\phi(p)}(f)).
\end{equation}
Since all stalk maps have to be induced by a homomorphism of sheaves of smooth functions, the germs 
$X_{\phi(p)}$ must be induced by a global smooth odd vector field $X$ on $\cm$.
\end{proof}

Prop.~\ref{l1sdiff} exhibits a factorization propery of morphisms of supermanifolds: each 
$\Lambda_1$-point of
$\scinfty(\cm,\cm)$ is an morphism of $\cm$ into itself composed with odd derivation of 
each stalk. This pattern will also show up in the general case.

\begin{prop}
\label{l2sdiff}
Let $\varphi:\cp(\Lambda_2)\times\cm\to\cm$ be a $\Lambda_2$-point of $\scinfty(\cm)$. Then $\varphi$ is
uniquely determined by its underlying morphism $\varphi_\realR$, two odd vector fields $X_1,X_2$
and an even vector field $X_{12}$ on $\cm$,
such that
\begin{equation}
\varphi=\varphi_\realR\circ \exp(\tau_1 X_1+\tau_2 X_2 +\tau_1\tau_2 X_{12}),
\end{equation}
where $\tau_1, \tau_2$ are the odd generators of $\Lambda_2$.
\end{prop}
\begin{proof}
By (\ref{sdiffeo}) we conclude again that the stalk map at $p\in M$ has the form
\begin{equation}
\varphi_{\phi(p)}(f)=\alpha_0(f)+\tau_1\alpha_1(f)+\tau_2\alpha_2(f)+\tau_1\tau_2\alpha_{12}(f).
\end{equation}
Using the homomorphism property $\varphi_{\phi(p)}(fg)=\varphi_{\phi(p)}(f)\varphi_{\phi(p)}(g)$, 
we obtain an equation for
each of the $\alpha_I$. For $\alpha_0$, we again obtain (\ref{a0}), and for $\alpha_1,\alpha_2$ an
equation of the form (\ref{a1}). Thus again, $\alpha_0$ can be identified with $\varphi_{\realR,{\phi(p)}}$, 
and $\alpha_1,\alpha_2$ correspond to odd vector fields $X_1,X_2$ composed with $\alpha_0$. Finally, for
$\alpha_{12}$, we obtain
\begin{multline}
\alpha_{12}(fg)=\alpha_{12}(f)\alpha_0(g)+\alpha_0(f)\alpha_{12}(g)+%
(-1)^{p(\alpha_1(f))}\alpha_1(f)\alpha_2(g)-\\
-(-1)^{p(\alpha_2(f))}\alpha_2(f)\alpha_1(g).
\label{12cond}
\end{multline}
This is clearly satisfied for
\[
\tau_1\tau_2\alpha_{12}(f)=\alpha_0(\tau_2X_2(\tau_1X_1(f))),
\]
but also for
\[
\tau_1\tau_2\alpha_{12}(f)=\alpha_0(\tau_1X_1(\tau_2X_2(f))).
\]
Any linear combination
\[
\tau_1\tau_2\alpha_{12}=c\cdot\alpha_0\circ(\tau_2 X_2(\tau_1 X_1))+%
(1-c)\cdot\alpha_0\circ(\tau_1 X_1(\tau_2 X_2))
\]
will therefore satisfy (\ref{12cond}), as well. Moreover, from (\ref{12cond}), we see that any two 
homomorphisms
$\alpha_{12},\alpha'_{12}$ which satisfy (\ref{12cond}) are allowed to differ by an even derivation, 
since inserting them yields
\[
(\alpha_{12}-\alpha'_{12})(fg)=(\alpha_{12}-\alpha'_{12})(f)\alpha_0(g)+%
\alpha_0(f)(\alpha_{12}-\alpha'_{12})(g)
\]
as the necessary condition. Thus the most general form of $\tau_1\tau_2\alpha_{12}$ can be written as
\begin{equation}
\label{a12}
\tau_1\tau_2\alpha_{12}=\frac{1}{2}\alpha_0\circ(\tau_1 X_1(\tau_2 X_2)+\tau_2 X_2(\tau_1 X_1))+%
\tau_1\tau_2 \alpha_0\circ X_{12},
\end{equation}
where $X_{12}$ is the germ of a smooth even vector field at $\phi(p)$. Summing up (\ref{a12}) and 
$\tau_i\alpha_0\circ X_i$ for $i=1,2$ and $\alpha_0$, we obtain
\begin{equation}
\alpha_0+\tau_1\alpha_0\circ X_1+\tau_2\alpha_0\circ X_2+%
\frac{1}{2}(\tau_1\alpha_1(\tau_2\alpha_2)+\tau_2\alpha_2(\tau_1\alpha_1))+%
\tau_1\tau_2 \alpha_0\circ X_{12}
\end{equation}
which matches exactly
\begin{equation}
\alpha_0\circ\exp(\tau_1X_1+\tau_2X_2+\tau_1\tau_2X_{12}).
\end{equation}
\end{proof}

It is now quite clear what the general formula will look like, and which strategy has to be applied to prove
it. Denote by $\mathfrak{S}(a_1\cdots a_n)$ the symmetrization of the ordered sequence $a_1\cdots a_n$, i.e.,
\[
\mathfrak{S}(a_1\cdots a_n)=\frac{1}{n!}\sum_{\sigma\in P(n)} a_{\sigma(1)}\cdots a_{\sigma(n)},
\]
where $P(n)$ is the group of permutations of $n$ elements. To keep the notation simple, let us also
introduce the following convention: the expression $I=I_1+\ldots+I_j$ shall denote the decomposition of 
the ordered set $I$ into an ordered $j$-tuple of subsets $I_1,\ldots,I_j$, each of which inherits an 
ordering from $I$. For example, $\{1,2\}=I_1+I_2$ consists of the four partitions
\[
\{\{\},\{1,2\}\},\quad \{\{1\},\{2\}\},\quad \{\{2\},\{1\}\},\quad \{\{1,2\},\{\}\}.
\]
The notation $I=I_1\cup\ldots\cup I_j$, on the other hand, denotes the decomposition of the ordered set 
$I$ into an
\emph{unordered} $j$-tuple of disjoint ordered subsets. So, $\{1,2\}=I_1\cup I_2$ consists of two partitions:
\[
\{\{\},\{1,2\}\},\quad \{\{1\},\{2\}\}.
\]

The following lemma will be useful.

\begin{lemma}
\label{sym}
Let $A$ be an algebra, $f,g\in A$, and let $a_1,\ldots,a_n$ be derivations of $A$. Then
\begin{equation}
\mathfrak{S}(a_1\circ\ldots\circ a_n)(fg)=\sum_{\{1,\ldots,n\}=K+L}
\mathfrak{S}(a_K)(f)\mathfrak{S}(a_L)(g),
\end{equation}
where for $K=\{k_1,\ldots,k_j\}$, $a_K$ denotes the composition
\[
a_K=a_{k_1}\circ\ldots\circ a_{k_j}.
\]
\end{lemma}
\begin{proof}
By the product rule, it is clear that 
$\mathfrak{S}(a_1\circ\ldots\circ a_n)(fg)$ will take the form
\[
\mathfrak{S}(a_1\circ\ldots\circ a_n)(fg)=\sum_{\{1,\ldots,n\}=K+L}\frac{N(K,L)}{n!}a_K(f)a_L(g),
\]
with some integer $N(K,L)$ denoting the multiplicity the $K,L$-summand.
Since the symmetrized product on the left hand side contains all possible orderings of the operators
$a_i$, all possible partitions of $\{1,\ldots,n\}$ into two ordered subsets will really 
appear on the right hand side. The summand with given $K$ and $L$ occurs exactly $(|K|+|L|)!/(|K|!|L|!)$
times, as one checks as follows: starting from an ordered sequence $K$ of indices, there are
$(|K|+|L|)!/|K|!$ ways to insert $|L|$ elements at arbitrary positions into it.
But since the ordering of $L$ is also fixed, one has to divide by the number 
of permutations of $L$. So we have
\begin{eqnarray}
\mathfrak{S}(a_1\circ\ldots\circ a_n)(fg) &=& \sum_{K,L\subseteq\{1,\ldots,n\}}%
\frac{(|K|+|L|)!}{|K|!|L|!n!}a_K(f)a_L(g)\\
&=& \sum_{\{1,\ldots,n\}=K+L}\mathfrak{S}(a_K)(f)\mathfrak{S}(a_L)(g)
\end{eqnarray}
\end{proof}

\begin{thm}
\label{sdiff}
Let $\varphi:\cp(\Lambda)\times\cm\to\cm$ be a $\Lambda$-point of $\scinfty(\cm,\cm)$. Then
$\varphi$ is uniquely determined by its underlying morphism $\varphi_\realR:\cm\to\cm$, as well as 
$2^{n-1}$ odd
and $2^{n-1}-1$ even vector fields $X_I$ on $\cm$ such that it can be expressed as
\begin{equation}
\label{toprove}
\varphi=\varphi_\realR\circ\exp(\sum_{I\subseteq\{1,\ldots,n\}}\tau_IX_I),
\end{equation}
where the sum runs over all unordered nonempty subsets and $\tau_I$ is defined by (\ref{taui}).
\end{thm}
\begin{proof}
For every $\phi(p)\in M$, we use (\ref{sdiffeo}) to write the stalk map $\varphi_{\phi(p)}$ as
\begin{equation}
\label{topro2}
\varphi_{\phi(p)}=\sum_{I\subseteq\{1,\ldots,n\}}\tau_I\alpha_I,
\end{equation}
with $\tau_I$ defined by (\ref{taui}) and each $\alpha_I$ a homomorphism $\co_{\cm,\phi(p)}\to\co_{\cm,p}$.
Since $\varphi_{\phi(p)}$ has to be a homomorphism, we require
\begin{equation}
\left(\sum_{K\subseteq\{1,\ldots,n\}}\tau_K\alpha_K(fg)\right)=%
\left(\sum_{I\subseteq\{1,\ldots,n\}}\tau_I\alpha_I(f)\right)\cdot%
\left(\sum_{J\subseteq\{1,\ldots,n\}}\tau_J\alpha_J(g)\right).
\label{hom}
\end{equation}
Identifying (\ref{toprove}) with the sum (\ref{topro2}) rephrases the claim of the theorem as
\begin{equation}
\label{topro3}
\tau_I\alpha_I=\sum_{j=1}^{|I|}\sum_{I=I_1\cup\ldots\cup I_j}\alpha_0\circ%
\mathfrak{S}\left((\tau_{I_1}X_{I_1})\circ\ldots\circ(\tau_{I_j}X_{I_j})\right).
\end{equation}
The summation runs over all partitions of $I$ into unordered tuples of subsets, each subset inheriting
an ordering from $I$ (cf.~the definition of the notation $I=I_1\cup\ldots\cup I_j$ above). We will prove this
formula by induction on $|I|$.

For indices $I$ of length $|I|=0,1,2$, the assertion holds by Props.~\ref{l1sdiff} and \ref{l2sdiff}.
Assume the statement has been proven for indices up to length $k$. Then let $I=\{i_1,\ldots,i_{k+1}\}$
be an index of length $k+1$.
We must assure that (\ref{hom}) holds, which means we must find the general solution $\alpha_I$ for
\begin{eqnarray}
\nonumber
\tau_I\alpha_I(fg) &=& \alpha_0(f)\tau_I\alpha_I(g)+(-1)^{p(f)}\tau_I\alpha_I(f)\alpha_0(g)\\
\label{topro4}
&& \sum_{\substack{I=K+L\\ K,L\neq\emptyset}}\tau_K\alpha_K(f)\tau_L\alpha_L(g).
\end{eqnarray}
Since $|K|,|L|\leq k$, it follows that $\tau_K\alpha_K$ and $\tau_L\alpha_L$ 
must have the form (\ref{topro3}). Therefore 
the sum in (\ref{topro4}) can be written as
\begin{multline}
\sum_{\substack{I=K+L\\ K,L\neq\emptyset}}\alpha_0\circ\left(\sum_{j=1}^{|K|}\sum_{K=K_1\cup\ldots\cup K_j}
\mathfrak{S}\left((\tau_{K_1}X_{K_1})\circ\ldots\circ(\tau_{K_j}X_{K_j})\right)(f)\circ\right.\\
\left.\sum_{l=1}^{|L|}\sum_{L=L_1\cup\ldots\cup L_l}
\mathfrak{S}\left((\tau_{L_1}X_{L_1})\circ\ldots\circ(\tau_{L_l}X_{L_l})\right)(g)\right).
\end{multline}
By Lemma \ref{sym}, this equals
\begin{equation}
\sum_{j=2}^{|I|}\sum_{I=I_1\cup\ldots\cup I_j}\alpha_0\circ%
\mathfrak{S}\left((\tau_{I_1}X_{I_1})\circ\ldots\circ%
(\tau_{I_j}X_{I_j})\right)(fg).
\end{equation}
The general solution to equation (\ref{topro4}) therefore reads
\begin{eqnarray}
\nonumber
\tau_I\alpha_I &=& \alpha_0\circ \tau_IX_I+\sum_{j=2}^{|I|}\sum_{I=I_1\cup\ldots\cup I_j}%
\alpha_0\circ\mathfrak{S}\left((\tau_{I_1}X_{I_1})\circ\ldots\circ(\tau_{I_j}X_{I_j})\right)\\
&=& \sum_{j=1}^{|I|}\sum_{I=I_1\cup\ldots\cup I_j}%
\alpha_0\circ\mathfrak{S}\left((\tau_{I_1}X_{I_1})\circ\ldots\circ(\tau_{I_j}X_{I_j})\right),
\end{eqnarray}
where $X_I$ is a derivation of parity $|I|$ of $\co_{\cm,\phi(p)}$.
\end{proof}

\subsection{The higher points of $\sdiff(\cm)$}

As a corollary of the results of the previous section, we obtain the criterion for the invertibility
of a supersmooth morphism $\cp(\Lambda)\times\cm\to\cm$, and thus a statement on the structure of
$\sdiff(\cm)$.

The underlying group $\sdiff(\cm)(\realR)$ of automorphisms of $\cm$ as a smooth supermanifold 
can have a very 
complicated structure. Its topology is determined by the topology of $\mathrm{Diff}(M)$, i.e., of the diffeomorphism 
group of its underlying manifold, and by the topology of the spaces of isomorphisms of smooth vector 
bundles on
$M$, since any smooth supermanifold can be realized as a smooth manifold and the exterior bundle of a
smooth vector bundle on it (Batchelor's theorem \cite{TSOS}). Even in the case of a $2|2$-dimensional smooth
supersurface, the diffeomorphism group of $M$ consists of infinitely many connected components
generated from its identity component by applying the mapping class group of the surface. 
All these
topological subtleties, however, pertain only to the underlying group $\sdiff(\cm)(\realR)$, while the higher points
$\sdiff(\cm)(\Lambda)$ have a much simpler structure.

\begin{thm}
\label{sdiffinv}
A supersmooth morphism $\varphi:\cp(\Lambda)\times\cm\to\cm$ is invertible if and only if its underlying
morphism $\varphi_\realR:\cm\to\cm$ is invertible. In this case, writing the sheaf map $\varphi$ as
\begin{equation}
\varphi=\varphi_\realR\circ\exp(\sum_{I\subseteq\{1,\ldots,n\}}\tau_IX_I)
\end{equation}
in accord with Thm.~\ref{sdiff}, its inverse is the sheaf map
\begin{equation}
\label{invhom}
\varphi^{-1}=\exp(-\!\!\sum_{I\subseteq\{1,\ldots,n\}}\!\tau_IX_I)\circ\varphi_\realR^{-1}.
\end{equation}
\end{thm}
\begin{proof}
We have to show that for each stalk $\co_{\cm,p}$, the action of the composition of the two exponentials
is the identity. Let $\varphi$ be the germ of our morphism. We have to show that
\begin{equation}
\label{topro5}
\exp(-\!\!\sum_{I\subseteq\{1,\ldots,n\}}\!\tau_I X_I)\circ%
\exp(\!\sum_{J\subseteq\{1,\ldots,n\}}\!\tau_J X_J)=\id_{\cm,p}.
\end{equation}
We can write
\begin{equation}
\label{topro6}
\exp(-\!\!\sum_{I\subseteq\{1,\ldots,n\}}\!\tau_I X_I)\circ\exp(\!\sum_{J\subseteq\{1,\ldots,n\}}\!\tau_J X_J)=%
1+\sum_K\tau_K\alpha_K
\end{equation}
by expanding both exponentials. Using (\ref{topro3}), we rewrite the expression on the left hand side as
\begin{multline}
1+\left(\sum_{j=1}^{|I|}\sum_{I=I_1\cup\ldots\cup I_j}%
\mathfrak{S}\left((-\tau_{I_1}X_{I_1})\circ\ldots\circ(-\tau_{I_j}X_{I_j})\right)\right)\circ\\
\left(\sum_{k=1}^{|J|}\sum_{J=J_1\cup\ldots\cup J_k}%
\mathfrak{S}\left((\tau_{J_1}X_{J_1})\circ\ldots\circ(\tau_{J_k}X_{J_k})\right)\right)
\end{multline}
Now $\tau_K\alpha_K$ on the right hand side of (\ref{topro6}) is a sum over all partitions of $K$ 
into ordered tuples of
subsets. Pick one such tuple $\{K_1,\ldots,K_n\}$; the tuple, and each of the $K_i$, is ordered, and 
their union is $K$. On the left hand side, we have the corresponding sum
\begin{equation}
\frac{1}{k!(n-k)!}\sum_{k=0}^n(-1)^k(\tau_{K_1}X_{K_1})\circ\ldots\circ(\tau_{K_n}X_{K_n})
\end{equation}
of all ways of realizing this sequence of indices by contributions from either two of the exponentials in
(\ref{topro5}). But
\[
\sum_{k=0}^n\frac{1}{k!(n-k)!}(-1)^k=\frac{1}{n!}(1+(-1))^n=0.
\]
Therefore, each $\alpha_K$ on the right hand side of (\ref{topro6}) receives only vanishing contributions,
and thus (\ref{topro5}) holds.
\end{proof}

\begin{cor}
The supergroup $\sdiff(\cm)$ is the restriction of $\scinfty(\cm,\cm)$ onto
$\Aut(\cm)\subset\scinfty(\cm,\cm)(\realR)$.
\end{cor}
\begin{proof}
This is a direct consequence of Thm.~\ref{sdiffinv}.
\end{proof}

Moreover, we can now phrase the factorization properties of $\sdiff(\cm)$ as follows. Let 
$\cx(\cm)$ denote the superrepresentable $\olr$-module of smooth sections of the tangent bundle of
$\cm$ (it exists due to Thm.~\ref{thm:sectsvbun}). By the above discussion it is clear that 
we obtain a unipotent group $\cn_\cm$ by
\begin{eqnarray}
\exp:\cx(\cm)^{nil} &\to& \cn_\cm\\
X &\mapsto& \exp(X)
\end{eqnarray}
because any $X\in\cx(\cm)(\Lambda)$, $\Lambda\neq\realR$ can be written as a sum 
\begin{equation}
\sum_{I\subseteq\{1,\ldots,n\}}\tau_IX_I,
\end{equation}
where again the $\tau_i$ are the free odd generators of $\Lambda$ and each $X_I$ is a vector field of
parity $|I|$.

\begin{thm}
The supergroup $\sdiff(\cm)$ splits as a semidirect product
\begin{equation}
\sdiff(\cm)=\Aut(\cm)\ltimes \cn_\cm.
\end{equation}
\end{thm}
\begin{proof}
Obviously we have $\Aut(\cm)\cap\cn_\cm=\{\id_\cm\}$ and by Thm.~\ref{sdiff} we know that
$\sdiff(\cm)=\Aut(\cm)\cn_\cm$. It remains to show that $\cn_\cm$ is normal. If 
$\varphi_\realR\in\Aut(\cm)$ is given and $X$ is a vector field, then we have for every
germ $f$ of a function on $\cm$
\begin{equation}
X(\varphi^*f)=\varphi^*\circ D\varphi(X)(f).
\end{equation}
If $\exp(\sum\tau_IX_I)$ is an element of $\cn_\cm$, this entails
\begin{equation}
\varphi_\realR\circ\exp(\!\sum_{I\subseteq\{1,\ldots,n\}}\!\tau_I D\varphi_\realR(X_I))=%
\exp(\!\sum_{I\subseteq\{1,\ldots,n\}}\!\tau_I X_I)\circ\varphi_\realR.
\end{equation}
Since $D\varphi_\realR$ is an isomorphism for an invertible $\varphi_\realR$, this implies that
\begin{equation}
\varphi_\realR\circ\cn_\cm=\cn_\cm\circ\varphi_\realR
\end{equation}
for all $\varphi_\realR\in\Aut(\cm)$.
\end{proof}

\subsection{\texorpdfstring{The action of $\sdiff(\cm)$ on vector fields and 1-1-tensor fields}{The action
of SD(M) on vector fields and 1-1-tensor fields}}

In order to analyse the structure of orbits of almost complex structures under pullback, we start by
investigating the pushforward operation for vector fields, i.e., the differential of a superdiffeomorphism.
The statement of Thm.~\ref{sdiff} can be rephrased as follows: apart from the action of its underlying
morphism, a diffeomorphism $\varphi:\cp(\Lambda)\times\cm\to\cm$ acts by Lie derivatives:
\begin{equation}
\varphi(f)=\varphi_\realR\circ\exp(\!\sum_{I\subseteq\{1,\ldots,n\}}\!\tau_I L_{X_I})(f).
\end{equation}
Intuitively, this corresponds to the fact that a deformation of a family of supermanifolds with fibre 
$\cm$ along the odd dimensions of the base has to be linear with respect to the coordinate of this dimension.
This in turn is due to the fact that one
cannot ``move a finite distance'' along an odd direction: there is no way of localizing ``at a point'' 
in these directions.
So all odd deformations have to be infinitesimal. We will, of course, discover the same phenomenon for
the induced actions on vector fields and other tensors. 

The Lie derivative of a vector field $Y$ along
a vector field $X$ is defined as $L_X(Y):=[X,Y]$, where the bracket is the superbracket of vector fields.
The Lie derivative of a (1-1)-tensor field $\sigma$ along some even vector field $X$ is then given by
\begin{eqnarray}
\label{lieder11}
(L_X\sigma)(Y) &:=& -\sigma([X,Y])+[X,\sigma(Y)]\\
&=& \left[-\sigma\circ L_X+L_X\circ\sigma\right](Y),
\end{eqnarray}
where $Y$ is any vector field.

\begin{prop}
Let $\varphi:\cp(\Lambda_n)\times\cm\to\cm$ be a $\Lambda_n$-point of $\sdiff(\cm)$ whose underlying
diffeomorphism is $\varphi_\realR$ and let $Y$ a vector field on $\cp(\Lambda)\times\cm$. Suppose 
$\varphi$ is given by an expression as in the statement of Thm.~\ref{sdiff}. 
Then we can write the differential $D\varphi$ as
\begin{equation}
\label{dvarphi}
D\varphi(Y)=\exp(-\!\!\sum_{I\subseteq\{1,\ldots,n\}}\!\tau_I L_{X_I})\circ D\varphi_\realR(Y).
\end{equation}
\end{prop}
\begin{proof}
The differential is determined by the commutative diagram of sheaf maps
\begin{equation}
\begin{CD}
\co_{\cp(\Lambda_n)\times\cm} @>\varphi>> \co_{\cp(\Lambda_n)\times\cm}\\
@V{D\varphi(Y)}VV @VV{Y}V\\
\co_{\cp(\Lambda_n)\times\cm} @>\varphi>> \co_{\cp(\Lambda_n)\times\cm}
\end{CD},
\end{equation}
where we have extended $\varphi$ to $\Pi_{\cp(\Lambda)}\times\varphi$, i.e., to a morphism in
$\catsman/\cp(\Lambda)$. On the stalk at $\phi(p)$, we must then have
\begin{equation}
(D\varphi(Y))_{\phi(p)}(f)=(\varphi^{-1})_{p}\circ Y\circ(\varphi_{\phi(p)}(f)).
\end{equation}
Inserting (\ref{invhom}) and (\ref{toprove}), this is equivalent to
\begin{equation}
\label{zw0}
D\varphi(Y)=\exp(-\sum_I\tau_IX_I)\circ\varphi_\realR^{-1}\circ Y\circ%
\varphi_\realR\circ\exp(\sum_I\tau_IX_I).
\end{equation}
Clearly,
\begin{equation}
\varphi_\realR^{-1}\circ Y\circ\varphi_\realR=D\varphi_\realR(Y).
\end{equation}
Using the Baker-Campbell-Hausdorff formula and abbreviating
\begin{equation}
\label{bchsum}
\mathcal{X}=\sum_I\tau_IX_I,
\end{equation}
the expression (\ref{zw0}) can be written as
\begin{equation}
\label{zw1}
D\varphi(Y)=\sum_{m=0}^n \frac{(-1)^m}{m!}\underbrace{[\mathcal{X},[,\ldots,[\mathcal{X}}_{m\,\,\,times}%
,D\varphi_\realR(Y)]\cdots]].
\end{equation}
Expanding $\cx$ into the sums (\ref{bchsum}) again, we obtain
\begin{equation}
D\varphi(Y)=\sum_{I\subseteq\{1,\ldots,n\}}\sum_{j=1}^{|I|}\sum_{I=I_1+\ldots+I_j}\frac{(-1)^j}{j!}%
[\tau_{I_1}X_{I_1},\ldots,[\tau_{I_j}X_{I_j},D\varphi_\realR(Y)]\cdots].
\end{equation}
This is clearly the same as
\begin{eqnarray}
\nonumber
D\varphi(Y) &=& \!\!\!\!\!\sum_{I\subseteq\{1,\ldots,n\}}\sum_{j=1}^{|I|}\sum_{I=I_1\cup\ldots\cup I_j}%
\!\!\!\mathfrak{S}((-\tau_{I_1}L_{X_{I_1}})\circ\ldots\circ (-\tau_{I_j}L_{X_{I_j}}))(D\varphi_\realR(Y))\\
&=& \exp(-\sum_{I\subseteq\{1,\ldots,n\}}\tau_I L_{X_I})\circ D\varphi_\realR(Y).
\end{eqnarray}
\end{proof}

One might have guessed this from general arguments: quantities, which transform contravariantly, 
i.e. by pullback, should transform by application of a transformation $\exp(\sum_I\tau_IX_I)$. 
Those which transform covariantly, like vector fields, will transform inversely. We will need the
analogous result for (1-1)-tensor fields.

\begin{prop}
\label{prop:pbtensor}
Let $\varphi:\cp(\Lambda)\times\cm\to\cm$ be a $\Lambda$-point of $\sdiff(\cm)$ whose underlying
diffeomorphism is $\varphi_\realR$, and let
$\sigma$ be an endomorphism of the tangent bundle, i.e. a 1-1-tensor field on $\cp(\Lambda)\times\cm$.
Also, suppose that $\varphi$ is given by an expression as in the statement of Thm.~\ref{sdiff}. 
Then we can express the pullback of $\varphi^*\sigma$ as
\begin{equation}
\varphi^*\sigma=\varphi_\realR^*\circ\exp(\!\sum_{I\subseteq\{1,\ldots,n\}}\!\tau_I L_{X_I})(\sigma).
\end{equation}
\end{prop}
\begin{proof}
The pullback of a germ $\sigma\in\cend(\ctm)_{(\{*\},\phi(p))}$ of an endomorphism of the 
tangent bundle of $\cp(\Lambda_n)\times\cm$ by a diffeomorphism $\varphi$ is given by
\begin{equation}
\label{dend}
(\varphi^*\sigma)_p=(D\varphi)^{-1}_{\phi(p)}\circ\sigma_{\phi(p)}\circ(D\varphi)_p.
\end{equation}
Inserting (\ref{dvarphi}) into this formula yields
\begin{equation}
(\varphi^*\sigma)_p=(D\varphi_\realR)^{-1}\circ\exp(\!\sum_{I\subseteq\{1,\ldots,n\}}\!\tau_I L_{X_I})%
\circ\sigma\circ\exp(-\!\!\sum_{J\subseteq\{1,\ldots,n\}}\!\tau_J L_{X_J})\circ(D\varphi_\realR)
\end{equation}
Therefore, the proposition is proven if we can show that
\begin{equation}
\label{topro7}
\exp(\!\sum_{I\subseteq\{1,\ldots,n\}}\!\tau_I L_{X_I})(\sigma)=%
\exp(\!\sum_{I\subseteq\{1,\ldots,n\}}\!\tau_I L_{X_I})\circ\sigma%
\circ\exp(-\!\!\sum_{J\subseteq\{1,\ldots,n\}}\!\tau_J L_{X_J}),
\end{equation}
where on the left hand side, the Lie derivatives act on $\sigma$ by (\ref{lieder11}), while on the right
hand side, they have to be understood as acting on vector fields. Expanding the right hand side
yields
\begin{multline}
\label{tar2}
\sum_{I\subseteq\{1,\ldots,n\}}\sum_{j=1}^{|I|}\sum_{I=I_1+\ldots+I_j}%
\sum_{a=0}^j\frac{(-1)^{j-a}}{a!(j-a)!}\cdot\\
\cdot L_{\tau_{I_1}X_{I_1}}\circ\ldots\circ L_{\tau_{I_a}X_{I_a}}\circ%
\sigma\circ L_{\tau_{I_{a+1}}X_{I_{a+1}}}\circ\ldots\circ L_{\tau_{I_j}X_{I_j}},
\end{multline}
while the left hand side can be written as
\begin{equation}
\label{tar3}
\sum_{I\subseteq\{1,\ldots,n\}}\sum_{j=1}^{|I|}\sum_{I=I_1+\ldots+ I_j}%
\frac{1}{j!}(L_{\tau_{I_1}X_{I_1}})\circ\ldots\circ(L_{\tau_{I_j}X_{I_j}})(\sigma).
\end{equation}
To see that (\ref{tar2}) equals (\ref{tar3}), we have to check which factor arises in front of
the summand
\[
L_{\tau_{I_1}X_{I_1}}\circ\ldots\circ L_{\tau_{I_a}X_{I_a}}\circ%
\sigma\circ L_{\tau_{I_{a+1}}X_{I_{a+1}}}\circ\ldots\circ L_{\tau_{I_j}X_{I_j}}
\]
in (\ref{tar3}) when one transforms it using the formula (\ref{lieder11}) for the Lie derivative
of (1,1)-tensors. The factor contains $1/j!$ because the length of the summand is $j$, and a 
factor $(-1)^{j-a}$ as one
can see from (\ref{lieder11}). But many summands of (\ref{tar3})
contribute to it, namely each one that contains $I_1,\ldots,I_a$ in this order and $I_{a+1},\ldots,I_j$
in the reverse order. There are precisely $j!/(j-a)!a!$ such summands, as one can see from the same reasoning as in 
the proof of Lemma \ref{sym}. Thus, the factors of each summand in (\ref{tar2}) and (\ref{tar3})
coincide and (\ref{topro7}) holds.
\end{proof}

\section{\texorpdfstring{The group $\Aut(\cm)$}{The automorphism group of a supermanifold}}
\label{sect:autm}

The discussion of the previous sections was exclusively concerned with the identity element and the action of the
``higher'' points (i.e., those for $\Lambda\neq\realR$) of the diffeomorphism supergroup. For
the endeavour of quotienting out the action of $\sdiff_0(\cm)$, we also need to investigate the
underlying ordinary group $\Aut(\cm)$. It turns out that its elements, again, can be split into a
``hard'' part, containing the analytic and topological subtleties typical of diffeomorphism groups, and
a ``nice'' algebraic part. In this case, however, we cannot achieve a splitting as a semidirect product.

Let a finite-dimensional supermanifold 
$\cm=(M,\co_\cm)$ be given, let $M$ be its underlying smooth manifold and let $E\to M$ be the associated smooth 
vector bundle, i.e., $E$ is the locally free sheaf $\Pi(\cn_\cm/\cn_\cm^2)$ (cf.~Section \ref{sect:explmor})
with its action of $C^\infty(M)\cong C^\infty(\cm)/\cn_\cm$.

The tangent sheaf $\ctm$ is, as a locally free module over $\co_\cm$, filtered by the powers of the
nilpotent ideal $\cn_\cm$. As in Section \ref{sect:explmor}, we introduce the notation
\begin{equation}
\ctm^{(j)},\qquad j\in 2\naturalN_0
\end{equation}
for the ideal in $\ctm$ which consists of even vector fields which are of degree $\geq j$ in the odd 
coordinates of $\cm$. That means that $X\in\ctm^{(j)}$ can everywhere locally be written as
\begin{eqnarray}
\label{vfield}
X &=& \sum_{k=j}^n\sum_{i,\alpha_1,\ldots,\alpha_k}f_i^{\alpha_1\ldots\alpha_k}(x)%
\theta_{\alpha_1}\cdots\theta_{\alpha_k}\pderiv{}{x_i}+\\
&&+ \sum_{k=j}^{n-1}\sum_{n,\alpha_1,\ldots,\alpha_{k+1}}f_n^{\alpha_1\ldots\alpha_{k+1}}(x)
\theta_{\alpha_1}\cdots\theta_{\alpha_{k+1}}\pderiv{}{\theta_n},
\end{eqnarray}
with the $f_k^A(x)$ ordinary smooth functions of the even coordinates $x_i$ on $\cm$.
The tangent space of $\sdiff(\cm)(\realR)$, i.e., the Lie algebra of $\Aut(\cm)$, is the algebra of
even vector fields on $\cm$. Since ordinary diffeomorphisms of the underlying manifold $M$ are involved
in any superdiffeomorphism, it is clear that the exponential map will not generate $\sdiff(\cm)(\realR)$.
But this problem does not occur for nilpotent vector fields, which simplifies the situation considerably.

\begin{lemma}
\label{lemmanil}
The Lie algebra $\ctm^{(2)}$ generates a nilpotent group $N_\cm$ of automorphisms of $\cm$, which act as the
identity on any associated vector bundle $(M,E)$.
\end{lemma}
\begin{proof}
Since the coefficient functions of any section $X$ of $\ctm^{(2)}$ are of degree $\geq 2$ in the odd
variable, the exponential $\exp(X)$ is a finite sum and therefore always well-defined. Additionally, one
easily verifies that $\exp(X)\exp(-X)=\exp(-X)\exp(X)=\One$, because the vector field $X$ is even and
hence commutes with itself.

For any $X$, $\exp(X)$ acts on $\co_\cm$ as an automorphism: its action (assuming it to be in the form 
(\ref{vfield})) can everywhere locally be expressed as
\begin{eqnarray}
\label{ngr}
x_i &\to& x_i+\sum_{j=1}^\infty\sum_{\alpha_1,\ldots,\alpha_{2j}}f_i^{\alpha_1\ldots\alpha_{2j}}(x)%
\theta_{\alpha_1}\cdots\theta_{\alpha_{2j}},\\
\nonumber
\theta_j &\to& \theta_j+\sum_{j=1}^\infty\sum_{\alpha_1,\ldots,\alpha_{2j+1}}f_n^{\alpha_1\ldots\alpha_{2j+1}}(x)
\theta_{\alpha_1}\cdots\theta_{\alpha_{2j+1}}.
\end{eqnarray}
Composing this map with the canonical epimorphism $q:\co_\cm\to\co_\cm/\cn_\cm^2$ evidently yields
the identity.
\end{proof}

Therefore, the space of sections of the ideal subsheaf $\ctm^{(2)}\subset\ctm$ forms the Lie algebra 
associated with the nilpotent group $N_\cm$, and generates this group by exponentiation. This leaves
only the vector fields $\ctm^{(0)}/\ctm^{(2)}$. They form a quotient algebra, since the bracket of two
vector fields of degree zero is again of degree zero. It is clear that they must be the infinitesimal
transformations associated with morphisms of the associated vector bundles. To express this rigorously,
we have to fix some notation.

\begin{dfn}
\label{def:autvb}
Let $p:E\to M$ be a smooth vector bundle. Let $\Aut_\catvbun(M,E)$ denote the automorphism group of
$(M,E)$ as a vector bundle, i.e., each of its elements is a pair of maps $(f:E\to E,g:M\to M)$ satisfying
\begin{enumerate}
\item $f,g$ are smooth and invertible,
\item they are compatible with the projection, i.e., $p\circ f=g\circ p$, and
\item $f$ is fiberwise a linear automorphism of vector spaces $f_x:E_x\to E_{g(x)}$.
\end{enumerate}
The group $\Aut_{\catvbun(M)}(E)$ is the group of automorphisms of $E$ as a vector bundle over $M$, i.e.,
it consists of all elements of the form $(f,\id_M)\in\Aut_\catvbun(M,E)$.
\end{dfn}

\begin{lemma}
\label{lem:autvb}
$\Aut_{\catvbun(M)}(E)$ is a normal subgroup of $\Aut_\catvbun(M,E)$, and the latter group splits
as a semidirect product
\begin{equation}
\Aut_\catvbun(M,E)=\mathrm{Diff}(M)\ltimes\Aut_{\catvbun(M)}(E).
\end{equation}
\end{lemma}
\begin{proof}
Let $(f:E\to E,g:M\to M)\in\Aut_\catvbun(M,E)$ be given.
Then this statement is a direct consequence of the well-known fact that the map $f:E\to E$ can always be 
factorized as $f=g_*\circ h$ \cite{W:Differential}, where
$g_*:g^*E\to E$ is the canonical projection of the pullback, and $h:E\to E$ is a morphism of vector
bundles with $p\circ h=p$.
\end{proof}

Then $\Aut(\cm)$ obviously ``consists of'' vector bundle morphisms and morphisms inducing higher nilpotent corrections.
Here, however, we cannot obtain a splitting into a semidirect product.

\begin{thm}
\label{thm:autsman}
Let $\cm$ be a smooth supermanifold and let $(M,E)$ be the associated canonical vector bundle over the underlying
manifold $M$. Then the automorphism group $\Aut(\cm)$ fits into an exact sequence
\begin{equation}
1\longrightarrow N_\cm\longrightarrow\Aut(\cm)\longrightarrow \Aut_\catvbun(M,E)\longrightarrow 1
\end{equation}
where $N_\cm$ is the nilpotent group constructed in Lemma \ref{lemmanil}.
\end{thm}
\begin{proof}
From our discussion above it should be clear that the elements of $N_\cm$ are precisely those automorphisms of
$\cm$ which induce the identity on the associated vector bundle $(M,E)$. Since there is no canonical map
$\cm\to(M,\wedge^\bullet E)$
we do not obtain a map $\Aut_\catvbun(M,E)\to\Aut(\cm)$ which we would need
in order to split $\Aut(\cm)$ into a semidirect product again.
\end{proof}

%% file: quotient.tex
\chapter{Super Teichm\"uller spaces}
\label{ch:quotient}

In this Chapter we will investigate whether one can divide out the action of $\sdiff_0(\cm)$
on the integrable almost complex structures. All Riemann surfaces and supermanifolds whose base
is a Riemann surface are assumed to be compact.
Following Fischer and Tromba \cite{TTiRG}, the goal would be to find
local slices for the pullback action. These would constitute patches of the Teichm\"uller space
$\ct_{\fvect^L(1|1)}(\cm)$
of $\fvect^L(1|1)$-structures on a given smooth compact closed supersurface $\cm$ of real dimension $2|2$,
i.e., the Teichm\"uller space of supercomplex structures in $2|2$ real dimensions. It will turn out that this
is not possible in the same way as in ordinary Teichm\"uller theory, since there remains a nontrivial
subgroup of automorphisms in the identity component $\sdiff_0(\cm)$ of the diffeomorphism supergroup.
The action of these automorphisms cannot be divided out from $\cc(\cm)$ without destroying the 
supermanifold structure of the quotient. We will
construct instead a supermanifold $\ct_{\fvect^L(1|1)}^{g,d}$ which parametrizes the 
supercomplex structures up to these automorphisms.
This can then be at best a semiuniversal family, but we will not try to prove semiuniversality here
since it would lead us deep into deformation theory. The underlying manifold of
$\ct_{\fvect^L(1|1)}^{g,d}$, however, can be shown to be universal among all families of
compact complex supermanifolds of dimension $1|1$ \emph{over a purely even base}. This is a simple
corollary of the fact that every complex $1|1$-dimensional supermanifold is equivalent to a Riemann
surface and a line bundle.

For the $\fkl(1|1)$-structures, we can restrict this construction to spin curves. We construct a supermanifold
$\ct_{\fkl(1|1)}^g$ whose base is the Teichm\"uller space of spin curves. Again, this 
supermanifold does not parametrize
a universal family of $\fkl(1|1)$-structures, since there remains a $\intZ_2$-ambiguity
stemming from the diffeomorphism $\theta\to -\theta$ of the supersurface, which is an automorphism
of the $\fkl(1|1)$-structure. $\ct_{\fkl(1|1)}^g$ is shown to be a closed subsupermanifold of 
$\ct_{\fvect^L(1|1)}^{g,g-1}$. The remaining $\intZ_2$-symmetry in the construction of the
Teichm\"uller space of super Riemann surfaces was already observed by several authors, e.g.,
\cite{CR:Super}, \cite{LBR:Moduli}. In the latter reference, the authors remove the remaining
automorphisms at the cost of destroying the supermanifold structure. They call the resulting space
a ``canonical superorbifold''. We do not pursue this idea here, but instead content ourselves
with the construction of $\ct_{\fkl(1|1)}^g$ as the base of a semiuniversal family.

\section{The space of nontrivial deformations}

Throughout this and the next sections, we will study neighbourhoods in $\ca(\cm)$ of a fixed 
integrable almost complex
structure $J$ on $\cm$. In order to try to take a local quotient for the $\sdiff_0(\cm)$-action on a 
neighbourhood
of $J$, we need to determine a direct sum decomposition of 
$\hat{T}_J^{int}=T_J\cc(\cm)=T_J\cs\oplus\cl J$, where
$\cl J$ comprises all integrable deformations of $J$ which arise as Lie derivatives, i.e., deformations
tangent to the $\sdiff_0(\cm)$-orbit. The quotient $T_J\cs$ then represents the true deformations of $J$.

\subsection{General setting}

The problem is local. It is therefore most convenient to describe the tangent directions to
$J$ in the form (\ref{hform}) and (\ref{hcomp}), i.e., we write
\begin{equation}
H=\left(\begin{array}{cc}
0 & A-iB\\
A+iB & 0
\end{array}\right),
\end{equation}
where $A-iB$ is a square matrix of size $1|1$ (in the standard format) which we write in components as
\begin{equation}
A-iB=\left(\begin{array}{cc}
\alpha & \beta\\
\gamma & \delta
\end{array}\right).
\end{equation}
The entries are local sections of $\co_\cm\otimes\Lambda$ if $H\in\hat{T}_J^{int}(\Lambda)$. Of these
four sections, $\alpha,\delta$ are even and $\beta,\gamma$ are odd. Let $\tau_1,\ldots,\tau_n$ be
the free odd generators of $\Lambda=\Lambda_n$. This means that each component function,
say $\alpha$, can be written as an expansion
\begin{equation}
\alpha=\sum_{I\subset\{1,\ldots,n\}}\tau_I f_I,
\end{equation}
where the sum runs over all increasingly ordered subsets and $\tau_I$ is the product of the corresponding
$\tau_i$'s. Each $f_I$ is a section of $\co_\cm$ of parity $|I|+p(\alpha)$. On the other hand, since we
have local complex coordinates $z,\theta$, we can write each local smooth section $f$ of $\co_\cm$ as
\begin{equation}
f=\left\{\begin{array}{ll}
f_0(z)+f_3(z)\theta\bar{\theta} & \textrm{if }p(f)=\bar{0},\\
f_1(z)\theta+f_2(z)\bar{\theta} & \textrm{if }p(f)=\bar{1},
\end{array}\right.
\end{equation}
where the $f_i(z)$ are ordinary smooth functions on the underlying manifold $M$. The question to be
answered is: for which tangent vectors $H$ of $J$ can we find an even real super vector field $X$
on $\cm$ such that
\begin{equation}
-iL_XJ=H
\end{equation}
holds, where $L_X$ denotes the Lie derivative. Let us express $X$ in the form
\begin{equation}
\label{vectform}
X=X^z\pderiv{}{z}+X^\theta\pderiv{}{\theta}+X^{\bar{z}}\pderiv{}{\bar{z}}+%
X^{\bar{\theta}}\pderiv{}{\bar{\theta}},
\end{equation}
where the coefficient functions are smooth real sections of $\co_\cm\otimes\Lambda$.

\begin{lemma}
\label{lieder}
Let $X$ be an even smooth real vector field given in the form (\ref{vectform}) and let 
$H\in\hat{T}_J^{int}(\Lambda)$ be a tangent vector to an integrable almost complex structure $J$. Then 
the condition $-iL_XJ=H$ is equivalent to the following equations:
\begin{gather}
2\pderiv{X^z}{\bar{z}}=\alpha, \qquad 2\pderiv{X^{\theta}}{\bar{z}}=\gamma\\
2\pderiv{X^z}{\bar{\theta}}=\beta, \qquad 2\pderiv{X^{\theta}}{\bar{\theta}}=\delta.
\label{cond}
\end{gather}
\end{lemma}
\begin{proof}
This is a straightfoward calculation in components. Write
\begin{equation}
J=i(\partial_{z}\otimes dz+\partial_{\theta}\otimes d\theta)-%
i(\partial_{\bar{z}}\otimes d\bar{z}+\partial_{\bar{\theta}}\otimes d\bar{\theta}).
\end{equation}
The Lie derivative of a 1-1-tensor field on a super manifold is given by the same formula as for an
ordinary manifold, except that a sign arises whenever two odd factors pass by each other. So,
the $i\partial_z\otimes dz$-component of $J$ contributes the terms
\begin{gather}
-\pderiv{X^z}{z}\partial_z\otimes dz-\pderiv{X^\theta}{z}\partial_\theta\otimes dz-%
\pderiv{X^{\bar{z}}}{z}\partial_{\bar{z}}\otimes dz-%
\pderiv{X^{\bar{\theta}}}{z}\partial_{\bar{\theta}}\otimes dz+\\
+\pderiv{X^z}{z}\partial_z\otimes dz+\pderiv{X^z}{\theta}\partial_z\otimes d\theta+%
\pderiv{X^z}{\bar{z}}\partial_z\otimes d\bar{z}+\pderiv{X^z}{\bar{\theta}}\partial_z\otimes d\bar{\theta}.
\end{gather}
The terms proportional to $\partial_z\otimes dz$ cancel, and this will happen as well for the other
contributions of the form $\partial_Z\otimes dZ$ for $Z=\bar{z},\theta,\bar{\theta}$. Analogously, 
the term $-i\partial_{\bar{z}}\otimes d\bar{z}$ contributes
\begin{gather}
\pderiv{X^z}{\bar{z}}\partial_z\otimes d\bar{z}+\pderiv{X^\theta}{\bar{z}}\partial_\theta\otimes d\bar{z}-%
+\pderiv{X^{\bar{z}}}{\bar{z}}\partial_{\bar{z}}\otimes d\bar{z}+%
\pderiv{X^{\bar{\theta}}}{\bar{z}}\partial_{\bar{\theta}}\otimes d\bar{z}-\\
-\pderiv{X^{\bar{z}}}{z}\partial_{\bar{z}}\otimes dz-%
\pderiv{X^{\bar{z}}}{\theta}\partial_{\bar{z}}\otimes d\theta-%
-\pderiv{X^{\bar{z}}}{\bar{z}}\partial_{\bar{z}}\otimes d\bar{z}-%
\pderiv{X^{\bar{z}}}{\bar{\theta}}\partial_{\bar{z}}\otimes d\bar{\theta}.
\end{gather}
The summands $i\partial_\theta\otimes d\theta$ and $-i\partial_{\bar{\theta}}\otimes d\bar{\theta}$
contribute similar expressions. Summing all up and equating this with $H$ yields equations (\ref{cond}).
\end{proof}

Just as one would expect, the vector field $X$ is an infinitesimal automorphism of the complex
structure (i.e., $L_XJ=0$) if and only if $X^z$ and $X^\theta$ do not depend on 
$\bar{z}$ and $\bar{\theta}$, i.e., if they are superholomorphic.

\subsection{The underlying tangent space}

Let us start with the problem of determining the above mentionend decomposition for the underlying
tangent space $\hat{T}_J^{int}(\realR)=T_J\cs(\realR)\oplus\cl J(\realR)$. In this case the analysis is
simplified by the absence of any odd parameters $\tau_i$. In this and the following sections we assume
that the complex $1|1$-dimensional supermanifold $\cm$ is given by a pair $(M,L)$ where $M$ is a
closed, compact Riemann surface of genus $g \geq 2$ and $L$ is a holomorphic line bundle of degree $d$ on
$M$ (cf.~Section \ref{sect:compss}).

\begin{prop}
\label{prop:tangsli0}
The underlying tangent space $\hat{T}_J^{int}(\realR)=T_J\cc(\cm)(\realR)$ of an integrable almost complex
structure $J$ possesses a direct-sum decomposition
\begin{equation}
\hat{T}_J^{int}(\realR)=V\oplus\cl J
\end{equation}
where $V$ is complex vector space of dimension $4g-3$, $g$ is the genus of the underlying surface $M$ and
\begin{equation}
\mathcal{L}J(\realR):=\left\{H\in \hat{T}_J^{int}(\realR)\mid \exists \textrm{ a smooth even vector field }X
\textrm{s.t. }L_XJ=H\right\}.
\end{equation}
\end{prop}
\begin{proof}
Since no odd $\tau_i$-parameters are involved, we can expand the 
coefficient functions $\alpha,\beta,\gamma,\delta$ of $H$ as
\begin{eqnarray}
\alpha &=& \alpha_0(z)+\alpha_3(z)\theta\bar{\theta}\\
\beta &=& \beta_1(z)\theta+\beta_2(z)\bar{\theta}\\
\gamma &=& \gamma_1(z)\theta+\gamma_2(z)\bar{\theta}\\
\delta &=& \delta_0(z)+\delta_3(z)\theta\bar{\theta}
\end{eqnarray}
where each of the expansion coefficients is just a locally defined smooth function on the underlying manifold
$M$. The integrability condition (\ref{int2}) implies $\beta_2=0=\delta_3$. Likewise, we can expand
each of the coefficient functions of a smooth even real vector field $X$ as
\begin{eqnarray}
X^z &=& X^z_0(z)+X^z_3(z)\theta\bar{\theta}\\
X^\theta &=& X^\theta_1(z)\theta+ X^\theta_2(z)\bar{\theta}
\end{eqnarray}
and similarly for $X^{\bar{z}}$ and $X^{\bar{\theta}}$.

Now we have to determine the obstructions to the solution of the equations (\ref{cond}). The equations
\begin{equation}
\beta=\beta_1\theta=2\pderiv{X^z}{\bar{\theta}}=-X_3^z\theta
\end{equation}
and
\begin{equation}
\delta=\delta_0=2\pderiv{X^\theta}{\bar{\theta}}=X_2^\theta
\end{equation}
are purely algebraic, and thus
can always be solved: for any given $\beta_1,\delta_0$, we will find a vector field $X$ that satisfies them.

So $\beta$ and $\delta$ have been found, and condition (\ref{int1}) then entails that
\begin{align}
\pderiv{\alpha}{\bar{\theta}} &=%
-\alpha_3(z)\theta=\pderiv{\beta}{\bar{z}}=\pderiv{\beta_1(z)}{\bar{z}}\theta\\
\pderiv{\gamma}{\bar{\theta}} &= \gamma_2(z) = \pderiv{\delta}{\bar{z}}=\pderiv{\delta_0(z)}{\bar{z}}
\end{align}
are also determined. The only equations that remain are
\begin{eqnarray}
\label{firsteq}
2\pderiv{X^z_0}{\bar{z}} &=& \alpha_0,\\
\label{seceq}
2\pderiv{X^\theta_1}{\bar{z}}\theta &=& \gamma_1\theta.
\end{eqnarray}
Let us first look at (\ref{firsteq}). The function $\alpha_0(z)$ is actually a smooth section of a holomorphic
line bundle on $M$: it contributes the summand
\[
\alpha=\alpha_0(z)\pderiv{}{z}\otimes d\bar{z}
\]
to $H$.
Therefore we can find a solution to (\ref{firsteq}) if and only if the Dolbeault cohomology class
$[\alpha]$ vanishes. We can view $\alpha$ as a $\opar$-closed smooth $(-1,1)$-form on $M$, and by the
Hodge decomposition theorem, we know that any such space of forms can be directly decomposed as
\begin{equation}
\ca^{p,q}=\opar\ca^{p,q-1}\oplus H\oplus\opar^*\ca^{p,q+1}.
\end{equation}
Here, $\opar^*=-*\opar *$ is the adjoint of $\opar$, $*$ is the Hodge star-operator (see, e.g., 
\cite{H:Complex},\cite{GH:Principles}) and $H$ is the space of harmonic forms. Thus we can write
\begin{equation}
\alpha=\opar\eta+\omega+\opar^*\kappa,
\end{equation}
but since $\opar\alpha=0$, we must have $\opar^*\kappa=0$. Therefore, the obstruction to the solution of
(\ref{firsteq}) is
\begin{equation}
H^{-1,1}_{\opar}(M)\cong H^1(M,K^{-1})\cong H^0(M,K^2)
\end{equation}
where the first ``$\cong$'' is due to the Dolbeault theorem, and second one follows from Serre duality. 
Moreover,
this space splits off directly from $\hat{T}_J^{int}(\realR)$ as the subspace of those deformations for which
$\alpha_0$ is harmonic.

Now everything works completely analogously for equation (\ref{seceq}). Here, however, 
\begin{equation}
\gamma_1(z)\theta=(\gamma_1(z)\theta)\partial_\theta\otimes d\bar{z}
\end{equation}
must be viewed as a smooth section of the line bundle $L\otimes L^{-1}$ since $\theta$ is
a local section of $L$ and $\partial_\theta$ is a section of $L^*\cong L^{-1}$. This means that the
obstructions to the solution of (\ref{seceq}) form the space
\begin{equation}
H^1(M,L\otimes L^{-1})\cong H^1(M,\co)\cong H^0(M,K).
\end{equation}
By the same argument as above, this space splits off directly from $\hat{T}_J^{int}$, and so altogether
we can write
\begin{equation}
\hat{T}_J^{int}(\realR)=H^0(M,K)\oplus H^0(M,K^2)\oplus\cl J,
\end{equation}
where the first two summands form a complex vector space of dimension $4g-3$.
\end{proof}

\subsection{\texorpdfstring{The $\Lambda_1$-points of $T_J\cs$}{The Lambda(1)-points of the deformation}}

Before treating the general $\Lambda$-points of $T_J\cs$, it is instructive to first find the
space $T_J\cs(\Lambda_1)$ separately. The reason is that $T_J\cs$ is supposed to be
a super vector space, and therefore the sets of $\realR$- and $\Lambda_1$-points contain the 
full information, while the
higher points can be expected to be generated from these two spaces.

\begin{prop}
\label{prop:tangsli1}
The $\Lambda_1$-tangent space $\hat{T}_J^{int}(\Lambda_1)=T_J\cc(\cm)(\Lambda_1)$ of an integrable 
almost complex structure $J$ possesses a direct-sum decomposition
\begin{equation}
\hat{T}_J^{int}(\Lambda_1)=\olv(\Lambda_1)\oplus\cl J,
\end{equation}
where $V$ is a super vector space of dimension
\begin{equation}
\dim_\complexC V=(4g-3|4g-4+h^0(L^{-1})+h^0(K^{-1}\otimes L))
\end{equation}
and
\begin{equation}
\mathcal{L}J(\Lambda_1):=\left\{H\in \hat{T}_J^{int}(\Lambda_1)\mid%
\exists\textrm{ a smooth even vector field }X \textrm{s.t. }L_XJ=H\right\}.
\end{equation}
\end{prop}
\begin{proof}
Denote the odd generator of $\Lambda_1$ by $\tau$. This time, each of the component functions of $H$ can
contain additional terms proportional to $\tau$. So we have
\begin{eqnarray}
\alpha &=& \alpha_0(z)+\tau\alpha_1(z)\theta+\tau\alpha_2(z)\bar{\theta}+\alpha_3(z)\theta\bar{\theta},\\
\beta  &=& \tau\beta_0(z)+\beta_1(z)\theta+\beta_2(z)\bar{\theta}+\tau\beta_3(z)\theta\bar{\theta}, 
\end{eqnarray}
and analogously for $\gamma$ and $\delta$. For the same reason, the coefficients of a generic smooth 
even vector field $X$ must be extended to
\begin{eqnarray}
X^z &=& X^z_0(z)+\tau X^z_1(z)\theta+\tau X^z_2(z)\bar{\theta}+X^z_3(z)\theta\bar{\theta},\\
X^\theta &=& \tau X^\theta_0(z)+X^\theta_1(z)\theta+ X^\theta_2(z)\bar{\theta}+%
X^\theta_3(z)\theta\bar{\theta},
\end{eqnarray}
and similarly for $X^{\bar{z}}$ and $X^{\bar{\theta}}$. As was already argued in Chapter \ref{ch:sdiff},
it may clarify things if one remembers that $H$, $J$ and $X$ are not really objects living on $\cm$,
but rather on the family $\cp(\Lambda_1)\times\cm$, i.e. they are deformations of vector fields and
tensors on $\cm$ along an odd dimension parametrized by $\tau$.

Now the strategy closely follows that of Prop.~\ref{prop:tangsli0}. Note, first of all, that
the integrability condition (\ref{int2}) tells us that
\begin{equation}
\beta_2=\beta_3=\delta_2=\delta_3=0.
\end{equation}
The equations
\begin{equation}
\beta=\tau\beta_0+\beta_1\theta=2\pderiv{X^z}{\bar{\theta}}=-\tau X_2^z-X_3^z\theta
\label{betacond}
\end{equation}
and
\begin{equation}
\delta=\delta_0+\tau\delta_1=2\pderiv{X^\theta}{\bar{\theta}}=X_2^\theta+\tau X_3^\theta\theta
\label{deltacond}
\end{equation}
are again purely algebraic and can always be solved: there always exists a vector field $X$ satisfying them.
The conditions (\ref{int1}) then fix $\alpha_2,\alpha_3$ and $\gamma_2,\gamma_3$. Thus the only remaining
equations are
\begin{eqnarray}
\label{f1eq}
\alpha_0+\tau\alpha_1\theta &=& 2\pderiv{X^z}{\bar{z}}=2\pderiv{X^z_0}{\bar{z}}+%
2\tau\pderiv{X^z_1}{\bar{z}}\theta\\
\label{f2eq}
\tau\gamma_0+\gamma_1\theta &=& 2\pderiv{X^\theta}{\bar{z}}=2\tau\pderiv{X^\theta_0}{\bar{z}}+%
2\pderiv{X^\theta_1}{\bar{z}}\theta.
\end{eqnarray}

We can use arguments analogous to those of the proof of Prop.~\ref{prop:tangsli0} to determine the
obstructions to the solution of these equations. First of all, $\alpha_0$ and $\gamma_1$ again contribute
the space $H^0(M,K^2)\oplus H^0(M,K)$. But now there are also obstructions proportional to $\tau$.
In (\ref{f1eq}) this is $\alpha_1(z)\theta$, a smooth section of the line bundle $L\otimes K^{-1}$, 
therefore the odd obstructions to the solution of (\ref{f1eq}) are
\begin{equation}
H^1(M,L\otimes K^{-1})\cong H^0(M,K^2\otimes L^{-1}),
\end{equation} 
where Serre duality has been used.
Analogously, $\gamma_0(z)$ is a smooth section of $L^{-1}$, and therefore the odd obstructions
to the solution of (\ref{f2eq}) are
\begin{equation}
H^1(M,L^{-1})\cong H^0(M,K\otimes L).
\end{equation}
By the Riemann-Roch theorem one finds
\begin{eqnarray}
h^0(M,K\otimes L)-h^0(M,L^{-1}) &=& \deg L+g-1,\\
h^0(M,K^2\otimes L^{-1})-h^0(M,K^{-1}\otimes L)&=& 3g-3-\deg L.
\end{eqnarray}
The Hodge decomposition theorem assures again that we can split off the obstructions as a direct submodule
from $\hat{T}_J^{int}(\Lambda_1)$. Summarizing all the above, we have obtained a decomposition
\begin{equation}
\hat{T}_J^{int}(\Lambda_1)=H^0(M,K^2)\oplus H^0(M,K)\oplus \tau\left(H^0(M,K\otimes L)\oplus%
H^0(M,K^2\otimes L^{-1})\right)\oplus \cl J,
\end{equation}
and the first four summands clearly are the $\Lambda_1$-points of a super vector space $V$ of the claimed 
dimension.
\end{proof}

\subsection{The general case}

It remains to be shown that we can indeed directly decompose the integrable deformations as
$\hat{T}_J^{int}=T_J\cs\oplus\cl J$, where $T_J\cs$ is a superrepresentable $\complexC$-module of
finite dimension. Of course, $T_J\cs$ will then have to coincide with $\olv$, where $V$ is the super
vector space constructed in the proof of Prop.~\ref{prop:tangsli1}.

\begin{thm}
\label{splittang}
The $\olc$-module $\hat{T}_J^{int}$ tangent to $J$ admits a direct sum decompostion
\begin{equation}
\hat{T}_J^{int}=T_J\cs\oplus\cl J,
\end{equation}
where $T_J\cs$ is a superrepresentable $\olc$-module of dimension 
\[
(4g-3|4g-4+h^0(L^{-1})+h^0(K^{-1}\otimes L)),
\]
and
\begin{equation}
\mathcal{L}J:=\left\{H\in \hat{T}_J^{int}\mid\exists\textrm{ a smooth even vector field }X 
\textrm{s.t. }L_XJ=H\right\}.
\end{equation}
Here, the definition of $\cl J$ has to be understood pointwise: it contains all 
$H\in\hat{T}_J^{int}(\Lambda)$ for which there exists a smooth vector field on $\cp(\Lambda)\times\cm$
such that $L_XJ=H$.
\end{thm}
\begin{proof}
Assume that $\Lambda=\Lambda_n$ and denote the free odd generators of $\Lambda$ by $\tau_1,\ldots,\tau_n$.
In this setting, each of the component functions of $H$, for example $\alpha$, can still be written as
\begin{equation}
\label{exp1}
\alpha=\alpha_0+\alpha_1\theta+\alpha_2\bar{\theta}+\alpha_3\theta\bar{\theta}.
\end{equation}
But now each of the expansion coefficients $\alpha_i$ can contain the odd parameters $\tau_i$. Since
$\alpha$ is even, we have
\begin{equation}
\label{exp2}
\alpha_0=\alpha_{00}+\tau_i\tau_j\alpha_{0ij}+\ldots=%
\sum_{\substack{I\subset\{1,\ldots,n\}\\|I|\,\,\mathrm{even}}}\tau_I\alpha_{0I}
\end{equation}
The sum here runs over all increasingly ordered subsets of even cardinality, and $\tau_I$ is the product
of the appropriate $\tau_i$'s in the same order. Each $\alpha_{0I}$ is just a local smooth function on
the underlying manifold $M$. An analogous expansion exists for $\alpha_3$. Since $\alpha_1,\alpha_2$ have
to be odd, their expansions into powers of $\tau$ run over the subsets of odd cardinality.

Likewise, each of the coefficient functions $X^z,X^\theta,X^{\bar{z}},X^{\bar{\theta}}$, has an expansion
like (\ref{exp1}) and each of the coefficients must then be expanded in a sum of the form
(\ref{exp2}), which runs over the subsets of even or odd cardinality, depending on the parity
of the coefficient.

Again, the integrability conditions (\ref{int2}) tell us that $\beta_2=\beta_3=0$ and
$\delta_2=\delta_3=0$. Also, the equations
\begin{equation}
\beta=\beta_0+\beta_1\theta=2\pderiv{X^z}{\bar{\theta}}=-X_2^z-X_3^z\theta
\end{equation}
and
\begin{equation}
\delta=\delta_0+\delta_1=2\pderiv{X^\theta}{\bar{\theta}}=X_2^\theta+X_3^\theta\theta
\end{equation}
remain purely algebraic and can always be solved: there always exists a vector field $X$ satisfying them.
Note that we have suppressed the $\tau_i$-expansions here and in (\ref{f1eqq}) and (\ref{f2eqq}) below.
The conditions (\ref{int1}) then still fix $\alpha_2,\alpha_3$ and $\gamma_2,\gamma_3$ completely. Thus,
we are left with the same equations as before, namely
\begin{eqnarray}
\label{f1eqq}
\alpha_0+\alpha_1\theta &=& 2\pderiv{X^z}{\bar{z}}=2\pderiv{X^z_0}{\bar{z}}+%
2\pderiv{X^z_1}{\bar{z}}\theta\\
\label{f2eqq}
\gamma_0+\gamma_1\theta &=& 2\pderiv{X^\theta}{\bar{z}}=2\pderiv{X^\theta_0}{\bar{z}}+%
2\pderiv{X^\theta_1}{\bar{z}}\theta.
\end{eqnarray}
But now each of the coefficients of the expansions (\ref{exp2}) may contribute obstructions. The equation
\begin{equation}
\alpha_0=2\pderiv{X^z_0}{\bar{z}}
\end{equation}
now actually expands into a system of $2^{n-1}$ equations, one for each increasingly ordered subset 
$I\subset\{1,\ldots,n\}$ of even cardinality. Each such equation has the form
\begin{equation}
\alpha_{0I}=2\tau_I\pderiv{X^z_{0I}}{\bar{z}}
\end{equation}
and thus each of these component equations contributes the space $H^0(M,K^2)$ as obstructions. Thus,
the full space of obstructions for the $\alpha_0$-equation is
\begin{equation}
\Lambda_{\bar{0}}\otimes H^0(M,K^2).
\end{equation}
By the same arguments, the obstructions to the solution of the $\gamma_1$-equation are
\begin{equation}
\Lambda_{\bar{0}}\otimes H^0(M,K),
\end{equation}
and those to the solution of the $\alpha_1$-equation are
\begin{equation}
\Lambda_{\bar{1}}\otimes H^0(M,K^2\otimes L^{-1}).
\end{equation}
Finally, the obstructions for $\gamma_0$ are, not surprisingly,
\begin{equation}
\Lambda_{\bar{1}}\otimes H^0(M,K\otimes L).
\end{equation}
The Hodge decomposition theorem again assures that we can split off the obstructions, component by component,
as direct summands. Defining a complex super vector space
\begin{equation}
\label{tottang}
V=H^0(M,K^2)\oplus H^0(M,K)\oplus\Pi\left(H^0(M,K^2\otimes L^{-1})\oplus H^0(M,K\otimes L)\right),
\end{equation}
this means that we have decomposed $\hat{T}_J^{int}(\Lambda)$ as
\begin{equation}
\hat{T}_J^{int}(\Lambda)=\olv(\Lambda)\oplus\cl J(\Lambda),
\end{equation}
as was to be shown.
\end{proof}

\subsection{Infinitesimal superconformal automorphisms}
\label{sect:sconfaut}

Before one can take the quotient $\cc(\cm)/\sdiff_0(\cm)$, one has to check whether there are automorphisms
of the supercomplex structure which are homotopic to the identity, i.e., whether there are global
vector fields generating superconformal transformations. In the case of the $\fvect^L(1|1)$-structure,
these would be just the global holomorphic sections of $\ctm$. 
If such vector fields exist, they will form the Lie algebra of a subgroup $\Aut_0(J)\subset\sdiff_0(\cm)$ of 
superconformal automorphisms of $J$, which is the stabilizer subgroup of $J\in\cc(\cm)$. Along
the $\sdiff_0(\cm)$-orbit of $J$, this group changes, but its conjugacy class is preserved.

We directly deduce the following corollary from Lemma \ref{lieder}:

\begin{cor}
Let $\cm$ be a complex $1|1$-dimensional supermanifold, and let $(M,L)$ be the Riemann surface and the 
line bundle which are equivalent to $\cm$ by Prop.~\ref{isosrs}. Then there exists a super vector space 
of infinitesimal automorphisms of the supercomplex structure of complex dimension
\begin{equation}
h^0(\ctm)=\dim_\complexC\,\mathrm{aut}_{\fvect^L(1|1)}=(1|h^0(L\otimes K^{-1})+h^0(L^{-1})).
\end{equation}
\end{cor}
\begin{proof}
By Lemma \ref{lieder}, a vector field $X$ has the property $L_XJ=0$ if and only if 
\begin{eqnarray*}
2\pderiv{X^z}{\bar{z}}=0, &\qquad& 2\pderiv{X^{\theta}}{\bar{z}}=0\\
2\pderiv{X^z}{\bar{\theta}}=0, &\qquad& 2\pderiv{X^{\theta}}{\bar{\theta}}=0.
\end{eqnarray*}
This means that $X^z$ and $X^\theta$ have to be superholomorphic. The coefficient $X^z_0$ is a section of
$K^{-1}$, the sheaf of holomorphic tangent vector fields on $M$. If the genus of $M$ is greater than one,
this sheaf has no global holomorphic sections. The other even contribution to the infinitesimal
automorphisms stems from the coefficient function $X^\theta_1\theta$, which is a section of $L\otimes L^{-1}$.
Indeed, this bundle has holomorphic sections:
\begin{equation}
H^0(M,L\otimes L^{-1})\cong H^0(M,\co_M)=\complexC.
\end{equation}
This is the 1-dimensional even subspace of $\mathrm{aut}_{\fvect^L(1|1)}$. The odd infinitesimal
automorphisms are afforded by $X^z_1\theta$, which contributes $H^0(M,L\otimes K^{-1})$, and by
$X^\theta_0$, which contributes $H^0(M,L^{-1})$.
\end{proof}

For the odd vector fields, one notes that the two spaces of holomorphic sections whose sum makes up
the odd infinitesimal automorphisms cannot be both nontrivial: if $h^0(L^{-1})\geq 0$ then we must
have $\deg(L)\leq 0$. But then $L\otimes K^{-1}$ cannot have any nontrivial holomorphic sections. There
is, in fact, a range of degrees where no infinitesimal odd automorphisms exist at all, namely for
\begin{equation}
\label{gdeg}
0 < \deg(L) < 2g-2.
\end{equation}
Interestingly, if $L$ is a spin bundle, as is the case for a $\fkl(1|1)$-structure, we are precisely in
this range.

Geometrically, the odd infinitesimal automorphisms act as translations. If $h(z)$ is, e.g., a holomorphic
section of $L^{-1}$, then we can generate a $0|1$-parameter subgroup from it which will locally act as
$\exp(\tau h(z)\pderiv{}{\theta})$
(here, $\tau$ is the odd generator of $\Lambda_1$). The effect of this subgroup is a translation 
$\theta\mapsto \theta+\tau h(z)$. Likewise, given a holomorphic section $g(z)$ of $L\otimes K^{-1}$,
we can generate a $0|1$-parameter subgroup which translates $z\mapsto z+\tau g(z)\theta$.

There is another reason to restrict the degree of $L$ to the range (\ref{gdeg}). We have seen that
the super Teichm\"uller space of $\fvect^L(1|1)$-structures (i.e., of supercomplex structures) has,
if it exists, the odd dimension $h^0(M,K^2\otimes L^{-1})+h^0(M,K\otimes L)$. But for line bundles of
certain degrees, these values \emph{depend on the complex structure on} $M$ \cite{ACGH:Geometry}. 
A rough idea of the problem is provided by the Riemann-Roch theorem, which states that
\begin{equation}
h^0(M,L)-h^0(M,K\otimes L^{-1})=\deg(L)-g+1
\end{equation}
holds for a Riemann surface $M$ of genus $g$ and a holomorphic line bundle $L\to M$. Now, if we are 
interested in $h^0(M,L)$,
then this formula is immediately useful only for $\deg(L)<0$ (when there are no nontrivial sections),
and for $\deg(L)> 2g-2$, because then $h^0(M,K\otimes L^{-1})=0$. One might call this range of
degrees the ``topological range'': $h^0(M,L)$ is then exclusively determined by the degree of $L$
and the genus. But for degrees in the range (\ref{gdeg}), there is no way to simply read off
$h^0(M,L)$ from the Riemann-Roch theorem, and indeed, the dimension in this case may depend on
the complex structure of the surface. The dimension can jump when one varies within the moduli space.
It is known that this can happen, for example, for spin bundles \cite{ACGH:Geometry}. 

We must exclude this phenomenon for the bundles $K^2\otimes L^{-1}$ and $K\otimes L$ in order to 
have a chance to define super Teichm\"uller space $\ct_{\fvect^L(1|1)}^{g,d}$ at least as a quotient
of a supermanifold.
Therefore, from now on we restrict the degree of the line bundle $L$ to which
$\co_{\cm,\bar{1}}$ is equivalent to the range (\ref{gdeg}). This also rules out the occurrence of
odd infinitesimal automorphisms. In this range, the tangent space of the slice always has the dimension
$(4g-3|4g-4)$. In fact, one can extend the construction to bundles $L$ of degree 0 if one excludes the
case of trivial bundles, i.e., those isomorphic to $\co_M$ \cite{M:New}. This corresponds to excluding the
zero section of the family of Jacobians $J(V_g)$ (see below). While this case could be handled in theory,
a thorough discussion is outside the scope of this work.

The even automorphisms cannot be disposed of so easily. They act by constant overall 
rescalings of the fibers of $L$. If $M\mapsto c\in\complexC$ 
is such a section, then we can generate a $1|0$-parameter family of automorphisms of the supercomplex
structure whose action is locally given by
\begin{equation}
\exp(c\theta\pderiv{}{\theta})=\sum_{n=0}^\infty \frac{c^n}{n!}\left(\theta\pderiv{}{\theta}\right)^n=%
\mathrm{e}^c\theta\pderiv{}{\theta},
\end{equation}
which simply maps $\theta\mapsto \mathrm{e}^c\theta$. This means that the vector fields 
$c\theta\pderiv{}{\theta}$ generate a subgroup of $\sdiff_0(\cm)$ which is isomorphic to $\complexC^\times$.
The presence of these automorphisms is indeed a severe problem for the construction of the super
Teichm\"uller space of supercomplex structures. In fact, the space $\ct_{\fvect^L(1|1)}^{g,d}$ that
we will construct below is \emph{not} the base of a universal family. The reason is that 
$\ct_{\fvect^L(1|1)}^{g,d}$ will still carry a $\complexC^\times$-action, and the elements of an orbit
of this action all parametrize the same complex structure. The problem is passed along to the
Teichm\"uller space $\ct_{\fkl(1|1)}^g\subset\ct_{\fvect^L(1|1)}^{g,g-1}$, which will carry a
residual $\intZ_2$ action. To see the effect of the $\complexC^\times$-ambiguity, it is best to
look directly at the cocycle defining the complex structure. Let $\{U_\alpha\}_{\alpha\in A}$ be
a cover of the underlying topological space $M$ of a complex supersurface. Then the cocycle defining
the complex structure on $\cp(\Lambda)\times\cm$ is given by transition functions
\begin{eqnarray}
\label{coc1}
z_\beta &=& f_{\alpha\beta}(z_\alpha)+\psi_{\alpha\beta}(z_\alpha)\theta_\alpha\\
\nonumber
\theta_\beta &=& \eta_{\alpha\beta}(z_\alpha)+g_{\alpha\beta}(z_\alpha)\theta_\alpha.
\end{eqnarray}
Here $f$ and $g$ are even sections of $\co_\cm\otimes\Lambda$ on $U_\alpha\cap U_\beta$,
while $\psi$ and $\eta$ are
odd sections. An automorphism $\theta\mapsto \mathrm{e}^c\theta$ rescales both $\theta_\alpha$ and 
$\theta_\beta$ by $\expe^c$, therefore it alters the cocycle to
\begin{eqnarray}
\label{coc2}
z_\beta &=& f_{\alpha\beta}(z_\alpha)+\expe^c\psi_{\alpha\beta}(z_\alpha)\theta_\alpha\\
\nonumber
\theta_\beta &=& \expe^{-c}\eta_{\alpha\beta}(z_\alpha)+g_{\alpha\beta}(z_\alpha)\theta_\alpha.
\end{eqnarray}
The cocycles (\ref{coc1}) and (\ref{coc2}) describe the same supercomplex structure. However, the 
$\complexC^\times$-action does not affect the underlying cocycle consisting of $f$ and $g$ alone, 
which describes
the complex structure of $\cm$. Therefore, the underlying Teichm\"uller space is not affected by
the automorphisms $\theta\mapsto\expe^c\theta$, and will turn out to be a complex manifold. For the
odd parameters of a deformation, we have found the following.

\begin{thm}
\label{vectaut}
Let $\cm$ be a complex $1|1$-dimensional supermanifold, and let $(M,L)$ be the Riemann surface and
the holomorphic line bundle to which $\cm$ is equivalent.
Define a $\complexC^\times$-action on the odd deformations of the complex structure $J$ of $\cm$ by
\begin{eqnarray}
\nonumber
\complexC^\times\times (H^0(M,K^2\otimes L^{-1})\oplus H^0(M,K\otimes L)) &\to&%
H^0(M,K^2\otimes L^{-1})\oplus H^0(M,K\otimes L)\\
\label{ctimesact}
(\alpha,v,w) &\mapsto& (\alpha v,\alpha^{-1}w).
\end{eqnarray}
Then all points of an orbit of this $\complexC^\times$-action parametrize the same deformation of $J$.
\end{thm}
\begin{proof}
Two deformations are equivalent if the complex structures $J',J''$ produced by them are the same, i.e.,
if they are related by an automorphism. A complex structure $J$ on $\cm$ is given by a cocycle of the form
\begin{eqnarray}
z_\beta &=& f(z_\alpha)\\
\nonumber
\theta_\beta &=& g(z_\alpha)\theta_\alpha,
\end{eqnarray}
for an open cover $\{U_\alpha\}_{\alpha\in A}$ of $M$ and local complex coordinates $(z_\alpha,\theta_\alpha)$
on $U_\alpha$. An odd deformation of $J$ produces a complex structure with additional terms $\psi,\eta$
of the form (\ref{coc1}). The collection $\{\psi_{\alpha\beta}\}$ defines an element $\tilde{v}$ of
$H^1(M,K^{-1}\otimes L)$, and likewise $\{\eta_{\alpha\beta}\}$ defines an element $\tilde{w}$ 
of $H^1(M,L^{-1})$. As shown in the preceding discussion, the map
\begin{equation}
(\tilde{v},\tilde{w})\mapsto (\alpha\tilde{v},\alpha^{-1}\tilde{w})
\end{equation}
alters the cocycle, but does not change the complex structure described by it.
Since $H^1(M,K^{-1}\otimes L)\cong H^0(M,K^2\otimes L^{-1})$ and $H^1(M,L^{-1})\cong H^0(M,K\otimes L)$,
this proves the claim.
\end{proof}

This unpleasant fact will prevent $\ct_{\fvect^L(1|1)}^{g,d}$ from being an effective parametrization of
the deformations of marked complex supersurfaces. If one tries to quotient out this $\complexC^\times$-action,
the result cannot be a supermanifold anymore. We will not pursue this problem further in this work. 
However, the situation becomes a bit better in the case of a $\fkl(1|1)$-surface, i.e., a super 
Riemann surface.

\begin{thm}
\label{srsaut}
Let $\cm$ be a $\fkl(1|1)$-surface. Then the $\complexC^\times$-action (\ref{ctimesact}) induces a
$\intZ_2$-action on the odd deformations of the $\fkl(1|1)$-structure given by
\begin{eqnarray}
\nonumber
\intZ_2\times H^0(M,K^2\otimes L^{-1}) &\to&%
H^0(M,K^2\otimes L^{-1})\\
\label{z2act}
(n,v) &\mapsto& nv.
\end{eqnarray}
The deformations $v$ and $nv$ describe the same $\fkl(1|1)$-structure.
\end{thm}
\begin{proof}
Let
\begin{eqnarray}
z_\beta &=& f(z_\alpha)\\
\nonumber
\theta_\beta &=& g(z_\alpha)\theta_\alpha,
\end{eqnarray}
be the cocycle defining the $\fkl(1|1)$-structure on $\cm$. An odd deformation brings it into the form
(\ref{coc1}), but the requirement that it remain a $\fkl(1|1)$-structure restricts the transition functions.
We have to make sure that $D=\pderiv{}{\theta}+\theta\pderiv{}{z}$ is preserved up to an invertible
factor, but (\ref{coc1}) entails
\begin{equation}
D_\alpha=(D_\alpha\theta_\beta)D_\beta+(D_\alpha z_\beta-\theta_\beta D_\alpha \theta_\beta)D^2_\beta.
\end{equation}
So we require $D_\alpha z_\beta=\theta_\beta D_\alpha \theta_\beta$, which is equivalent to
\begin{eqnarray}
\label{rest1}
\psi_{\alpha\beta} &=& g_{\alpha\beta}\eta_{\alpha\beta}\\
\label{rest2}
g_{\alpha\beta}^2 &=& f'_{\alpha\beta}+\psi_{\alpha\beta}\psi'_{\alpha\beta}.
\end{eqnarray}
The map $\theta\mapsto\expe^c\theta$ is only an automorphism of the $\fkl(1|1)$-structure
if (\ref{rest1}) and (\ref{rest2}) remain satisfied. This enforces
\begin{equation}
\expe^c=\pm 1.
\end{equation}
This means that the odd deformations of a $\fkl(1|1)$-structure are effective up this $\intZ_2$ ambiguity.
\end{proof}

So the space $\ct_{\fkl(1|1)}^g$ that we will construct below cannot be the base of a universal family either.
In this case, however, one could define a ``superorbifold'' (as is done, e.g., in \cite{LBR:Moduli}) by
quotienting out the residual automorphisms. We will leave the investigation of the resulting
superorbifold for later work and content ourselves with the space $\ct_{\fkl(1|1)}^g$.

\section{Earle's family of Jacobians over Teichm\"uller space}
\label{sect:earle}

The Teichm\"uller space of complex compact closed supersurfaces of dimension $1|1$ will turn out
not to be a supermanifold because of the presence of nontrivial infinitesimal automorphisms.
Its underlying space, however, is a complex manifold, namely
the fiber space $J(V_g)\to\ct_g$ of Jacobian varieties over (ordinary) Teichm\"uller space $\ct_g$. The
space $J(V_g)$ was introduced by Earle \cite{E:Families} along with an embedding of the universal
Teichm\"uller curve into it. The universal Teichm\"uller curve is universal among all 
embeddings of families of Riemann surfaces
into families of compact complex manifolds over the same base. We will briefly review this construction
in the next section, before proceeding to construct $\ct_{\fvect^L(1|1)}^{g,d}$.

\subsection{The universal Teichm\"uller curve}

Let $\ct_g$, $g\geq 2$, be the Teichm\"uller space of compact Riemann surfaces of 
genus $g$, and let $\Gamma$ be
the fundamental group of a smooth surface of genus $g$. Then the Bers fiber space \cite{B:Fiber}
$F_g\subset\ct_g\times\complexC$ is a subspace of $\ct_g\times\complexC$ with the following properties:
\begin{enumerate}
\item $\Gamma$ acts freely, properly discontinuously and fiberwise on $F_g$ as a group of biholomorphic maps
\begin{equation}
\gamma(t,z)=(t,\gamma^t(z))\qquad\forall \gamma\in\Gamma,(t,z)\in F_g,
\end{equation}
such that for every $t\in\ct_g$, $\gamma^t$ is a M\"obius transformation,
\item the fiber $U(t)$ of $F_g$ over $t$ is simply connected, and $U(t)/\Gamma$ is the Riemann surface
represented by $t$,
\item the projection map $\ct_g\times\complexC\to\ct_g$ descends to a well-defined holomorphic map
$\pi_0:V_g:=F_g/\Gamma\to\ct_g$ defining a holomorphic family of Riemann surfaces.
\end{enumerate}
The family $\pi_0:V_g\to\ct_g$ is called the universal Teichm\"uller curve. One can express the action of
$\Gamma$ on $F_g$ by the choice of $t$-dependent elements of $PSL(2,\complexC)$. These matrices will vary
holomorphically with $t$, since $\Gamma$ is supposed to act biholomorphically on the 
total space $F_g$.
Only loxodromic\footnote{A M\"obius transformation is called loxodromic if it is conjugate in
$PSL(2,\complexC)$ to the 
matrix $M=\mathrm{diag}(\lambda,\lambda^{-1})$ with $|\lambda|\neq 1$. Hyperbolic transformations
are a special case of loxodromic ones, namely those for which $\mathrm{tr}(M^2)=\lambda^2+\lambda^{-2}>4$.} 
elements of $PSL(2,\complexC)$ will occur, since otherwise, the quotients would not be compact.

By the Riemann mapping theorem, each $U(t)$ is biholomorphic to the upper half-plane $\mathbbm{H}$.
This is just the statement of the uniformization theorem: each compact Riemann surface of genus $g\geq 2$
has $\mathbbm{H}$ as its universal covering space.
One can then describe the action of $\Gamma$ as the action of a finitely generated discrete subgroup
of $\mathrm{Isom}(\mathbbm{H})=PSL(2,\realR)$. It is, however, impossible to uniformize all fibers of
$F_g$ simultaneously by $\mathbbm{H}$ if one wants to keep the action of $\Gamma$ holomorphic in $t$.

\subsection{The basis of normalized differentials}

A set of $2g$ canonical generators for $\Gamma$ is a set $\{A_1,\ldots,A_g,B_1,\ldots,B_g\}\subset\Gamma$ 
such that
\begin{equation}
\prod_{i=1}^g A_iB_iA_i^{-1}B_i^{-1}=1,
\end{equation}
holds, and all other relations in $\Gamma$ are consequences of this single one. On the quotient
$\Sigma_t=U(t)/\Gamma$, these generators are represented by nontrivial loops on the Riemann surface which
generate the fundamental group $\pi_1(\Sigma_t)$. These loops have intersection numbers
\begin{equation}
A_j\cdot A_k=B_j\cdot B_k=0,\qquad A_j\cdot B_k=\delta_{jk}\qquad 1\leq j,k\leq g
\end{equation}
(cf., e.g., \cite{J:Compact} for a geometric treatment of the intersection pairing). 
Bers \cite{B:Holomorphic} 
has shown that one can construct a family of $g$ linearly independent holomorphic 
functions $\alpha_j(t,z)$ on $F_g$ which satisfy
\begin{equation}
\label{treq}
\alpha_j(t,z)=\alpha_j(\gamma(t,z))\pderiv{\gamma}{z}(t,z)\qquad\forall \gamma\in\Gamma,
\end{equation}
and
\begin{equation}
\label{integ}
\int_{z}^{A_k^t(z)}\alpha_j(t,w)dw=\delta_{jk}\qquad\forall t\in\ct_g,z\in U(t),1\leq j,k\leq g.
\end{equation}
The integral in (\ref{integ}) may be computed along any path in $U(t)$. Equation (\ref{treq}) just means
that the functions $\alpha_j$ will descend to holomorphic 1-forms on the quotient $\Sigma_t$,
while equation (\ref{integ}) states that these 1-forms are the basis of $\Omega^{1,0}(\Sigma_t)$ 
dual to the loops $A_1,\ldots,A_g$. It is a classical fact that these requirements determine the
functions $\alpha_j$ uniquely \cite{J:Compact}. We can then calculate the Riemann period matrix associated
with our chosen basis of loops $A_1,\ldots,B_g$:
\begin{equation}
\label{periodm}
\tau_{ij}(t):=\int_{z}^{B_i^t(z)}\alpha_j(t,w)dw,
\end{equation}
where $z$ is just any point in $U(t)$. Since both $B_i$ and $\alpha_j$ are holomorphic with respect to $t$,
the entries of the period matrix are as well.

\subsection{The family of Jacobians}

It is a classical fact that the $g$ column vectors $\tau_i$, $1\leq i\leq g$ of the period matrix and the
$g$ standard basis vectors $e_1,\ldots,e_g$ of $\complexC^g$ are linearly independent over 
the reals \cite{J:Compact}.
They form a lattice subgroup $(\One,\tau)$ in $\complexC^g$, and the Jacobian variety of a
Riemann surface $\Sigma$ is a $g$-dimensional torus defined as
\begin{equation}
\Jac(\Sigma):=\complexC^g/(\One,\tau).
\end{equation}
It is Earle's result \cite{E:Families} that this construction can be carried out holomorphically on the
entire universal curve.

\begin{thm}[Earle]
The $\intZ^{g}\times\intZ^g$-action on $\ct_g\times\complexC^g$ given by
\begin{eqnarray}
\Lambda:\intZ^g\times\intZ^g\times(\ct_g\times\complexC) &\to& \ct_g\times\complexC\\
(m,n,t,z) &\mapsto& (t,z+m+\tau(t)n)
\end{eqnarray}
is a free, properly discontinuous action by biholomorphisms. The projection $\ct_g\times\complexC^g\to\ct_g$
descends to a holomorphic family $\pi:J(V_g):=(\ct_g\times\complexC^g)/\Lambda\to\ct_g$ of complex tori.
The family $J(V_g)$ is a topologically trivial fibration, and each fiber $\pi^{-1}(t)$ is canonically
isomorphic to $\Jac(\Sigma_t)$, the Jacobian of the corresponding Riemann surface.
\end{thm}

\subsection{The Picard variety}

The isomorphism classes of holomorphic line bundles on a Riemann surface $\Sigma$ form an abelian group,
the Picard group $\Pic(\Sigma)$, with group operation the tensor product of bundles. This group can 
be written in a variety of ways. First of all, one can
view $\Pic(\Sigma)$ as the group of linear equivalence classes of divisors on $\Sigma$. From this point
of view, the Picard variety inherits naturally the degree map of divisors
\begin{equation}
\deg:\Pic(\Sigma)\to\intZ,
\end{equation}
which is easily seen to be a homomorphism of abelian groups. On the other hand, the set of isomorphism
classes of line bundles is naturally isomorphic to the first \v Cech cohomology group 
$H^1(\Sigma,\co^*)$, with $\co^*$ the sheaf
of invertible holomorphic functions on $\Sigma$. The two pictures can be connected using the
exponential sequence of sheaves
\begin{equation}
0 \longrightarrow \underline{\intZ}\stackrel{\iota}{\longrightarrow} \co%
\stackrel{\exp}{\longrightarrow} \co^*\longrightarrow 0,
\end{equation}
where $\underline{\intZ}$ is the locally constant sheaf with values in $\intZ$ 
and $\exp$ is the map which sends each germ $f$ to $\exp(2\pi if)$. This sequence induces a 
long exact sequence in cohomology, of which a part reads
\begin{equation}
\ldots\to H^1(\Sigma,\underline{\intZ})\longrightarrow H^1(\Sigma,\co)\longrightarrow%
H^1(\Sigma,\co^*)\stackrel{c}{\longrightarrow} H^2(\Sigma,\underline{\intZ})\to\ldots
\end{equation}
Since $\Sigma$ is connected, we have $H^2(\Sigma,\underline{\intZ})\cong\intZ$.
The connecting homomorphism $c$ sends a line bundle to its Chern class, so in our case it is
simply the degree map. Hence, we can conclude that for the line bundles of degree zero, we have
\begin{equation}
\Pic^0(\Sigma)\cong \exp(H^1(\Sigma,\co)/H^1(\Sigma,\underline{\intZ})).
\end{equation}
Line bundles of degree zero form a normal subgroup of $\Pic(\Sigma)$ which can be identified with
the Jacobian via the Abel map. Choose $p_0\in\Sigma$. Then the Abel map is given by
\begin{eqnarray}
\ca_{p_0}:\Sigma &\to& \Jac(\Sigma)\\
\ca_{p_0}(p) &=& \left(\int_{p_0}^p\alpha_1,\ldots,\int_{p_0}^p\alpha_g\right)
\end{eqnarray}
This mapping depends on the base point $p_0$ and on our choice of a basis $A_1,\ldots,B_g$
(see above) for the first
homology of $\Sigma$.\footnote{Actually we had introduced $A_1,\ldots,B_g$ as the generators
of $\pi_1(\Sigma)$. But by a theorem of van Kampen (see, e.g., \cite{J:Compact}) one has
$H_1(\Sigma,\intZ)\cong(\pi_1(\Sigma)/[,])$, where $[,]$ denotes the commutator subgroup of
$\pi_1(\Sigma)$. Thus $H_1$ is isomorphic to the free abelian group generated by
$A_1\ldots,B_g$.}
The 1-forms being integrated over in the Abel map are those dual to the $A_i$.
The Abel map is an embedding and induces a homomorphism
\begin{equation}
\ca_{p_0}:\mathrm{Div}(\Sigma)\to\Jac(\Sigma),
\end{equation}
where $\mathrm{Div}$ is the abelian group of divisors on $\Sigma$. For divisors of degree 0, the 
Abel map is independent of the base point. Moreover, Abel's theorem asserts that
\begin{equation}
\ca_{p_0}(D)=0\quad\Leftrightarrow\quad D\in\mathrm{PDiv(\Sigma)},
\end{equation}
where $\mathrm{PDiv(\Sigma)}$ is the space of principal divisors (those which can occur as
the divisors of meromorphic functions). This shows that one can let $\ca_{p_0}$ descend to a map
on the group of linear equivalence classes of divisors, and that it is injective there. In addition,
the Jacobi inversion theorem states that $\ca_{p_0}$ is surjective. Therefore, we have
the sequence
\begin{equation}
\ca_{p_0}:\Jac(\Sigma)\cong(\mathrm{Div(\Sigma)}/\sim)\cong\Pic^0(\Sigma),
\end{equation}
of isomorphisms of abelian varieties,
where $\sim$ denotes linear equivalence. Since $\deg:\Pic\to\intZ$ is a homomorphism, we see that the
set of line bundles of degree $d$ can be obtained as the coset
\begin{equation}
\Pic^d(\Sigma)\cong L \otimes\Pic^0(\Sigma)
\end{equation}
of any line bundle $L$ of degree $d$. None of these cosets except for $d=0$ is a group, but each 
one is isomorphic as a complex manifold to $\Pic^0(\Sigma)\cong\Jac(\Sigma)$
\cite{J:Compact}, \cite{GH:Principles}. 

\subsection{The embedding of $V_g$ into $J(V_g)$}

With the family $J(V_g)\to\ct_g$ at hand, one can ask oneself whether it is possible to extend 
the Abel embedding 
$\ca_{p_0}:\Sigma\to\Jac(\Sigma)$ to the entire family $J(V_g)$. The problem here is that one cannot
choose a point $p_0(t)$ on $\Sigma(t)$ which varies holomorphically with the coordinates $t$ of
$\ct_g$: the universal curve $V_g$ has no holomorphic sections if $g\geq 3$ \cite{H:Sur}.

As before,
let $A_1,\ldots,B_g$ denote a canonical basis of $H_1(\Sigma)$, and let $\omega_1,\ldots,\omega_g$
be the dual basis of holomorphic 1-forms on $\Sigma$. For any base point $p_0\in\Sigma$, we define 
the vector $K(p_0)$ of Riemann constants as the point of $\Jac(\Sigma)$ whose $j$-th component is
\begin{equation}
K_j(p_0)=-\frac{\tau_{jj}}{2}+\sum_{k\neq j}\int_{A_j}\omega_j(w)\int_{p_0}^w\omega_k(s),
\end{equation}
where $\tau$ is the period matrix (\ref{periodm}) of $\Sigma$. The vector $K$ plays an important role
in the study of theta functions of Riemann surfaces, as well as for spin structures, as will be shown below.
It also plays a key role in the construction of an embedding $V_g\to J(V_g)$ as Earle has shown.

\begin{thm}[Earle \cite{E:Families}]
The map $\eta:F_g\to\complexC^g$ defined by
\begin{equation}
\eta_j(t,z):=\frac{1}{1-g}%
\left(\sum_{k\neq j}\int_{z}^{A_j(z)}\alpha_j(w)\left(\int_{z}^{w}\alpha_k(u)du\right)dw-%
\frac{\tau_{jj}}{2}\right)
\end{equation}
is holomorphic and satisfies
\begin{equation}
\label{transfeta}
\pderiv{\eta_j}{z}=\alpha_j,\qquad \eta(A_k(t,z))=\eta(t,z)+e_k,\qquad \eta(B_k(t,z))=\eta(t,z)+\tau(t)e_k
\end{equation}
for $j,k=1,\ldots,g$. Here $e_k$ is the $k$-th unit vector of $\complexC^g$.
\end{thm}

The map $\eta$ just assigns to every point $(t,z)$ the (normalized) vector of Riemann constants 
$\frac{1}{1-g}K(t,z)$ relative to this point. The properties (\ref{transfeta}) show that $\eta$ descends
to a map $\psi:V_g\to J(V_g)$, the desired embedding.

\begin{thm}[Earle \cite{E:Families}]
\label{thm:earle}
Let $\pi:J(V_g)\to\ct_g$ and $\pi_0:V_g\to\ct_g$ be the family of Jacobians and the universal
Teichm\"uller curve, respectively. Then the map $\psi:V_g\to J(V_g)$ defined by the map $\eta$ above is
a holomorphic embedding such that $\pi_0=\pi\circ\psi$.
\end{thm}

The Abel map with base point $p_0$ can be viewed as a map $\Sigma\to\mathrm{Div}^0(\Sigma)$ which 
assigns to each $p\in\Sigma$ the divisor class $[p-p_0]$. Conversely, if we have any divisor $D$ of degree
1 on $\Sigma$, then the map $p\mapsto [p-D]$ is an embedding $\Sigma\to\Jac(\Sigma)$, and all translates
of the Abel map can be obtained this way. Therefore, Thm.~\ref{thm:earle} defines a divisor class $[D(t)]$
on every fiber $\Sigma(t)$ of $V_g$, and this class varies holomorphically in $t$. The class
$[(g-1)D(t)]$ is the vector of Riemann constants, by construction. Moreover, a class of
degree 1 and the Jacobian variety are sufficient to reproduce all equivalence classes of divisors, and 
therefore the entire Picard variety.

\subsection{Spin bundles on Riemann surfaces}

It was already mentioned in Chapter \ref{ch:sconf} that a spin bundle on a Riemann surface $\Sigma$ is 
just a square root of the canonical bundle $K$, and that the number of inequivalent spin bundles 
on $\Sigma$ is given by $\big| H^1(\Sigma,\intZ_2)\big|=2^{2g}$. This already makes it clear that
the divisors belonging to spin bundles form a discrete subset of $\Jac(\Sigma)$, and that there is no
nontrivial way to continuously deform a spin bundle in such a way that it remains a spin bundle 
throughout the deformation. 

The distribution of spin bundle divisors on $\Jac(\Sigma)$ can be easily determined with the help of
the following well-known theorem.

\begin{thm}
Let $C\in\Jac(\Sigma)$ denote the divisor class of the canonical bundle, and let $K(p_0)$ be the vector
of Riemann constants relative to the point $p_0\in\Sigma$. Then
\begin{equation}
2K(p_0)=C\qquad\textrm{for all } p_0\in\Sigma.
\end{equation}
\end{thm}

Since for a spin bundle $S$, we have $K=S\otimes S$, we must also have $2[S]=C$ for the divisor class $[S]$
of $S$. Therefore the set of points of $\Jac(\Sigma)$ which belong to divisors of spin bundles
is a translation of the half-period lattice
\begin{equation}
\label{hplat}
\Delta=\alpha+\tau\cdot\beta,\qquad \alpha,\beta\in\frac{1}{2}\intZ^g
\end{equation}
($\tau$ the period matrix) by the vector of Riemann constants. The elements of 
$\Delta$ are also called theta characteristics or
points of 2-torsion of $\Jac(\Sigma)$. One distinguishes odd and even spin structures by the parity of
the integer $4(\alpha\cdot\beta)=4\sum\alpha_k\beta_k$. 

The number of holomorphic sections of a
spin bundle on a Riemann surface is, in general, difficult to determine. The Riemann-Roch theorem is not of
much use for spin bundles: inserting a square root $S$ of $K$, one can only conclude that
$\deg(S)=g-1$, which was clear anyway. Moreover, spin bundles lie in the ``non-topological'' range
(\ref{gdeg}), and one can indeed show that the number $h^0(\Sigma,S)$ depends on the particular
complex structure $\Sigma$, i.e., it can jump when one runs along a path in Teichm\"uller space
\cite{ACGH:Geometry}.
At least for odd spin structures, a holomorphic section is guaranteed since one can show that
$h^0(\Sigma,S)\equiv$ parity of $S\mod\,2$.

Earle's embedding $V_g\to J(V_g)$ makes it possible to keep track of a particular spin structure
along the entire universal curve, since the vector of Riemann constants as the reference point is
given as a global holomorphic section of $J(V_g)$. It therefore makes sense to compare spin
structures on different fibers of the universal curve. One says that two spin structures on
two respective fibers $\Sigma_1,\Sigma_2\subset V_g$ are equal, if they correspond to the 
same theta characteristic
relative to the class $[(g-1)D(t)]$ which represents the vector of Riemann constants on each fiber. 

\begin{prop}
\label{spstr}
Let $\cs\to V_g$ be a holomorphic line bundle over the universal curve $V_g$ such that each restriction
$\cs(t)$
to a fiber $\Sigma(t)=\pi^{-1}(t)$ of $V_g$ is a spin bundle on $\Sigma(t)$. Then each $\cs(t)$ has the
same spin structure.
\end{prop}
\begin{proof}
Earle's embedding gives us a global section $[(g-1)D(t)]$ of $J(V_g)$ which can be interpreted
as the image under the Abel map of the divisor of degree $g-1$ corresponding to the vector of 
Riemann constants. Furthermore, the half-period lattice (\ref{hplat}) is given as $\Delta(t)$ on every 
fiber and varies holomorphically with $t$, since the period matrix $\tau$ is holomorphic in $t$. Therefore,
we can globally define the translated lattice $[(g-1)D(t)]+\Delta(t)$ on $V_g$, which consists of a set
of $2^{2g}$ holomorphic sections of $J(V_g)$ describing the spin structures of $\Sigma(t)$.
It is clear that the divisor class of $\cs(t)$ has to vary holomorphically in $t$ if $\cs$ is holomorphic
in $t$. Therefore, this class corresponds to one of the sections of $J(V_g)$ defined by 
$[(g-1)D(t)]+\Delta(t)$.
\end{proof}

The same does not hold anymore if one allows diffeomorphisms of $\Sigma$ which are not homotopic to
$\id_\Sigma$. Then the spin structures can get interchanged, but their parity remains preserved
\cite{ACGH:Geometry}.
The moduli space of spin curves therefore consists of two connected
components \cite{C:Moduli}. But in this work, we will confine ourselves to the construction of
Teichm\"uller spaces, where these problems do not arise and Prop.~\ref{spstr} holds. For a very detailed
account on the moduli of spin curves and theta characteristics see \cite{C:Moduli}, \cite{ACGH:Geometry}.

\subsection{The Teichm\"uller space of pairs of Riemann surfaces and line bundles}
\label{sect:modpair}

The discussion of the previous section can be summed up by saying that the Jacobian $\Jac(\Sigma)$ of a given
Riemann surface $\Sigma$ can be viewed as the moduli space of line bundles of some fixed degree $d$ on
$\Sigma$: it effectively parametrizes all isomorphism classes. The existence of Earle's family of
Jacobians allows one to extend this to pairs of complex structures and holomorphic line bundles. 

Let $M$ be a given smooth compact closed oriented surface $M$ together with
a smooth orientable vector bundle $p:E\to M$ of rank 2 and degree $d$. Denote by $\cc^{g,d}$ the set of pairs
$(J_E,J_M)$ of almost complex structures on $E$ and $M$, respectively, such that
\begin{enumerate}
\item $J_E$ is integrable, and
\item the bundle projection $p$ becomes a holomorphic map with respect to $J_E$ and $J_M$.
\end{enumerate}
The space of these pairs is obviously the space of pairs $(M,L)$ consisting of a Riemann surface of
genus $g$ and a holomorphic line bundle $p:L\to M$ of degree $d$. On $\cc^{g,d}$, we have an action of 
the group
$\Aut_{\catvbun}(M,E)$ of automorphisms of $E$ as a smooth vector bundle (see Definition \ref{def:autvb}). 
We recall that this group splits as a semidirect product
\begin{equation}
\Aut_{\catvbun}(M,E)=\mathrm{Diff}(M)\ltimes\Aut_{\catvbun(M)}(E),
\end{equation}
where $\Aut_{\catvbun(M)}(E)$ is the group of automorphisms of $E$ over $M$, i.e. the subgroup of
$\Aut_{\catvbun}(M,E)$ whose elements act as the identity on $M$.

The elements of $\mathrm{Diff}(M)$ act on $\cc^{g,d}$ by pullback of the entire bundle along with $J_E$ and
$J_M$. The elements of $\Aut_{\catvbun(M)}(E)$ preserve $J_M$ and act only by real linear automorphisms
on the fibers of $E$.
Quotienting out the acion of $\Aut_{\catvbun}(M,E)$ would produce the moduli space of pairs $(M,L)$, which
is too ambitious for our purposes.
Let us instead restrict the underlying diffeomorphisms to the group $\mathrm{Diff}_0(M)$, 
i.e. to those homotopic to the identity. This gives us the group $\Aut_{\catvbun}(M,E)_0$. The quotient of
$\cc^{g,d}$ by the action of this group is the Teichm\"uller space of pairs $(M,L)$.

\begin{prop}
\label{cor:modpair}
The family $J(V_g)\to\ct_g$ of Jacobians is the Teichm\"uller space of pairs $(M,L)$ consisting of
a Riemann surface of genus $g$ and a holomorphic line bundle $L\to M$ of degree $d$. 
\end{prop}
\begin{proof}
Given the space $\cc^{g,d}$, we can quotient out the action of $\Aut_{\catvbun(M)}(E)$ to obtain the
space $J(\cc^{g,d})$ of pairs $(M,[L])$ of complex structures on $M$ and equivalence classes of holomorphic
line bundles of degree $d$ on $M$. This space then carries an action of
\begin{equation}
\Aut_{\catvbun}(M,E)_0/\Aut_{\catvbun(M)}(E)\cong\mathrm{Diff}_0(M).
\end{equation}
The quotienting out of this action produces the space of pairs $([M]_{\mathrm{Diff}_0(M)},[L])$ of
$\mathrm{Diff}_0(M)$-orbits of complex structures on $M$ and isomorphism classes of holomorphic line
bundles. By Earle's result, this space can be given a natural complex structure, namely that of the
family $J(V_g)$.
\end{proof}

\section{\texorpdfstring{The super Teichm\"uller space $\ct_{\fvect^L(1|1)}^{g,d}$}
{The Teichm\"uller space of supercomplex structures in 2|2 real dimensions}}

\subsection{Versal, semiuniversal and universal families}

The Teichm\"uller space $\ct_g$ of Riemann surfaces of genus $g$ is the base of a universal family of
marked Riemann surfaces. This means that any other complex analytic family of marked surfaces can
be obtained by pullback from it \emph{in a unique way}.
The supermanifold $\ct_{\fvect^L(1|1)}^{g,d}$ that will be constructed in this section is, however, 
not the base
of a universal family of complex $1|1$-dimensional supermanifolds. This is due to the existence
of automorphisms of these structures which cannot be divided out without destroying the supermanifold
structure of the base (cf.~Thm.~\ref{vectaut}). Before proceeding to the construction of
$\ct_{\fvect^L(1|1)}^{g,d}$, we want to introduce some basic terminology of deformation theory and 
argue that $\ct_{\fvect^L(1|1)}^{g,d}$
can still be regarded as the super Teichm\"uller space of $\fvect^L(1|1)$-structures.
For the sake of clarity, we look at ordinary complex manifolds first. For detailed accounts, consult,
e.g., \cite{M:Lectures}, \cite{K:Complex}, \cite{SU:Advances}.

Let $\pi:M\to B$ be a complex analytic family of compact complex manifolds.
This means that $\pi$ is a proper
holomorphic map and $\pi_*:T_pM\to T_{\pi(p)}B$ is surjective at every $p\in M$. For every $b\in B$,
such a family is also called a \emph{deformation} of the fiber $M_b:=\pi^{-1}(b)$. Assume for
simplicity that $B$ is a polydisc (since we are only interested in local properties around $b\in B$,
this is not a restriction). We denote by $\Gamma_{hol}(B,TB^{1,0})$ the holomorphic sections of the 
holomorphic tangent bundle of the base and by $\Gamma(M,TM^{1,0})$ the smooth holomorphic vector fields on
$M$. The \emph{Kodaira-Spencer map} $\ck S_\pi$ is defined as
\begin{eqnarray}
\ck\cs_\pi:\Gamma_{hol}(B,TB^{1,0}) &\to& H^1(M,T(M/B)^{1,0})\\
\gamma &\mapsto& [\bar{\partial}\eta],
\end{eqnarray}
where $\eta\in\Gamma(M,TM^{1,0})$ is any vector field such that $\pi_*\eta=\gamma$ and
$T(M/B):=\ker\pi_*$ denotes the relative tangent bundle. It is easy to show
that this map is well-defined, i.e., does not depend on the chosen lift $\eta$ of $\gamma$. The
map $\ck\cs_\pi$ induces a map
\begin{equation}
KS_{\pi,b}:T_bB^{1,0}\to H^1(M_b,TM_b^{1,0})
\end{equation}
which one obtains from the map that $\ck\cs_\pi$ induces on the stalks by dividing out the ideal of
germs of vector fields vanishing at $b$. This map is also called the Kodaira-Spencer map at $b$. The
Kodaira-Spencer mapping plays a central role in deformation theory. Originally it was invented to
study families of complex manifolds, but one can already see from the above definitions that it is
possible to apply it in a variety of other contexts. Vaintrob \cite{V:Deformations} has developed a
Kodaira-Spencer theory for complex superspaces and showed its usefulness in several examples.
One can interpret $KS_{\pi,b}$ as a means to measure how much the complex structure on the fibers of
the family varies in a neighbourhood of $b$. Geometrically speaking, the cohomology group 
$H^1(M_b,TM_b^{1,0})$ describes the possible variations of the transition functions of $M_b$ which
cannot be induced by mere coordinate changes, i.e., those which can define a new complex
structure which is not biholomorphic to the one on $M_b$.

The map $KS_{\pi,b}$ is also used to classify families in the sense of ``how well'' they parametrize the
objects in question (in the above definition: compact complex manifolds). For this we need to have
a look at the behaviour of $KS_\pi$ under pullback. Let $\phi:C\to B$ be a holomorphic map and let
$b=\phi(c)$. We obtain a pullback family $M\times_B C$ whose projections we denote as
\begin{equation}
\begin{CD}
M\times_B C @>\hat{\phi}>> M\\
@V{\hat{\pi}}VV @VV{\pi}V\\
C @>\phi>> B
\end{CD}\qquad.
\end{equation}
Then one finds (see, e.g., \cite{M:Lectures}):

\begin{thm}
We have
\begin{equation}
KS_{\hat{\pi},c}=KS_{\pi,b}\circ\phi_*:T_cC\to H^1(M_b,TM_b^{1,0}).
\end{equation}
\end{thm}

\begin{dfn}
Let $\pi:M\to B$ be a complex analytic family of compact complex manifolds and let $b\in B$ be given
together with the Kodaira-Spencer map $KS_{\pi,b}:T_bB\to H^1(M_b,TM_b^{1,0})$. Then the family $\pi$ is
called
\begin{itemize}
\item \emph{versal} at $b$ if $KS_{\pi,b}$ is surjective, and for every family 
$p:N\to C$ around $c\in C$ there exists
a morphism $\phi:C\to B$ with $\phi(c)=b$ such that the pullback $M\times_B C$ is isomorphic to $p:N\to C$,
\item \emph{semiuniversal} at $b$ if it is versal and $KS_{\pi,b}$ is a bijection, and
\item \emph{universal} at $b$ if $KS_{\pi,b}$ is a bijection and if for every family $p:N\to C$ 
around $c\in C$ there
exists a unique morphism $\phi:C\to B$ with $\phi(c)=b$ such that the pullback $M\times_B C$ 
is isomorphic to $p:N\to C$.
\end{itemize} 
\end{dfn}

Semiuniversal families are also called \emph{Kuranishi} families. The interpretation of this classification
is quite clear: a universal family parametrizes objects ``optimally'', in the sense that it contains all
possible deformations of the fiber $M_b$ and that every family which contains a fiber isomorphic to
$M_b$ can be pulled back from the universal one in a unique way.
The Teichm\"uller space of Riemann surfaces is the base of such a universal family for marked Riemann
surfaces. Teichm\"uller space does not, however, universally parametrize unmarked families of Riemann
surfaces. Without a marking, a Riemann surface may possess an additional discrete group of automorphisms,
and in this case there is no unique way to obtain a family containing it as a pullback
from the family over Teichm\"uller space. But it is still semiuniversal: it is possible to
pull back any family from it, albeit non-uniquely, and the tangent space at Teichm\"uller space can
naturally be identified with $H^1(M,TM^{1,0})\cong H^1(M,K)\cong H^0(M,K^2)$ (if $M$ is a Riemann surface).
The moduli space of Riemann surfaces is the quotient of Teichm\"uller space by the mapping class group, whose
elements represent these additional possible automorphisms. It is universal in the sense that it parametrizes
the deformations ``optimally'', but the price is that it is not a manifold anymore. Therefore, a
universal family of Riemann surfaces of genus $g\geq 2$ does not exist (at least not in the category
of complex manifolds, but in the larger category of stacks it does).
Even semiuniversal families do not, in general, exist. A versal family still contains all 
deformations of the fiber $M_b$, but may contain redundancies. Versal families are also called 
\emph{complete}.

\subsection{Construction of $\ct_{\fvect^L(1|1)}^{g,d}$}

The result of Section \ref{sect:modpair} is already almost the answer for the underlying Teichm\"uller
space of $\fvect^L(1|1)$-structures. A smooth supermanifold of dimension $2|2$ is, since smooth supermanifolds
are always split, isomorphic to a supermanifold of the form $\cm=(M,\wedge E)$, where $E$ is a smooth
vector bundle of rank 2 over the smooth underlying manifold $M$. We assume that $M$ is compact, closed and
orientable, and also that $E$ is orientable. The genus $g\geq 2$ of $M$ and the degree $d$ of $E$ are
invariants of $\cm$, i.e., they are preserved by superdiffeomorphisms.
A complex supermanifold of dimension
$1|1$, as was shown in Prop.~\ref{isosrs}, is always of the form $(M,\co_M\oplus \Gamma(L))$, where $M$ is a
Riemann surface and $L$ is a holomorphic line bundle over $M$.
Since there is only one odd coordinate available, it always remains in this exterior bundle form. 
The underlying 
smooth supermanifold, however, may deviate from the form $(M,\wedge^\bullet E)$:
its transition functions can contain terms proportional to $\theta_1\theta_2$. So there seems to
be a mismatch: when we switch back from $(M,L)$ to the underlying smooth bundle $(M,\wedge^\bullet E)$, 
we always
obtain the canonical bundle form $\cm^{split}$, but never the nilpotent corrections proportional to 
$\theta_1\theta_2$ that may occur in the smooth supermanifold $\cm$.

The solution to this riddle is that a real form on $\cm$, i.e., an involution of the complex supermanifold
whose fixed point locus is again the underlying real analytic supermanifold \cite{NoS_QFaSACfM}, need 
not express the even complex coordinate $z$ solely as a function of the even real coordinates $x,y$.
It only must produce an even real superfunction, which can involve terms proportional
to $\theta_1\theta_2$ as well.

\begin{thm}
\label{sl1}
The family $\pi:J(V_g)\to\ct_g$ is the underlying manifold of $\ct_{\fvect^L(1|1)}^{g,d}$, i.e., each point
of $J(V_g)$ represents exactly one $\sdiff_0(\cm)(\realR)$-orbit in $\cc(\cm)(\realR)$.
\end{thm}
\begin{proof}
Let the associated exterior bundle of $\cm$ be denoted as $\cm^{split}=(M,\wedge^\bullet E)$, 
where $E$ is a smooth, orientable vector bundle of rank 2 and degree $d$.
By Thm.~\ref{thm:autsman}, the group $\sdiff_0(\cm)(\realR)$ can be written as
\begin{eqnarray}
\sdiff_0(\cm)(\realR) &\cong& (\mathrm{Diff}_0(M)\ltimes\Aut_{\catvbun(M)}(E))\ltimes N_\cm\\
&=& \Aut_{\catvbun}(M,E)_0\ltimes N_\cm\\
&=& \Aut(\cm^{split})_0\ltimes N_\cm.
\end{eqnarray}
Therefore, the theorem will be proven if we can divide out the action of $N_\cm$. Let $J$ be an integrable
almost complex structure on $\cm$. The group $N_\cm$ is generated by vector fields of order $j\geq 2$
in the odd variables (see Section \ref{sect:autm}). So if $n\in N_\cm$ is such a unipotent diffeomorphism 
and $(x,y,\theta_1,\theta_2)$ is a local smooth coordinate system on $\cm$, the action of $n$ can
locally be expressed by the sheaf map
\begin{equation}
n=\exp(\cx)=(\One+g(x,y)\theta_1\theta_2\pderiv{}{x}+h(x,y)\theta_1\theta_2\pderiv{}{y}),
\end{equation}
where $g$ and $h$ are ordinary smooth functions on the base $M$. The action of $n$ on $J$ is therefore
\begin{equation}
n^*J=J+\theta_1\theta_2L_XJ,
\end{equation}
where $X=g(x,y)\pderiv{}{x}+h(x,y)\pderiv{}{y}$. If $L_XJ\equiv 0$ were
possible for some $X\neq 0$, the action would not be free. But this cannot happen: by Prop.~\ref{tcm},
this would imply that $X$ is holomorphic with respect to $J$, i.e., it would be a holomorphic vector field 
with holomorphic coefficients on the
underlying Riemann surface. But for genus $g\geq 2$, such holomorphic sections of $TM^{1,0}$ do not
exist. Therefore, $N_\cm$ acts freely on $\cc(\cm)(\realR)$. Since the action is also infinitesimal, i.e.,
by Lie derivatives, its freeness is sufficient for allowing us to divide it out.
This is done by quotienting $\co_\cm\to\co_\cm/\cj^2$ ($\cj$ is the nilpotent ideal of $\co_\cm$), and
by reducing all modules over $\co_\cm$ correspondingly. This eliminates all tensor fields whose
coefficient functions are of degree greater than one in the odd variables.
As was shown in Section \ref{sect:explmor}, this reduces $\cm$ to $\cm^{split}=(M,\wedge^\bullet E)$.
The remaining integrable almost complex structures are complex structures on
the vector bundle $(M,E)$, which turn it into a holomorphic line bundle over a Riemann surface, i.e.,
we have
\begin{equation}
\cc(\cm)(\realR)/N_\cm\cong\cc^{g,d},
\end{equation}
where $\cc^{g,d}$ is the set of all pairs of Riemann surfaces and holomorphic line bundles defined in the
previous section. But then Prop.~\ref{cor:modpair} implies that the remaining action of 
$\Aut_{\catvbun}(M,E)_0$ on $\cc^{g,d}$ can be quotiented out and one obtains $J(V_g)$ as the resulting
underlying Teichm\"uller space.
\end{proof}

This Theorem implies that once one has fixed a marking on the underlying Riemann surface $M$ of
a complex $1|1$-dimensional supermanifold $\cm$, there exists a universal family for the the purely
even deformations of the complex structure of $\cm$. This means that among all families of complex
$1|1$-dimensional supermanifolds over purely even bases, the one over $J(V_g)$ is universal.
On the other hand, it is clear from
Thm.~\ref{vectaut} that it will be impossible to construct a supermanifold which is the base of a universal
family of marked complex $1|1$-dimensional supermanifolds. The supermanifold $\ct_{\fvect^L(1|1)}^{g,d}$
that we will construct below is only the base of a semiuniversal family.

Let $\pi:J(V_g)\to\ct_g$ be the family of Jacobians over (ordinary) Teichm\"uller space. Then each 
point $p\in J(V_g)$ denotes a pair $(M,L)$ consisting of an isomorphism class of complex structures 
and an isomorphism class of line bundles which are holomorphic with respect to this complex structure.
Let the degree of $L$ be $d$ with $d> 2g-2$. Then we can construct a holomorphic vector bundle
$p:H\to J(V_g)$ by assigning to each $p\in J(V_g)$ the space of holomorphic sections of
the bundle $L$. 

We employ this construction to define a specific bundle $p:\ch\to J(V_g)$ whose fiber at a point $(M,L)$ is
the space $H^0(M,K^2\otimes L^{-1})\oplus H^0(M,K\otimes L)$ where the degree of $L$ lies in the range
$0\leq d\leq 2g-2$. This is a holomorphic vector bundle of rank $4g-4$ over $J(V_g)$.

\begin{conj}
\label{citthm}
Let $J$ be an integrable almost complex structure on a smooth closed oriented supersurface $\cm$ of
dimension $2|2$. Then $\cm$ is equivalent to a Riemann surface $M$ and a holomorphic line bundle $L\to M$
by Prop.~\ref{isosrs}. Assume that the genus of $M$ is $g\geq 2$ and that $0<\deg(L)<2g-2$. Then
the supermanifold $\ct_{\fvect^L(1|1)}^{g,d}:=(J(V_g),\wedge^\bullet\ch)$ parametrizes a semiuniversal
family of complex $1|1$-dimensional supermanifolds.
\end{conj}

The genus $g$ and the degree $d$ remain, of course, constant in this family: they are supersmooth invariants
of $\cm$. It is clear that the condition that the Kodaira-Spencer map be an isomorphism at every point is
satisfied by construction for the family parametrized by $\ct_{\fvect^L(1|1)}^{g,d}$. What has 
to be shown is that every
other family $p:\cv\to\cn$ of compact closed complex $1|1$-dimensional supermanifolds with genus $g$ and
degree $d$ can be obtained as a pullback from the one over $\ct_{\fvect^L(1|1)}^{g,d}$. This would
require a much deeper analysis of the deformation theory of complex superspaces. Although
Vaintrob developed this theory in \cite{V:Deformations}, there still is quite some work to do to prove
this claim. We will leave it for further investigations.

\section{\texorpdfstring{The Teichm\"uller space $\ct_{\fkl(1|1)}^{g}$}
{The Teichm\"uller space of N=1 super Riemann surfaces}}

It was shown in Prop.~\ref{isofkl} that the set of $\fkl(1|1)$-structures on a given smooth supersurface
of dimension $2|2$ can be characterised as the set of pairs $(M,S)$, where $M$ is the underlying
smooth surface equipped with a complex
structure and $S$ is a spin bundle. Spin bundles have degree $g-1$, so they lie
in the range
(\ref{gdeg}) of allowed degrees and there are no odd infinitesimal automorphisms. Inserting a spin bundle
$S$ as the line bundle in our considerations above yields 
\begin{equation}
W=H^0(M,K^2)\oplus H^0(M,K)\oplus\Pi\left(H^0(M,K^{3/2})\oplus H^0(M,K^{3/2})\right),
\end{equation}
as the space of nontrivial deformations.
Here, $K^{3/2}$ is a shorthand for $K\otimes S$, so this bundle also depends on the spin bundle $S$.
This is a complex super vector space of dimension $4g-3|4g-4$.
But the space $W$ does not describe the true deformations of a $\fkl(1|1)$-structure, it only describes
the deformations of a $\fvect^L(1|1)$-structure whose line bundle $L$ happens
to be a spin bundle. These deformations would allow a deformation of the isomorphism class $[L]$, but we have
to assure that it always remains a spin bundle, so the class really has to stay fixed. It will turn out 
that we have to restrict the space $W$
to a certain subspace. In order to do that, we first reinterpret the tangent space $V$ to the slice
as a cohomology group, which is the form usually used in deformation theory.

\begin{lemma}
\label{lem:h1ctm}
Let $\cm$ be a complex supersurface of dimension $1|1$, and let $(M,L)$ be the Riemann surface and
line bundle to which $\cm$ is equivalent.
Then the space $V$ constructed in (\ref{tottang}) is $H^1(M,\ctm)$, the first \v Cech 
cohomology group
of the sheaf $\ctm$ considered as a $\intZ_2$-graded module over the sheaf $\co_M$ of holomorphic
functions on the base manifold $M$.
\end{lemma}
\begin{proof}
Since a complex supersurface $\cm$ of dimension $1|1$ can always be interpreted as a Riemann surface $M$
with a parity-reversed line bundle $L$, the tangent sheaf $\ctm$ can indeed always be interpreted as a
locally free module over $\co_M$ of rank $2|2$. A local basis is given by
\begin{equation}
\pderiv{}{z},\quad\theta\pderiv{}{\theta},\quad\theta\pderiv{}{z},\quad\pderiv{}{\theta}.
\end{equation}
The first two of these are even, the last two are odd. Interpreting $\theta$ as a local section of $L$
and $\pderiv{}{\theta}$ as a local section of $L^{-1}$ yields the result.
\end{proof}

The interpretation of this result is as follows. Associating to every open set of $M$ the space
of sections of $\ctm$, one obtains a presheaf with values in the Lie 
algebra $\fvect^L(1|1)$. This is the algebra of infinitesimal local automorphisms of the supercomplex
structure. Its first cohomology describes precisely those variations of the supercomplex structure which
cannot be achieved by a global coordinate change, i.e., by a diffeomorphism.

In the case of a $\fkl(1|1)$-structure, the algebra of local infinitesimal automorphisms is, of course,
$\fkl(1|1)\subset\fvect^L(1|1)$. We must therefore determine the first \v Cech cohomology of the sheaf
of $\fkl(1|1)$-vector fields as the tangent space to $\ct_{\fkl(1|1)}^{g}$. Although it is a slight
abuse of notation, we will denote the sheaves whose sections are elements of a Lie
algebra $\mathfrak{g}$ also as $\mathfrak{g}$. For example, $\fkl(1|1)$ also denotes the subsheaf of
$\ctm$ whose local sections form the algebra $\fkl(1|1)$.

\begin{lemma}
\label{lem:split}
Let $\cm$ be a complex supersurface with a $\fkl(1|1)$-structure, i.e., $\cm$ is complex $1|1$-dimensional
and it is endowed with a maximally nonintegrable distribution $\cd$ of rank $0|1$ (cf.~Definition \ref{ddef}).
Then there exists a direct decomposition of $\co_\cm$-modules
\begin{equation}
\ctm=\fvect^L(1|1)=\fkl(1|1)\oplus\cd.
\end{equation}
\end{lemma}
\begin{proof}
The statement is local, so we can work in local complex coordinates $z,\theta$ and can assume that
$\cd$ is generated by $D=\pderiv{}{\theta}+\theta\pderiv{}{z}$. From Prop.~\ref{p1} we know that
every element of $\fkl(1|1)$ can be obtained by taking an arbitrary function $f\in\co_\cm$ and
calculating its associated contact vector field
\begin{equation}
K_f=(2-\theta\pderiv{}{\theta})(f)\pderiv{}{z}+(-1)^{p(f)}\pderiv{f}{\theta}\pderiv{}{\theta}+%
\pderiv{f}{z}\theta\pderiv{}{\theta}.
\end{equation}
Setting $f=f_0+f_1\theta$, where $f_0,f_1$ are holomorphic functions of $z$, we see that
\[
Df=(-1)^{p(f)}(\pderiv{f_0}{z}\theta-f_1)
\]
and thus
\begin{equation}
K_f=2f\pderiv{}{z}+(-1)^{p(f)}(Df)D
\end{equation}
Therefore, $K_f$ lies in $\cd$ if and only if $f\equiv 0$. Since $D$ and $\pderiv{}{z}=\frac{1}{2}D^2$
generate $\ctm$ over $\co_\cm$, this proves the claim.
\end{proof}

This direct splitting allows us to calculate the \v Cech cohomology of the sheaf $\fkl(1|1)$ from that
of $\cd$, which is easier than a direct attempt.

\begin{prop}
Let $\cm$ be a $\fkl(1|1)$-surface defined by a  Riemann surface $M$ with spin bundle $S$. Then the 
first \v Cech cohomology
group of the subsheaf of $\fkl(1|1)$-vector fields considered as a $\intZ_2$-graded $\co_M$-module 
is given by
\begin{equation}
H^1(M,\fkl(1|1))=H^0(M,K^2)\oplus\Pi(H^0(M,K^{3/2})).
\end{equation}
\end{prop}
\begin{proof}
From Lemma \ref{lem:split} we deduce that the cohomology of $\ctm$ as a $\co_M$-module of rank $2|2$ can
be split as
\begin{equation}
H^q(M,\ctm)=H^q(M,\cd)\oplus H^q(M,\fkl(1|1))\qquad q=0,1,\ldots.
\end{equation}
Locally, all sections of $\cd$ can be written as
\begin{equation}
(f_0+f_1\theta)(\pderiv{}{\theta}+\theta\pderiv{}{z})=f_1\theta\pderiv{}{\theta}+f_0D.
\end{equation}
The first summand contributes $H^1(M,\co_M)\cong H^0(M,K)$ as the even subspace of $H^1(M,\cd)$. The
second summand contributes $H^1(M,S^{-1})=H^1(M,K\otimes S)\cong H^0(M,K^{3/2})$ as the odd subspace.
Thus, $H^1(M,\fkl(1|1))$ must be isomorphic to $H^0(M,K^2)\oplus\Pi(H^0(M,K^2\otimes S^{-1}))$.
\end{proof}

Therefore, the Teichm\"uller space of $\fkl(1|1)$-structures ($N=1$ super Riemann surfaces in physics
terminology) has dimension $3g-3|2g-2$, as was already found by many authors, e.g., 
\cite{F:Notes}, \cite{Docsos}, \cite{CR:Super}, \cite{LBR:Moduli}. 
The missing $g|2g-2$ dimensions as compared to the deformations of the
supercomplex ($=\fvect^L(1|1)$-)structure alone also have a nice interpretation. The $g$ even
dimensions are those of the Jacobian variety, i.e., the vertical dimensions of $J(V_g)$. We cannot deform
the structure in these directions, since the divisor class of a spin bundle is rigid (see 
Prop.~\ref{spstr}). This fact also accounts for the loss of the $2g-2$ odd directions. As one could see in the
proofs of Prop.~\ref{prop:tangsli1} and Thm.~\ref{splittang}, the subspace $H^0(M,K\otimes L)$ that
we lose as valid deformations originates in the component $\gamma_0$ of the tangent vector to the complex
structure $J$. The entry $\gamma$ describes the contribution of $\pderiv{}{\bar{z}}$ to the
infinitesimally deformed vector field $\pderiv{}{\theta'}$, i.e., to the deformation of the line bundle $L$.
Since the isomorphism class of $L$ must remain constant, no such deformations are possible anymore.

The remaining even deformations describe the ordinary variation of the complex structure of the underlying
Riemann surface $M$, i.e., ordinary Teichm\"uller space. The odd ones are the true ``super''-deformations,
which can only occur over a base which contains odd dimensions. They, too, only deform the complex
structure of $M$, but in a way which would be impossible in classical complex geometry, namely by mixing
in even combinations of elements of $\Pi(H^0(M,K^{3/2}))$ and the odd variables of the base of the
deformation.

Denote by $\Xi$ the set of $2^{2g}$ disjoint holomorphic sections of $J(V_g)\to\ct_g$ corresponding
to the divisors of spin bundles defined in Prop.~\ref{spstr}. Each of these sections is biholomorphic
to ordinary Teichm\"uller space $\ct_g$, so $\Xi$ can be viewed as a trivial $2^{2g}$-sheeted covering
of $\ct_g$. Then we can restrict the bundle $p:\ch\to J(V_g)$ with fiber
$H^0(M,K^2\otimes L^{-1})\oplus H^0(M,K\otimes L)$ above to the subbundle
\begin{eqnarray}
p':\ch' &\to& J(V_g)\\
H^0(M,K^2\otimes L^{-1}) &\to& J(V_g).
\end{eqnarray}
Then the above discussions make it clear that $\ct_{\fkl(1|1)}^g$ is a subsupermanifold of 
$\ct_{\fvect^L(1|1)}^{g,g-1}$ which contains all isomorphism classes of spin curves and only those odd 
deformations that preserve the spin bundle.

\begin{thm}
The supermanifold $\ct=(\Xi,\wedge\ch'\big|_\Xi)$ parametrizes a semiuniversal family of 
$\fkl(1|1)$-structures on a smooth closed compact supermanifold of dimension $2|2$ whose underlying
manifold is a surface of genus $g$.
\end{thm}
\begin{proof}
By construction, $\ct$ is the closed subsupermanifold of $\ct_{\fvect^L(1|1)}^{g,g-1}$ whose underlying
manifold parametrizes all isomorphism classes of marked Riemann surfaces of genus $g$ together 
with a spin bundle, and whose odd directions parametrize all inequivalent odd deformations of such
supersurfaces. The semiuniversality of the family parametrized by $\ct$ was proven by
\cite{V:Deformations} and \cite{LBR:Moduli}.
\end{proof}

In \cite{LBR:Moduli}, the residual $\intZ_2$-symmetry (cf.~Thm.~\ref{srsaut}) is also quotiented out and a
``superorbifold'' is obtained.